\documentclass[11pt]{article}

\usepackage[dvipsnames]{xcolor}

\usepackage{amssymb, amsmath,amsthm,amsfonts,latexsym, dsfont, color, xcolor} 

\usepackage[title,titletoc]{appendix}
\usepackage{array}

 \usepackage{graphics}
 \usepackage{MnSymbol}
\usepackage{enumitem}
\usepackage[normalem]{ulem}
\usepackage{mathrsfs}
\usepackage[utf8]{inputenc}
\usepackage[T1]{fontenc}
\usepackage{setspace} 
\usepackage[american]{babel} 
\usepackage{thmtools, thm-restate}
\usepackage{mathtools}

\newcommand{\hatStwofour}{\hat{\mathbb{S}}^{2,4} }

\newcommand{\M}{\mathcal{M}} 
\newcommand{\Mh}{\mathcal{M}_h}

\newcommand{\Fix}{\mathsf{Fix}} 

\newcommand{\proj}{\mathsf{proj}} 
\newcommand{\spann}{\mathsf{span}}

\newcommand{\V}{\mathcal{E}}

\newcommand{\Gtwo}{\mathsf{G}_2}

\newcommand{\Gtwosplit}{\mathsf{G}_2'}

\newcommand{\Ein}{\mathsf{Ein}} 
\newcommand{\tu}{\tilde{u}}

\newcommand{\imoct}{\mathsf{Im}\, \Oct'}
\newcommand{\bioct}{ (\Oct')^\C }

\newcommand{\Gr}{\mathsf{Gr}} 
\newcommand{\Stab}{\mathsf{Stab}} 

\newcommand{\R}{\mathbb R}
\newcommand{\RP}{\mathbb R \mathbb P}
\newcommand{\CP}{\mathbb C \mathbb P}
\newcommand{\C}{\mathbb C}
\newcommand{\Z}{\mathbb Z}

\newcommand{\Ha}{\mathbb{H}} 
\newcommand{\F}{\mathbb{F}}
\newcommand{\g}{\mathfrak{g}}
\newcommand{\frakk}{\mathfrak{k}} 
\newcommand{\frakp}{\mathfrak{p}} 
\newcommand{\h}{\mathfrak{h}}

\newcommand{\K}{\mathcal{K}}
\newcommand{\End}{\mathsf{End}}

\newcommand{\Stwofour}{\mathbb{S}^{2,4}} 
\newcommand{\quadric}{\hat{ \mathbb{S}}^{2,4}}
\newcommand{\SL}{\mathfrak{sl}}

\newcommand{\GL}{\mathsf{GL}}
\newcommand{\gl}{\mathfrak{gl}}

\newcommand{\tr}{\text{tr}}
\newcommand{\id}{\text{id}}
 
  \newcommand{\Aut}{\mathsf{Aut}}
 
 \newcommand{\Hom}{\mathsf{Hom}}
  
 \newcommand{\Fr}{\mathsf{Fr}\,} 
 
 \newcommand{\fraks}{\mathfrak{s}} 
 
 \newcommand{\Ad}{\text{Ad}} 
 \newcommand{\ad}{\text{ad}} 
 \newcommand{\del}{\partial}
\newcommand{\delbar}{\overline{\partial}}
 
\newcommand{\Ann}{\mathsf{Ann}}
 
\newcommand{\bm}{\boldsymbol} 
\newcommand{\gr}{\mathsf{gr}}

\newcommand{\der}[1]{\frac{\partial}{\partial #1}} 
\newcommand{\deriv}[2]{\frac{\partial #1}{\partial #2}} 

\newcommand{\sgn}{\text{sgn}}
\newcommand{\zbar}{\bar{z}} 

\newcommand{\selfarrow}{\lcirclearrowleft}

\newcommand{\im}{\text{im}}

\newcommand{\Diff}{\mathsf{Diff}} 
\newcommand{\Isom}{\mathsf{Isom}} 
\newcommand{\Eintwothree}{\mathsf{Ein}^{2,3}} 
\newcommand{\Oct}{\mathbb{O}} 
\newcommand{\diag}{\mathsf{diag}}

\theoremstyle{plain}

\newtheorem{theorem}{Theorem}[section]
\newtheorem{proposition}[theorem]{Proposition}
\newtheorem{lemma}[theorem]{Lemma}
\newtheorem{corollary}[theorem]{Corollary}

\newtheorem{definition}[theorem]{Definition}
\newtheorem{conjecture}[theorem]{Conjecture} 
\newtheorem{remark}[theorem]{Remark}
\newcommand{\alignL}{\begin{flushleft}}
\newcommand{\alignLend}{\end{flushleft}}

\usepackage{soul}

\title{Polynomial Almost-Complex Curves in $\quadric$} 
\author{Parker Evans\footnote{The author acknowledge(s) support from U.S. National Science Foundation grants DMS 2005551, 1107452, 1107263, 1107367 "RNMS: Geometric Structures and Representation Varieties" (the GEAR Network) as well as NSF DMS-1745670.} 
\footnote{This work is based off of part of the author's PhD thesis \cite{Eva24}.} }

\date{April 2024}

\usepackage[margin=.75in]{geometry}
\numberwithin{equation}{section}

\newcommand{\U}{\mathcal{U}}

\usepackage[nottoc,notlot,notlof]{tocbibind}
\usepackage{hyperref} 
\hypersetup{colorlinks=true,linktocpage, citecolor=ForestGreen, linkcolor=blue,urlcolor=blue}
\usepackage[
backend=bibtex,
date = year,
maxnames = 4,
style=alphabetic,
]{biblatex}
 \renewbibmacro{in:}{} 
 
\AtEndDocument{  \par
  \medskip
 \begin{tabular}{@{}l@{}} 
    \textsc{Department of Mathematics, Rice University, Houston, TX 77005}\\
    \textit{E-mail address}: \texttt{pge1@rice.edu}
	 \end{tabular} }

\begin{document}
\maketitle

\begin{abstract}

For solutions to the $\g_2$ affine Toda field equations in $\C$ with respect to \emph{polynomial} holomorphic sextic differential $q$, we study the associated almost-complex curves $\nu_q: \C \rightarrow \quadric$. The asymptotic boundary $\Delta := \partial_{\infty}(\nu_q)$ of $\nu_q$ 
is found to be a polygon in $\Eintwothree$ with $\deg q + 6$ vertices. The polygon $\Delta$ satisfies an \emph{annihilator property}, which is related to a $\Gtwosplit$-invariant discrete metric $d_3: \Eintwothree \times \Eintwothree \rightarrow \{0,1,2,3\}$ on $\Eintwothree$. In fact, we show $\Gtwosplit = \Isom(d_3) \cap \Diff(\Eintwothree)$. 
The asymptotic boundary defines a map $\alpha: \mathsf{MS}_{k} \rightarrow \mathsf{MP}_{k+6}$ between the equidimensional moduli spaces of holomorphic polynomial sextic differentials of degree $k$ and of annihilator polygons with $k+6$ vertices and is conjectured to be a homeomorphism onto its image.
We also discuss the relationship between $\nu_q$ and a related minimal surface $f_q: \C \rightarrow \Gtwosplit/K$ in the symmetric space $\Gtwosplit/K$, showing how to realize their mutual harmonic lift to $\Gtwosplit/T$ geometrically. Before beginning the geometry, we prove the existence and uniqueness of a complete (real) solution to the $\g_2$ affine Toda field equations in $\C$ associated to 
polynomial $q \in H^0(\K_\C^6)$.

 \end{abstract}  

\newpage

\tableofcontents

\section{Introduction} 

We first provide a brief overview of the main objects and results. 
Almost complex curves $\phi: \Sigma \rightarrow (\mathbb{S}^6, J_{\mathbb{S}^6})$ from a Riemann surface $\Sigma$ are now well-studied, though the same is not true for almost-complex curves in $(\quadric, J_{\quadric})$. Let $\Gtwosplit$ denote the split real form of the complex exceptional Lie group $\Gtwo$.
In 2010, Baraglia constructed $\rho$-equivariant almost-complex curves $\nu: \widetilde{\Sigma} \rightarrow \quadric$ associated to representations $\rho: \pi_1 S \rightarrow \Gtwosplit$ in the $\Gtwosplit$-Hitchin component. However, Baraglia does not examine the geometry of such curves, and this geometry has remained unexplored until the recent works \cite{CT23, Nie24}. This chapter studies the asymptotic geometry of almost-complex curves $\nu_q: \C \rightarrow \quadric$, now in the planar setting.

The curves $\nu_q$ are associated to a holomorphic sextic differential $q \in H^0(\K^6_\C)$ by solving a $2 \times 2$ coupled system of PDE called the $\g_2$ affine Toda (field) equations (with appropriate reality conditions). First, we prove the existence and uniqueness of complete solutions to these equations when $q$ is a polynomial (cf. Theorem A). 

We next show that the asymptotic boundary $\Delta = \partial_{\infty}( \, \nu_{q} )$ of a polynomial almost-complex curve $\nu_q: \C \rightarrow \quadric$ is a polygon in the space $\Eintwothree$, the frontier of $\Stwofour$ in $\mathbb{P}( \R^{3,4})$. The polygon $\Delta$ satisfies a condition we call the \emph{annihilator property} (cf. Theorem B, Theorem B'). We then give descriptions of $\Gtwosplit$ as the subgroup of the diffeomorphism group $\Diff(\Eintwothree)$ preserving various forms of annihilators
(cf. Theorem C). Moreover, Theorem B allows us to define a map $\alpha: \mathsf{TS}_{k} \rightarrow \mathsf{TP}_{k+6}$ between a 
vector space of holomorphic polynomial sextic differentials of degree $k$ in $\C$ 
and a moduli of space of marked annihilator polygons with $k+6$ vertices.
The space $\mathsf{TS}_{k} \cong_{\mathbf{Top}} \R^{2(k-1)}$ is a cell of real dimension $2(k-1)$. On the other hand, the generic locus in $\mathsf{TP}_{k+6}$ consists of a disjoint union of cells of dimension $2(k-1)$. 
Hence, this dimension count invites the question of whether the map $\alpha$ is injective and 
moreover whether $\alpha$ furnishes a homeomorphism with a component of $\mathsf{TP}_{k+6}$.  

Along the way, we discuss the geometric relationship between the minimal surface $f_q: \C \rightarrow \Gtwosplit/K$ in the symmetric space associated to $q$ and the almost-complex curve $\nu_q$. Each map carries equivalent data, as they each lift to a harmonic map in $\Gtwosplit/T$,
where $T < K$ is a maximal torus. 

\subsection{Background} 

Harmonic maps from a Riemann surface $\Sigma$ into $\mathbb{S}^n, \RP^n, \CP^n$ are now classical \cite{BW92, EW83}. Among such target spaces, $\mathbb{S}^6$ is 
exceptional since it carries a \emph{canonical} almost-complex structure $J_{\mathbb{S}^6}$.
Any almost-complex curve $\phi: \Sigma \rightarrow (\mathbb{S}^6, J_{\mathbb{S}^6})$ is weakly conformal, harmonic, and still harmonic when regarded as a map into $\RP^6$ or $\CP^6$. Some rich geometry has been discovered by studying these almost-complex curves \cite{BPW95, BVW94, Bry82, Fer14, Gra69, MP18, ZJ12}.\\

Let us now relate $(\mathbb{S}^6, J_{\mathbb{S}^6})$ to the present setting. The almost-complex structure $J_{\mathbb{S}^6}$ is invariant under the compact real form $\Gtwo^c$ of $\Gtwo$. We study instead the almost-complex manifold $(\quadric, J_{\quadric})$, now invariant under the \emph{split} real form $\Gtwosplit$. These Lie groups can be realized as $\Gtwosplit := \Aut_{\R-\mathbf{alg}}(\Oct')$ and $ \Gtwo^c := \Aut_{\R-\mathbf{alg}}(\Oct)$, where $\Oct$ and $ \Oct'$ are the \emph{octonions} and \emph{split octonions}, respectively. Both $\Oct$ and $\Oct'$ are real algebras (with unit) over 8-dimensional vector spaces and are endowed with various algebraic structures $q, \odot, \times$, all of which are invariant under the relevant 
real form of $\Gtwo$; see Section \ref{Gtwosplit} for details.
These algebraic structures on $\Oct$ and $\Oct'$ lead to the almost-complex structures $J_{\quadric}$ and $J_{\mathbb{S}^6}$. For brevity, we describe $\Oct$ first
and note the relevant differences for $\Oct'$. The octonions are equipped with a Euclidean inner product $q_{\Oct}$. Taking the \emph{imaginary} octonions by orthogonally splitting $\Oct = \R \{1_{\Oct}\} \oplus \mathsf{Im} \, \Oct$, the space $\mathsf{Im} \, \Oct$ comes equipped with a \emph{cross-product} $\times_{\Oct}: \mathsf{Im} \, \Oct \times \mathsf{Im} \, \Oct \rightarrow \mathsf{Im} \, \Oct$ in the sense that $\times_{\Oct}$ is skew-symmetric, orthogonal (with respect to $q_{\Oct}$), and normalized by $q_{\Oct} (u \times_{\Oct} v) = q_{\Oct}(u)q_{\Oct}(v) - q_{\Oct}(u, v)^2$. Denote $Q_+(V, q) := \{ x \in V \, | \, q(x) = +1\}$. Since $q_{\Oct}$ is Euclidean, identify $\mathbb{S}^6 = Q_+(\mathsf{Im} \, \Oct)$.  Then $J_{\mathbb{S}^6}$ is defined at $x \in \mathbb{S}^6$ by identifying $T_x \mathbb{S}^6 \cong [x^\bot \subset \mathsf{Im} \, \Oct ]$ and setting $J_x(u) = x \times u$.\footnote{In fact, the complex structure on $\mathbb{S}^2$ can be realized this way. Splitting $\Ha = \R \oplus \mathsf{Im} \Ha$, one finds $(\R^3, \times) \cong (\mathsf{Im} \Ha,\times_{\Ha}) $. Then $J_{\mathbb{S}^2}$ defined by $J_{x}(u) = u \times_{\Ha} x$ gives a (trivially integrable) almost-complex structure on $\mathbb{S}^2 =Q_+(\mathsf{Im}\Ha)$.} On the other hand, $q_{\Oct'}$ is of split signature (4,4).  The pseudosphere $\quadric := Q_+(\imoct)$ also obtains a non-integrable almost-complex structure via $J_{x}(u) : = x \times_{\Oct'} u$, now using the cross-product on $\imoct$. We are naturally led to ask about almost-complex curves $\nu: \Sigma \rightarrow \quadric$ from a Riemann surface $\Sigma$. \\

The same argument as in the $\mathbb{S}^6$-case tells us almost-complex curves $\nu: \Sigma \rightarrow \quadric$ are weakly conformal, harmonic, and also harmonic as maps to $\RP^6$ and $\CP^6$.  A key difference from the $\mathbb{S}^6$-case is that the pseudo-Riemannian metric $q_{\Oct'}|_{\quadric}$ is of signature $(2,4)$. Hence, the signature of the induced metric on $S = \im\, \nu$ is not a priori prescribed. We study \emph{timelike} almost-complex curves (negative-definite induced metric) with \emph{spacelike} second fundamental form.\footnote{Alternatively, dualize with $\overline{\imoct}:= (\imoct, -q_{\imoct})$. Then the aforementioned almost-complex curves $\nu$ in $\quadric$ correspond to 
 spacelike surfaces with timelike second fundamental form in pseudo-hyperbolic space $\hat{\Ha}^{4,2} := Q_-(\overline{\imoct})$. This perspective is taken in \cite{CT23}.} Such almost-complex curves were first discussed by Baraglia in his PhD thesis \cite{Bar10}. However, Baraglia studied $\rho$-equivariant almost-complex curves $\nu_q: \widetilde{\Sigma} \rightarrow \quadric$ associated to holomorphic sextic differentials $q \in H^0(K_{\Sigma}^6)$ by solving a $2 \times 2$ coupled elliptic system of PDE called the $\g_2$ affine Toda (field) equations, where $\rho: \pi_1 \Sigma \rightarrow \Gtwosplit$. These almost-complex curves are associated to a holomorphic sextic differential $q \in H^0(K_{\Sigma}^6)$ by solving a $2 \times 2$ coupled elliptic system of PDE called the $\g_2$ affine Toda (field) equations. 
 Instead, we focus on almost-complex curves $\nu_q: \C \rightarrow \quadric$
associated to these equations for polynomial $q \in H^0(\K^6_\C)$, in the planar (i.e., non-equivariant) setting. We study the asymptotic geometry of such curves. 

\subsection{Main Results } 

We now describe the main results. As noted, we study planar almost-complex curves $\nu_{q}: \C \rightarrow \quadric$ associated to holomorphic sextic differential $q \in H^0(\K_\C^6)$ that are \emph{polynomial}, i.e., $q = \tilde{q} \, dz^6$  for $\tilde{q}$ a holomorphic polynomial. 
To furnish such curves, we must solve the $\g_2$ affine Toda equations \eqref{HitEuc} associated to $q$ on $\C$. To this end, we prove: \\

\textbf{Theorem A}: Let $q \in H^0(\K_\C^6)$ be a holomorphic polynomial sextic differential. Then there exists a unique complete, real solution to the $\g_2$ affine toda equations for $q$
on the complex plane. \\

The definition of complete can be found in Definition \ref{CompletenessDefinition}. 
Given polynomial $q \in H^0(\K_\C^6)$, the unique aforementioned complete solution $\bm{u} = (u_1, u_2)$ to the affine Toda equations provides
a unique almost-complex curve $\nu = \nu_q: \C \rightarrow \quadric$, up to global $\Gtwosplit$-isometry. We show that one recovers $q$ from $ \nu_q$, up to an explicit constant $C$, by $ q = C (\nu_z \times \nu_{zz}) \cdot \nu_{zzz}$. 
We embed the surface in projective space via $\overline{\nu}: \C \rightarrow \Stwofour \hookrightarrow \mathbb{P}\imoct \cong \mathbb{P}(\R^{3,4}) $ and show that the \emph{asymptotic boundary} of $\Sigma = \im \overline{\nu}$ is a \emph{polygon} $\Delta$ with $ \deg q+ 6$ vertices. Here, we mean polygon in the following sense: a piecewise linear curve formed by concatenating (projective) line segments 
between adjacent vertices. Furthermore, we characterize the $\Gtwosplit$-properties of these polygons. 
The polygon $\Delta$ is contained in $\Eintwothree := \{ \, [x] \in \mathbb{P} \imoct \; | \; q(x) = 0\}$, the space of null lines in $\imoct$.
Defining $\Ann(u) := \{l \in \mathbb{P}(\imoct) \; | \; l \times_{\Oct'} u = 0 \}$ as in \cite{BH14}, we find $\Delta$ satisfies the \emph{annihilator property}: denote the cyclic ordering of the vertex set of $\Delta$ as $(p_i)_{i \in \mathbb{Z}_n}$ and 
 $\Ann(p_i) = \mathbb{P} ( \spann_{\R} \langle p_{i-1}, p_{i}, p_{i+1} \rangle )$.
 That is, linear combinations of $p_i$ and its neighbors are exactly all the lines in $\mathbb{P} (\imoct)$ which annihilate the given line $p_i$. To summarize, our second main result is as follows: \\

\textbf{Theorem B}: Let $q \in H^0(\K_\C^6)$ be polynomial, $\nu_q: \C \rightarrow \Stwofour$ its associated almost-complex curve.
Then the boundary at infinity $\Delta = \partial_{\infty} \nu_q \subset \Eintwothree$ is an annihilator polygon with $\deg q+6$ vertices. \\

We now give two reformulations of Theorem B after discussing some structures on $\Eintwothree$. 
First, as a corollary to the proof \cite[Theorem 16]{BH14} that the diagonal action of $\Gtwosplit$ on $\Eintwothree \times \Eintwothree$ has exactly four  orbits,
we show there is a $\Gtwosplit$-invariant discrete metric $d_3: \Eintwothree \times \Eintwothree \rightarrow \{0,1,2,3\}$ defined by labeling the orbit the pair comes from. The space $\Eintwothree$ also carries a canonical (2,3,5) distribution $\mathscr{D}$, invariant under $\Gtwosplit$. That is, $\mathscr{D}$ is maximally non-integrable in the sense that 
$\mathscr{D}^1: = \mathscr{D}+ [\mathscr{D}, \mathscr{D}]$ is 3-dimensional and 
$\mathscr{D}^2 := \mathscr{D}^1 + [\mathscr{D}^1, \mathscr{D}^1] $ is 5-dimensional. In fact, $\g_2'$ was first described by Cartan as the space of infinitesimal
symmetries of $(\Eintwothree, \mathscr{D})$ \cite{Agr08}. We refer the reader to \cite{Agr08} for further history of $\Gtwo$,
however, we remark that $\mathscr{D}$ was only later connected to annihilators as we have defined here 
(cf. \cite{Sag06}, \cite[Appendix C]{BM09}). Both $\mathscr{D}, d_3$ allow us to give reinterpretations of the annihilator condition.\\

\textbf{Theorem B'}: Let $(p_i)_{i=1}^n = \Delta \subset \Eintwothree$ be the boundary at infinity of a polynomial almost-complex curve. Writing the vertices cyclically as $(p_i)_{i \in \Z_n}$, we have
	\begin{itemize} 
		\item $d_3( p_i, p_{i+j } ) = j$ for $j \in \{0,1,2,3\}$.
		\item Away from the vertices, $\Delta$ is an integral curve of the (2,3,5) distribution $\mathscr{D}$ on $\Eintwothree$. 
	\end{itemize} 

Take indices $i \leq j \in \{1,\dots, n\}$. Define $d_{\mathsf{graph}}(p_i, p_j) := \min \{j-i, i+ n-j \}$ as the distance between vertices on $\Delta$ as a cyclic graph. Then Theorem B' says $d_3(p_i, p_j) = d_{\mathsf{graph}}(p_i, p_j)$ if $d_{\mathsf{graph}}(p_i, p_j) \leq 3$. We currently do not know whether the points are in general position -- that is, whether $d_{3}(p_i, p_{i+j} ) =3$ when $ d_{\mathsf{graph}}(p_i, p_j) > 3$.\\

We also describe the relationship between $\Gtwosplit$ and annihilators in various formulations, as we now describe.
Denote $\mathsf{Diff}(\Eintwothree)$ as the group of diffeomorphisms of $\Eintwothree$. 
By \cite[Theorem 12]{BH14}, $\Ann([u]) \cong_{\mathbf{Diff}} \RP^2$ is a 2-dimensional submanifold 
of $\Eintwothree$ for $[u] \in \Eintwothree$. We say $\varphi \in \Diff(\Eintwothree)$ preserves \emph{annihilator submanifolds} when
$\varphi( \Ann(x)) = \Ann(\varphi(x))$ for $x \in \Eintwothree$. We say $\varphi$ preserves \emph{annihilators infinitesimally} when
$ \varphi^* \mathscr{D} = \mathscr{D}$. We have one final notion of annihilators. Namely, given $x, y \in \Eintwothree$ such $d_3(x, y ) = 1$,
or equivalently, $x \times_{\Oct'} y=0$ and $x \neq y$, we form the projective line $\mathscr{L}_{x,y} := \mathbb{P}\,( \spann_{\R} \langle x, y \rangle \, ) \subset \Eintwothree$, since $x \cdot y = 0$. Then 
we say $\varphi$ preserves \emph{annihilator lines} when $\varphi (\mathscr{L}_{x,y} ) = \mathscr{L}_{\varphi(x), \varphi(y)} $ for all $x,y \in \Eintwothree$
with $d_3(x,y) = 1$. We show that all of these
annihilator conditions are equivalent. (Here, recall $\Gtwosplit := \Aut_{\R-\mathbf{alg}}(\Oct')$ so that $\psi \in \Gtwosplit$ induces a diffeomorphism 
of $\Eintwothree$ that we denote by $\psi|_{\Eintwothree}$.) \\

\textbf{Theorem C}: Let $F: \Gtwosplit \rightarrow \Diff(\Eintwothree)$ denote the map $F(\psi) = \psi |_{\Eintwothree}$. Then
$F$ is injective and $\varphi \in \Diff(\Eintwothree)$ satisfies $\varphi \in \mathsf{image}(F)$ if and only if 
	\begin{itemize}
		\item \cite{Car10, Sag06} $\varphi$ preserves annihilators infinitesimally. 
		\item $\varphi$ preserves annihilator submanifolds. 
		\item $\varphi \in \Isom(d_3)$.  
		\item $\varphi$ preserves annihilator lines. 
	\end{itemize} 
Thus, the group $\Gtwosplit$ is fundamentally linked to annihilators.

We also discuss the relationship between the almost-complex curve $\nu: \C \rightarrow \quadric$ and the harmonic map $f: \C \rightarrow \mathcal{S}_{\Gtwosplit}$ to the $\Gtwosplit$-symmetric space $\mathcal{S}_{\Gtwosplit} = \Gtwosplit/K$, where $K<\Gtwosplit$ is a maximal compact subgroup. In particular, the commutativity of the diagram in Figure \ref{HarmonicMapDiagram} is established. We explain some of the diagram now. By work of Baraglia, associated to a (real) solution to the $\g_2$ affine Toda 
is a harmonic lift into $ \Gtwosplit/T$ of the map $f$, where $T < K$ is a maximal torus \cite{Bar15}. 
This lift is related to the notion of a cyclic surface in \cite{Lab17} and to that of a $\tau$-primitive map in \cite{BPW95}. 
While Baraglia's work is done for general split real forms of complex simple Lie groups, we attempt to explain this lift in as concrete and geometric fashion as possible in the $\Gtwosplit$-setting. To this end, we identify 
$\Gtwosplit/K \cong_{\mathbf{Diff}} \mathscr{X}:= \{ P \in \Gr_{(3,0)}(\imoct) \; | \; P \times_{\Oct'} P = P \}$ and $\Gtwosplit/T$ as a space of certain block-decompositions $\mathscr{Y}$ of $\imoct$. Using these geometric descriptions, we explain how an almost-complex curve $\nu$ admits a natural lift $G_{\nu}: \C \rightarrow \mathscr{Y}$.
We also realize a `geometric' version of the harmonic map in the symmetric space as $f: \C \rightarrow \mathscr{X}$. 
We then clarify the naturality of these constructions. The Toda frame field $e^{\Omega}$ of the almost-complex curve defines a map $\hat{\mathcal{F}_{\nu}}: \C \rightarrow \Gtwosplit$ such that the projections $F_K := \pi_K \circ \hat{\mathcal{F}_{\nu}}$ and $F_T:= \pi_T \circ \hat{\mathcal{F}_{\nu}}$ are harmonic, where $\pi_T: \Gtwosplit \rightarrow \Gtwosplit/T$ and $\pi_K: \Gtwosplit \rightarrow \Gtwosplit/K$. There is also a natural projection $\pi_{\Stwofour}: \mathscr{Y} \rightarrow \Stwofour$,
where $\Stwofour = \{ [x] \in \mathbb{P} \imoct \; | \; q(x) > 0 \}$ is the projective model for the pseudosphere. The almost-complex curve $\nu$ is recovered (projectively) from the lift $G_{\nu}$ by $\nu = \pi_{\Stwofour} \circ G_{\nu}$. We show the equivalence of $F_K, f$ and $F_T, G_{\nu}$ in a large commutative diagram in Theorem \ref{THM:GeometricCorrespondence}. 
Identifying $\mathscr{Y}$ with $\Gtwosplit/T$, the map $G_{\nu}$ is harmonic and serves as a mutual lift of $\nu$ and $f$. 

\subsection{Related Work} 
 
Theorem A is closely related to recent work of Li \& Mochizuki \cite{ML20a, ML20b} on the existence and uniqueness of complete solutions to the 
$\SL_n\C$ affine Toda equations on non-compact Riemann surfaces. Also related are the works of Biquard \& Boalch \cite{BB04} and Mochizuki \cite{Moc14} on solutions to Hitchin's equations for
cyclic wild Higgs bundles over $\CP^{1}$. However, the $\g_2$ affine Toda equations are not known to be linearly equivalent to the $\SL_n\C$ affine Toda equations for any $n$. We remark further in Section \ref{HitchinsEquationsSubsection} why there is no naive equivalence. Guest, Its, and Lin studied similar ``Toda-like'' equations for $\SL_n\C$ in the complex plane \cite{GIL15}. Conversely, for a compact Riemann surface $\Sigma$ of genus $g > 1$, Baraglia proved the existence and uniqueness of (real) solutions to the affine Toda field equations for any complex simple Lie algebra \cite[Theorem 2]{Bar15}.

Theorem B is a natural continuation of a line of work on maps $\phi:= \phi_q: \C \rightarrow X$ associated to 
polynomial holomorphic data $q \in H^0(\K_{\C}^j)$ with polygonal boundary $\Delta := \partial_{\infty}(\im \, \phi_q) $ \cite{HTTW95, DW15, Tam19, TW25}. 
This theme begins in \cite{HTTW95}, where Han, Tam, Treibergs, and Wan show that any orientation-preserving harmonic map $\phi: \C \rightarrow H^2,$ where $H^2$ is the hyperbolic plane, 
that is a smooth embedding with polynomial Hopf differential $q_2$ has frontier an ideal polygon with $k+2$ vertices on $\partial_{\infty} H^2$.
Even more closely related are the works \cite{DW15, Tam19, TW25}, in higher rank settings, where the geometric maps $\phi: \C \rightarrow X$ associated to polynomial holomorphic 
differential have asymptotic boundary a linear polygon and the map $\phi$ factors
a harmonic map $f: \C \rightarrow \mathcal{S}_{G^\R} $ into the symmetric space $\mathcal{S}_{G^\R}$ of a split real semi-simple rank two Lie group $G^\R$. For $G^\R = \mathsf{PSL}_2 \R \times \mathsf{PSL}_2 \R $, Tamburelli studied spacelike maximal surfaces in $\mathsf{AdS}_3$ 
associated to a polynomial differential $q_2 \in H^0(\K_{\C}^2)$. For $G^\R= \mathsf{SL}_3\R$, Dumas and Wolf studied hyperbolic affine spheres $\phi: \C \rightarrow \R^3$ with polynomial pick differential $q_3 \in H^0(\K_{\C}^3)$. Most recently, for $G^\R = \mathsf{Sp}_4\R$, Tamburelli and Wolf studied spacelike maximal surfaces in pseudo-hyperbolic space $\sigma: \C \rightarrow \Ha^{2,2} $ associated to a polynomial quartic differential $q_4 \in H^0(\K_{\C}^4)$. In each case, the appropriate boundary at infinity $\Delta = \partial_{\infty} S$ of the surface $S=\im \, \phi$ is a polygon with $\deg q + j$ vertices, where $q \in H^0(\K_{\C}^j)$ is a holomorphic polynomial $j$-differential. Theorem B extends the polynomial to polygon results to $\Gtwosplit$, the remaining split real simple rank two Lie group.

Much of the work in this thesis is done using Higgs bundles. 
Most critically, we adapt Baraglia's ideas for studying \emph{cyclic} $\Gtwo$-Higgs bundles over a compact Riemann surface $\Sigma$.
Baraglia uses the holomorphic vector bundle $\V_{\Sigma} = \K_{\Sigma}^3 \oplus \K_{\Sigma}^2 \oplus \K_{\Sigma}^1 \oplus \mathcal{O} \oplus \K_{\Sigma}^{-1} \oplus \K_{\Sigma}^{-2} \oplus \K_{\Sigma}^{-3} $ to form a Higgs bundle $(\V, \varphi)$. He reduces the structure group of the bundle $\V$ to $\Gtwo$ by identifying
the fibers of $\V$ with $(\imoct)^\C$ so that that local coordinate transitions for $\V$ are $\Gtwo$-transformations. See Remark \ref{HiggsBundleRemark}. 
We use Baraglia's same identifications in the bundle $\V_\C = \bigoplus_{i=3}^{-3} \K^i_{\C}$, but now the identification is done once-and-for-all in a single holomorphic coordinate $z = x+ i y$ for $\C$.  We remark that while we work in the complex plane, the use of the same bundle $\V_\C$ allows for the constructions to extend locally to the case of $\V_{\Sigma}$ when the same Higgs bundle is instead considered over a closed Riemann surface $\Sigma$. 

Since the first draft of this paper, Collier and Toulisse produced the work \cite{CT23} on the differential geometry of alternating almost-complex curves in $\quadric$.
They dualize the metric and treat the curves in $\quadric$ as maximal surfaces in pseudo-hyperbolic space $\hat{\Ha}^{4,2} = Q_-(\R^{4,3})$ under the natural anti-isometry 
with $\quadric$. One of their main results is that an alternating almost-complex curve in $\quadric$ is equivalent to a \emph{cyclic surface} in the space $\Gtwo^\C/T$, where $T< K < \Gtwosplit$ is a maximal torus in the maximal compact subgroup $K<\Gtwosplit$. Their construction of a complex Frenet frame gives the natural bijection between the two classes of objects. 
Moreover, the correspondence holds regardless of whether the almost-complex curves are equivariant or not. 

In \cite[Section 2.4]{Bar10}, Baraglia relates minimal surfaces in quadrics $Q_{\pm}(\R^{k,l})$ and $\mathsf{SO}(p,q)$ Higgs bundles, but this relationship is not extended to the case of $\Gtwo < \mathsf{SO}(3,4)$ Higgs bundles and almost-complex curves in $\quadric$. Recent work of Nie extends Baraglia's work, showing correspondences between minimal surfaces in pseudo-hyperbolic spaces, harmonic maps in the $\mathsf{SO}(p,p+1)$ symmetric space, and cyclic $\mathsf{SO}(p,p+1)$ Higgs bundles \cite{Nie24}. 
Nie also handles the case of $\Gtwosplit$ explicitly and he constructs the map $f$ in $\mathscr{X} = \Gr_{(3,0)}^\times \imoct $ from $\nu$ in the same fashion we do, but does not discuss the mutual lift of $f$ and $\nu$ to $\Gtwosplit/T$. While our work in Section \ref{AlmostComplexCurve} is subsumed by the work of Collier-Toulisse, the constructions given here are a bit more geometrically concrete and from a different perspective, so we provide them nonetheless. 

\subsection{Outline} 

We begin in Section \ref{G2HiggsBundle} by defining the Higgs bundle $(\V, \varphi)$ of interest. Section \ref{AlmostComplexCurve} describes the construction of the almost-complex curves via Higgs bundle techniques and discusses many associated maps. We note that Section \ref{GeometricCorrespondence} is logically independent from the rest of the paper, in which the geometric correspondence between the almost-complex curve $\nu$ and the minimal surface $f$ in the symmetric space is described.
In Section 3, Hitchin's equations for the Higgs bundle are shown to equivalently describe the integrability of the Toda frame $\Omega$ for $\nu$;
these equations are then related to certain geometric invariants of the \emph{harmonic map sequence} of $\nu$, analogous to work of Baraglia \cite{Bar10}. We then prove the existence of a \emph{complete} solution to Hitchin's equations in $\C$ in Section \ref{Existence} for $q \in H^0(\K^6_\C)$ polynomial and we prove the uniqueness of a complete solution in Section \ref{Uniqueness}. In the proof of existence, we produce sub and super-solutions $\bm{u}^-$ and $\bm{u}^+$, respectively, to Hitchin's equations. Denote 
$\tilde{\bm{u}} := \bm{u} - \bm{u}^-$ as the error term. We prove the precise exponential rate of decay of $\tilde{\bm{u}} $ in Section \ref{ErrorEstimates} in special coordinates called
\emph{natural coordinates} for $q$. We then compute the asymptotic parallel transport along geodesic rays in the flat $|q|^{1/3}$ metric in a natural coordinate in Section \ref{ParallelTransport}. We find \emph{stable sectors} as well as \emph{critical angles} where the limits along rays change. In Section \ref{AsymptoticBoundary}, we prove the main theorem (Theorem \ref{THM:AnnihilatorPolygon}) using the asymptotics of Section \ref{ParallelTransport} to show the asymptotic boundary $\Delta: = \partial_{\infty} \nu_q$ polynomial $\nu_q$ is an  polygon in $\Eintwothree$ with $\deg q + 6$ vertices. A (local) comparison of $\Delta$ to the boundary $\Delta_0$ of the model surface associated to $q \equiv 1$
shows that $\Delta$ has the annihilator property. We note that Sections \ref{AnnihilatorPolygons} and \ref{d3}, in which we prove Theorem C (Corollary \ref{GtwoEquivalence}) on $\Gtwosplit$ and its relationship to annihilators, is logically independent of Sections 3-7.

\subsection*{Acknowledgements} 

I would like to thank Mike Wolf and Andrea Tamburelli for many useful and stimulating conversations about various 
components of this paper. I thank J\'er\'emy Toulisse and Brian Collier for useful discussions on $\g_2$ Lie theory  
and $\Gtwosplit$-geometry. Thanks are in order to Christos Mantoulidis for pointing out some errors and typos in Section 4 of a previous draft. I also thank Franz Pedit for a hosting a trip to University of Massachusetts Amherst to discuss $\Gtwosplit$ and some results of the paper.

\section{Background Material} \label{Background} 

In this section, we provide an introduction to $\Gtwo$, focused on the split real (adjoint) form $\Gtwosplit$. 
The reader can find further exposition on $\Gtwo$ in \cite{Fon10} or \cite{HL82}. 
The key object here is the \emph{split octonions} $\Oct'$, an $\R$-algebra over an 8-dimensional vector space $V$, also endowed with a 
quadratic form $q$ of split signature $(4,4)$. 
The group $\Gtwosplit := \Aut_{\R-\textbf{alg}}(\Oct')$ is the Lie group of linear automorphisms of $V$ also preserving the algebra product $\odot$.
After introducing all the structures on $\Oct$, we discuss the notion of a \emph{complex cross-product basis} for $\imoct^\C$. We then use such a basis to
reduce the structure group of the bundle $\V = \bigoplus_{i=3}^{-3} \K_{\C}^i$ to $\Gtwo$. We place compatible real structures $\hat{\tau}_{\V}$ and $ \hat{\tau}$ on $\V$ and $\End(\V)$, respectively. 
We then write Hitchin's equations for a harmonic metric $h \in \V$ 
and obtain a flat connection $\nabla = \nabla_{\del, h} + \varphi + \varphi^{*h}$ that preserves $\hat{\tau}_{\V}, \hat{\tau}$ and the cross-product $\times$ on $\V$.
All of the algebraic structures in $\V$ are defined following Baraglia's work on cyclic $\Gtwo$-Higgs bundles in Section 3.6 of \cite{Bar10}. 

\subsection{$\Oct'$ and $\Gtwosplit$}\label{Gtwosplit}
We now provide a brisk tour through the algebraic structures on $\Oct'$. \\

The split octonions $\Oct'$ are a real algebra over the $8$-dimensional vector space $\Ha \oplus \Ha$, where $\Ha$ 
denotes the quaternions. For $u = (a,b)\in \Oct'$, the algebra multiplication is
	\begin{align}\label{OctMultiplication} 
		(a,b) \odot (c,d) := (ac + db^*, a^*d+ cb),
	\end{align} 
where $*$ is quaternionic conjugation given by $x^* =x$ for $x \in \R$ and $x^*= -x$ for $x \in \spann_{\R} \langle i,j,k \rangle$. 
While not associative in the strict sense, there is some associativity in $\Oct'$. Indeed, $\Oct'$ is 
\emph{alternative}, i.e., the subalgebra generated by any two elements is associative. 
In particular, the identities $(x^2)y = x(xy)$ and $(xy)x = x(yx)$ hold for any $x,y \in \Oct'$. 

The \emph{Moufang} identities, true in any associative algebra, are also very useful \cite{Fon10}: for any $u,v,w \in \Oct'$,
	\begin{align}\label{Moufang} 
		\begin{cases} 
			(uvu)w &= u(v(uw)) \\%\\ \label{Moufang1}
			(uv)(wu) &= u(vw)u %\eqname{Moufang2} 
		\end{cases} 
	\end{align}
The complex split-octonions $(\Oct')^\C = \Oct' \otimes_{\R} \C$\footnote{We exclusively use $i$ to denote a split-octonion and
use $\sqrt{-1}$ for the new scalar added by the complexification.} is the complexification of $\Oct'$ and carries analogous versions of all the structures $\Oct'$ carries. We use the shorthand $\Gtwo^\F$ for $\Gtwo^\R = \Gtwosplit$ and $\Gtwo = \Gtwo^\C $. For now, we use the working definition $\Gtwo^\F := \{ \psi \in \GL(\imoct)^\F \; | \; \psi(x \odot y) = \psi(x) \odot \psi(y) \; \}$, however we note useful equivalent definitions at the end of the section. The group $\Gtwosplit $ has trivial center (i.e., is the adjoint form) but has $\pi_1(\Gtwosplit) = \Z_2$, while $\Gtwo^\C$ is simply-connected and has trivial center \cite{Fon10}. \\

We now define a particularly useful basis for $\Oct'$. Consider, for abuse of notation, $i = (i,0)$, $j := (j, 0), l:= (0, 1) $ in $\Ha \oplus \Ha$. 
Defining the element $k := i\odot j$, we obtain a vector space basis 
\begin{align} \label{MultiplicationBasis}
	\mathcal{M} = (1, i, j, k, l, li, lj, lk) = (m_i)_{i=0}^7
\end{align} 
for $\Oct'$. We will call this basis 
$\mathcal{M} = (m_i)_{i=0}^7$ the \emph{standard multiplication basis} for $\Oct'$. The multiplication table for $\Oct'$ in this basis is shown in Figure 1.

We identify a distinguished copy of the quaternions in $\Oct'$ by $\Ha := (\Ha, 0)$ and then obtain 
an intrinsic direct sum decomposition $\Oct' = \Ha \oplus l\Ha$.\footnote{Care is needed here. We use $(a,b) \in \Ha \oplus \Ha$ to denote $a+ lb$. Some others authors instead prefer to use $(a,b)$ to denote $a+ bl$. With the latter identification, $\Oct'$ multiplication on $\Ha \oplus \Ha$ is given by $ (a,b) \odot (c,d) = (ac+d^*b, da+ bc^*)$.} 
\begin{figure}[ht]
\centering
\includegraphics[width = .6\textwidth]{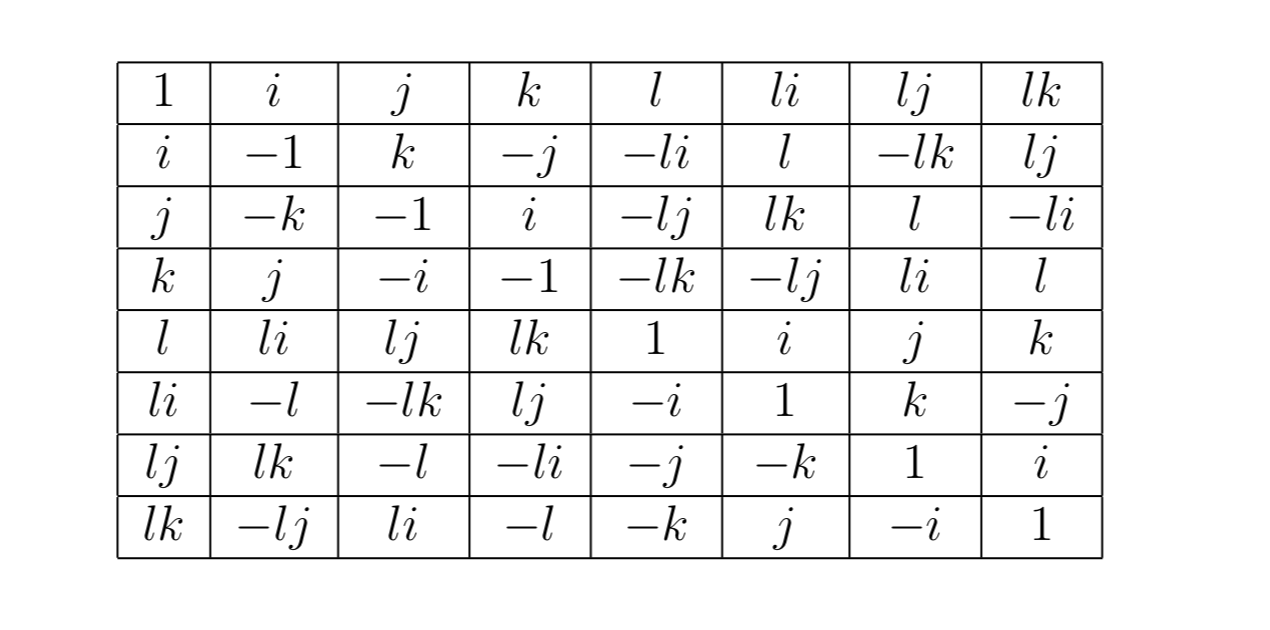}
\caption{Multiplication table for the split octonions $\Oct'$, from \cite{Bar10}.} 
\end{figure} 
Using the decomposition
$\Oct' = \Ha \oplus l\Ha$, the signature (4,4) inner product $q$ is given by 
$$ q(a + l b , c + l d ) \; = \; \langle a, c \rangle_{\Ha} - \langle b, d \rangle_{\Ha} $$ 
Write $q(u) = q(u,u)$ for the induced quadratic form. In the basis $\mathcal{M}$, 
set $ u = \sum_{i=0}^7 c_i m_i $ and we have $q(u) = c_0^2 + c_1 ^2 + c_2^2 + c_3^2 -  c_4^2 - c_5^2 - c_6^2 - c_7^2$. 
In particular, the direct sum decomposition $\Oct' = \Ha \oplus l\Ha$ is an orthogonal decomposition. \\

The split-octonions come with an $\R$-linear involution called \emph{conjugation} $*: \Oct' \rightarrow \Oct'$, which satisfies $ x^* = x $ if and only if $x \in \R$.
In particular, $\Oct'$ splits orthogonally as $ \Oct' = \R \oplus \imoct$, where $\imoct$ is the \emph{imaginary} split-octonions, the $(-1)$-eigenspace of $*$.
The quadratic form $	q$ is also given by
\begin{align}\label{QuadraticForm} 
	q(x) = xx^*. 
\end{align}

Just as in the quaternions, $ (xy)^* = y^* x^*$ for $x,y \in \Oct'$. Now by definition, any $\varphi \in \Gtwo^\F$ satisfies $\varphi(1) = 1$. Thus, it is standard to restrict the action of $\Gtwo^\F$ to $\imoct^\F$, and the representation $\Gtwo^\F \rightarrow \GL(\imoct^\F)$ is now irreducible. Moreover, $\g_2$ admits no irreducible faithful representation on a vector space of dimension less
than 7 \cite{Fon10}. \\

The space $\imoct$ is also endowed with a skew-symmetric map $\times: \imoct \times \imoct \rightarrow \imoct$ defined by 
\begin{align}
	x \times y : = x \odot y \; - \;q( x, y ). \label{CrossProductDefinition}
\end{align}
We call $\times$ a \emph{cross-product} due to it satisfying the familiar properties of $(\R^3, \times_{\R^3})$. Namely, $\times$ is
(i) bilinear, (ii) skew-symmetric, (iii) orthogonal (i.e. $x \times y \; \bot \; x,y$), and (iv) satisfies the normalization 
\begin{align}\label{CrossProductNormalization} 
	 q(x \times y) = q(x) q(y) - q(x,y)^2.
\end{align}
We will use often that \eqref{CrossProductDefinition} implies that $uv = u \times v$ when $u, v \in \imoct$ are orthogonal. \\

 \textbf{Structures on $\bioct$.}  

We now define all the same structures on $\bioct$. The algebra $\bioct$ is equipped with the $\C$-bilinear extension $q_{\C}$ of the inner product $q$ on $\Oct'$, the $\C$-linear extension of $*$, and the natural extension of $\odot$ to $\bioct = \Oct' \otimes_{\R} \C$. The cross product on $\imoct^\C$ 
is then defined by the same equation \eqref{CrossProductDefinition} as in the real case, now using the multiplication and inner-product on $\bioct$.
The cross-product on $\bioct$ also satisfies (i)-(iv). Finally, as a complexification of real vector space, $\bioct$ is equipped with a complex-conjugation
$x \mapsto \overline{x}$\,. For readability, we will fall back on the more standard notation of $\cdot$ to denote $q_{\imoct}$ or $q_{\imoct}^\C$ going forward. \\

There is one final crucial structure on $\Oct'$: the scalar triple product $\Omega$ given by
\begin{align}\label{Calibration} 
	\Omega(x,y, z) = (x \times y) \cdot z.
\end{align} 
$\Omega$ is skew-symmetric due to the skew-symmetry of $\times$ and the symmetry of $q$. More surprisingly, $\Omega$
provides a \emph{calibration} in the sense of Harvey \& Lawson \cite{HL82} and is intimately related to $\Gtwo^\F$ geometry. 
In fact, $\Gtwosplit$ can be \emph{defined} in terms of $\Omega$, as we note below.  \\

We will use the following non-trivially equivalent characterizations of $\Gtwosplit$ as groups of linear transformations \cite{Fon10}: 
\begin{enumerate}
	\item Split-octonion algebra product automorphisms: $\{ \psi \in \GL(\Oct') \; | \; \psi( x \odot y) = \psi(x) \odot \psi(y) \}$, transformations which turn out to automatically 
	also preserve $\times, q, *$.
	\item Cross-product preserving transformations: $\{ \psi \in \GL(\imoct) \; | \; \psi(x \times y) = \psi(x) \times \psi(y) \}$. 
	\item Stabilizing transformations of $\Omega: \{ \psi \in \GL(\imoct) \; | \; \psi^*\Omega= \Omega \}$. 
\end{enumerate} 

\subsection{Some $\Oct'$ Identities} 

We record the following identities, which we learned the excellent text \cite{Fon10}, for future use.\\

 The \textbf{double cross-product} identity holds for any $x,y \in \imoct$. 
\begin{align} \label{DCP} 
	x \times (x \times y) = (x \cdot y) \, x - q(x) \, y.
\end{align} 
 When applied twice to $x = z+w$, the previous identity gives us the \textbf{generalized double cross-product} identity.
\begin{align}
	z \times (w \times y) + w \times (z \times y) = (z \cdot y) w + (w \cdot y) z - 2(z \cdot w) \, y. \label{GDCP} 
\end{align} 
A very useful corollary to \eqref{GDCP} occurs in the case that $x,y ,z$ are pairwise orthogonal. Indeed, in this case,
\begin{align}\label{DCP_Orthogonal} 
	x \times (y \times z) = - y \times (x \times z) .
\end{align} 
Finally, we note a very useful orthogonality property of the cross product. 
 	\begin{lemma}[Orthogonality of Cross-Product] \label{CrossProductOrthogonality} 
		Let $u, v, w \in \imoct $ be pairwise orthogonal. Then we have $ (w \times u) \; \bot (w \times v) $.  
		\end{lemma} 

\subsection{$\quadric$ and its Almost-Complex Structure}

We now define $ \quadric: = \{ \, x \in \imoct \; | \, q(x) = +1 \}$ as the unit sphere in the imaginary split-octonions. 
The space $\quadric$ is a 2-1 cover of $\Stwofour := \{ [x] \in \mathbb{P} \imoct \; | \; q(x) > 0\}$, the space of positive-definite lines.
We may use the projective model $\Stwofour$ later. \\

We begin by identifying $ T_{x} \quadric \cong [x^\bot \subset \imoct ] $. Then $q_{\Oct'}$ defines a $\Gtwosplit$-invariant signature (2,4) pseudo-Riemannian
metric on $\quadric$. The natural almost-complex structure $J$ on $\imoct$ comes from the cross-product $\times$. Recall here that an \emph{almost-complex structure} on a manifold $M$ is a bundle endomorphism $J: TM \selfarrow$ satisfying $J^2 = -\id$.
Moreover, an \emph{almost-complex curve} is a map $\phi: (M, J_M) \rightarrow (X, J_X)$ satisfying $ d\phi \circ J_M = J_X \circ d\phi$. \\

We now define $J$. First, take a position vector $w \in \quadric$ and denote $C_{w}:= (v \mapsto w \times v) $ as the cross-product endomorphism of $w$. 
 The double cross-product identity \eqref{DCP} implies $ (C_{w} \circ C_{w})|_{w^\bot} = -\mathsf{id}_{w^{\bot}}$. Hence, $J$ is given pointwise by $J_{w}: T_w \mathbb{S}^{2,4} \rightarrow  T_w \mathbb{S}^{2,4}$ via $J_w := C_w$, or 
\begin{align}\label{AlmostComplexStructure} 
	J_{w}( u) = w \times u.
\end{align}
Of course, since $\Gtwosplit$ preserves $q_{\imoct}$ and the cross-product $\times$ on $\imoct$, the group acts by isometries on $\Stwofour$ that preserve $J$. While $\Gtwosplit$ is not the full isometry group of $(\quadric, q|_{\quadric})$, the group $\Gtwosplit$ is the full automorphism group of $J$-preserving isometries: $\Gtwosplit = \Aut(\quadric, q|_{\quadric}, J_{\quadric})$. 

The almost-complex structure $J_{\quadric}$ is non-integrable as is the case for the canonical almost-complex structure $J_{\mathbb{S}^6}$ \cite{CT23, Leb87}.\\

We now finish the preliminaries by providing a Stiefel-manifold description of $\Gtwosplit$, which will be useful in the construction of the almost-complex curve in Section 
\ref{AlmostComplexCurve}. The idea of understanding the $\Gtwosplit$ action of $\imoct$ via Stiefel triplets is a foundational 
one, and is extremely useful \cite{BH14, Fon10, HL82, Nie24}. 
To this end, we define:

\begin{definition}\label{MultiplicationFrame}
We call a basis $\mathcal{F} = (f_i)_{i=1}^7$ for $\imoct$ a \textbf{multiplication basis} when $\mathcal{F} = \varphi(\mathcal{M})$ for $\varphi \in \Gtwo$,
where $\mathcal{M}$ is the standard multiplication basis \eqref{MultiplicationBasis}. That is, $\mathcal{F}$ is a multiplication basis when the linear map $\varphi := (m_i \mapsto f_i)$ preserves the cross-product. 
\end{definition} 

The following proposition tells us a $\Gtwosplit$ transformation is determined by its 
action on an algebra basis of $\Oct'$, in particular, by the action of $\varphi$ on the algebra generators $(i,j, l)$. 
Note that the signs in the subscript refer to the $q_{\imoct}$-norm of the triplets. Indeed, $q_{\imoct}(x) = -x^2$, so for $x\in \imoct$, $x^2 = \pm 1$ if and only if $q(x) = \mp 1$. 

\begin{proposition}[ $(+,+,-)$ Stiefel Description of $\Gtwosplit$] \label{StiefelTripletModel} 
Consider the set 
$$ V_{(+,+,-)}(\imoct) := \, \{ \, (x,y,z) \; | \; \; x,y,z \in \imoct, \; x^2 = y^2 = -1, \; z^2 = 1, \; \, x \cdot y \, = \, x \cdot z =\,  y\cdot z = 0 , \; z \, \cdot \, (x\times y) = 0\; \} $$
of (ordered) orthonormal triplets. Then $\Gtwosplit$ acts simply transitively on $V_{(+,+,-)}(\imoct)$. 
\end{proposition} 

\begin{proof}
Given any triplet $p = (x,y,z) \in V_{(+,+,-)}(\imoct)$, one naturally creates an orthonormal (ordered) vector space basis $\mathcal{B} = (b_i)_{i=1}^7$ for $\imoct$
in the same way we get $\mathcal{M}$, \eqref{MultiplicationBasis}, from $i,j,l$. Namely, 
\begin{align}\label{BasisCreation} 
	\mathcal{B}= (x, y, xy, z, zx, zy, z(xy)). 
\end{align} 
We consider the linear map $\varphi := m_i \mapsto b_i$. In particular, $\varphi(i) = x, \varphi(j)= y, \varphi(l) = z$.
If we can show $\varphi \in \Gtwosplit$, then $\Gtwosplit$ acts transitively on $ V_{(+,+,-)}$. \\

To prove $\varphi \in \Gtwosplit$, one can use some general
identities of $\Oct'$, adapting the argument of Lemma A.8 of \cite{HL82}. 
By the orthogonality property $(x \times y) \cdot z = 0$, if we define the copy of the quaternions $H := \spann_{\R-\mathsf{alg}} <x,y>$,
then $z \in H^\bot$. Then the argument from Lemma A.8 of \cite{HL82} tells us that for $a,b,c,d \in H$, 
\begin{align} \label{Multiplication} 
 (a + zb ) (c + zd) = ac + db^* + z(a^*d+ cb).
\end{align} 
We will see that this identity implies $\varphi \in \Gtwosplit$. 
Now, the fact that $\mathcal{B}$ is a basis implies we have a direct sum decomposition $ \Oct' = H \oplus z H$. Then equation \eqref{Multiplication}
says that if we take $(a,b) \in H \oplus H$ corresponding to $a+ zb$, then the multiplication of $\Oct'$ in this model is given by
exactly the formula \eqref{OctMultiplication}. In particular, the map $\varphi$ preserves $\odot$. 
We conclude that $\Gtwosplit$ acts transitively on $V_{(+,+,-)}(\imoct)$. On the other hand, suppose $\varphi, \psi \in \Gtwosplit$ act with $\varphi \cdot p = q, \; \psi \cdot p = q $ for $p, q \in V_{(+,+,-)}(\imoct)$, then one finds $\psi (\mathcal{B}_p) = \varphi(\mathcal{B}_p) = \mathcal{B}_q$ for the basis creation process \eqref{BasisCreation} since $\varphi$ preserves multiplication. But then $\varphi = \psi$ by linearity.
Hence, the action of $\Gtwosplit$ on $ V_{(+,+,-)}$ has no fixed points. 
\end{proof} 

As a corollary to the proof, we learn the following: 

\begin{corollary}[$\Gtwosplit$ and Multiplication Frames] 
The group $\Gtwosplit$ acts simply transitively on the set of all multiplication frames for $\imoct$. 
\end{corollary} 

\begin{corollary}[$\Gtwosplit$ and $\quadric$] 
$\Gtwosplit$ acts transitively on $\quadric$. 
\end{corollary} 

\subsection{A Complex Cross-Product Basis for $(\imoct)^{\C}$}

Following Baraglia \cite{Bar15}, we will use the following carefully selected basis \\
$\mathcal{B} = ( u_3, u_2, u_1, u_0, u_{-1}, u_{-2}, u_{-3} )$ for $\imoct^\C$.
The basis $(u_i)_{i=3}^{-3}$ will be useful for understanding almost-complex curves in $\Stwofour$ in Section \ref{AlmostComplexCurve}. The name `complex cross-product basis' is motivated by property (3) of Proposition \ref{PropBaragliaBasis}. 

\begin{align} \label{BaragliaBasis}
	\begin{cases} 
	u_{3} &= \frac{1}{\sqrt{2}}( \, jl + \sqrt{-1} kl).\\
	u_{2} &=\frac{1}{\sqrt{2}}( \, j + \sqrt{-1} k).\\
	 u_{1} &= \frac{1}{\sqrt{2}}( \, l + \sqrt{-1} il).\\
	 u_{0} &= i.\\
	 u_{-1} &= \frac{1}{\sqrt{2}}( \, l- \sqrt{-1} il).\\
	 u_{-2}&= \frac{1}{\sqrt{2}}( \, j - \sqrt{-1} k).\\
	 u_{-3} &= \frac{1}{\sqrt{2}}( \, jl - \sqrt{-1} kl).
	 \end{cases} 
\end{align}

Observe that $u_{-i} = \overline{u_i}$, where $x \mapsto \overline{x}$ denotes complex conjugation with respect to the standard complex 
structure on $\bioct$. We state all of the useful properties of $\mathcal{B}$ we know, of which only (1) is given in \cite{Bar15}. 

\begin{proposition}[Properties of Complex Cross-Product Basis]\label{PropBaragliaBasis}
The basis $\mathcal{B}$ satisfies:
	\begin{enumerate}
		\item The inner-product with respect to $\mathcal{B}$ is given by
		 	\begin{align}	\label{InnerProductBaragliaBasis}
			 	q_{\C}( \mathbf{x}, \mathbf{y} ) = \mathbf{x}^T \; \begin{pmatrix} & & & & & & -1\\
				 & & & & & 1 & \\
				  & & & &-1 & & \\
				 & & & 1& & & \\
				  & & -1& & & & \\
				 &1 & & & & & \\
				  -1& & & & & &  \end{pmatrix} \;\mathbf{y},
			\end{align}
		\item Define the operator $C_{u_0}: \imoct^\C \rightarrow \imoct^\C$ by $C_{u_0}:= (u_0 \times \, \cdot)$. Then the $\pm \sqrt{-1}$ eigenspaces are given by 
		$$ \mathbb{E}_{\sqrt{-1}}(C_{u_0}) = \spann_\C < u_3, u_{-2}, u_{-1} > \; \; \text{and} \; \;  \mathbb{E}_{-\sqrt{-1}}(C_{u_0} ) = \spann_{\C} < u_{-3}, u_{2}, u_1 > .$$
		\item More generally, we have $ u_{i} \times u_j = C_{ij} u_{i+j}$ for constants $C_{ij} \in \C$, possibly 0, where we set $u_{i} = 0$ for $ | i | >3$.
		\item For constants $C_{ijk} \in \C$, the 3-form $\Omega$ in the basis $\mathcal{B}$ has the form 
			\begin{align}\label{G2_3Form}
				\Omega = \sum_{i + j + k = 0} C_{ijk} \, u_i^* \wedge u_j^* \wedge u_k^*.
			\end{align}
		\item A transformation $D \in \GL(\imoct^\C)$ of the form $D = \diag(d_i)_{i=3}^{-3} $ in the basis $\mathcal{B}$ satisfies $D \in \Gtwo^\C$ if
		and only if $d_i d_j d_k = 1$ for all pairwise distinct triplets $(i,j,k)$ such that $i + j + k = 0$. 
	\end{enumerate} 
\end{proposition} 

\begin{proof}
We sketch the proof. First, using that $q_{\imoct} = \mathsf{diag}(+1,+1,+1,-1,-1,-1,-1)$ in the basis $\mathcal{M} = (m_i)_{i=1}^7$, one can directly 
compute that (1) is the correct expression for $q_{\imoct}^\C$ with respect to $\mathcal{B}$. For example, it is immediately evident that $u_i \cdot u_j = 0$
for $j \notin \{-i, +i\}$. \\

We can see (2) by direct calculation, using that $C_{u_0}$ acts as the standard complex structure $\begin{pmatrix} 0 & -1 \\ 1 & 0 \end{pmatrix} $ 
on the 2-planes $\spann_{\R} \langle m_2, m_3 \rangle$, $\spann_{\R} \langle m_4, m_5 \rangle,$ $\spann_{\R} \langle m_6, m_7\rangle$ in the ordered bases $(m_2, m_3)$, $(m_5, m_4)$, $(m_7, m_6)$, respectively.  \\

(3) is a tedious and direct calculation. Here are the other multiplicative relations amongst the basis \eqref{BaragliaBasis}. 
\begin{itemize}
	\item $u_{3} \times u_{-1} = \sqrt{2} u_2$.
	\item $u_{3} \times u_{-2} = \sqrt{2} u_1 $. 
	\item $u_{2} \times u_1 = \sqrt{2} u_{3}.$
	\item $u_{i} \times u_{-i} = \sgn(i) \sqrt{-1} u_0$, for $i \in \{1,2,3\}$. 
\end{itemize} 
Since $u_{-i} = \overline{u}_i$, the equation $\overline{xy} = \overline{x} \; \overline{y}$ for $x,y \in \bioct$ implies the remaining multiplicative relations.% the relation relation $u_i \times u_j = C_{ijk} u_k$ for $C_{ijk} \in \C$ is equivalent to the relation $u_{-i} \times u_{-j} = \overline{C_{ijk}} u_{-k}$. \\

(5) is an immediate consequence of (4). If $D$ is diagonal, then $D$ preserves $\Omega$ iff it preserves each decomposable tensor $u^{i,j,k}$
and the claim follows. 
\end{proof} 

\begin{remark}
Later, we use an analogous basis called a `real cross-product basis' in Section \ref{Sec:RealCrossProductBasis}, adapted to the real vector space $\imoct$.
\end{remark}

\subsection{A Model for $\g_2$} 

Recall that $\Gtwo^\F = \{ \psi \in \GL(\imoct^\F) \; | \; \psi(u \times v) = u \times v\}$. Consequently, the Lie algebra $\g_2^\F$ is the set of transformations 
infinitesimally preserving the cross-product, i.e., derivations of the cross product:  
$$ \g_2^\F = \{ \, A \in \mathsf{End}(\imoct^\C) \; | \;A( x \times y)=  (A \,x) \times y + x \times (Ay) \; \}$$ 

For symmetry of notation, we later write $\g_2' := \g_2^\R$. Fixing the (ordered) basis $\mathcal{B} := (u_i)_{i=3}^{-3}$ for $\imoct^\C$, we obtain a matrix representation $\Psi$ of $\g_2$ by writing these endomorphisms with respect to $\mathcal{B}$. 
 In short, $\Psi: \g_2 \rightarrow \gl_7\C$ by $ A \; \stackrel{\Psi}{\longmapsto} \; [A]_{\mathcal{B}}$. 
Going forward, we will conflate $\g_2$ with its image under this representation $\Psi$. When we use a different basis for $\imoct^\C$, we will make this clear. \\

For completeness, we include the full details of $\g_2$ Lie theory in the basis $\mathcal{B}$ in the Appendix. Specifically,
Section \ref{RootSpace} contains the root space decomposition, Section \ref{KostantTheory} contains our choice of model principal 3-dimensional subalgebra $\fraks$ of Kostant \cite{Kos59}, and Section \ref{HitchinTheory} contains the Cartan involution $\sigma$ from \cite{Hit92}.

\subsection{Cyclic $\Gtwosplit$-Higgs Bundles} \label{G2HiggsBundle} 

We now form the Higgs bundle $\V$ in which we work for most of the paper. While the vector bundle used below is, of course, a trivial holomorphic  vector bundle, all of the structure here is intentionally set up to be completely analogous to the Higgs bundle construction $\V_{\Sigma}$ 
considered over a closed Riemann surface $\Sigma$. 

Denote $\K$ as the holomorphic cotangent bundle of $\C$ and we fix the holomorphic vector bundle
\begin{align}
	\mathcal{E} =  \K^3 \oplus \K^2 \oplus \K \oplus \mathcal{O} \oplus \K^{-1} \oplus \K^{-2} \oplus \K^{-3}.
\end{align} 
We place an orthogonal structure $E$ on the bundle by taking a coordinate vector $\mathbf{x} = (x_i)_{i=3}^{-3}$ in the basis $ (dz^i)_{i=3}^{-3}$ and forming
the metric 
	$$E( \mathbf{x}, \mathbf{y} ) = \mathbf{x}^T R \, \mathbf{y} ,$$
where 
$$R = \begin{pmatrix} & & & & && -1 \\ & & & & & 1 &  \\ & & & & -1& & \\ & & & 1 && &  \\ &&-1&&&& \\ & 1& & & && &  \\  -1&&&&&& \end{pmatrix} . $$

Then $E$ reduces the structure group of $\V$ to $\mathsf{SO}_7\C$. We then further reduce the structure group by identifying our $\C^7$ fibers with $\imoct^\C$ as follows.
Let $( \mathbf{e}_{3},  \mathbf{e}_{2},  \mathbf{e}_{1}, \mathbf{e}_{0}, \mathbf{e}_{-1}, \mathbf{e}_{-2}, \mathbf{e}_{-3} ) $ be the standard basis of $\C^7$, indexed to match the powers of $\K$. We fix once-and-for-all the identification of the fibers of $\V$ with complex split-octonions by identifying 
$\mathbf{e}_i$ with $u_i \in \mathcal{B}$ from our complex cross-product basis from Proposition \ref{BaragliaBasis}. By equation \eqref{InnerProductBaragliaBasis}, this identification respects our initial choice of orthogonal structure. That is to say, the $\imoct$ dot-product $q$ in the basis $(u_i)_{i=3}^{-3}$ is the same as is the dot-product $E$ in
the basis $(\mathbf{e}_i)_{i=3}^{-3}$. This identification $\mathbf{e}_i \leftrightarrow u_i$ defines a reduction of structure to $\Gtwo$ by considering only frames that preserve the $(\imoct)^\C$ cross-product. 

\begin{remark} \label{HiggsBundleRemark} We emphasize that the identification of fibers of $\V_{\Sigma} := \bigoplus_{i=3}^{-3} \K_{\Sigma}^i$ with $(\imoct)^\C$ extends to any Riemann surface to give a reduction of structure to $\Gtwo$. Let $(z_{\alpha}, U_{\alpha})$ be an atlas for $\Sigma$. Then applying the aforementioned identification $dz_{\alpha}^i \leftrightarrow u_i$ in the local holomorphic frames $\mathcal{B}_{\alpha} = (dz_{\alpha}^i)_{i=3}^{-3} $, one finds that the local transitions are in $\Gtwo$, i.e., they respect the cross-product $\times$ on $(\imoct)^\C$. Indeed, if $z_{\beta} = g_{\alpha \beta} z_{\alpha} $, then set $ \zeta := \deriv{ g_{\alpha \beta}}{z_{\alpha}}$ and we have
$\mathcal{B}_{\beta} = \psi \, \mathcal{B}_{\alpha} $, where 
$$\psi = \mathsf{diag}(\zeta^3, \zeta^2, \zeta^1, 1, \zeta^{-1}, \zeta^{-2}, \zeta^{-3}).$$
By Proposition \ref{PropBaragliaBasis} part (5), $\psi \in \Gtwo$. Thus, these local identifications yield a globally well-defined 
bundle map $\times_{\V}: \V \times \V \rightarrow \V$ defined by using any of the local cross products. \end{remark} 

The (Kostant) invariants of $\g_2$ are 1 and 5. Thus, the $\Gtwosplit$-Hitchin component is parametrized by pairs of holomorphic quadratic and sextic differentials $q_2, q_6$. Since we are interested in only \emph{cyclic} Higgs bundles, we set $q_2 =0$ and denote $q:= q_6$.
Following the definition of the Hitchin section, the Higgs field is of the form $\varphi = \tilde{e} + q e_{\gamma}$, where $\tilde{e}$ comes from our 
distinguished principal Kostant 3DS $\fraks = \spann_{\C} < x, e, \tilde{e} > $ and $e_{\gamma}$ is the (normalized) highest weight vector given by 
			$$ e_{\gamma} = \begin{pmatrix} 0& & & & & 1 &\\
				      & 0&  & & & & 1\\
			  	      & &0 & & & &\\
			 	      & & &0 && &\\
				      & & & & 0& &\\
				      & & & &&0&\\
				      & & & & &&0\\ \end{pmatrix} .$$ 
Thus, the Higgs field in these coordinates is 
	\begin{align}
		\varphi &= \begin{pmatrix} 0&  & & & & q&\\
				     \sqrt{3} & 0&  & & & & q\\
			  	      & \sqrt{5}&0 & & & &\\
			 	      & & -\sqrt{-6}&0 && &\\
				      & & & -\sqrt{-6}& 0& &\\
				      & & & &\sqrt{5} &0&\\
				      & & & & & \sqrt{3}&0\\
				     \end{pmatrix} .	
	\end{align} 

\subsubsection{Harmonic Metric \& Hitchin's Equation} \label{HitchinsEquationsSubsection} 

We now search for a Hermitian metric $h$ on $\V$ such that the connection $\nabla = \nabla_{\delbar, h} + \varphi + \varphi^{*h}$ flat. The flatness of $\nabla$ is equivalent 
to the metric $h$ satisfying \emph{Hitchin's equations} \cite{Hit92}. 
\begin{align}
	F_{h} + [\varphi, \varphi^{*h} ] &= 0 \label{Hitchin1} \\
	\nabla_{\delbar, h}^{0,1} \varphi &= 0. \label{Hitchin2}
\end{align}

Here, $F_{h}$ is the curvature of the Chern connection $\nabla_{\delbar, h}$. 
By our construction of the Higgs field, equation \eqref{Hitchin2} is equivalent to the holomorphicity of $q$. 
Due to our interest in \emph{cyclic} Higgs bundles, we look for a metric $h$ that is diagonal in the standard holomorphic coordinates \cite{Bar15}.
Let $r, s > 0$ be positive functions. Then in the standard coordinates $(dz^i)_{i=3}^{-3}$, we search for $h$ of the form 
\begin{align}\label{HitchinMetric}
	h( \mathbf{x}, \mathbf{y} ) = \mathbf{\overline{x}}^T \, \mathsf{diag} (r^{-1}s^{-1}, r^{-1}, s^{-1}, 1, s, r, rs)\, \mathbf{y} =  \mathbf{\overline{x}}^T \,H \, \mathbf{y}.
\end{align} 
We will prove there is a unique such metric under some completeness conditions in Section \ref{Uniqueness}. 

Now, the $h$-adjoint of $\varphi$ is given by $\varphi^{*h} = H^{-1} \overline{ \varphi}^T H$.
Thus,
\begin{align}
		\varphi^{*h} &= \begin{pmatrix} 0& \sqrt{3}\, s & & & & &\\
				      & 0& \sqrt{5} \frac{r}{s}& & & & \\
			  	      & &0 &\sqrt{-6}s  & & &\\
			 	      & & &0 &\sqrt{-6}s & &\\
				      & & & & 0&\sqrt{5} \frac{r}{s}&\\
				      \frac{ \overline{q}}{ r^2s }& & & & &0& \sqrt{3} \,s \\
				      & \frac{ \overline{q}}{ r^2s } & & & &&0\\
				     \end{pmatrix} 	.
	\end{align} 
The curvature of the Chern connection is $ F_{h} = \delbar ( H^{-1} \del H )= (\log H)_{z\zbar} \; d\zbar \wedge dz.$ Writing $\varphi = \widehat{\varphi} \, dz , \varphi^{*h} = \widehat{\varphi^{*h}} d\zbar$, Hitchin's equation \eqref{Hitchin1} becomes the following equation
\begin{align}\label{HitchinsEquation1}
	 (\log H)_{z\zbar} d\zbar  \wedge dz= [\widehat{\varphi}, \widehat{\varphi^{*h}} ] d\zbar  \wedge dz.
\end{align} 
 Taking the second and third components along the diagonal, equation \eqref{Hitchin1} is equivalent to the following $2 \times 2$ coupled system of PDE for $r,s$.
We use $\Delta := \partial_{\zbar} \circ \partial_{z} $ going forward. 
 \begin{align}\label{HitchinsEquations_rs}
 	\begin{cases} 
 		\Delta \log r  &=  5 \frac{r}{s} - 3s- \, \frac{|q|^2}{r^2 s} \\%\label{HitchinsEquation1} \\ 
		\Delta \log s  &= 6s-  5\frac{r}{s} .%\label{HitchinsEquation2}
	\end{cases} 
 \end{align} 
Substituting $u_1= \log r + \frac{1}{2}\log s$, $u_2 = \frac{1}{2}\log s$, the system \eqref{HitchinsEquations_rs} becomes 
  \begin{align}\label{HitEuc}
 	\begin{cases}  
		 2 \Delta u_1&= 5 e^{(u_1-3u_2)} - 2 e^{-2u_1}\,|q|^2 .\\
		2 \Delta u_2 &=  6e^{2u_2} -5e^{(u_1-3u_2) }. \end{cases} 
 \end{align} 
 
 Later, it will be useful to have a form of these equations with the constants all equal. 
Hence, define $ c, d > 0 \in \R$ which uniquely solve the system
\begin{align}\label{cdSystem}
 	\begin{cases} 5c^{5/6}d^{-1/2} - 2c^{-5/3} &= 0 \\
	-5c^{5/6}d^{-1/2} +6d^{1/3} &= 0. \end{cases} 
\end{align}%This system is equivalent to $5 c^{5/2} = 2 d^{1/2}$ and $ 5c^{5/6} = 6d^{5/6}$. 
The exact solution to \eqref{cdSystem} is $d = \frac{5}{6\sqrt{3}}, c = (\frac{2}{5})^{2/5} d^{1/5}.$ \\

Define $b= 3 d^{1/3} $. Then $\bm{w} = (w_1, w_2)$ solves \eqref{HitEuc_Clean} if and only if $\bm{u} = (u_1, u_2)$ defined by $u_1 = w_1 + \log(c^{5/6} ) $
and $ u_2 = w_2 + \log(d^{1/6}) $ solves \eqref{HitEuc}. 
 \begin{align}\label{HitEuc_Clean}
 	\begin{cases}  
		   \Delta w_1&= b\left( e^{(w_1-3w_2)} - e^{-2w_1}\,|q|^2 \right) \\
		  \Delta w_2 &= b\left( e^{2w_2}  -e^{(w_1-3w_2) } \right) . \end{cases} 
 \end{align} 
 For later purposes, we also re-write the system \eqref{HitEuc} with respect to a different conformal metric $\sigma = \sigma(z) |dz|^2$. Given a solution
 $\bm{u}$ of \eqref{HitEuc}, we define $\bm{\psi} = (\psi_1, \psi_2)$ by 
 $$\begin{cases}
 	e^{\psi_1} \sigma^{5/2} &= e^{u_1} \\
	e^{\psi_2} \sigma^{1/2} &= e^{u_2} .
 \end{cases} $$
A calculation shows that $\bm{\psi} $ solves the system \eqref{G2Hitchin_GeneralMetric_old}: 
\begin{align}\label{G2Hitchin_GeneralMetric_old} 
	\begin{cases}  
		2 \Delta_\sigma \psi_1 &= 5e^{\psi_1 - 3 \psi_2} - 2 |q|^2_{\sigma} \, e^{-2\psi_1} + \frac{5}{2} \kappa_{\sigma} \\
		2 \Delta_\sigma \psi_2 &= -5e^{\psi_1 - 3\psi_2} + 6 e^{2\psi_2} + \frac{1}{2} \kappa_\sigma  \end{cases} ,
\end{align} 
where $\Delta_{\sigma} := \frac{1}{\sigma} \Delta$, $|q|^2_\sigma := \frac{q \bar{q}}{\sigma^6} $ and $\kappa_{\sigma} = - \frac{2}{\sigma} \Delta \log \sigma$.\\

\begin{remark} 
There is no \emph{naive} linear equivalence between the $\g_2$ affine Toda equations and the $\SL_6\C$ or $\SL_7\C$ affine Toda equations. 
The affine Toda equations for $\SL_n\C$ depend on a holomorphic differential $q \in H^0(\K^n_\C)$. Thus, the $\SL_7\C$ affine Toda equations
have the wrong order differential. 
One can then show that no linear substitution yields an equivalence between the $\g_2$ affine Toda equations and is the $\SL_6\C$ affine Toda equations. 
To the author's best knowledge, there is no relationship between the solutions to the $\g_2$ and $\SL_6 \C$ affine Toda equations. 
\end{remark} 

\subsubsection{Parabolic Structure} 

We now define a parabolic structure on $\V$.
Take a section $s = s_i dz^i \in \Gamma(\C, \V)$ that is meromorphic in the usual sense, i.e., each $s_i $ is meromorphic. Define $(s_i)_{\infty} \in \Z$ as the order of $s_i$ at $\infty$. We then define a parabolic structure on $\V$ as follows. 
We define weights $(\alpha_{i})_{i=3}^{-3} := ( -\frac{n}{2},  -\frac{n}{3}, -\frac{n}{6},0,  \frac{n}{6}, \frac{n}{3}, \frac{n}{2})$. For each meromorphic section $s \in \Gamma(\C, \V)$,
we define the order at infinity by 
$$v_{\infty}(s) := \max_{-3 \leq i \leq 3} \{ (s_{i})_{\infty} + \alpha_i \} .$$
We can explain where this definition comes from by noting the asymptotics proven later for the unique complete real solution  
$h = \mathsf{diag} (\frac{1}{rs}, \frac{1}{r}, \frac{1}{s}, 1, s, r, rs) $ to Hitchin's equations for $\V$ in the standard frame $(dz_i)_{i=3}^{-3}$. 
We prove in Theorem \ref{ExistenceTheorem} that 
$r \asymp |q|^{2/3} \asymp |z|^{ \frac{2n}{3} }$ and $s \asymp |q|^{\frac{1}{3}} \asymp |z|^{ \frac{n}{3} }$, where $q$ is a monic polynomial of degree $n$. 
% using our sub and super-solutions from Lemmas \ref{SuperSolution}, \ref{SubSolution}. 
We write $f(z) \asymp g(z)$ here to denote asymptotically comparable in the sense that $f(z) \in O(g(z)) $ and $g(z) \in O(f(z))$. 

In particular, the asymptotics for $r,s$ imply that the canonical section $\sigma_i := dz^i \in \Gamma(\C, \K_\C^i)$ satisfies that $h(\sigma_i, \sigma_i) \asymp |z|^{2 \alpha_i }$ for the aforementioned constants $\alpha_i$.
 Thus, the metric $h$ respects the parabolic structure in the sense that for any meromorphic section $s \in \Gamma(\C, \V)$, we have 
 \begin{align}\label{MetricAsymptoticGrowth} 
  h(s(z), s(z) ) \asymp \, |z|^{2 \, v_{\infty}(s)} .
  \end{align} 

We can think of the weights $(\alpha_i)_{i=3}^{-3}$ as being canonically associated to the bundle $(\V, \varphi)$ in the sense that there is a unique 
complete solution (in the sense of definition \eqref{CompletenessDefinition}) $h = \mathsf{diag}(h_i)_{i=3}^{-3} \in \Gtwo$ with $h_i : \C \rightarrow \R_+$  to the equation \eqref{Hitchin1}. Then
this solution $h$ satisfies $h_i \asymp |z|^{2\alpha_i}$.

On the other hand, we have a unique diagonal metric $h$ on $(\V, \varphi)$ solving Hitchin's equations if we instead demand the asymptotic \eqref{MetricAsymptoticGrowth}
instead of the completeness condition \eqref{CompletenessDefinition}. Indeed, given any two metrics 
$h_1 =\mathsf{diag} (\frac{1}{r_1s_1}, \frac{1}{r_1}, \frac{1}{s_1}, 1, s_1, r_1, r_1s_1), \; h_2 =\mathsf{diag} (\frac{1}{r_2s_2}, \frac{1}{r_2}, \frac{1}{s_2}, 1, s_2, r_2, r_2s_2)$, 
then using our same substitution to define $u^{(i)}_1 := \log r_i + \frac{1}{2}\log s_i$, $u_2^{(i)} = \frac{1}{2}\log s_i$ and 
$ \bm{u}^{ (i) } := (u^{ (i) }_1,u^{ (i) }_2)$, by the asymptotic \eqref{MetricAsymptoticGrowth}, we find that $\bm{u}^{ (1) }, \bm{u}^{(2)}$ are mutually bounded. Hence,
by Lemma \ref{UniquenessMutuallyBounded}, we have $\bm{u}^{(1)} = \bm{u}^{(2)}$. \\

We refer the reader to \cite{FN21} for further details on parabolic Higgs bundles over $\CP^1$.

\subsubsection{Real Forms on $\V$ and $\End( \V )$}  \label{Subsection_RealForms}

Previously, we reduced structure group of $\V$ to $\Gtwo$. 
We now define a $\nabla$-parallel real structure on $\End( \V )$, for $\nabla = \nabla_{\delbar, h} + \varphi + \varphi^{*h}$
the flat connection. \\

First, we define a compact involution $\hat{\rho}$ on $\End(\V)$ as a conjugate of our model compact involution $\rho = (A \mapsto -\overline{A}^T )$ 
from \ref{ModelRealForms}. Set
\begin{align}
	\hat{\rho} = H^{-1} \rho H = \Ad_{H^{-1}} \circ \rho.
\end{align} 
Observe that $\Fix(\hat{\rho}) = \mathfrak{u}(H)$, the set of $H$-skew-hermitian matrices.
Now, the involution $\hat{\rho}$ commutes with the Hitchin involution $\sigma$.
The global split real form is given by $\hat{\tau} = \sigma \hat{\rho} $, or in coordinates, 
\begin{align}
	\hat{\tau}(A) =  H^{-1}Q \overline{A} Q H,
\end{align} 
where 
\begin{align}\label{Qmatrix} 
	Q = \begin{pmatrix} & & & & & &1\\
					& & & & &1&\\
				   	& & & & 1& & \\
				  	& & & 1& & & \\
				  	& & 1&& & &\\
				   	& 1& & & & & \\
				   	1 & & & & & & \end{pmatrix}.
\end{align} 

Writing $\hat{\tau} = \Ad_{H^{-1}} \circ \tau$ shows that $\hat{\tau}$ defines a split real form, where $\tau$ is the model real form from \ref{ModelRealForms}. Hence,  $\Fix(\hat{\tau}) $ defines a $\g_2'$-fibered sub-bundle
of endomorphisms. We write $\End( \V )^\R := \mathsf{Fix}(\hat{\tau})$. \\

We now define a real structure $\hat{\tau}_{\V}$ on $\V$ that is compatible with our real structure $\hat{\tau}$ on $\End(\V)$. Indeed, $\psi \in \End(\V)^\R$ will preserve the real sub-bundle $\V^\R := \Fix( \, \hat{\tau}_{\V} \, )$. To this end, define
\begin{align}\label{ReaIBundleInvolution}
 	\hat{\tau}_{\V}( \mathbf{x} ) = H^{-1}Q \, \overline{\mathbf{x}},
\end{align}
in the basis $\mathcal{B}$. As a matrix expression, $\hat{\tau}_{\V}$ is given by
\begin{align}\label{ParallelConjugation} 
	\hat{\tau}_{\V}(\mathbf{x}) :=  \begin{pmatrix} & & & & && \frac{1}{rs} \\ & & & & & \frac{1}{r} &  \\ & & & & \frac{1}{s} & & \\ & & & 1 && &  \\ &&s&&&& \\ & r& & & && &  \\  rs&&&&&& \end{pmatrix} \overline{\mathbf{x}} .
\end{align} 
 We directly exhibit a $\hat{\tau}_\V$-real frame for $\V$ in the next section.\\
 
As noted by Baraglia \cite{Bar10}, $\End(\V)^\R$ preserves $\V^\R$ by
a certain `multiplicative' compatibility: 
 \begin{align}\label{MultiplicativeCompatibility}
 	\hat{\tau}_{\V}(\, A \mathbf{x} \,) = H^{-1}Q\overline{A\mathbf{x}} = (H^{-1}Q\overline{A}QH)(H^{-1}Q\overline{x}) = \hat{\tau}(A) \, \hat{\tau}_{\V}(\mathbf{x} ).
 \end{align} 
Indeed, equation \eqref{MultiplicativeCompatibility} says that if $\psi \in \End( \V)^\R, \mathbf{x} \in \V^\R$, then $\psi \mathbf{x} \in \V^\R$.\\

We now observe by direct calculation that $\hat{\tau}( \varphi + \varphi^{*h} ) = \varphi + \varphi^{*h}.$ Let $D_H $ denote the Chern connection 1-form and 
we have trivially that $D_H$ preserves $h$, so $\hat{\rho}(D_H) = D_H$. Since $D_H \in \h$ point-wise, $\sigma(D_H) = D_H$. Thus, $\hat{\tau}(D_H) = D_H$. We conclude that 
both $\hat{\tau}, \hat{\tau}_{\V}$ are $\nabla$-parallel. 

 \section{The Minimal Surface $f$ $\&$ the Almost-Complex Curve $\nu$}  \label{AlmostComplexCurve} 
 
 In this section, we discuss a detailed picture of the geometric relationship between the almost-complex curve $\nu$ 
 and the minimal surface in the symmetric space. We proceed in steps towards proving Theorem \ref{THM:GeometricCorrespondence},
 in which we prove the commutativity of the diagram in Figure \ref{HarmonicMapDiagram}. \\
 
 In Section \ref{BaragliasConstruction}, we review Baraglia's construction $\nu$ in the Higgs bundle $\V$ \cite[Section 3.6]{Bar10}
 We then discuss the \emph{harmonic map sequence} of $\nu$ and use it to recover Hitchin's equations from the almost-complex curve.
 In Section \ref{Toda}, we discuss Toda frames and the $\g_2$ affine Toda Field equations; we show the integrability equations of the Toda frame $e^{\Omega}$ for $\nu$ are equivalently Hitchin's equations. The Toda frame defines a global lift $\hat{\mathcal{F}}_{\nu}: \C \rightarrow \Gtwosplit$ of $\nu$. In Section \ref{GeometricCorrespondence}, we use the harmonic map sequence to define an abstract geometric lift $G_{\nu}$ of $\nu$ to
 a model space $\mathscr{Y} \cong_{\mathsf{Diff}} \Gtwosplit/ T$. Here,  $T$ is a maximal torus in the maximal compact $K < \Gtwosplit$. We also give a geometric interpretation of the metric $h$ on $\V$ as defining a map $f: \C \rightarrow \mathscr{X} \cong_{\mathsf{Diff}} \Gtwosplit/ K$. Theorem \ref{THM:GeometricCorrespondence} shows the compatibility of the 
maps $f, G_{\nu}$ with the projections $F_K: = \pi_K \circ \hat{\mathcal{F}}_{\nu}$ and $F_T: =  \pi_K \circ \hat{\mathcal{F}}_{\nu}$, where $\pi_K: \Gtwosplit \rightarrow \Gtwosplit/K$
and $\pi_T: \Gtwosplit \rightarrow \Gtwosplit/T$ are natural projections.

 \subsection{Constructing $\nu$ in the Higgs Bundle} \label{BaragliasConstruction} 
 
 We now recall how the harmonic metric $h$ in the Higgs bundle $(\V, \varphi) $ gives rise to an almost-complex curve $\nu: \C \rightarrow \quadric$ \cite{Bar10}. 
Surprisingly, $\nu$ arises from the tautological 
section $\sigma$ of the bundle $\V$.
 
Let $\sigma:  \C \rightarrow \V$ be the tautological section, given by $\sigma(z) = (0,0,0,1,0,0,0)$ in standard coordinates. 
Recall the fixed identifications of each fiber of $\V$ with $(\imoct)^\C$ in Section \ref{G2HiggsBundle}. We can interpret $\sigma$ 
 as a map $\nu: \C \rightarrow (\imoct)^\C$. 
Indeed, the correspondence $\nu \, \leftrightarrow \, \sigma$ is that $\nu(z) \in \imoct^\C$ is given by $\nabla$-parallel translation of $\sigma(z) \in \V_z$
to $\nu(z) \in \V_{p_0} = \imoct^\C$, where $p_0$ is the origin. Further, the section $\sigma$ obeys reality conditions. That is, since $\sigma \in \Gamma(\C, \V^\R)$ is a 
$\hat{\tau}_{\V}$-real section, we have 
$\nu(z) \in \imoct  $, identified as $\Fix(\hat{\tau}_\V)_{p_0}$. Thus, we regard $\nu$ as a map $\nu: \C \rightarrow \imoct$. 
Moreover, $\sigma(z)$ identifies with $u_0 = i \in \imoct$ in each fiber. In particular, $q_{\imoct}(\sigma) = +1$. As $\nabla$-parallel translation preserves $q_{\imoct}^\C$,
the map $\nu$ satisfies $q_{\imoct}(\nu(z)) = +1$. Hence, we may regard $\nu$ as a map $\nu: \C \rightarrow \quadric$. 

We will repeatedly use the correspondence between $u: \C \rightarrow \imoct^\C$ and sections 
of $\V$ by conflating $u$ with $\sigma_{u} \in \Gamma(\C, \V)$ given by $\nabla$-parallel translation of $u(z) \in \V_{p_0}$ to $\sigma_{u}(z) \in \V_z$. Moreover, 
under this correspondence, $d\nu(X)$ corresponds to $\nabla_X \sigma$. To see this, write $\nu(z) = c^i(z) \mathbf{e}_i$ in Einstein summation. Then denote
$\widetilde{\mathcal{B}} := \widetilde{\mathbf{e}}_i$ as the $\nabla$-parallel frame obtained by parallel translating $ (\mathbf{e}_i|_{p_0} )_{i=3}^{-3}$.
Observe that $\sigma_u(z) = c^i(z)  \widetilde{\mathbf{e}}_i|_{z}$. Since $\nabla = d$ in the frame $\widetilde{\mathcal{B}}$, the correspondence follows.

Next, we show that $\nu$ is an almost-complex curve. We recall the connection 1-form $\mathcal{A}$ of the flat connection $\nabla$, which splits into (1,0) and (0,1) parts, $\mathcal{A} = \mathcal{A}^{1,0} + \mathcal{A}^{0,1} $ as follows:

$$ \mathcal{A}^{0,1}  =\varphi^{*h} = \begin{pmatrix} 0& \sqrt{3}\, s & & & & &\\
				      & 0& \sqrt{5} \frac{r}{s}& & & & \\
			  	      & &0 &\sqrt{-6}s  & & &\\
			 	      & & &0 &\sqrt{-6}s & &\\
				      & & & & 0&\sqrt{5} \frac{r}{s}&\\
				      \frac{ \overline{q}}{ r^2s }& & & & &0& \sqrt{3} \,s \\
				      & \frac{ \overline{q}}{ r^2s } & & & &&0\\
				     \end{pmatrix} 	\; \; \text{and}$$

$$ \mathcal{A}^{1,0} = \begin{pmatrix} -\del (\log r + \log s)& & & & & q&\\
				      \sqrt{3} & -\del \log r  & & & & &q \\
			  	      &\sqrt{5} & -\del \log s & & & &\\
			 	      & &- \sqrt{-6}&0 &  & &\\
				      & & & -\sqrt{-6}&  \del \log s& &\\
				      & & & &\sqrt{5} & \del \log r& \\
				      & & & & & \sqrt{3} &\del( \log r + \log s)\\
				     \end{pmatrix} 	. $$

In the standard coordinate $z = x + \sqrt{-1}y$, we now calculate derivatives of $\nu$ with the above correspondence. Since the coefficients of $\sigma$
are constant in $\mathcal{B}$ coordinates, $ \nabla \sigma = \mathcal{A} \sigma$. Hence, 
	\begin{align}
		\nabla_{\der{z}} \sigma &= -\sqrt{-6} u_{-1}. \\
		 \nabla_{\der{\zbar}} \sigma = &= \sqrt{-6} \, s\,u_{1}.
	\end{align}
	
Since $u_{-1}, u_1$ are isotropic split-octonions, we conclude that $ \nu_z,\nu_{\zbar}$ are isotropic. Thus, $\nu$ is weakly conformal.
We will compute the induced metric shortly and shortly and see that $\nu$ is conformal. 

We now see that $\nu$ is an almost-complex curve since $\nu_{z} $ corresponds to $ -\sqrt{-6} u_{-1}:$ 
$$ J_{\quadric} \circ d \nu \left ( \der{z} \right)  = \nu \times \nu_{z} = \sqrt{-1} \, \nu_{z}  = (d\nu \circ J_{\R^2}) \left( \der{z} \right) ,$$
where we use that $ u_{-1} \in \mathbb{E}_{\sqrt{-1}}(\mathcal{C}_{u_0} ) $ by Proposition \ref{PropBaragliaBasis}.

We can now write down a global $h$-unitary, $\hat{\tau}_{\V}$-real multiplication frame 
$$\mathcal{M}_h = (w_1,\, w_2, \,w_3,\, w_4, \,w_5,\, w_6, \,w_7).$$
That is, the linear map $(\mathcal{M} \mapsto \mathcal{M}_h) \in \Gtwo^\C$, where $\mathcal{M}$ is from \eqref{MultiplicationBasis}. This frame will be convenient to relate $f$ and $\nu$ later. 

Recall equation \eqref{ParallelConjugation} describing $\hat{\tau}_\V$. By direct calculation, we find the following imaginary split octonions are fixed by $\hat{\tau}_{\V}$.
In particular, $\V^\R = \spann_{\R} (w_i)_{i=1}^7$. 
	\begin{align} \label{HUnitaryMultiplicationFrame}
		\begin{cases}
		w_1 &= i = u_0.\\
		w_2 &= \,\frac{1}{\sqrt{2}} \, ( \, r^{1/2} u_{2} + r^{-1/2} u_{-2} \, ) \\ %  \;   \; \; \sim j$ \\ % in the model case. \\
		w_3 &=  - \, \frac{ \sqrt{-1} }{\sqrt{2}} ( \, r^{1/2} u_{2} - r^{-1/2} u_{-2} \, ) \\ % $ \; \; $\sim k$% in the model case.\\
		w_4& = \,\frac{1}{\sqrt{2}} \, ( \, s^{1/2} u_{1} + s^{-1/2} u_{-1} \, ) \\% \; ~ \; $\sim l$% in the model case.
		w_5& = \,\frac{ \sqrt{-1} }{\sqrt{2}} \, ( \, s^{1/2} u_{1} - s^{-1/2} u_{-1} \, ) \\% $ \; ~ \; $\sim li$% in the model case. 
		w_6& = -\,\frac{1}{\sqrt{2}} \, ( \, (rs)^{1/2} u_{3} + (rs)^{-1/2} u_{-3} \, ) \\ % \; ~ \; $\sim lj$% in the model case.
		w_7 &= \,\frac{ \sqrt{-1} }{\sqrt{2}} \, ( \, (rs)^{1/2} u_{3} - (rs)^{-1/2} u_{-3} \, ) \\ % $ \; ~\; $\sim lk$% in the model case.
		\end{cases} 
	\end{align}
One calculates that $\mathcal{M}_h$ is $h$-unitary using the expression \eqref{HitchinMetric} for $h$. Moreover, $\mathcal{M}_h$ is a \emph{multiplication frame} for $\imoct$. 
To see this, observe following relations for $\mathcal{M} = (i, \, j, \, k, \, l, \, li,\, lj, \, lk)$: 
	\begin{align}\label{StandardMultiplicationFrame} 
		\begin{cases}
		m_1 &= i \\
		m_2 &= \,\frac{1}{\sqrt{2}} \, ( \, u_{2} + u_{-2} \, ) \; = j \\ 
		m_3 &=  - \, \frac{ \sqrt{-1} }{\sqrt{2}} ( \, u_{2} -u_{-2} \, ) = k \\
		m_4 &= \,\frac{1}{\sqrt{2}} \, ( \, u_{1} + u_{-1} \, ) = l \\
		m_5 &= \,\frac{ \sqrt{-1} }{\sqrt{2}} \, ( \, u_{1} - u_{-1} \, ) = li \\
		m_6 &= -\,\frac{1}{\sqrt{2}} \, ( \, u_{3} + u_{-3} \, ) = lj \\ 
		m_7 &= \,\frac{ \sqrt{-1} }{\sqrt{2}} \, ( \, u_{3} - u_{-3} \, ) = lk
		\end{cases}
	\end{align} 
Now, observe that $ H^{-1/2} \mathcal{M} = \Mh$, where $H$ is the matrix representative of $h$ from \eqref{HitchinMetric}
in the basis $(u_i)_{i=3}^{-3}$. For example, set $g:= H^{-1/2}$ and $g \cdot u_{+1} = s^{1/2} u_{+1}, \; g \cdot u_{-1} = s^{-1/2} u_{-1}$. Hence, 
$g \cdot m_4 = w_4$ by the linear relations \eqref{StandardMultiplicationFrame}, \eqref{HUnitaryMultiplicationFrame}.

On the other hand, since  $H  \in \Gtwo$, it follows that 
$$H^{-1/2} = \diag (\, (rs)^{1/2}, \; r^{1/2},\; s^{1/2}, \;1,\; s^{-1/2}, \; r^{-1/2}, \; (rs)^{-1/2} \,) \in \Gtwo$$ 
as well by Proposition \ref{PropBaragliaBasis} part (5). 
Hence, $H^{-1/2} \in \Gtwo$ maps the standard multiplication basis $\mathcal{M}$
to the multiplication frame $\mathcal{M}_h  $. 

We now compute the real tangent vectors $\nu_x, \nu_y$.  In the $\hat{\tau}_{\V}$-real coordinates, $(w_i)_{i=1}^7$ from \eqref{HUnitaryMultiplicationFrame}: 
\begin{align}
	\sigma_x &= \nabla_{\der{z}} \sigma + \nabla_{\der{\zbar}} \sigma = (\varphi_0+ \varphi^{*h}_0 )\sigma= -\sqrt{-6}\, u_{-1} +\sqrt{-6} \, s \, u_{+1} = \sqrt{12s} \, \; w_5\\
	\sigma_y &= \sqrt{-1} \left(\nabla_{\der{z} } \sigma - \nabla_{\der{\zbar} } \sigma \right)  =  \sqrt{-1} (\varphi_0 - \varphi^{*h}) \sigma  = \sqrt{-1} ( -\sqrt{-6}u_{-1} -\sqrt{-6}\, s u_{+1})
				=  \sqrt{12s} \, \; w_4.
\end{align}
Since $q_{\imoct}(w_4) = -1 = q_{\imoct}(w_5) $ are unit time-like vectors, the pullback metric on $\nu$ is  
\begin{align}\label{AC_InducedMetric} 
	g_{\nu} = -12 s\, |dz|^2. 
\end{align}
In particular, since $s > 0$ is positive everywhere by definition, $\nu$ is conformal and timelike. 

 \subsection{Harmonic Map Sequence of $\nu$} \label{HarmonicMapSequence} 
In this subsection, we explain one way to recover the harmonic metric $h$ in $\V$ from the almost-complex curve $\nu$. 
 The key to the correspondence is the \emph{harmonic map sequence}, which we very briefly recall. 
We include this subsection for two reasons. First, while Baraglia explains the connection between minimal surfaces in quadrics and affine Toda equations in \cite[Section 2.4]{Bar15}, this construction is never explicitly applied to the almost-complex curve $\nu$. Secondly, we will need the harmonic map sequence to relate $f$ and $\nu$ in the Section \ref{GeometricCorrespondence}. \\

 Following \cite{Bar15, BPW95,BVW94,BW92, EW83}, we define the \textbf{harmonic map sequence} of a harmonic map
$\psi: \Sigma \rightarrow \CP^n$ from a Riemann surface $\Sigma$ (in local coordinates) by first taking a local lift $\widetilde{\psi}: U \rightarrow \C^{n+1}$ such that $\proj_{\widetilde{\psi}} \; (\widetilde{\psi}_{\zbar}) = 0$. One then defines 
 \begin{align}\label{HarmonicSequence} 
 	\widetilde{\psi_{i+1}} = \mathsf{proj}_{\widetilde{\psi_i}^\bot}( \; (\widetilde{\psi_i})_z ) = (\widetilde{\psi_i})_z - \frac{ \langle  \, (\widetilde{\psi_i})_z, \,  \widetilde{\psi_i }  \rangle \,}{ h_i } \widetilde{\psi_i},
\end{align}
where $h_i =  \; \langle\widetilde{\psi_i}, \widetilde{\psi_i} \rangle $ and $\langle \cdot, \cdot \rangle $ denotes a hermitian metric on $\C^{n+1}$.
As long as $\widetilde{\psi_i} \nequiv 0$, the map $\psi_{i+1}$ into $\CP^n $ can be extended across any zeros and
gives a well-defined map globally by $\psi_{i+1} = [ \widetilde{\psi}_i]$ that is also harmonic \cite{BVW94}. If $h_i$ is not identically zero, the singularities are isolated \cite{BW92}.

We now apply the harmonic map sequence to $\nu$. First, observe that the map $[\nu]: \C \rightarrow \Stwofour $
is harmonic since $\nu_{z\zbar} \equiv 0 \,(\text{mod}\, \nu)$. 
Similarly, $\nu$ is harmonic
when included into $\RP^6$ or $\CP^6$. 
We note that since $ q_{\Oct'}^\C( \nu) = +1$, we have $ \nu_{z}, \nu_{\zbar} \; \bot \nu$, where the orthogonality is with respect to the hermitian metric $\langle u,v \rangle = q_{\Oct'}^\C(u, \overline{v} )$. Thus, the current situation special as the harmonic map $[\nu]:\C \rightarrow (\imoct)^\C$ to which we run the construction comes with a canonical lift
$\nu: \C \rightarrow (\imoct)^\C$ that satisfies $\proj_{[\nu]}(\nu_{\zbar}) = 0$. Hence, we hold in our hands the next element of the harmonic map sequence already. We define the (local) harmonic map sequence by setting $ \widetilde{\nu}_0:= \nu$ and inductively defining   
the (truncated) sequence $(\widetilde{\nu_0}, \, \widetilde{\nu_1}, \, \widetilde{\nu_2}, \, \widetilde{\nu_3} )$ via \eqref{HarmonicSequence}. In the construction, we must use
the hermitian metric $\langle u,v \rangle \, = q_{\Oct'}^\C(u, \overline{v})$. 
For the truncated sequence $(\widetilde{\nu_0}, \, \widetilde{\nu_1}, \, \widetilde{\nu_2}, \, \widetilde{\nu_3} )$, we are singularity-free in the sense that $h_i < 0$ or $h_i > 0$
for $i \in \{1,2,3\}$ as shown in Proposition \ref{Prop:HarmonicMapSequence} below. 
We recall that by $\nabla$-parallel translation, each map $\widetilde{\nu}_i: \C \rightarrow \imoct^\C$ corresponds uniquely to a section $\sigma_i \in \Gamma(\C, \V)$. The harmonic sequence of $\nu$ is seen in the bundle $\V$ as follows:

\begin{proposition}[Harmonic Map Sequence] \label{Prop:HarmonicMapSequence} 
The harmonic map sequence of $\nu$ is given by 
\begin{itemize} 
	\item $\sigma_1 = -\sqrt{-6} u_{-1} $ and $q_{\C}(\nu_1) = 0$ 
	\item $\sigma_2 = -\sqrt{-30} u_{-2} $ and $q_{\C}(\nu_2) = 0$. 
	\item $\sigma_3 = -\sqrt{-30}(q u_{3} + \sqrt{3} u_{-3} ) $. Hence, $ q_{\Oct'}^\C(\widetilde{\nu}_3, \widetilde{\nu}_3) = 60\sqrt{3} \; q.$
\end{itemize} 
Moreover, the invariants $h_i$ are given by 
\begin{itemize}
	\item $h_1 = -6 s < 0. $
	\item $h_2 = 30 r > 0. $
	\item $h_3 = -30(\frac{|q|^2}{rs} + 3rs ) < 0.$
\end{itemize} 
\end{proposition} 

\begin{proof}
These calculations are straightforward. Since $\sigma = (0,0,0,1,0,0,0)$, we directly compute $\sigma_{1} = \nabla_{\der{z}}\sigma = \mathcal{A} \sigma = \varphi \sigma = -\sqrt{-6}u_{-1}$. 
Similarly, 
$$\sigma_{2} = \proj_{u_{-1}^\bot} \left ( \nabla_{\der{z}} \sigma_1 \right) = -\sqrt{-6} \proj_{u_{-1}^\bot} \left ( \mathcal{A} u_{-1}  \right)  =  \proj_{u_{-1}^\bot} \left( -\sqrt{-30} u_{-2} - \sqrt{-6} \deriv{s}{z}  \, u_{-1} \right) = -\sqrt{-30} u_{-2} .$$
The calculation of $\sigma_{3}$ is similar. 

Next, we need that $\hat{\tau}_\V$ is the $\nabla$-parallel version of complex conjugation. That is, $\nu_i \leftrightarrow \sigma_i$ 
implies that $\overline{\nu}_i$ corresponds to $\hat{\tau}_\V( \sigma_i)$. 
We compute using \eqref{ParallelConjugation} $\hat{\tau}_\V(\sigma_1) = \sqrt{-6} s \, u_{1},\; \hat{\tau}_\V(\sigma_2) = \sqrt{-30} r \, u_{2}$,
and $\hat{\tau}_\V(\sigma_3) = \sqrt{30} \left( \frac{\overline{q}}{rs} u_{-3} + \sqrt{3} rs \, u_{3} \right)$. Then since $ q_{\Oct'}^\C(u_i, \, u_{-i} ) = \mathsf{sgn}(i) \in \{-1, +1\}$ by Proposition \ref{PropBaragliaBasis} part (1), we get the expressions for each $h_i$. 
\end{proof} 

Proposition \ref{Prop:HarmonicMapSequence} says the defining components $r,s$ of the harmonic metric $h$ are recovered from the harmonic sequence invariants $h_1, h_2$. 

We now remark on a way to recover $q$, without having to compute the harmonic map sequence. 
Proposition \ref{Prop:HarmonicMapSequence} says that $(\widetilde{\nu}_1 \times \widetilde{\nu}_2) $ corresponds to $ -\sqrt{-6} u_{-1} \times  -\sqrt{-30} u_{-2} $, which is $Cu_{-3}$, up to a fixed scalar $C$. By Proposition \ref{PropBaragliaBasis} part (1), we find $(\widetilde{\nu}_1 \times \widetilde{\nu}_2) \cdot \widetilde{\nu}_3 = C' \, q$, up to a fixed constant $C'$. We argue now that, in fact, $(\nu_{z} \times \nu_{zz}) \cdot \nu_{zzz} = C'' \, q$ for some constant $C''$. Since $\widetilde{\nu}_2  = \nu_{zz} + f \,\nu_{z}$ for some function $f$, we see that $\widetilde{\nu}_2 \times \widetilde{\nu}_1 = \nu_{zz} \times \nu_z$. Moreover, $\widetilde{\nu}_3 = \nu_{zzz} + w$ for some function $w \in \spann_{\C} \langle \nu_{z}, \nu_{zz} \rangle$. 
In particular, this means $(\widetilde{\nu}_1 \times \widetilde{\nu}_2) \cdot \widetilde{\nu}_3 = (\widetilde{\nu}_1 \times \widetilde{\nu}_2) \cdot \nu_{zzz}$. Combining these facts, we have shown that $(\nu_z \times \nu_{zz} ) \cdot \nu_{zzz} = C'' q(z)$ recovers $q$ up to a fixed constant $C''$. Alternatively, using the harmonic map sequence, Proposition \ref{Prop:HarmonicMapSequence} shows $q_{\Oct'}^\C(\widetilde{\nu}_3, \widetilde{\nu}_3) = C\, q$ recovers the sextic differential, as in (Proposition 2.4.1 \cite{Bar10}). 

\begin{remark} In the language of Baraglia \cite[Definition 2.4.1]{Bar10}), $\nu$ is \emph{superconformal}. The map $\nu$ has \emph{isotropy order} 2 since $q_{\Oct'}^\C(\widetilde{\nu}_i) = 0$ for $i \in \{1,2\}$ and $q_{\Oct'}^\C(\widetilde{\nu}_3) \nequiv 0$. The reader may consult \cite{EW83} for further discussion of isotropy,
focused on the case of ``totally isotropic'', or \emph{superminimal} harmonic maps $\phi: \Sigma \rightarrow \CP^n$.  
\end{remark} 
We can now show that these harmonic maps are \emph{linearly full} when $q$ is a non-constant polynomial.

\begin{definition}\cite{EW83} 
A map $f: M \rightarrow \RP^n$ is \textbf{linearly full} when $\im \,f$ is not contained a proper projective subspace of $\RP^n$.  
\end{definition} 

In our case, linear fullness of $\nu$ is equivalent to the nonexistence of a global orthogonal vector $N \in \imoct$ such that $N \,  \bot\, \nu(z) $ for all $z \in \C$.
We later note in Remark \ref{GlobalPolar} that in the case $q$ is constant, $\nu_{q}$ has a timelike global orthogonal line, so the proof below handles only
the non-constant case. 

\begin{proposition}[Linear Fullness]\label{LinearlyFull} 
Let $\nu: \C \rightarrow \hatStwofour$ be the almost-complex curve associated to polynomial $q \in H^{0}(\K^6_\C)$. Then $\nu$ is linearly full 
if and only if $q$ is non-constant.
\end{proposition} 

\begin{proof}
We include $\hatStwofour \hookrightarrow \RP^n \hookrightarrow \CP^n$. Now, we recall some classical theory on harmonic maps in $\CP^n$. Define $V(z) := \spann_{\C} < \nu(z), \nu_i(z) , \overline{\nu_{i}}(z) >$ for $i \in \{1,2,3\}$, where $\nu_i$ comes from the harmonic map sequence. On the other hand, consider $W:=  \bigcup_{z \in \C} \spann_{\C} < \nu(z) > $ the smallest subspace of $\C^{7}$ containing the image of $\nu$ in $\CP^6$. Then $W = V(z)$ for $z \in \C$ except possibly at a set of isolated singularities \cite[Lemma 1.1, 1.2]{BW92}. Hence, to prove linear fullness, it suffices to see $\dim V(z) = 7$ almost everywhere.

By the proof of Proposition \ref{Prop:HarmonicMapSequence}, we see that $\dim V \geq 6$ everywhere. Moreover, $\dim V(z) = 6$ if and only if $\nu_3(z)$ and $\overline{\nu}_{3}(z)$ are linearly dependent. We now recall that up to a fixed constant, $\sigma_{3} = qu_{3} + \sqrt{3}u_{-3}$ and $\sigma_{-3} = \frac{ \overline{q}}{rs} \, u_{-3} + \sqrt{3}rs \, u_{3}$. 
Thus, $\sigma_{3},  \sigma_{-3}$ are linearly dependent exactly when $ \sigma_{3} = \zeta \sigma_{-3}$ for some function $\zeta: \C \rightarrow \C$. Equating $\imoct^\C$ components, we see $\sigma_3, \sigma_{-3}$ are linearly dependent only if $ q = \zeta \sqrt{3}rs \;(1)$ and $ \sqrt{3}= \zeta \frac{ \overline{q}}{rs} \;(2)$. These equations (1), (2) imply $\overline{q} = \frac{1}{\zeta} \sqrt{3}rs$, so that
 $\frac{1}{\zeta} = \overline{\zeta}$. Thus, $|\zeta|^2 = 1$. Hence, if $\sigma_3$ and $ \sigma_{-3}$ are linearly dependent, then $|q|^2 = 3r^2s^2$. 

Suppose $\nu$ is not linearly full. Then $\dim V = 6$ almost everywhere and hence $|q|^2 = 3r^2s^2$ almost everywhere. Since $|q|^2$ is continuous, $|q|^2 = 3r^2s^2$ identically. We show $q$ is a constant. Substituting $|q|^2 = 3r^2s^2$ into \eqref{HitchinsEquations_rs}, we find $\log(rs) $ is harmonic. Recall that $r, s > 0$ and $r =e^{\frac{1}{2}(u_1 - u_2)}$ and $ s = e^{u_2}$. 
Then the existence and uniqueness results from Theorems \ref{ExistenceTheorem} and \ref{UniquenessTheorem} give asymptotics for $r,s$. Indeed, the sub, super-solutions tell us $ r \asymp |q|^{2/3}$ and $s \asymp |q|^{1/3}$, where $f \asymp g$ denotes $f \in O(g)$ and $g \in O(f)$. Hence, $\log(rs) \asymp \log(|q|)$ is a constant by Liouville's theorem. Thus, $q$ is a constant as well. 
\end{proof} 

We saw earlier that $r,s$ are recovered from the harmonic map sequence invariants $h_1,h_2$. We now note that Hitchin's equations appear in a new context here
for $h_1, h_2$. Indeed, for a harmonic map $\phi: S \rightarrow \mathbb{CP}^n$ of isotropy order $\geq i$, we have the equation \cite[Proposition 2.4.1]{Bar10} or \cite[Equation (1.15)]{BW92}
\begin{align} \label{HarmonicMapSequenceEquation} 
	( \log h_i)_{z\zbar} = \frac{h_{i+1}}{h_i} - \frac{ h_i}{h_{i-1}} .
\end{align} 
Since $\nu$ has isotropy order 2, we can apply these equations to $h_1, h_2$. 

We find that the equations \eqref{HarmonicMapSequenceEquation} for $i \in \{1,2\}$ are Hitchin's equations \eqref{HitchinsEquations_rs}: 
\begin{align}
	(\log h_1)_{z\zbar} = (\log s)_{z\zbar} &= -5 \frac{r}{s} + 6s =  \frac{h_{2}}{h_1} - \frac{ h_1}{h_{0}} . \\
	(\log h_2)_{z\zbar} = (\log r)_{z\zbar} &= 5 \frac{r}{s} - \frac{|q|^2}{r^2 s} - 3 s = \frac{h_{3}}{h_2} - \frac{ h_2}{h_{1}} .
\end{align} 

Note that $h_0 := +1$ by convention, since $\nu \in Q_+(\imoct)$. However, this construction still does not \emph{directly}
explain how we obtain the almost-complex curve $\nu$ by solving these equations. To this end, we now give an interpretation of \eqref{HitchinsEquations_rs} 
as integrability equations for a particular frame field of $\nu$.

\subsection{Toda Frames for $\nu$ and for $\V$}\label{Toda}

We first relate the affine Toda field equation to the Higgs bundle. 
Write $h= e^{-2\Omega}$ uniquely for a positive-definite matrix $\Omega \in \h$, as in \cite[Proposition 2.2.1]{Bar10}. 
Unraveling Hitchin's equations as equations for $\Omega$, rather than $h$, we recover the \emph{affine Toda field equations} for $(\Omega, q)$ \cite[Equation (20)]{Bar15}
as originally explained by Baraglia \cite{Bar15}. We clarify here that how to consider $\Omega$ as a frame field for $\nu$ and reinterpret the Toda field equations as integrability conditions. 

Start with equation \eqref{HitchinsEquation1} concretely in $\g_2$. We use
the fixed Chevalley generators $e_{\alpha}, t_{\alpha}$ from Section \ref{RootSpace} here, a CSA $\h < \g$ with roots $\Delta$, along with the Kostant P3DS $\fraks$ from \ref{KostantTheory}. Recall $\varphi = \widetilde{e} + qe_{\gamma} = \sqrt{3} e_{-\beta} + \sqrt{5} e_{-\alpha} + q\, e_{\gamma}$ and $\varphi^{*h} = \sqrt{3}\, s\, e_{\beta}  + \sqrt{5} \frac{r}{s} \,e_{\alpha}+ \frac{ \overline{q}}{r^2 s} \, \, e_{-\gamma}$, using that $e_{-\alpha} = \overline{e_{\alpha}}^T$.  
Hence, equation \eqref{HitchinsEquation1} becomes
\begin{align}
	(-2 \Omega)_{z\zbar} = ( \log H)_{z\zbar} = [\varphi, \varphi^{*h}] =- 3s \,t_{\beta}  - 5\, \frac{r}{s} \,t_{\alpha}+  \, \frac{|q|^2}{r^2 s}\, t_{\gamma},
\end{align} 
using the relation $[e_{\delta}, e_{-\delta}] = t_{\delta}$ for $\delta \in \Delta$.
Writing $\varphi^{*h} =  \Ad_{e^{2\Omega}}(\overline{\varphi}^T)$, the previous equation becomes: %and using $\Ad \circ \exp = \exp \circ \ad$,
%one finds $ F_{h} + [\varphi, \varphi^{*h} ] = 0$ becomes the following equation:
\begin{align}\label{AffineTodaFieldEquation} 
	2\Omega_{z\zbar} = \sum_{i=1}^2 k_i \, e^{2 \alpha_i(\Omega) } \, t_{\alpha_i} + |q|^2 e^{-2\gamma (\Omega) } t_{-\gamma} ,
\end{align} 
where $ \alpha_1: = \alpha,\; \alpha_2 := \beta$ are the primitive roots for $\Delta$, $k_1 := 5, \, k_2 : = 3$, and $t_{\sigma} \in \h$ is the co-root to $\sigma$. 
Equation \eqref{AffineTodaFieldEquation} is the $\g_2$ \emph{affine Toda field equation} \cite{Bar15, BPW95}.\footnote{We remark that in \cite{BPW95},
the authors refer to a Toda frame only away from the zeros of $q$ and after taking a local \emph{natural coordinate} for $q$, in which $q \equiv 1$.}
Note that $\Omega \in \spann_{\R} < t_{\alpha}, t_{\beta} > $ satisfies $\hat{\rho}(\Omega) = - \Omega$; yet, this reality condition
entails that the frame $e^{\Omega}$ is $\hat{\tau}$-real, as we explain shortly. 

We now recover the Toda frame $\Omega$ from the almost-complex curve $\nu$, using the harmonic map
sequence $(\,\widetilde{\nu}_0, \, \widetilde{\nu}_1, \, \widetilde{\nu}_2, \,\widetilde{\nu}_3) $. Define $ F_i := \frac{\widehat{\nu}_i}{ |h_i| ^{1/2} }$,
where again $h_i = \, q_{\imoct}^\C( \widetilde{\nu}_i, \overline{\widetilde{\nu}_i} ) $.  Then define the frame 
\begin{align}\label{TodaFrameConstruction} 
	F_{\nu} := \left ( -\frac{1}{\sqrt{2}} \overline{F_1 \times F_2}, \; -\sqrt{-1} \;  \overline{F_2}, \; -\sqrt{-1} \; \overline{F}_1, \; \nu, \;\sqrt{-1} F_1, \; \sqrt{-1} F_2,  \; -\frac{1}{\sqrt{2}} F_1 \times F_2 \right). 
\end{align} 
More precisely, writing $F = (f_i)_{i=3}^{-3}$ in column-vector form, $F$ is the $\C$-linear transformation 
$ u_i \mapsto f_i$ for $\mathcal{B} = (u_i)_{i=3}^{-3}$ the basis \eqref{BaragliaBasis}. As before with $\nu$ and $\sigma$, the frame $F$ corresponds to a section
$\hat{F}_\nu \in \Fr(\V)$ uniquely. Using Proposition \ref{Prop:HarmonicMapSequence}, one calculates $\hat{F} \in \Fr(\V)$ is
\begin{align}
 	\hat{F_\nu} =  \left ( \sqrt{rs} \, u_3, \;\sqrt{r} \, u_2, \;\sqrt{s} \, u_1, \; u_0, \;  \frac{1}{\sqrt{s}}\, u_{-1}, \frac{1}{\sqrt{r}}\, u_{-2}, \, \frac{1}{\sqrt{rs}} \, u_{-3} \right ) = e^{\Omega}.
\end{align} 
In particular, $F \in \Gtwo^\C$. Hence, the construction \eqref{TodaFrameConstruction} recovers the Toda frame for $\V$ as a moving frame for $\nu$. We now explain
how the Maurer-Cartan integrability equation for $\omega_{\nu} = F^{-1}_{\nu} dF_{\nu}: \C \rightarrow \g_2$ coincides with Hitchin's equations. 
The argument is similar to the proof of \cite[Proposition 2.3.1]{Bar10}, but is motivated from a different perspective, so we provide it for completeness. 

A key object needed here is $\tau := \exp( \frac{ 2\pi i }{6} \, x) \in \Gtwo$, where $x = \diag(3,2,1,0,-1,-2,-3) \in \h$ is a grading element of $\g_2$ with respect
to the base $\Pi = \{\alpha, \beta\}$. That is, $[x, e_{\alpha} ] = \mathsf{height}(\alpha) \, e_{\alpha}$. It follows that $ \Ad_{\tau}( e_{\alpha} ) = e^{\frac{2\pi i}{6} \, \mathsf{height}(\alpha)\, } e_{\alpha}$. Since the longest root $\gamma = 2 \alpha + 3\beta$ has height 5,
we get an eigen-decomposition of $\g_2$ under the action of $\Ad_{\tau}$ into\footnote{Since $\zeta = e^{\frac{2\pi k}{6}}$ has $\zeta^{5} = \zeta^{-1}$, we can use index from -1 to 4
rather than 0 to 5. This choice makes the grading components appear symmetric in the following discussion.} 
\begin{align}
	\g_2 = \bigoplus_{i=-1}^4 \g_{\tau, k},
\end{align} 
where $\g_{\tau, k} := \{ X \in \g_2 \; | \; \tau(X) = e^{ \frac{2\pi i\, k}{6} } \; X \}$. Moreover, $[ \g_{\tau, j}, \g_{\tau, k} ] \subseteq \g_{\tau, j+k}$. We will decompose the Maurer-Cartan
equation into $\g_{\tau, i}$ components. First, we compute $\omega_\nu $. In the standard (Higgs) frame $\mathcal{B}$, we have $\mathcal{B}^{-1} d\mathcal{B} = \omega_{\mathcal{B}} = (D_H + \varphi )dz + \varphi^{*h} d\zbar$. The transformation equation for Maurer-Cartan shows
\begin{align} 
	\omega_{\nu} =  \left( - \Omega_{z} + \Ad_{e^{-\Omega}} (\varphi) \right) dz + \left ( \Omega_{\zbar} + \Ad_{e^{-\Omega} }(\varphi^{*h} ) \right) d\zbar. 
\end{align} 
Note that $\g_{\tau,+1}= \bigoplus_{\sigma \in \{\alpha, \beta, -\gamma \} } \g_{\sigma} $ and $\g_{\tau,-1}= \bigoplus_{\sigma \in \{-\alpha, -\beta, +\gamma \} } \g_{\sigma} $.
In particular, one finds $\varphi = \widetilde{e} + qe_{\gamma} \in \g_{\tau, -1}$ and $\varphi^{*h} \in \g_{\tau, +1}$. It follows that  
$\Ad_{e^{-\Omega} }(\varphi) \in \g_{\tau, -1}, \; \Ad_{e^{-\Omega} }(\varphi^{*h}) \in \g_{\tau, +1}$ as well. 
Now, decompose $\omega_{\nu}= (A_0 + A_{-1} ) dz + (B_0 + B_{+1}) d\zbar$ with $A_i, B_i \in \g_{\tau, i}$. That is,
$$A_0 =- \Omega_{z}, \; A_{-1}=- \Omega_{z},\; B_0 =  \Omega_{\zbar}, \; B_1 =  \Ad_{e^{-\Omega} }(\varphi^{*h} ).$$  
Then the equation $d\omega_{\nu} = -\frac{1}{2}[ \omega_\nu, \omega_\nu]$
decomposes into $\g_{\tau, -1}, \g_{\tau, 0}, \g_{\tau, +1}$ components, respectively, as follows: 
\begin{align}
	\begin{cases} 
		\; (A_{-1} )_{\zbar} + [B_0, A_{-1}] &= 0. \\
		\; (A_0)_{\zbar} - (B_0)_z &= [A_{-1}, B_{+1} ]. \\
		\; (B_{+1})_{z} + [ A_0, B_{+1} ] & =0. 
	\end{cases} 
\end{align} 
Applied explicitly to the components of $\omega_{\nu}$, these equations in order are: 
\begin{align}
	\begin{cases}\label{ConcreteMC} 
	\;\; (\Ad_{e^{-\Omega}}(\varphi) \, )_{\zbar} + [ \Omega_{\zbar}, \Ad_{e^{-\Omega}}(\varphi) ]= 0 \\ 
	\;\; -2\Omega_{z\zbar} = [ \Ad_{e^{-\Omega}}(\varphi),  \Ad_{e^{-\Omega}}(\varphi^*)] = \Ad_{e^{-\Omega}} [\varphi, \varphi^*] =  [\varphi, \varphi^*] \\
	\;\; (\Ad_{e^{-\Omega}} (\varphi^{*h} )\, )_{z}  - [\Omega_{z}, \Ad_{e^{-\Omega}}(\varphi^*)\, ] = 0 .
	\end{cases}
\end{align} 
Note that in the frame $\hat{F}_{\nu}=  e^{\Omega}$, the hermitian metric $h$ is the identity. Hence, the Chern connection 1-form is the unitary 1-form 
$D_H = (-\del \Omega + \delbar \Omega)$ in this frame. It follows that the Maurer-Cartan equations \eqref{ConcreteMC} are, in order, equivalent to: 

\begin{align*} 
	\begin{cases}
	 \;\;\nabla^{0,1}_{\delbar, h} \varphi = 0\\
	\;\; F_h = [\varphi, \varphi^{*h} ] \\ 
	 \;\;\ \nabla^{1,0}_{\delbar, h} \varphi^{*h} = 0.
	 \end{cases}
\end{align*} 

 In particular, if we start with $\varphi$ holomorphic, then 
then $\omega_{\nu}$ is integrable if and only if $ \nabla^{1,0}_{\delbar, h} \varphi^{*h} = 0$ and $F_h = [\varphi, \varphi^{*h}]$, which happens if and only 
if $\nabla = \nabla_{\delbar_E,h} + \varphi + \varphi^{*h}$ is flat. Moreover, $F_h = [\varphi, \varphi^{*h}]$ is precisely the $\g_2$ affine Toda field equation \eqref{AffineTodaFieldEquation}.

Finally, we discuss reality conditions for $\omega_{\nu}$. As $e^{\Omega}$ is $h$-unitary the compact involution $\hat{\rho}$ on $\V$ is in this frame is the standard
one $\rho := A \mapsto -\overline{A}^T$. The equation $\rho(-\del \Omega+ \delbar \Omega) = -\del \Omega + \delbar \Omega$ follows since $\overline{\Omega} = \Omega$;
this is why $\hat{\rho}(\Omega) = -\Omega$ is the correct reality condition. Observe that $\rho(A_{-1}) = -B_{+1}$, $\rho(B_{+1}) = -A_{-1}$
and also $\sigma(A_{-1}) = -A_{-1},  \; \sigma(B_{+1}) = -B_{+1}$. We conclude that $\hat{\tau}(\omega_{\nu}) = \omega_{\nu}$. Thus, $\omega_{\nu}$ defines a unique map $\mathcal{F}_{\nu}: \C \rightarrow \Gtwosplit$ with $\mathcal{F}_{\nu}^{-1} d\mathcal{F}_{\nu} = \omega_{\nu}$,
up to global $\Gtwosplit$ isometry. Moreover, this construction shows how to avoid Higgs bundles in defining $\nu$. Indeed, we can define $\nu$ associated to $q$ instead by
solving equation \eqref{AffineTodaFieldEquation} for $(\Omega, q) $ and defining $\nu:= \mathcal{F}_{\nu}(u_0)$. 

We now define a few of the maps in Figure \ref{HarmonicMapDiagram} for the main theorem in the section. 
Since $\sigma |_{\h} = \id_{\h}$, it follows that $\Fix(\hat{\tau}|_{\h} ) = \Fix(\hat{\rho}|_{\h})$. Thus, $\mathfrak{t}:= \Fix( \, \hat{\tau} |_\h )$ is a maximal abelian subalgebra of $\mathfrak{k} < \g_2'$. Let $T  <\Gtwosplit$ be generated by $\exp(\mathfrak{t})$. Then $T$ is a maximal torus in $K$. Let $\pi_{K}: \Gtwosplit \rightarrow \Gtwosplit/K$ and $\pi_T: \Gtwosplit \rightarrow \Gtwosplit/ T$ be natural projections. 
Define $ F_{T} = \pi_T \circ \mathcal{F}_{\nu}$ and $F_{K} = \pi_K \circ \mathcal{F}_{\nu}$. The integrability equations of $\omega_{\nu}$ imply the harmonicity of $F_{T}$ \cite[Proposition 2.3.1]{Bar15}. 

 \subsection{Geometric Correspondence Between $f$ and $\nu$ } \label{GeometricCorrespondence} 
In this subsection, we first define geometric model spaces $\mathscr{X}, \mathscr{Y}$ and give $\Gtwosplit$-equivariant identifications $\Gtwosplit/K \cong_{\mathbf{Diff}}\mathscr{X} $ and 
$\Gtwosplit/T \cong_{\mathbf{Diff}} \mathscr{Y}$. 
We then define a geometric lift $G_{\nu}$ of $\nu$ to $\mathscr{Y}$. 
The map $G_{\nu}$ is indeed a lift of $\nu$ as there is a natural projection $\pi_{\Stwofour}: \mathscr{Y} \rightarrow \Stwofour$ such that $\pi_{\Stwofour} \circ G_{\nu} = \nu$. 
We also define a geometric version of the map to the symmetric space by $f: \C \rightarrow \mathscr{X}$. The maps $F_K, f $ and $F_T, G_{\nu}$ are shown to be equivalent
in Theorem \ref{THM:GeometricCorrespondence}.  
 
 We now give three descriptions of the maximal compact in $\Gtwosplit$, summarizing results across the literature. 
 Here, $\mathsf{SU}(2) $ is regarded as the unit sphere in the quaternions $ \Ha$. 
 
 \begin{lemma}[The Maximal Compact in $\Gtwosplit$ ] \label{MaximalCompact} 
 The following subgroups coincide and define a copy of the maximal compact subgroup $\mathsf{SO}(4) \cong K < \Gtwosplit$: 
 	\begin{enumerate}
		\item \cite{Hel78} $H_1 = O(7) \cap \Gtwosplit$.\footnote{To be precise, by $O(7)$, we mean $O( \imoct, g)$, where $\mathcal{M}$ from \eqref{MultiplicationBasis} is an orthonormal frame for the metric $g$.} 
		\item  \cite{Yok77} $H_2 = \{ \psi \in \Gtwosplit \; |  \; \psi( \Ha) \subseteq \Ha \} = \{ \psi \in \Gtwosplit \; |  \; \psi( l \Ha) \subseteq l \Ha \}$. 
		\item  \cite{BM09} $H_3 = \mathsf{image}(\Psi)$ for the group homomorphism $ \Psi: \mathsf{SU}(2) \times \mathsf{SU}(2) \rightarrow \Gtwosplit $ with kernel $\{ (+1, +1), (-1,-1) \}$ given 
		in the splitting $\Oct' = \Ha \oplus l\Ha$ by 
		$$(p,q) \mapsto \Psi_{p,q}:= \;a+ l b\mapsto pap^{-1} + l \, (pbq^{-1} ) .$$
	\end{enumerate} 
 \end{lemma} 
 
 \begin{proof} 
 By general theory, we know (1) defines a copy of the maximal compact \cite{Hel78}. By (3), we see that $H_3 \cong (\mathsf{SU}(2) \times \mathsf{SU}(2))/ \{ \pm (1,1) \} \cong \mathsf{SO}(4)$. Thus, it suffices to see these subgroups are equivalent. 
 
 Given $\psi \in H_2$, then $\psi$ preserves the orthogonal splitting $\Ha \oplus l\Ha$.
 Let $g_{n}$ denote the standard Euclidean metric on $\R^n$. Thus, $\psi$ preserves the metric $q_{\imoct} = g_3 \oplus -g_4$ and preserves the splitting $\mathsf{Im} \Ha \oplus l \Ha$,
 so it preserves $g_7 = g_3 \oplus g_4$ too.  
 Thus, $H_2 \subseteq H_1$. Conversely, if $\varphi \in H_1$, then $\varphi$ preserves the metrics $g_3 \oplus g_4 $ and $g_3 \oplus -g_4$, so it must preserve 
 the space $l\Ha$. But $\varphi \in \Gtwosplit$ preserves $l\Ha$ if and only if it preserves $\Ha = [ \, (l \Ha)^\bot \subset \Oct' ]$. 
 
By a direct calculation using the multiplication formula \eqref{OctMultiplication} for $\Oct'$, one finds that $\Psi_{p,q} \in \Gtwosplit$. A separate direct calculation
shows that $\Psi $ is a group homomorphism. If $(p,q) \in \ker \Psi$, then one first finds $p$ commutes with all of $\Ha$ so that $p \in \R$. Hence, $p = \pm 1$.
Then one find that $q = \pm 1$ accordingly, with the signs matching. Hence, $H_3 \cong \mathsf{SO}(4)$.
 
 Let $F: \Gtwosplit \rightarrow  V_{(+,+,-)}$ be the coordinate  map $\varphi \mapsto (\varphi(i), \varphi(j), \varphi(l))$ as in Proposition \ref{StiefelTripletModel}. 
 We show that $H_2, H_3$ correspond to the same triplets under $F$.  
First, observe that $H_2 $ corresponds to the set of triplets $V_{\mathsf{Im} \, \Ha} := \{ (x,y, z) \in V_{(+,+,-)} \; | \; x, y \in \mathsf{Im} \Ha, \; z \in l\Ha\}.$
Now, also by definition, it is clear that $F(H_3) \subset V_{\mathsf{Im} \, \Ha}$. 
Conversely, take any such triplet $(x_0,y_0,z_0) \in V_{\mathsf{Im} \, \Ha}$. We unpack the 2-1 map $\mathsf{SU}(2) \rightarrow \mathsf{SO}(3) = \Aut_{\R-\mathsf{alg}}(\Ha)$
by $ p \mapsto (q \mapsto pqp^{-1} ) $ to complete the proof. Indeed, we can find
 $p \in \mathsf{SU}(2)$ such that $p i p^{-1} = x_0, p j p^{-1} = y_0$. Then we may write $z_0 = lpq^{-1}$ for a (unique) element $q \in \Ha$ because $z_0, lp \in l\Ha$. Hence,
 $F( \Psi_{p,q} ) = (x_0,y_0,z_0)$. We conclude that $F(H_3) = V_{\mathsf{Im} \Ha} = F(H_2).$ This means $H_2 = H_3$. 
 \end{proof} 
 
 As a corollary, we can describe points in the symmetric space geometrically. 
  \begin{corollary}[The Geometry of $\mathsf{Sym}(G_2')$] \label{GeometricSymmetricSpace} 
 Let $\mathsf{Sym}(\Gtwosplit) := \Gtwosplit/K$ be the $\Gtwosplit$ symmetric space. Then 
 $$\mathsf{Sym}(\Gtwosplit) \cong_{\mathbf{Diff}} \mathscr{X} := \mathsf{Gr}_{(3,0)}^\times(\imoct) = \{ \, P \in \mathsf{Gr}_{3,0}(\imoct) \, | \,P \times_{\imoct} P = P \},$$ 
 under the $\Gtwosplit$-equivariant diffeomorphism $\beta_{\mathscr{X}}: gK \mapsto g \cdot \mathsf{Im} \Ha $. 
 \end{corollary}
 
 \begin{proof}
 We can use an Orbit-Stabilizer argument. $\Gtwosplit$ acts on $ \mathscr{X}$ since it preserves $\times$. 
 By Lemma \ref{MaximalCompact}, $ \Stab( \mathsf{Im} \, \Ha) = K$. Now, given $P \in \mathscr{X}$, write $P = \spann_{\R} <x, y, x \times y>$ for $x,y \in \quadric$ that are orthogonal. We extend $(x,y)$ to $(x,y,z) \in V_{(+,+,-)}(\imoct)$ from Proposition \ref{StiefelTripletModel}. By the same Proposition,
take $\varphi \in \Gtwosplit$ such that $\varphi \cdot (i,j,k)= (x,y,z)$. Hence, $\varphi(P) = \mathsf{Im} \, \Ha$. 
 We conclude that $\beta_{\mathscr{X}}$ is a bijection. Since $\beta_{\mathscr{X}}$ is clearly $\Gtwosplit$-equivariant, the Equivariant Rank Theorem finishes the job. 
 \end{proof} 

The previous lemma gives a geometric model for $\mathsf{Sym}(\Gtwosplit)$. On the other hand, it does not geometrically explain how the metric $h$ in the bundle,
which is Euclidean on $\V^\R$, defines a map to $\mathscr{X}$. We can explain this point with the use of the following lemma. 
Here, we denote the model space of Euclidean inner products on $\R^7$ by 
$$\mathscr{M} := \{ A \in \mathsf{Mat}_7 \R \; | \; A > 0, \,A^T = A, \det A = +1\}.$$

 \begin{lemma}\label{3,0_OrthogonalMultiplicationFrame} 
Let $A \in \mathscr{M}$ be a Euclidean inner product on $\imoct$ in the basis $\mathcal{M}$. If $\mathcal{F} = (f_{i})_{i=1}^7$, $\mathcal{G} = (g_i)_{i=1}^7$ are 
$A$-orthogonal multiplication frames for $\imoct$, then 
$$ \mathcal{F}^{3,0} := \spann_{\R} <f_1, f_2, f_3 > \, = \spann_{\R} < g_1, g_2, g_3> =: \mathcal{G}^{3,0} .$$
Thus, there is a natural map $m_{(3,0)}: \mathscr{M} \rightarrow \mathscr{X}$ defined by $A \mapsto \mathcal{F}^{3,0}$. 
\end{lemma} 

\begin{proof}
Let $\mathcal{F} = (f_i)_{i=1}^7, \mathcal{G} = (g_i)_{i=1}^7$ be $A$-orthogonal multiplication frames for $\imoct$. 
Now, by definition of a multiplication frame, the linear transformation
$\psi:= (f_i \mapsto g_i ) \in \Gtwosplit$. But moreover $\psi \in \Gtwosplit \cap O(\imoct, A)$. 
As in the proof of Lemma \ref{MaximalCompact}, $\psi$ must preserve the orthogonal splitting $\mathcal{F}^{3,0} \oplus \mathcal{F}^{0,4}$, where $\mathcal{F}^{0,4} =\spann_{\R} < f_4, f_5, f_6, f_7>$. Hence, $\psi( \mathcal{F}^{3,0}) =  \mathcal{F}^{3,0}$. On the other hand, $\psi( \mathcal{F}^{3,0} ) =  \mathcal{G}^{3,0} $ by construction
and we conclude $\mathcal{F}^{3,0} = \mathcal{G}^{3,0}.$
\end{proof}

We now define a ``geometric'' version of the harmonic map in the symmetric space as a map $f: \C \rightarrow \mathscr{X}$, via the Euclidean metric $h$ on $\V^\R$.
Define $f(z)$ as the $\nabla$-parallel translation of $P_z \in \Gr_{3,0}^\times(\V^\R_z) $ to the fiber $\V_{p_0}$, for $p_0$ the origin, where $P_z$ is the (3,0) part of any $\hat{\tau}_{\V}$-real $h(z)$-unitary multiplication frame. In particular, we may use the plane 
$P =\spann_{\R} <w_1, w_2, w_3>$, in terms of our $h$-unitary frame \eqref{HUnitaryMultiplicationFrame}. We will see the compatibility of $f$
with the map $F_K: \C \rightarrow \Gtwosplit/K$ in Theorem \ref{THM:GeometricCorrespondence}.\\

Now, we move on to describing the space $\Gtwosplit/T$ geometrically, where $T<K$ is a maximal torus. Recall here that $\mathcal{C}_x := (y \mapsto x \times y)$. We consider the
space
$ \mathscr{Y} $ of tuples $(x, P, Q, R)$ such that $x \in \quadric$, $P,Q, R$ are 2-planes with signatures $(2,0), (0,2), (0,2)$, respectively, and each 2-plane is set-wise preserved by $\mathcal{C}_x$, and the following splitting is orthogonal:
$$\imoct = \R \{x\} \oplus P \oplus Q \oplus R. $$
 For such a splitting, one finds that $R$ is necessarily given by $R = P \times Q$. In fact, this is a corollary of the transitivity 
 of the action described in the upcoming Lemma \ref{G/T}. Hence, such a decomposition is specified by either triplet $\{ x, P, Q\}$ or $\{x, P, R \} $.

We define a basepoint\footnote{One can think of $p_0$ as a ``natural''
basepoint for $\Gtwosplit/T$, insofar as one accepts the multiplication frame $\mathcal{M} = (i, j, k, l, li, lj, lk)$ as a natural $\imoct$ basis.} 
\begin{align} \label{Basepoint} 
 p_0 = \left ( \, i, \, \spann_{\R} \langle j,k \rangle,  \, \spann_{\R} \langle l, li \rangle, \; \spann_{\R} \langle lj, lk \rangle \, \right) \in \mathscr{Y}. 
\end{align} 
We set $x_0 = i$, $P_0 = \spann_{\R} \langle j, k \rangle, \; Q_0 = \spann_{\R} \langle l, li \rangle, \; R_0 = P_0 \times Q_0 = \; \spann_{\R} \langle lj, lk \rangle $.
$\Gtwosplit$ acts naturally on $\mathscr{Y}$ via the diagonal action. 

In the following Lemma, we show $ T:= \Stab(p_0) \cong_{\mathbf{Lie}} \mathsf{SO}(2) \times \mathsf{SO}(2)$ is a maximal torus in the maximal compact subgroup $K < \Gtwosplit$.
We identify $\Gtwosplit/T$ with $\mathscr{Y}$ as homogeneous $\Gtwosplit$-spaces.

\begin{lemma}[Model for $\Gtwosplit/T$.] \label{G/T}
 $\Gtwosplit/T \cong_{\mathbf{Diff}} \mathscr{Y}$ under the map $\Gtwosplit$-equivariant map 
$\beta_{\mathscr{Y}}: \,gT \mapsto g \cdot p_0$.
\end{lemma} 

\begin{proof}

We first show transitivity of the action on $\mathscr{Y}$. Take any $ p= (x, P, Q) \in \mathscr{Y}$. We show that there is $\varphi \in \Gtwosplit$
such that $\varphi \cdot p_0 = p$.  To this end, select any $y \in P$, $z \in Q$ such that $q(y) = +1$ and $q(z) = -1$. By hypothesis, 
$x, y, z$ are mutually orthogonal. Since $P$ is stable under $\mathcal{C}_x$, we have $x \times y \in P$. Since $P \, \bot  \, Q $, we conclude $(x \times y) \, \bot \, z $. 
By Proposition \ref{StiefelTripletModel}, one finds that there is $\varphi \in \Gtwosplit$ such that $\varphi(i) = x, \, \varphi(j) = y, \varphi(l) = z$. 
Since $P= \spann_{\R} \langle y, x \times y \rangle$, $Q = \spann_{\R} \langle z, x \times z \rangle$ by $C_x$-stability, it follows that $\varphi \in \Gtwosplit$ has $\varphi(P_0) = P, \; \varphi(Q_0) = Q$.
Thus, $\varphi \cdot p_0 = p$. 

We now consider $H = \Stab(p_0):= \Stab(i) \cap \Stab(P_0) \cap \Stab(Q_0) \cap \Stab(R_0)$. In fact, $H = \Stab(i) \cap \Stab(P_0) \cap \Stab(Q_0)$. Indeed, if $\varphi \in \Gtwosplit$ preserves $P_0, Q_0$, then it preserves $R_0 = P_0 \times Q_0 $ automatically. 
 
Observe that $\mathsf{Im} \Ha = \spann_{\R} \langle i, j, k \rangle= \R\{x_0\} \oplus P_0 $ is fixed by any $\psi \in H$. Hence, $H < K$ is a subgroup of our fixed maximal compact subgroup,
by Lemma \ref{MaximalCompact}. 
We can now finish the proof by describing $H$ in the Stiefel model $V_{(+,+,-)}$ from Proposition \ref{StiefelTripletModel}. Recall that $\varphi \in \Gtwosplit$
identifies as $(\varphi(i), \varphi(j), \varphi(l)) \in V_{(+,+,-)}$. Then $H < V_{(+,+,-)}$ 
identifies as the triplets of the form $(i, x, y)$ with $q(x) = +1, q(y) = -1$ and 
$ x \in P_0, \; y\in Q_0$. It then follows that the map $\Phi: H \rightarrow \mathsf{SO}(P_0) \times \mathsf{SO}(Q_0)
\cong \mathsf{SO}(2) \times \mathsf{SO}(2)$ by $\Phi( g) = (g|_{P_0}, \; g|_{Q_0} ) $ is a Lie group isomorphism. If $ \psi \in \Gtwosplit$ commutes with $H$, then 
one finds $\psi \in \tilde{H} := \Stab(P_0) \cap \Stab(Q_0)$. 
Now, $H$ is not a maximal abelian subgroup, as $\tilde{H}$ is abelian. However, $\tilde{H} = H \times \Z_2$ is disconnected,
where the $\Z_2$-generator is $\varphi \in \Gtwosplit$ corresponding to $(-i, j, l) \in V_{(+,+,-)}$. 
Note that the Lie subalgebra $\mathfrak{t} < \frakk$ corresponding to $H$ is a maximal 
abelian subalgebra of $\frakk$. 
We conclude that $T:= H$ is a maximal torus, i.e., a maximal compact connected abelian Lie subgroup of $K$. The map $\beta_{\mathscr{Y}}$ is clearly equivariant and thus a diffeomorphism by the Equivariant Rank Theorem. 
\end{proof} 

\begin{corollary}[Natural Projections] \label{NaturalProjections} 
There are natural projections $\pi_{\mathscr{X}}: \mathscr{Y} \rightarrow \mathscr{X} $, $\pi_{\quadric}:\mathscr{Y} \rightarrow \quadric$
\begin{align}
	\pi_{\mathscr{X}}:=& (x, P, Q , R) \longmapsto \R \{x\} \oplus P   \\
	\pi_{\quadric} :=& (x, P, Q, R) \longmapsto x 
\end{align} 
\end{corollary} 

Before we use the model $\mathscr{Y}$ to define a lift $G_{\nu}$ of $\nu$, we need some structural equations on almost-complex curves;
the proof is exactly the same as in the case for $\mathbb{S}^6$ \cite[Page 413]{BVW94}.

\begin{proposition}\label{StructureEquations1} 
Let $\nu: \Sigma \rightarrow \mathbb{S}^{2,4}$ be an almost-complex curve. Then $\nu$ is weakly conformal, harmonic,
and satisfies the following for $(\tilde{\nu}_0, \tilde{\nu_1}, \tilde{\nu_2}, \tilde{\nu_3} )$ its truncated local harmonic map sequence:
	\begin{align}
		\nu \times \widetilde{\nu_1} &= \sqrt{-1}  \, \widetilde{\nu_1} \\
		\nu \times \widetilde{\nu_2} &= \sqrt{-1} \,  \widetilde{\nu_2} \\ 
		\widetilde{\nu_1} \times \widetilde{\nu_2} &+ \nu \times \widetilde{\nu_3} = \sqrt{-1} \, \widetilde{\nu_3}.\label{nu3}
	\end{align} 
\end{proposition}

\begin{corollary}\label{OrthogonalityProperties} 
Under the same hypotheses as Proposition \ref{StructureEquations1}, we have \\$\nu, \, \Re \widetilde{\nu_1}, \, \Im \widetilde{\nu_1}, \, \Re \widetilde{\nu_2}, \, \Im \widetilde{\nu_2}$ 
are pairwise orthogonal under $q_{\imoct}$. 
\end{corollary} 

We now relate $\widetilde{\nu}_2$ from the harmonic map sequence to the second fundamental form $\Pi$ of $\Sigma := \im(\nu)$. Denote $\Pi^\C$ as the $\C$-bilinear extension of $\Pi$. 
Then by definition of $\widetilde{\nu_2}$, we have 
$ \Pi^\C( \der{z}, \der{z} ) = \proj_{\nu_z^\bot} (\nu_z) = \widetilde{\nu}_2$. In fact, the outputs of $\Pi$ determined by $\widetilde{\nu}_2$. 
A calculation using $\tr( \Pi )= 0$ shows that $\Pi( \der{x}, \der{x} ) = 2 \Re(\widetilde{\nu}_2)$ and $ \Pi( \der{x}, \der{y} ) = -2\Im(\widetilde{\nu}_2)$. In particular, we have
$\Pi = 2 \begin{pmatrix} \Re \widetilde{\nu}_2 & - \Im \widetilde{\nu}_2 \\ -\Im \widetilde{\nu}_2 & -\Re \widetilde{\nu}_2 \end{pmatrix} $. 
This expression shows $\Pi( X, J \, Y ) =J \Pi (X , Y )$ since $\nu \times \widetilde{\nu}_2 = \sqrt{-1}\widetilde{\nu}_2$ corresponds to $\nu \times \Im \widetilde{\nu}_2 = \Re \widetilde{\nu}_2$
and $\nu \times \Re(\widetilde{\nu}_2) = - \Im \widetilde{\nu}_2$. We call the $\nu(x)$-complex line $ P(x) =\{  \Pi(X,Y) \; | \; X, Y \in T_{\nu(x)}\Sigma \}$
the \textbf{normal line} of the surface. In particular, the two-planes  
$ \spann_{\R} \langle \Re(\widetilde{\nu_1}),\, \Im (\widetilde{\nu_1}) \rangle, \spann_{\R} \langle \Re(\widetilde{\nu_2}), \Im (\widetilde{\nu_2}) \rangle$ do not depend on choice of local lifts $\tilde{\nu}_1, \tilde{\nu}_2$ from the harmonic map sequence. Note that $h_2 > 0$ corresponds to $P$ being spacelike. \\

We now show the orthogonality and metric properties of the truncated harmonic map sequence $(\widetilde{\nu_0}, \, \widetilde{\nu_1}, \, \widetilde{\nu_2}) $
determine a lift of $\nu$ to $\mathscr{Y}$. Note that the following lemma needs no assumption on the source 
Riemann surface $\Sigma$. 

\begin{lemma}[Geometric Lift of Almost-Complex Curves] 
Let $\hat{\nu}: \Sigma \rightarrow \quadric$ be an alternating almost-complex curve. 
Define $ P = \spann_{\R} \langle \Re(\widetilde{\nu_2}), \Im (\widetilde{\nu_2}) \rangle ,$ $Q = \spann_{\R} \langle \Re(\widetilde{\nu_1}), \Im(\widetilde{\nu_1}) \rangle$,
and $R = P \times Q$. Then $\hat{\nu}$ admits the canonical lift $ G_{\hat{\nu}}: \Sigma \rightarrow \mathscr{Y} $ 
associated to the tuple $(\hat{\nu}, P, Q, P \times Q) \in \mathscr{Y}$. 
\end{lemma} 

\begin{proof}
We check that the splitting does indeed live in $\mathscr{Y}$. By Corollary \ref{OrthogonalityProperties}, the spaces $\R\{\nu\}, P, Q$ are pairwise orthogonal and $P,Q$ are $\mathcal{C}_{\nu}$-stable. By hypothesis, $P$ is spacelike and $Q$ is timelike. Thus, $R := P \times Q$ is timelike by the multiplicativity of $q$. Choosing any elements
$x \in P, y \in Q$, we find that $P = \spann_{\R} \langle x, \nu \times x \rangle, \;Q = \spann_{\R} \langle y,  \nu \times y \rangle$. Hence,
we find $R = \spann_{\R} \langle x \times y, \nu \times (x \times y) \rangle $ by the identity \eqref{DCP_Orthogonal}. Hence, $R$ is $\mathcal{C}_{\nu}$-stable. 
Clearly $(x \times y) \cdot  x = 0 = (x \times y) \cdot y =0 $. Then by Lemma \ref{CrossProductOrthogonality}, we find $(x \times y) \, \bot \, (\nu \times x)$ since $x\, \bot \, \nu$ and $y \bot \, x $ and $y \bot \nu$.
Thus, $(x \times y) \, \bot \, P$. Similarly, one finds $(x \times y) \, \bot \,(\nu \times y)$ so $(x \times y )\, \bot\,  Q$. We need only show that $R \, \bot\, \nu$ to finish the proof.
To see this, note that $(x \times y) \cdot \nu = \Omega(x, y, \nu) = - \Omega(x, \nu, y) = - (x \times \nu) \cdot y = 0$. Finally,
$\nu \, \bot \,( \nu \times (x \times y)) $ holds by definition of cross-product. Hence,  
$(x, P, Q, R) \in \mathscr{Y}$.
\end{proof} 

While not employing this geometric model, Collier and Toulisse study \emph{cyclic surfaces} in $\Gtwo^\C/T$ in \cite{CT23}. The above map $G_{\hat{\nu}}$ is
such a cyclic surface, after including $\Gtwosplit/T \hookrightarrow \Gtwo^\C/T$. They then show that this operation of lifting the almost-complex curve produces a bijection between the moduli space of equivariant alternating almost-complex curves and the moduli space of such cyclic surfaces. 

We now describe the lift of $G_{\hat{\nu}}$ corresponding to our almost-complex curve $\hat{\nu}$ explicitly in the bundle $\V$. 
We use again the correspondences $\widetilde{\nu}_i \leftrightarrow \sigma_i$. 
Recall that complex conjugation on $\widetilde{\nu}_i$ corresponds to the $\nabla$-parallel bundle endomorphism $\tau_{\V}$ \eqref{ParallelConjugation} on $\sigma_i$ in $\V$. We find that $\Re( \widetilde{\nu}_2) = ( \nu_2 + \tau_{\V}(\widetilde{ \nu}_2 ) \, )$ and $\Im(\widetilde{\nu}_2) = \sqrt{-1}( \widetilde{\nu}_2 - \tau_{\V} (\widetilde{\nu}_2) \, ) $ correspond to $ -\sqrt{15r} \; w_3$ and $ -\sqrt{15 r} \, w_2$. In particular, the normal line of the surface is given in the bundle by 
$P = \spann_{\R} \langle w_2, w_3 \rangle$. On the other hand, we calculated earlier that the (real) tangent space of $\hat{\nu}$ is given by
$ Q = \spann_{\R} \langle w_4, w_5 \rangle$. Combining all of these observations, the lift $G_{\hat{\nu}}: \C \rightarrow \mathscr{Y}$ corresponds to the following tuple  in the fibers of $\V^\R$: $(w_1, \spann_{\R} \langle w_2, w_3 \rangle, \, \spann_{\R} \langle w_4, w_5 \rangle, \, \spann_{\R} \langle w_6,w_7 \rangle )$.

We saw earlier that the map $f: \C \rightarrow \mathscr{X}$ corresponds to the sub-bundle $ \spann_{\R} \langle w_1, w_2, w_3 \rangle$ of $\V^\R$. Hence, we see how the 
map $\nu$ geometrically recovers $f$: we have $ f = \pi_{\mathscr{X} } \circ G_{\hat{\nu}}$, where $\pi: \mathscr{Y} \rightarrow \mathscr{X}$ is the natural projection in Corollary \ref{NaturalProjections}.\footnote{Similar ideas were independently discovered by Nie \cite[Corollary 4.14]{Nie24}. However,
he only recovers the harmonic map in the symmetric space, but not the lift to $\Gtwosplit/T$.} Moreover, by the other projection map $\pi_{\Stwofour}: \mathscr{Y} \rightarrow \quadric$, we find that the lift $G_{\hat{\nu}}$ recovers $\hat{\nu}$ by
$ \hat{\nu} = \pi_{\quadric} \circ G_{\hat{\nu}}$. Thus, we can summarize the situation with the following commutative diagram; the following result holds
equally well in a local chart for a compact Riemann surface $\Sigma$.  

\begin{theorem}[Geometric Relationships Between Harmonic Maps] \label{THM:GeometricCorrespondence} 
Let $\nu: \C \rightarrow \quadric$ be an almost-complex curve associated to polynomial $q \in H^0(\K^6_\C)$. Then $F_T, F_K$ are harmonic and the following diagram commutes.
\end{theorem} 

\begin{figure}[ht]
\centering 
\includegraphics[width=.4\textwidth]{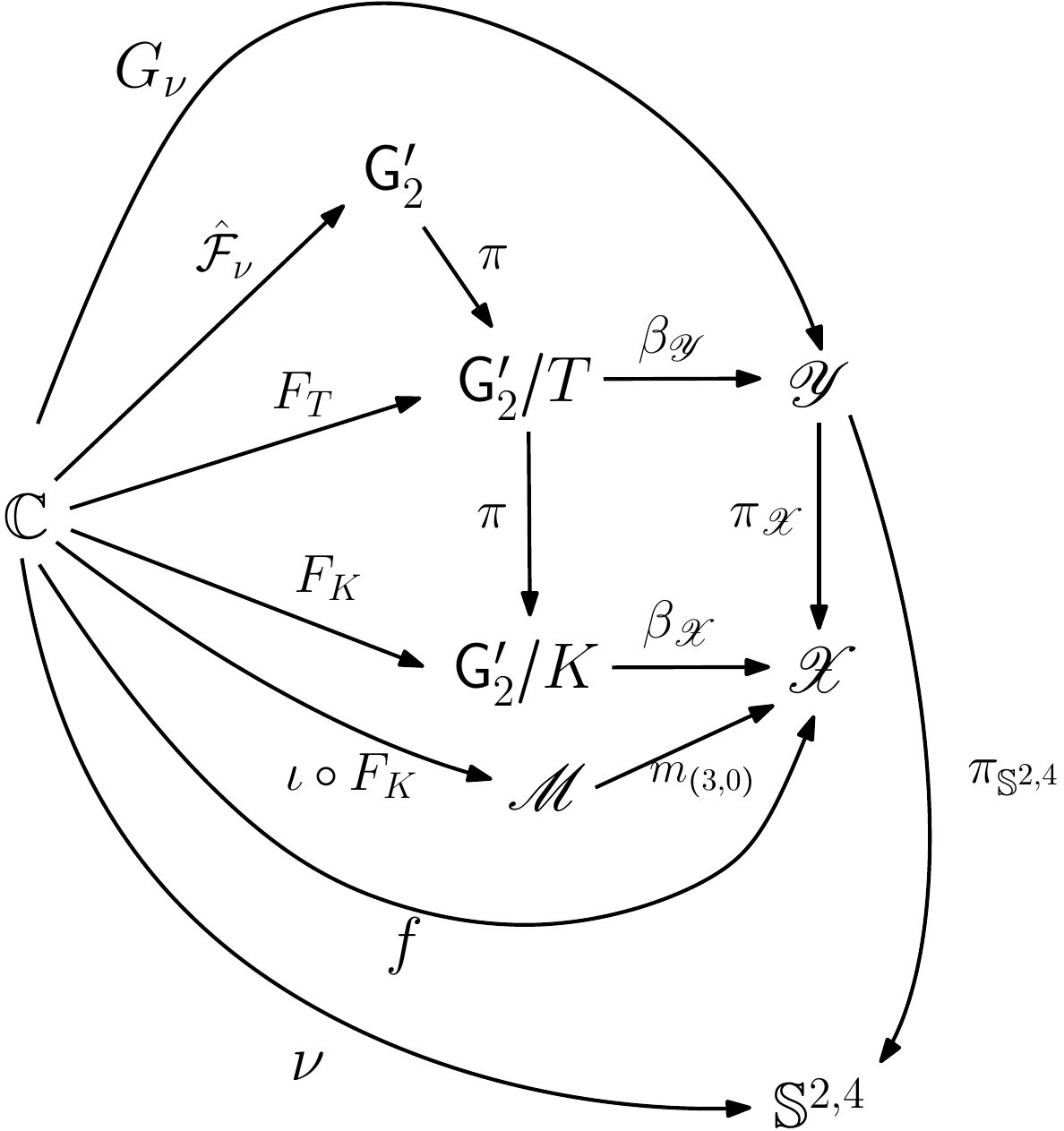}
\caption{\small{\emph{The geometric relationship between the minimal surface in the symmetric space associated to $h$ and the almost-complex curve $\nu$.}}} 
\label{HarmonicMapDiagram} 
\end{figure} 

\begin{proof}
We first verify the commutativity around $\Gtwosplit/K$ and $\mathscr{X}$. Here, $\iota: \Gtwosplit/K \hookrightarrow \mathscr{M}$
is the composition $\iota := \left(\Gtwosplit/K \hookrightarrow \mathsf{SL}_7\R/\mathsf{SO}_7\R \stackrel{\cong}{\longrightarrow} \mathscr{M}\right)$, where we first embed 
$\Gtwosplit/K$ totally geodesically in $\mathsf{SL}_7\R/\mathsf{SO}_7\R$ \cite[Theorem 7.2]{Hel78} and then identify $[g] \in \mathsf{SL}_7\R/\mathsf{SO}_7\R$
with the metric $(g^{-1})^T g^{-1} \in \mathscr{M}$. 

Recall that $F_K = e^{\Omega} K$. Hence, in the bundle $\V$, the map
$\iota \circ F_K$ corresponds to the metric $ e^{-2\Omega} = h$. We then recall that $\mathcal{M}_h = H^{-1/2} \mathcal{M} = e^{\Omega} \mathcal{M}$ from \eqref{HUnitaryMultiplicationFrame} is an $h$-orthogonal multiplication frame. By Lemma \ref{3,0_OrthogonalMultiplicationFrame}, we see 
$$ m_{(3,0)} \circ \iota \circ F_K = e^{\Omega} \cdot \mathsf{Im} \, \Ha = \beta_{\mathscr{X}} \circ F_K.$$
Earlier, we saw that $f$ could be realized as $e^{\Omega} \cdot \mathsf{Im} \, \Ha$ in the bundle, hence $f = \beta_{\mathscr{X}} \circ F_K$, finishing this portion of the diagram.

The commutativity $F_T = \pi_T \circ \hat{\mathcal{F}}_{\nu}, \; F_K = \pi_K \circ \hat{\mathcal{F}}_{\nu}$ holds by definition.
We earlier verified $\nu = \pi_{\Stwofour} \circ G_{\nu}$. The only thing remaining to verify is the equivalence of $G_{\nu}, F_T$ as well as $f, F_K$ under the identifications of $\Gtwosplit/T \leftrightarrow \mathscr{Y}$ and $\Gtwosplit/K \leftrightarrow \mathscr{X}$, respectively. We recall the basepoints \\
$p_{\mathscr{X}} := \spann_{\R} \langle  i, j, k \rangle \; \in \mathscr{X}$ and $p_{\mathscr{Y}} = p_0$ from \eqref{Basepoint}.  Since $\hat{\mathcal{F}_{\nu}} = e^{\Omega}$ in $\V$, we need to check that 
\begin{align} 
	e^{\Omega} \cdot p_{\mathscr{X}}&=  f = \spann_{\R} \langle w_1, w_2, w_3 \rangle  \label{Decomposition1} \\
	 e^{\Omega} \cdot p_{\mathscr{Y}} &= G_{\nu} = (w_1, \, \spann_{\R} \langle w_2, w_3 \rangle, \, \spann_{\R} \langle w_4, w_5 \rangle, \, \spann_{\R} \langle w_6,w_7 \rangle ) \label{Decomposition2} 
\end{align} 
Now, $ e^{\Omega} = \mathsf{diag} \left( \sqrt{rs}, \sqrt{r}, \sqrt{s}, 1, \frac{1}{\sqrt{s}}, \frac{1}{\sqrt{r}} ,\frac{1}{\sqrt{rs}} \right)$ in the basis $\mathcal{B}$ from \eqref{BaragliaBasis}.
Recall the linear relations \eqref{StandardMultiplicationFrame} for $\mathcal{M}$ along with
$e^{\Omega} \mathcal{M} = \mathcal{M}_h = (w_i)_{i=1}^7$. Let $g:= e^{\Omega}$ and by a direct calculation, one finds $g(m_i) = w_i$. In particular,
$$ g \cdot p_{\mathscr{X}} = g \cdot \spann_\R \langle m_1, m_2, m_3 \rangle = \spann_{\R} \langle w_1, w_2, w_3 \rangle $$
and similarly $g \cdot p_{\mathscr{Y}} = G_{\hat{\nu}}$ easily follows from \eqref{Decomposition2}. 
Also, \eqref{Decomposition1}, \eqref{Decomposition2} imply $\pi_{\mathscr{X}} \circ G_{\hat{\nu}} = f$
and $\pi_{\quadric} \circ G_{\hat{\nu}} = \hat{\nu}$ follows from $w_1 = m_4 =i$. Hence, the diagram commutes. 
\end{proof} 

\section{Uniqueness of Complete Solutions} \label{Uniqueness} 

In this section, we show the uniqueness of a complete solution to the equations \eqref{HitEuc}. Here, the appropriate 
notion of completeness is as follows:

\begin{definition}\label{CompletenessDefinition} 
We call a solution $\bm{u} = (u_1, u_2)$ to \eqref{HitEuc} \textbf{complete} when the metrics $\sigma_1:= e^{u_1 -3u_2} |dz|^2$, $\sigma_2: = e^{2u_2} |dz|^2 $, and $\sigma_3 = |q|^2e^{-2u_1}|dz|^2$ are complete (completeness of $\sigma_3$ applies away from the zeros of $q$).
\end{definition} 
	
We remark where this definition comes from. Let $f: \C \rightarrow \Gtwosplit/K$ be the minimal surface
in the symmetric space associated to a solution to \eqref{HitEuc}. The induced metric
$g_q: =  f_q^*g_{\mathsf{Sym}(\Gtwosplit) }$ is given by 
\begin{align}\label{SymmetricSpaceMetric} 
	g_q = 2 \; \tr( \varphi \varphi^{*h}) = 2\left ( 18e^{2u_2} + 10 e^{u_1-3u_2} + 2|q|^2 e^{-2u_1} \right).
\end{align} 
Thus, a complete solution yields a complete minimal surface in the symmetric space. We show that when $\bm{u} =(u_1, u_2)$ is complete, each metric $\sigma_i$ is quasi-isometric to $|q|^{1/3}$.
In particular, the induced metric $g_{\nu} =-12e^{2u_2}$ of $\nu_q$ constructed in Section \ref{BaragliasConstruction} is complete. We do not know if the completeness of $g_q$ entails the completeness of each $\sigma_i$. The completeness of $\sigma_3$ will be crucial
for certain upper bounds of complete solutions proven in Lemma \ref{CompleteUpperBounds}. 

  \subsection{Global Estimates} 
 
The $\g_2$ affine Toda equations give immediate lower bounds on curvature for each $\sigma_i$, just as in the case for $\SL_n\C$ \cite[Lemma 3.15]{ML20a}.
 
 \begin{proposition}[Curvature Bounds] 
 Let $\bm{u} = (u_1, u_2)$ solve equation \eqref{HitEuc}. Then $\sigma_1:= e^{u_1 -3u_2} |dz|^2$, $\sigma_2: = e^{2u_2} |dz|^2 $, and $\sigma_3 = |q|^2e^{-2u_1}|dz|^2$ satisfy 
 $$ \kappa_{\sigma_1}  \geq -20, \; \kappa_{\sigma_2} \geq -12, \; \kappa_{\sigma_3} \geq -4 .$$ 
 \end{proposition} 
 
 \begin{proof}
 These are all straightforward, with the signs of \eqref{HitEuc} conspiring to yield success. 
  \begin{align*}
 	\kappa_{\sigma_2}  &= \frac{-2}{\sigma_2} \Delta \log(\sigma_2 )= \frac{2}{e^{2u_2}} (5e^{u_1- 3u_2} - 6e^{2u_2} ) \geq -12.\\
 	\kappa_{\sigma_1} &= \frac{-2}{\sigma_1} \Delta \log(\sigma_1) = \frac{1}{e^{u_1-3u_2}} ( -20e^{u_1 -3 u_2} + 2e^{-2u_1} |q|^2 + 18e^{2u_2} ) \geq  -20.
 \end{align*} 
Similarly, 
$$ \kappa_{\sigma_3} = \frac{-2}{\sigma_3} \Delta \log(\sigma_3) =  \frac{2e^{2u_1}}{|q|^2}(  -2e^{-2u_1} |q|^2 +5 e^{u_1 -3 u_2})  = -4 +10 |q|^{-2} e^{3u_1-3u_2} \geq -4,$$
so that $\kappa_{\sigma_3} $ is bounded below away from the zeros of $q$.
 \end{proof} 

We now recall the Cheng-Yau maximum principle from \cite[Theorem 8]{CY75}.
 
 \begin{lemma}[Cheng-Yau Maximum Principle]\label{ChengYau}
 Let $(M,g)$ be a complete Riemannian manifold (without boundary) with Ricci curvature bounded below. Then if $u$ is a $C^2$ function on $M$ with 
 $\Delta u \geq f(u)$ such that there exists a continuous function $g: [a, \infty) \rightarrow \R_{+}$ such that 
 \begin{itemize} 
 	\item $g$ is non-decreasing 
 	\item $\int_b^{\infty}\left( \int_a^{x} g(\tau) d\tau \right)^{-1/2} dx < \infty$ for some $b \geq a$
	\item $ \lim \inf_{t \rightarrow \infty} \frac{f(t)}{g(t)} > 0 $,
\end{itemize}
then $u$ is bounded above. Moreover, if $f$ is lower semicontinuous, $f(\sup u) \leq 0 $. 
 \end{lemma} 
 
 We will also need a version of Cheng-Yau for manifolds with boundary from \cite[Lemma 3.4]{ML20a}. 
 
  \begin{lemma}[Cheng-Yau Maximum Principle with Boundary]\label{ChengYauBoundary}
 Let $(M,g)$ be a complete Riemannian manifold with smooth compact boundary and
 Ricci curvature bounded below. Then if $u$ is a $C^2$ function on $M$ with 
 $\Delta u \geq f(u)$ such that there exists a continuous function $g: [a, \infty) \rightarrow \R_+$ such that 
 \begin{itemize} 
 	\item $g$ is non-decreasing 
 	\item $\int_b^{\infty}\left( \int_a^{x} g(\tau) d\tau \right)^{-1/2} dx < \infty$ for some $b \geq a$
	\item $ \lim \inf_{t \rightarrow \infty} \frac{f(t)}{g(t)} > 0 $,
\end{itemize}
then $u$ is bounded above. 
Suppose $f$ is lower semicontinuous. Then either $ \sup_{p \in M} u(p) = \sup_{p \in \partial M}$
or $f(\sup u) \leq 0$. 
 \end{lemma}

  We now show that the ratios $\frac{\sigma_i}{\sigma_j}$ are each bounded due to the completeness hypotheses on $\sigma_1, \sigma_2, \sigma_3$.
  
  \begin{lemma}[Global Estimates] \label{GlobalEstimates} 
  Let $\bm{u}$ be a complete solution to \eqref{HitEuc}. Then the following inequalities hold on all of $\C$. 
  	\begin{align}
  		1 \leq \frac{\sigma_1}{\sigma_2} = e^{u_1- 5u_2} &\leq \frac{6}{5} \label{GlobalBound1}\\
		\frac{\sigma_3}{\sigma_2} =  |q|^{2} e^{-2u_1 -2u_2} &\leq 3 \label{GlobalBound2} \\ 
		\frac{\sigma_3}{\sigma_1} = |q|^{2} e^{-3u_1 + 3u_2} &\leq \frac{5}{2}  \label{GlobalBound3}
	  \end{align} 
Furthermore, each of the upper bounds is strict everywhere or is an equality everywhere. 
  \end{lemma} 
 
  \begin{proof} 
 First, we consider $\alpha = \log( |q|^2 e^{-2u_1- 2u_2} )$. Let $Z:= q^{-1}(0)$. Since $q$ is holomorphic,
 $$ \Delta_{{\sigma_2}}\alpha = 2|q|^{2} e^{-2u_1-2u_2} -6 = 2e^{\alpha} - 6.$$ 
 Since $\alpha \rightarrow- \infty$ as $z \rightarrow z_0$ for $z_0$ a zero of $q$, the sup of $\alpha$ occurs away from the zeros. 
Define  $X = \C - U$ for $U = \bigcup_i U_i$, where each $U_i$ is a small domain containing a unique zero of $q$. 
Choosing $U_i$ small enough, we can guarantee $\sup_{z \in \C - Z} \alpha(z)= \sup_{z \in \mathsf{int}(X)} \alpha(z) = : M$. 
We can now apply Lemma \ref{ChengYauBoundary} on $X$ with 
$f(t) = 2e^t -6$ and $g(t) = e^t$. We conclude that $\alpha$ is bounded above on $X$.
Suppose $p \in \mathsf{int}(X)$ achieves $\sup \alpha$. Then the equation $\Delta \alpha(p) \leq 0$ implies $ \widetilde{M} := e^{M} \leq 3$,
which is the inequality \eqref{GlobalBound2}. 
 
Next, we handle equation \eqref{GlobalBound1}. Compute 
\begin{align}\label{TightnessBound}
 2\Delta_{\sigma_2} (u_1 -5u_2) = 30 e^{u_1 -5u_2} -30 -2|q|^{2}e^{-2u_1-2u_2} \geq 30e^{u_1- 5u_2} - 36,
 \end{align} 
 where the last inequality follows from \eqref{GlobalBound2}. Applying Cheng-Yau again now with $f(t) = 30e^t-36, g(t) = e^t$, we learn $\alpha = u_1 -5u_2$ is bounded above, as desired. Moreover, the final statement of Cheng-Yau gives the upper bound of \eqref{GlobalBound1}. 

Let $\alpha = e^{5u_2- u_1}$. We use the inequality $\Delta u \,  \geq u \, \Delta \, \log u$. Hence, 
$$ 2 \, \Delta \alpha \geq 2\alpha\,( \Delta \log \alpha ) = 2\alpha ( 5 \Delta u_2 - \Delta u_1 ) = \alpha(-30e^{u_1-3u_2} + 30 e^{2u_2} + 2|q|^2 e^{-2u_1} )$$ 
Hence, using the complete metric $\sigma_1$, 
$$ 2 \, \Delta_{\sigma_1} \alpha = \alpha ( -30 + 30 \alpha + 2|q|^2 e^{-3u_1 + 3u_2} ) \geq 30(\alpha^2 - \alpha).$$
We can now apply Cheng-Yau, with $f(t) = 30(t^2 -t)$, $g(t) = t^2$. We find that $\alpha$ is bounded above and $f( \sup \alpha ) \leq 0$ gives
the lower bound of \eqref{GlobalBound1}. 
 
The proof of the bound \eqref{GlobalBound3} is similar, using the upper bound of \eqref{GlobalBound1} and Cheng-Yau with respect to $\sigma_1$ on $\alpha = \log(|q|^2 e^{-3u_1 +3u_2})$. 

The strong maximum principle implies each of the inequalities is globally strict or is a global equality \cite[Theorem 3.3.1]{Jos13}.   \end{proof}
  
As a corollary, $u_2$ is subharmonic, or equivalently, $ \kappa_{\sigma_2} \leq 0$. 

 \begin{corollary}
 Let $\bm{u} = (u_1, u_2)$ be a complete solution to \eqref{HitEuc}. Then $\kappa_{\sigma_2} \leq 0$. Moreover, $\kappa_{\sigma_1} = 0$ at a point if and only if $\kappa_{\sigma_2} \equiv 0$. Also, $\kappa_{\sigma_2} < 0 $ if and only if $q$ is non-constant. 
 \end{corollary} 
 
 \begin{proof}
 $\kappa_{\sigma_1} \leq 0$ is clear from the bounds, and the strong maximum principle proves the tightness claim. Finally, if 
 $q$ is constant, we show later that $\kappa_{\sigma_1} \equiv 0$. On the other hand, suppose $ \kappa_{\sigma_1} \equiv 0$.
 Then $5e^{u_1-5u_2} \equiv 6$. Hence, by \eqref{TightnessBound}, we see $|q|^{2} e^{-2u_1 -2u_2}\equiv 6$, which implies
 $q$ has no zeros and thus is constant. 
 \end{proof}
 
 In particular, the induced metric on the almost-complex curve is strictly negatively curved unless $q$ is a constant.  

  \subsection{Mutual Boundedness of Complete Solutions} 
 
We now show any complete solution $\bm{u} = (u_1, u_2)$ to \eqref{HitEuc} has $ e^{u_1} \sim |q|^{5/6}$ and $ e^{u_2} \sim |q|^{1/6}$ as $ |z| \rightarrow \infty$. Thus, the asymptotics of a complete solution is determined by the polynomial $q$. We begin by showing that any complete solution to \eqref{HitEuc} dominates the naive solution 
 $$\bm{v} = \left ( \frac{5}{3} \log( c|q| ),  \frac{1}{3} \log( d|q|)\right) $$ to \eqref{HitEuc}, away from the zeros of $q$. 
 Here, $c,d$ come from \eqref{cdSystem}. 
 
 \begin{lemma}[Lower Bounds] \label{CompleteLowerBounds} 
 Let $\bm{u}$ be a complete solution to \eqref{HitEuc}. Then $\bm{u} \geq \bm{v}$:
  	\begin{align}
		u_1 &\geq  \frac{5}{3} \log( c|q| )\\
		u_2 &\geq \frac{1}{3} \log( d|q|) 
	\end{align} 
 \end{lemma} 
 
 \begin{proof}
We now use to the flat metric $g_q:= |q|^{1/3} |dz|^2$. Then $g_q$ is complete 
and $\kappa_{g_q} \equiv 0$. We work on on $X = \C - U$ again, for $U = \cup_i U_i$ a union of small open neighborhoods of the zeros of $q$. Make the substitution $ \bm{\psi}= ( \psi_1, \psi_2)$ as earlier so that 
$e^{\psi_1} |q|^{5/6} = e^{u_1} , e^{\psi_2} |q|^{1/6} = e^{u_2}$. Then by \eqref{G2Hitchin_GeneralMetric_old}, $\bm{\psi}$ solves
$$\begin{cases}  
		2 \Delta_{g_q} \psi_1 &= 5e^{\psi_1 - 3 \psi_2} - 2 \, e^{-2\psi_1}\\
		2 \Delta_{g_q} \psi_2 &= -5e^{\psi_1 - 3\psi_2} + 6 e^{2\psi_2}  
\end{cases} .$$ 
The inequality $5e^{u_1- 5u_2} \leq 6$ now tells us $ 5e^{\psi_1 - 5\psi_2} \leq 6$, or 
\begin{align}\label{PsiInequality1} 
	\psi_1 - 5\psi_2 \leq \log \left ( \frac{6}{5} \right).
\end{align} 
Similarly, $2|q|^2 e^{-2u_1 - 2u_2} \leq 6$ now yields $2 e^{-2\psi_1 - 2\psi_2} \leq 6$, or
\begin{align}\label{PsiInequality2} 
	 \psi_1 + \psi_2 \geq \frac{1}{2} \log \left( \frac{1}{3} \right).
\end{align} 
Combining \eqref{PsiInequality1}, \eqref{PsiInequality2}, one finds $ 6 \psi_2 = (\psi_1 + \psi_2) - (\psi_1 - 5\psi_2) > \frac{1}{2} \log(1/3) - \log(6/5) = \log( \frac{5}{6\sqrt{3} } )= \log( d)$.
Hence, $ \psi_2 \geq \log(d^{1/6})$ , giving $ u_2 \geq v_2$. 

Applying the lower bound for $\psi_2$, we have
$$ \Delta_{g_q}(-\psi_1) = 2e^{-2\psi_1} - 5 e^{\psi_1} e^{-3\psi_2} \geq 2e^{-2\psi_1} - 5d^{-1/2} e^{\psi_1}.$$
Now, set $\beta = -\psi_1$ and we can apply Cheng-Yau with respect to $f(t) = 2e^{2t} - 5d^{-1/2}e^{-t}, g(t) = e^{2t}$ and the complete metric $(g_q)|_X$. 
Since $-\psi_1 \rightarrow -\infty$ as $z$ approaches a zero of $q$, maxima of $\beta$ cannot occur near the zeros of $q$. Modifying $U$, if necessary, we can demand 
 $\overline{\beta} := \sup_{z \in \C -Z(q)} \beta(z) = \sup_{z \in \mathsf{int}(X)} \beta(z)$ and Cheng-Yau gives
$ 2e^{2\overline{\beta} } - 5d^{-1/2} e^{- \overline{ \beta} } = f( \overline{ \beta}) \leq 0.$
Using $ \log(c^{5/6} ) = \frac{1}{3} \log( \frac{2}{5} d^{1/2} ) $ by \eqref{cdSystem}, we find $ \psi_1 \geq \log( c^{5/6} )$, which proves $u_1 \geq v_1$. 
 \end{proof} 

We now show a similar upper bound for a complete solution, away from the zeros of $q$. Here, we need the completeness hypothesis of the metric $\sigma_3$.
 
 \begin{lemma}[Upper Bounds for Complete Solutions] \label{CompleteUpperBounds} 
 Let $q \in H^0(\K^6)$ be a polynomial. Then there exist constants $A_i > 0$ such that any complete solution 
$ \bm{u}$ to \eqref{HitEuc} satisfies $ \bm{u} \leq \bm{\tilde{u}} $ for $\bm{\tilde{u}} =( \tilde{u}_1, \tilde{u}_2) $ 
for any $c > 0$. 
   \begin{align}\label{SuperBounds} 
 	\tilde{u}_1 &=  \frac{1}{6} \log \left( A_1 \left( |q|^5 + c \right) \right) \\
	\tilde{u}_2  &= \frac{1}{6} \log \left( A_2 \left( |q| +c \right) \right) 
 \end{align} 
 \end{lemma} 
 
 \begin{proof}
 We first consider $\eta = \log(|q|^{-2} e^{2u_1+2u_2} )$. Then a calculation shows that $ \Delta_{\sigma_3} \eta = 6 e^{\eta} -2.$
Hence, we may apply Cheng-Yau to conclude that $\eta$ is bounded away from the zeros of $q$.
 
 Similarly, we show $\beta = \log( |q|^{-2}e^{3u_1- 3u_2}) $ is bounded above away from the zeros of $q$.
 $$ \Delta_{\sigma_3} \beta = 15 e^{\beta} - 3 -9e^{\eta} .$$ 
By the boundedness of $\eta$, away from the zeros of $q$, we have $ \Delta_{\sigma_3} \beta \geq 15 e^{\beta} - C_1,$
 for some constant $C_1$. Again, the Cheng-Yau maximum principle gives boundedness of $\beta$ away from the zeros of $q$.
 
 Let us now fix a compact set $K$ containing the zeros of $q$. 
Combining the bounds for $\eta$ and $\beta$, on $\C - K$, there is $D_1 > 0$  such that 
\begin{align}\label{Super1} 
	|q|^{-5}e^{6u_1} = (|q|^{-3} e^{3u_1+ 3u_2} )( |q|^{-2}e^{3u_1- 3u_2} ) = (e^{\frac{3}{2} \eta} ) e^{\beta} \leq D_1. 
\end{align} 
Combining \eqref{Super1} with Lemma \ref{CompleteLowerBounds}, we see similarly that on $\C - K$, there is $D_2 > 0$ such that 
\begin{align}\label{Super2} 
	|q|^{-1} e^{6u_2} = (|q|^{-6} e^{6u_1 +6u_2} )(|q|^{5}e^{-6u_1}) =  e^{3 \eta} ( e^{-6\psi_1} ) \leq D_2. 
\end{align} 
The inequalities \eqref{Super1}, \eqref{Super2} imply our desired inequalities. Indeed, for any $c > 0$ and the functions
 $ c^{-1} e^{6u_2} , c^{-1}e^{6u_1} $ are trivially bounded on any compact set $K$. Hence, for $c >0$, \eqref{Super1} and \eqref{Super2} imply that 
$\frac{e^{6u_1}}{|q|^5 + c} \leq D_3$ and $\frac{e^{6u_2}}{|q| + c} \leq D_4$, and one finds the desired bounds.  \end{proof} 

 \begin{definition} 
 Let $\bm{u} = (u_1, u_2), \, \bm{w} = (w_1, w_2)$. We say $\bm{u}, \bm{w}$ are mutually bounded when $|u_i -w_i|$ is bounded for $i \in \{1,2\}$. 
\end{definition}
 
 \begin{corollary}[Mutual Boundedness of Complete Solutions] \label{MutualBoundedness} 
Let $\bm{u}, \bm{w}$ be complete solutions to equation \eqref{HitEuc}. Then $\bm{u}, \bm{w}$ are mutually bounded. \end{corollary} 
 
 \begin{proof}
 We need only see $\tilde{\bm{u}} - \bm{v} > 0 $ is bounded as $|z| \rightarrow \infty$, where $\tilde{\bm{u}}, \bm{v}$
 come from Lemma \ref{CompleteUpperBounds} and \ref{CompleteLowerBounds}. But this follows from the fact that 
 $|q(z) | \rightarrow \infty$ as $|z| \rightarrow \infty$ since $q$ is a polynomial. 
 \end{proof}

\subsection{Uniqueness of Mutually Bounded Complete Solutions} 
 
In this subsection, we prove any two complete, mutually bounded solutions to equation \eqref{HitEuc} must coincide, using the Omori-Yau maximum principle along with an
argument inspired by \cite{ML20a}. 
Combined with our earlier a priori mutual boundedness, we have the uniqueness of all complete solutions to \eqref{HitEuc}. 

\begin{lemma}[Omori-Yau Maximum Principle \cite{Omo67} \cite{Yau75}] 
Let $(M,g)$ be a complete Riemannian manifold with lower bounded Ricci curvature. If $u$ is a $C^2$ function bounded above, then there exists 
a sequence of points $(x_k)_{k=n}^{\infty} $ such that for $\overline{u} = \sup_{x \in M} u(x)$, 
	\begin{itemize}
		\item $ u(x_k) \geq \overline{u} - \frac{1}{k} $
		\item $ |\nabla_g u_k|  \leq \frac{1}{k} $ 
		\item $ \Delta_g u_k \leq \frac{1}{k}. $ 
	\end{itemize} 
\end{lemma} 

Before we state the next result, let us sketch the proof of uniqueness from mutual boundedness that we are headed towards. Suppose we have 
complete, mutually bounded solutions $\bm{u},\bm{w}$ to \eqref{HitEuc}. Then we may form the associated hermitian metrics $h_{\bm{u}}, h_{\bm{w}}$ on $\V$. 
For example, $h_{\bm{u}}$ is represented by the matrix 
$$H_{\bm{u}} = (e^{-u_1 -u_2}, e^{-u_1 +u_2}, e^{-2u_2}, 1, e^{2u_2}, e^{u_1- u_2}, e^{u_1 +u_2}).$$ 
That is, $h_{\bm{u}}(z_1, z_2) = \overline{z_1}^T H_{\bm{u}} z_2$, as in \eqref{HitchinMetric}. To prove $\bm{u} = \bm{w}$, we may instead prove $H_{\bm{u}} = H_{\bm{w}}$. Since $H_{\bm{u}}, H_{\bm{w}}$ are mutually diagonal in the frame $(dz^i)$, there is a diagonal 
matrix $A: \C \rightarrow \mathsf{SL}_7\C$ such that $H_{\bm{u}} = H_{\bm{w}} A$ with $\det A = 1$. Note the inequality $\tr A \geq 7$ follows from $\det A = 1$ by AM-GM. 
We will show $\tr(A) \leq 7$, from which $\tr(A) = 7$ implies $A = I$. The next result will aid us towards the reverse inequality. \\

In the following lemma, we use the same notation and assumptions as above. In particular, the result says that if 
$A$ almost commutes with $\varphi$, then $A$ is close to the identity matrix $I$. The statement and proof come from \cite[Lemma 3.9]{ML20a}.
The proof below is nearly exactly the same as the original, with insignificant modifications due to our Higgs field being slightly different. 
Recall that if $(e_i)_i$ is an orthogonal $h$-frame, then $|| B ||_{h} := \sum_{i} \frac{ ||Be_i ||_{h} }{||e_i||_h} $. 
Take any complex numbers $c_i \neq 0 \in \C$ and $\alpha \in \C$ and construct the endomorphism $\theta \in \SL_7\C$ by 
$$	\theta = \begin{pmatrix} 0&  & & & & \alpha &\\
				     c_1 & 0&  & & & & \alpha \\
			  	      &c_2 &0 & & & &\\
			 	      & & c_3 &0 && &\\
				      & & & c_4 & 0& &\\
				      & & & & c_5 &0&\\
				      & & & & & c_6&0\\
				     \end{pmatrix} .	
$$ 

Of course, $\theta$ serves as an abstract representative for the Higgs field $\varphi(z)$. In the following lemma, the constant $C_2$ has no dependence on $\alpha$. 

\begin{lemma}\label{CommutatorBound}
Suppose $h, h '$ are mutually diagonalizable hermitian metrics on $\C^7$ with $h'(z_1, z_2)= h(z_1, Az_2)$ for $A \in \mathsf{SL}_7\C$.
Suppose further that there are constants $B, C$ such that 
$\frac{ ||e_i||_h}{ ||e_{i+1} ||_h} \leq C$ for all $i$ and $ || A ||_h \leq B$. 
There exists a constants $\epsilon_1 > 0, C_2 > 0$ such that if $ 0 < \epsilon < \epsilon_1$, then 
$$|| \, [\theta, A] A^{-1/2} \, ||_{h_{\bm{u}} } < \epsilon \; \mathrm{implies} \;  || \, A- I\, ||_{h_{\bm{u}} } < C_2 \epsilon .$$
\end{lemma} 

Combining a key inequality of Simpson with Omori-Yau will tell us that $A$ almost commutes with $\varphi$, as desired. 
The following proof is modeled on the proof of \cite[Theorem 3.23]{ML20a}. 

\begin{lemma}[Uniqueness of Complete, Mutually Bounded Solutions] \label{UniquenessMutuallyBounded}
Let $\bm{u} ,\bm{w} $ be complete, mutually-bounded solutions to \eqref{HitEuc}. Then $\bm{u} = \bm{w}$. 
\end{lemma} 

\begin{proof} 
We first note that the lower bounds from Lemma \eqref{CompleteLowerBounds} imply that if $\bm{u}$ is a complete solution
to \eqref{HitEuc}, then $e^{-2u_2}$ and $e^{-u_1+ 3u_2}$ are bounded above. Write $H_{\bm{u}} = (h^i)_{i=1}^7$, temporarily changing our indexing.
By direct inspection $\frac{h^i}{h^{i+1}} \in \{e^{-2u_2}, e^{-u_1+ 3u_2} \}$ for any $1 \leq i \leq 6$. In particular, if $(e_i)_{i=1}^7$ is identified with $(dz^i)_{i=3}^{-3}$, changing our indices again, then we have a global bound $C$ such that $\frac{||e_i||_{h_{\bm{u}} }}{||e_{i+1}||_{h_{\bm{u}}}} \leq C$. 
 The mutual boundedness
of $ \bm{u} ,\bm{w} $ then tells us that if $H_{\bm{w}} = H_{\bm{u}} A$, then $|| A||_{h_{\bm{u}}}$ is bounded. The crux of the proof comes from  
Simpson. He splits the flat connection $\nabla$ into pieces by $D'_h:= D_h^{1,0} + \varphi^{*h}$ and $D_h'' := \delbar_E + \varphi$. 
He then considers a generalized
Laplacian $\Delta_h' := \sqrt{-1} \Lambda D_h'' D_h'$, where $\Lambda$ is contraction with the K\"ahler form. In our setting, $\sqrt{-1} \Lambda(dz \wedge d\zbar) = 2$ 
for the K\"ahler metric $g = |dz|^2 = \frac{1}{2} (dz \otimes d\zbar + d\zbar \otimes dz)$. Now, by \cite[Lemma 3.1(b, c)]{Sim88}, we have 
\begin{align}\label{SimpsonEquation} 
	\tr(\Delta_h' A) = - || D_h''(A) A^{-1/2} ||_{h_{\bm{u}}, g}^2, 
\end{align} 
since $h$ is a self-adjoint section of $(\End(\V), h_{\bm{u}})$ and $\det h = 1$. On the other hand, $\tr(\Delta'_h A) = - 2 \Delta \tr(A)$, where again we write $\Delta = \partial_{z}\partial_{\zbar}$. 
Note that $||dz|_g^2 = 2$ under our choice of hermitian metric $g= |dz|^2$ on $\C$. Writing out the terms of \eqref{SimpsonEquation} shows that for some $C > 0$, we have 
$$ \Delta \tr(A) = C || \, [\varphi, A] A^{-1/2} \, ||^2_{ h_{\bm{u}}, g } + C || \, \delbar_\V(A) \, A^{-1/2} \, ||^2_{h_{\bm{u}}, g } \geq C || \, [\varphi, A] A^{-1/2} \, ||^2_{ h_{\bm{u}} ,g } .$$
Since $A$ is bounded, so is $\tr(A)$. Denote $M = \sup_{z \in \C} \tr(A)$. 
Applying Omori-Yau, we get a sequence of points $(p_k)_{k=m}^{\infty}$ such that $\Delta \tr(A)(p_k) \leq \frac{1}{k} $ and $\tr(A)(p_k) \geq M - \frac{1}{k}$. 
Writing $\varphi = \hat{\varphi} dz$, the previous inequality implies for some $C_3 > 0$ that 
$$ || \, [\hat{\varphi}, A] A^{-1/2} \, ||^2_{ h_{\bm{u}} }(p_k) \leq C_3/ k. $$ 
For $N$ large enough, $k > N$ implies by 
Lemma \ref{CommutatorBound} that for some $C_4 > 0$
$$ || A - I ||_{ h_{\bm{u}}} (p_k) \leq C_4 /\sqrt{k}.$$ 
Writing $A = \diag (A_i)_{i=1}^7$, we conclude that 
$$\tr(A)(p_k) - 7 \leq \sum_{i=1}^7 |A_i(p_k)-1| \leq  7C_4 /\sqrt{k} .$$
 Hence, as $k \rightarrow \infty$, we find $\tr(A) \leq 7$. Since $A$ is diagonal, this is possible only
if $ A= I$. \end{proof} 

\alignL Combining Corollary \ref{MutualBoundedness} and Lemma \ref{UniquenessMutuallyBounded}, we have proven: \alignLend 

\begin{theorem}[Uniqueness of a Complete Solution] \label{UniquenessTheorem} 
Let $q \in H^0(\K_\C^6)$ be a polynomial. Then there is at most one complete solution to \eqref{HitEuc}. 
\end{theorem} 

\section{Existence of Complete Solutions}\label{Existence} 

In this section, we prove the existence of a complete solution to the system \eqref{HitEuc_Clean} on $\C$ when $q$ is a polynomial using the method of sub and super-solutions. We need the following properties of $q$: 
$|q| \rightarrow \infty$ as $|z| \rightarrow \infty$, $q$ has finitely many zeros, and $|q_z| \in O(|q|)$. 
First, we produce global sub and super-solutions $\bm{u}^-, \bm{u}^+$, respectively. 
Then we furnish a global solution $\bm{u}$ by invoking a general existence result. 
In particular, we apply a result of Li \& Mochizuki \cite{ML20a} which says that a system of PDE with appropriate convexity 
which has sub and super solutions $\bm{u}^- <\bm{u}^+$ must have a solution $\bm{u}$ satisfying $ \bm{u}^- \leq \bm{u} \leq \bm{u}^+ $. 

Define the operator $F: \R^2 \rightarrow \R^2$ by
\begin{align}\label{F_system}
	F(u_1, u_2) = b ( e^{(u_1-3u_2)} - e^{-2u_1}\,|q|^2, \;  -e^{(u_1-3u_2) } + e^{2u_2}),
\end{align} 
where $b>0 $ is a constant defined earlier in Section \ref{HitchinsEquationsSubsection}.
Write $F = (F_1, F_2)$ and $\bm{u} = (u_1, u_2)$ and the system \eqref{HitEuc_Clean} becomes $\Delta \bm{u} = F( \bm{u} )$. 
The key convexity property of $F$ here is that for $i \neq j$, 
\begin{align} \label{Convexity}
	\deriv{F_i}{u_j} \leq 0.
\end{align} 
 The property \eqref{Convexity} is precisely the convexity needed to invoke the result of Li \& Mochizuki \cite{ML20a}. 

As noted earlier, away from the zeros of $q$, an exact solution to the system \eqref{HitEuc_Clean} is given by 
\begin{align}\label{ExactSolution} 
	\bm{v} = \left( \; \frac{5}{6} \log( \, |q| \, ), \; \frac{1}{6} \log( \, |q| \, ) \right).
\end{align} 
We will alter the equation \eqref{ExactSolution} to obtain global sub and super-solutions. 

\begin{definition}
A \textbf{sub-solution}, resp. \textbf{super-solution} to $\Delta u = F(u)$ on $\Omega$ is a function $\bm{u} \in W^{1,2}_{loc}(\Omega) \cap C^0(\Omega)$
such that $\Delta u^+_i \geq F(u^+_i) $, resp. $\Delta u^+_i \leq F(u^+_i) $, in the weak sense. 
 \end{definition} 
 
 \subsection{Sub-Solution} 
 
For the sub-solution, we can adapt an idea of Wan \cite{Wan92} to our setting, leveraging the hyperbolic density function,
along with \eqref{ExactSolution}. These ideas have been used similarly in \cite{DW15, ML20a, TW25}.

\begin{lemma}[Sub-Solution]\label{SubSolution} 
Let $q \in H^0(\K_\C^6)$ be a polynomial and take take $\bm{v}$ from \eqref{ExactSolution}. Then there is an $M > 0$ such that 
$$\bm{u}^- = \begin{cases} 
	 \max \{ \bm{v}, \bm{w} \} \ ; & |z| \leq M \\
	 \bm{v}			\; & |z| > M 
	 \end{cases} 
$$ is a sub-solution to \eqref{HitEuc} on $\C$,
where $ \bm{w} = ( \, \frac{5}{2} \log ( g_M), \frac{1}{2} \log ( g_M) )$ and $g_{M}|dz|^2$ is the metric of constant curvature $-2$ on the ball $B_{2M}(0) = \{z \in \C \; : \; |z| < 2M\}$ given by
	$$ g_{M} = \frac{1}{2} \; \left ( \frac{8M}{4M^2-|z|^2} \right)^2.$$
In other words, $\Delta \log g_M = g_M. $ 
\end{lemma} 

\begin{proof}
For a system $\Delta \bm{u} = F(\bm{u})$ satisfying \eqref{Convexity}, the max of two sub-solutions is a sub-solution
\cite[Lemma 5.4]{ML20a}. Since $\bm{v}$ is a sub-solution away from the zeros, it suffices to show $\bm{w} $ is a sub-solution and to 
 find $M$ so that $\bm{u}^-$ is continuous.

We handle continuity first.  Let $B_M:= \{ z \in \C :  |z| < M \}$. Choose $M$ large enough that $\{ \, z \; : \; |q(z) | \leq 1 \} \subseteq B_{\frac{1}{2}M}$.
Then $\bm{u}^-$ is continuous on $\{ |z| < M\}$ and $ \{ |z| > M\}$ by definition, since $\bm{w}$ is bounded away from zero. 
Take $M > \frac{4\sqrt{2}}{3} $ so that for $|z| \leq M$, we have $g_{2M}(z) < 1$ and consequently $ w_1, w_2 < 0 $ on $B_M$,
while $u_1, u_2 > 0 $ on $\{ \frac{1}{2} M < |z| < M \}$. Thus, $\bm{v} > \bm{u}$ on a neighborhood of $\{ |z| = M \} $ and continuity of $\bm{u}^-$ follows. 

We now show $\bm{w}$ is a sub-solution. Recalling $ d= \frac{5}{6\sqrt{3}} $ from \eqref{cdSystem}, we find $ 0 < b = 3 \, d^{1/3}< \frac{5}{2}$. Hence, 
	\begin{align*}
		F_1( \bm{w} ) &= b( g - g^{-5}|q|^2 ) \leq b g < \frac{5}{2} g= \Delta w_1 .
	\end{align*}
Also, $ F_2( \bm{w} ) = 0 \leq \frac{1}{2}g = \Delta w_2$. Thus, $\bm{u}^-$ is a sub-solution. 
\end{proof} 

We now construct a super-solution, also by altering $\bm{v}$ from \eqref{ExactSolution}. 
The argument is more technical than that of the sub-solution, but is analogous to the super-solution in \cite[Lemma 2.4]{TW25}.

 \subsection{Super-Solution} 
 
\begin{lemma}[Super-Solution]\label{SuperSolution} 
Let $q \in H^0(\K^6_\C)$ be a non-constant polynomial. Then there is a constant $A > 1$, depending continuously on $q$, such that $\bm{u}^+ = (u^+_1, u^+_2)$ 
is a super-solution to \eqref{HitEuc_Clean} on $\C$, where
	$$ \bm{u}^+= \left( \frac{5}{12} \log(|q|^2+ A), \; \frac{1}{12} \log( |q|^2+ 2A) \; \right).$$
\end{lemma} 

\begin{proof}
A direct calculation shows 
$\Delta u_1^+ = \frac{5}{12} \frac{ A|q_z|^2}{ (|q|^2 + A)^2} $
and 
 $\Delta u_2^+ = \frac{1}{6} \frac{ A|q_z|^2}{ (|q|^2 + 2A)^2}. $

Then evaluating $F_i$, we find 
\begin{align}
F_1(u_1^+, u_2^+ ) &=b\,( |q|^2+A\,)^{5/12}(\, |q|^2+2A \,)^{-1/4}- b(\, |q|^2+A \, )^{-5/6} \, |q|^2\\
F_2(u_1^+, u_2^+) &= -b (\, |q|^2+A\, )^{5/12}(\, |q|^2+2A \, )^{-1/4} + b(\, |q|^2+2A \,)^{1/6} 
\end{align}
The equations for $u$ to be a super-solution, $ \Delta u_1^+ \leq F_1(u_1^+, u_2^+) $ and $ \Delta u_2^+ \leq F_2(u_1^+, u_2^+)$, are:
\begin{align}
	\label{Supersolution_1} 	\frac{5}{12} A|q_z|^2 &\leq b( |q|^2+A)^{29/12}(|q|^2+2A)^{-1/4}-b(|q|^2+A)^{7/6}|q|^2 \\
	\label{Supersolution_2} 	\frac{1}{6} A|q_z|^2 &\leq -b(|q|^2+A)^{5/12}(|q|^2+2A)^{7/4} +b(|q|^2+2A)^{13/6} .
\end{align}
For $A \in \R_{>0}$ fixed, we are led to define the functions
\begin{align*} 
	 f_1(z) &= b(|q|^2+A)^{29/12}(|q|^2+2A)^{-1/4}- b( |q|^2+A)^{7/6}|q|^2 - \frac{5}{12}A |q_z|^2  \\
	f_2(z) &= -b(|q|^2+A)^{5/12}(|q|^2+2A)^{7/4} +b(|q|^2+2A)^{13/6} -\frac{1}{6}A |q_z|^2 .
\end{align*} 
By definition, \eqref{Supersolution_1}, \eqref{Supersolution_2} hold if and only if $f_1, f_2 \geq 0$.
Hence, it suffices to prove that for $A$ sufficiently large, $f_1, f_2 \geq 0$ for all $z$. We prove this by showing that
$f_1, f_2 \rightarrow +\infty$ uniformly as $A \rightarrow \infty$. 

We now take the asymptotic expansion of $f_i$, for $A$ fixed. The (a priori) highest order terms $\asymp (|q|^2)^{13/6}$ cancel for $f_1, f_2$
and we find the asymptotics
\begin{align}
	\label{asymptotic_f}	f_1(z) &= \frac{3}{4} b \, A \, ( |q|^2)^{7/6}  + \, O( \; (|q|^2)^{1/6} \; ) \\
	\label{asymptotic_g} f_2(z) &= \frac{5}{12} b\, A \, (|q|^2)^{7/6} + \, O( \; (|q|^2)^{1/6} \; ).
\end{align}

The asymptotic expressions \eqref{asymptotic_f}, \eqref{asymptotic_g} imply that each $f_i$ has a global minimum for $A$ fixed, using here that
 $|q| \rightarrow +\infty$ as $|z| \rightarrow +\infty$. We define 
$z_i(A) \in \C$ as the location of the global minimum of $f_i$. To finish the proof, we now show $f_i( z_i(A)) \rightarrow +\infty$ 
as $A \rightarrow +\infty$.

\textbf{Case 1.} Suppose that $ z_1(A) $ remains bounded in a ball $B$ as $A \rightarrow \infty$. 
Then set $M = \sup_{z \in B} |q(z)|^2$ and $M' = \sup_{w \in B} |q_z(w)|^2$. Observe that
$$ f_1( z_1(A)) \geq b \frac{ A^{29/12} }{( M+ A)^{1/4}} - b(M+A)^{7/6}M - \frac{5}{6}A M'  \asymp A^\frac{13}{6}.$$
Thus, $f_1( z_1(A)) \rightarrow \infty$ in this case. The argument is very similar for $f_2$ if $z_2(A)$ remains
bounded.\\

\textbf{Case 2.} Otherwise, $|z_1(A)|, |z_2(A) |$ are unbounded as $A \rightarrow +\infty$. We can disinguish three subcases
here depending on the asymptotics of $|\, q(z_1(A) )\, |^2$ and $|\,q(z_2(A)) \,|^2$.

	\emph{Case 2(i)}. First, suppose that $|\, q(z_1(A) )\, |^2$ is superlinear in $A$. Then a similar asymptotic expansion of $f_1(z_1(A))$, similar to \eqref{asymptotic_f}, 
	applies. In particular, one finds $A^{7/6} \in o(f_1(z_1(A)))$. Thus, $f_1(z_1(A)) \rightarrow +\infty$ in this case. 
	The argument is similar for $f_2$. 
	
	\emph{Case 2(ii).} Suppose that $|\, q(z_1(A) )\, |^2$ is sublinear in $A$. Using $|q_z| \in O( \, |q| \,)$, we expand 
	$$ f_1(z_1(A)) = b( \, o(A) + A)^{29/12}( o(A) + 2A)^{-1/4} - b( o(A) + A)^{7/6}o(A) - o(A^2) \asymp A^{13/6}.$$
	Hence, $ f_1(z_1(A)) \rightarrow +\infty$ again. Similarly, when $|\, q(z_2(A) )\, |^2 $ is sublinear, 
	$f_2(z_2(A)) \asymp A^{13/6}$. 
	
	\emph{Case 2(iii).} In the final case, we have roughly linear growth. More precisely, by excluding the previous cases, we have 
	$(A_n)_{n=1}^{\infty}$ such that $A_n \rightarrow \infty$ and $|\, q(z_1(A_n) ) \, |^2 = r_n A$ for a sequence $(r_n)$
	with $\alpha < r_n < \beta$ for some fixed constants $0 < \alpha <\beta$. 
	Take subsequences as necessary so that $r_n \rightarrow r > 0$ as $n \rightarrow \infty$.
	Then 
	$$ f_{1}(z_{1}(A_n)) = bA r_n \, (r_n A+A)^{7/6}\, \left( \,  \frac{ (1+ \frac{1}{r_n} \,)^{5/4}}{(1+ \frac{2}{r_n} )^{1/4}} - 1\right) + o( \, A^{13/6} \, )$$
	Using Bernoulli's inequalities, $ (1+ \frac{1}{r_n} \,)^{5/4} \geq \frac{5}{4} r_n \geq (1+ \frac{5}{r_n} \,)^{1/4}$, so again $f_1(z_1(A)) \asymp A^{13/6}$.
	The argument is similar but slightly simpler for $f_2$ in this case. 
\end{proof} 

We can now invoke the aforementioned existence theorem from \cite[Proposition 5.2]{ML20a} since we have produced sub and super-solutions. 
For $\bm{u} = (u_j), \bm{w} = (w_j)$, we write $\bm{u} \leq \bm{w}$ if and only if $u_j \leq w_j$ for all $j$. 

\begin{proposition}\label{PropExistenceTheorem} 
Let $\Delta_g u_k = F(x, u_1, u_2)$ be a system of PDE for $\bm{u} = (u_1, u_2)$ with $\deriv{F_i}{x_j} \leq 0 $ on a non-compact manifold $(M,g)$. If there exist sub and super-solutions $\bm{u}^- < \bm{u}^+$, then there is a smooth solution $\bm{u}$ satisfying $\bm{u}^- \leq \bm{u} \leq \bm{u}^+$. 
\end{proposition} 

We can now prove existence of a complete solution.

\begin{theorem}[Existence of Complete Solution] \label{ExistenceTheorem} 
Let $q \in H^0(\K^6_\C)$ be a polynomial. Then equation \eqref{HitEuc} has a smooth solution $\bm{u}$ satisfying $ \bm{u}^- \leq \bm{u} \leq \bm{u}^+$. Moreover, $
\bm{u}$ is complete. 
\end{theorem} 

\begin{proof}
By Proposition \ref{PropExistenceTheorem}, there is a smooth solution $\bm{u}$ satisfying $ \bm{u}^- \leq \bm{u} \leq \bm{u}^+$. 
We need only check completeness. The asymptotics of $\bm{u}^-, \bm{u}^+$ imply that $e^{u_1}$ is quasi-isometric to $ |q|^{5/6}$
and $e^{u_2}$ is quasi-isometric to $ |q|^{1/6}$. It follows that $\sigma_1, \sigma_2, \sigma_3$ of each quasi-isometric to $|q|^{1/3}$
and hence complete. 
\end{proof} 

By Theorem \ref{UniquenessTheorem}, the solution in Theorem \ref{ExistenceTheorem} is the unique complete solution. 

\begin{remark} 
The equivalent sub and super-solutions to \eqref{HitEuc}, rather than \eqref{HitEuc_Clean} are given by\\
$$ \bm{u}^- = \left( \, \frac{5}{6} \log(c\, |q| ), \frac{1}{6}\log(d\, |q|) \, \right), \; \; \bm{u}^+ =  \left( \, \frac{5}{12} \log(c^2(|q|^2+A) ), \frac{1}{12}\log(d^2(|q|^2 +2A)) \, \right).$$
\end{remark} 

\section{Estimates of the Error Term} \label{ErrorEstimates} 

In this section, we prove the precise rate of exponential decay of the error term $\boldsymbol{u} - \boldsymbol{u}_-$ in a set of coordinates called \emph{natural coordinates} arising from the polynomial sextic differential $q$. We also discuss a collection of natural coordinates for $q$ covering a neighborhood of infinity, which will be crucial in the proof
of the main theorem in Section \ref{BoundaryAnnihilatorPolygon}.

\subsection{Standard Half-Planes for Sextic Differentials} \label{NaturalCoordinate} 

We now introduce the key coordinates in which the error will decay exponentially. 

\begin{definition} Given $q \in H^0(\K_\C^6)$ polynomial, a \textbf{natural coordinate} for $q$ on a domain $\Omega$ is a holomorphic function $w: \Omega \rightarrow \C$ such that $dw^6 = q(z)dz^6$. 
We call a pair $(U,w)$ a \textbf{$q$-half-plane}, when $U\subset \C$ is open and $w: U \rightarrow \Ha$ 
is a natural coordinate mapping bijectively onto the upper-half plane $\Ha = \{ z = x+iy \in \C \; | \; y > 0\}$. \end{definition} 

Locally, one can compute a natural coordinate away from the zeros of $q$ simply by setting $w(z)= \int_{z_0}^z q^{\frac{1}{6}}(\zeta) d\zeta $
for a holomorphic sixth root $q^{\frac{1}{6}}$. Later, we shall need a
a collection of natural coordinates that cover $\C$ up to a compact set, where the coordinates have convenient transition maps. 
We describe these coordinates shortly. 
 
It will be useful for us to also define $ R (\theta) := \{ z \in \C \, | \, z= re^{i \theta}, \, r \geq 0 \}$ to be the Euclidean ray at angle $\theta \in [0, 2\pi)$, based at the origin. Moreover, we will need quasi-rays as well. 
 
 \begin{definition}\label{QuasiRay} 
 A path $\gamma: \R_{\geq 0} \rightarrow \C $ of the form $ \gamma(t) = L(t) + \delta(t)$, where $L(t)= te^{i\theta}$ is a unit speed Euclidean ray and $|\delta(t)| \in o(t)$ is called a \textbf{quasi-ray}. We call $\theta$ the \textbf{angle} of the ray $L$. We call any Euclidean ray $L(t)$ such that $\gamma(t) =L(t) +\delta(t)$, where $|\delta(t)| \in o(t)$ is an \textbf{associated ray} of $\gamma$. 
We call $\gamma$ a \textbf{q-ray} when it is a Euclidean ray in a natural coordinate for $q$. 
 \end{definition} 

\begin{remark}\label{angle} 
The angle of a $q$-ray is uniquely prescribed up to $\frac{\pi}{3}$, since any two natural coordinates $w_1, w_2$
satisfy $w_2 = \beta w_1 + C$ on their overlap for $\beta$ a fixed sixth root of unity. Moreover, if $\gamma$ is a quasi-ray in $(w, U)$,
then all associated rays $L$ in $U$ have the same angle. Thus, quasi-rays have unique angles $\mod \frac{\pi}{3}$ as well. 
\end{remark} 

Before discussing the collection of natural coordinates necessary for general monic polynomial $q$, we discuss a model example for reference
and geometric motivation. Suppose $q(z) = z^n dz^6$
and consider the Euclidean rays $\alpha_k(s) = se^{ \frac{2\pi i \,}{n+6} \left(k+ \frac{ 1}{4} \right)}$ at angle $\frac{ \pi}{2(n+6)} \mod \frac{2\pi}{n+6}$. 
We will later compute the projective limit of the almost-complex curve $\nu$ along these rays. Define $ \theta_k := \frac{2\pi \, k}{n+6}$.  Consider the Euclidean sectors $U_k$ centered at angle $ \theta_k + \frac{\pi}{n+6}$ of radius $\frac{3\pi}{n+6}$: 
$$ U_k := \left \{ \; z \in \C : \; \left | \arg(z) - \left( \frac{2 \pi k}{n+6}+ \frac{\pi}{n+6} \right ) \right | < \frac{3\pi}{n+6} \,  \right \} $$ 
 Then we define $\zeta_k: U_k \rightarrow \Ha$ by $ \zeta_k (z) = \,\beta_k \, \frac{6}{n+6} z^{ \frac{n+6}{6} }$, where $\beta_k = e^{-\frac{2\pi i(k-1)}{6} }$ is chosen so that
 under $\zeta_k$, the ray $R\left( \theta_{k-1} \right) $ corresponds to $\R_{\geq 0}$ and $R\left( \theta_{k+2} \right)$ corresponds to $\R_{\leq 0}$.
 One directly computes that $ (d \zeta_k)^6 = z^n dz^6 = q$, so that $\zeta_k$ is a natural coordinate for $q$. By construction, $\zeta_k$ is a 
 biholomorphism. Furthermore, the coordinates satisfy $\zeta_{k+1} = \zeta^{-1} \zeta_k$ for $\zeta = \frac{2\pi i}{6}$. In the $\zeta_k$ coordinate, 
 the rays $\alpha_{k-1}, \alpha_k, \alpha_{k+1}$ are now rays $\gamma_i$ of angles $ \frac{\pi}{12}, \, \frac{5\pi}{12}, \, \frac{9\pi}{12}$, respectively.
 We illustrate the construction below in Figure \ref{Fig:NaturalHalfPlaneRays}. 
 
 \begin{figure}[ht]
 \centering 
 \includegraphics[width=.7\textwidth]{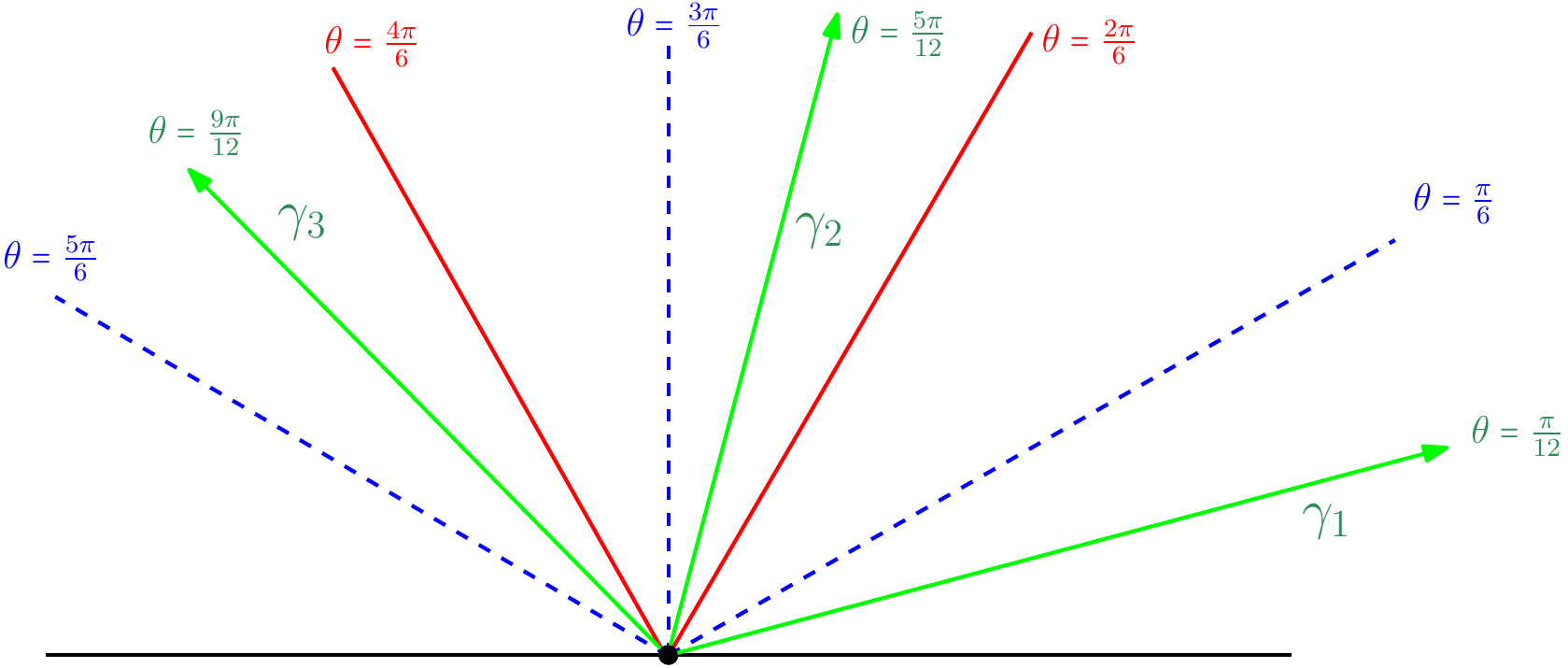} 
 \caption{\emph{\small{The image of a natural coordinate $\zeta_k$ for the monic polynomial differential $q(z) = z^n dz^6$.
 In Section \ref{AsymptoticBoundary}, we show $\lim_{s \rightarrow \infty} [\nu(\gamma(s))] $ is the same for $\gamma$ any $q$-ray of angle $\theta \in S_k:= (\frac{(k-1)\pi}{3}, \frac{k\pi}{3})$. The red lines are critical lines, or `Stokes lines', in the sense that the projective limit of $\nu$ along $q$-rays changes as we cross them. The rays $\gamma_i$
 serve as representative rays for the sector $S_i$. The dotted blue lines are `Stokes lines' for $\psi \psi_0^{-1}$, but not for $\nu$.} }} 
 \label{Fig:NaturalHalfPlaneRays} 
 \end{figure} 
 
We now note that a version of the construction holds in general for any monic polynomial $q$, as shown by Dumas-Wolf, 
by treating the general case as a deformation of the model case of a monomial \cite{DW15}. The essential difference is that the Euclidean rays $\alpha_k, \alpha_{k-1}, \alpha_{k+1}$ are now only eventually contained in $U_k$ and moreover these Euclidean rays in the $z$ coordinate become quasi-rays in a natural coordinate $\zeta_k$ for $q$. 
This collection of half-planes from \cite[Proposition 3.2]{DW15} will be crucial for the proof of the main theorem, Theorem \ref{THM:AnnihilatorPolygon}. 
 
\begin{lemma}[Standard Half-Planes]\label{StandardHalfPlanes} 
Let $q \in H^0( \K_\C^6)$ be a monic polynomial of degree $n$ and $K$ a compact set containing the zeros of $q$. We can find a compact set $K' \supset K$ 
and a set $\{ (w_k, U_k) \}_{k=1}^{n+6}$ of $q$ half-planes such that  
	\begin{enumerate}
		\item $ \C - \bigcup_{k=1}^{n+6} U_k = K' $. 
		\item The rays $R \left ( \frac{2\pi }{n+6}\left( j + \frac{1}{4} \right) \right)$ are eventually contained in $U_k$ for $j \in \{k-1, k, k+1\}$. 
		\item The rays $R \left( \frac{2\pi (k \pm 2) }{n+6} \right)$ are disjoint from $U_k$. 
		\item Euclidean rays are quasi-rays in each $q$-half-plane $(w_k, U_k)$. Moreover, $R \left ( \frac{2\pi }{n+6}\left( j+ \frac{1}{4} \right) \right)$
		are $q$-quasi-rays of angles $\frac{\pi}{12},\frac{5 \pi}{12}, \frac{9\pi}{12}$ for  $j = k-1, k, k+1$, respectively. 
		\item On overlaps $U_{k} \cap U_{k+1}$, we have $w_{k+1} = \zeta^{-1} w_k + c$ for $\zeta = e^{\frac{2\pi i}{6} \, }$ 
			and $w_{k}, w_{k+1}$ map $U_k \cap U_{k+1}$ onto a Euclidean sector of angle $\frac{2\pi}{3}$.
		\item Every $q$-ray $\gamma: \R_{+} \rightarrow \C$ is eventually contained in some $U_i$. 
	\end{enumerate} 
\end{lemma} 
\begin{proof} 
All of (1) -- (5) are explained in the proof from \cite{DW15}, except (6), which we now check. Let $\gamma: \R_{+} \rightarrow \C$ be a $q$-ray. 
Since $g_q: =|q|^{1/3}$ is complete and $\gamma$ is a $g_q$-geodesic by definition, $\gamma$ leaves compacta. 
Thus, for some $s_0$, we have $\gamma(s) \in \bigcup_{k=1}^{n+6} U_j$ for $s \geq s_0$. 
We claim that for some index $j$, we have $\gamma(s) \in U_j$ for $s \geq s_1 > s_0$. 
To see this, write $\gamma(s_0) \in U_i$.
Then let $\theta \in [0, 2\pi)$ be the angle of $\gamma$ in $U_i$. If $\theta \in (0, \pi)$, we find $\gamma(s) \in U_i $ for $s \geq s_0$. 
If $\theta \in [\pi, \frac{4\pi}{3} )$, we find $ \gamma$ moves into $U_{i+1}$, where it has angle $\theta - \frac{\pi}{3} \in U_{i+1}$
and hence remains in $U_{i+1}$. The argument is similar if $\theta \in ( \frac{5\pi}{3}, 2\pi] $, where $\gamma$ moves into $U_{i-1}$ and remains there. 
In the final case that $\theta \in [\frac{4\pi}{3}, \frac{5\pi}{3}]$, then $\gamma$ travels into an adjacent half-plane $U_{i \pm 1}$ in which it has angle $\theta \pm \frac{\pi}{6}$,
reducing to a previous case. We conclude that (6) holds. \end{proof} 

Denote $\mathcal{Z} := q^{-1}(0)$ to be the zeros of $q$ and we define $t: \C \rightarrow \R_{\geq 0}$ by 
\begin{align}\label{tdistance} 
	t(p) := \mathsf{dist}_{g_q}(p, \mathcal{Z} ).
\end{align} 
Later, we will need $q$-half-planes from \cite[Proposition A.1]{DW15} that have nice asymptotics with respect to $t$. 

\begin{proposition}[Good Asymptotic Half-Planes]\label{t_distance} 
Let $q \in H^0(\K_\C^6)$ be a monic polynomial and $K$ a compact set containing the zeros of $q$. Then there are constants $A, R, a > 0 $ such that
for any $p \in \C$ with $ t(p) > R$, there is a $q$-half-plane $(w, U)$ with $U \cap K = \emptyset$ and 
$$ \Im(w(p)) > t(p) - A. $$
Moreover, for $x \in \partial U$ in the boundary of this half-plane, we have 
$$ t(x) > a| \Re( w(x)) |\; \; \text{as} \; \; x \rightarrow \infty. $$
\end{proposition} 

 \subsection{Asymptotics of Decay Term} \label{AsymptoticDecay}

In this subsection, we prove the precise a precise asymptotic for the error term $\bm{\tilde{u}} := \bm{u} - \bm{u}^- $ in a natural coordinate for $q$,
where $\bm{u}$ solves \eqref{HitEuc}. 

We begin with first estimate for the error term $\tilde{\bm{u}}:= \bm{u} - \bm{u}^-$, which will be crucial for proving 
stronger estimates on the the precise rate of decay of $\tilde{\bm{u}}$.

\begin{corollary}[A First Estimate]\label{G2PDE_Bound}
Let $q \in H^0(\K_\C^6)$ be a monic polynomial of degree $n$ and $t: \C \rightarrow \R_{\geq 0}$ be the $|q|^{1/3}$-distance to the zero set of $q$. 
Then there are constants $C, R, \beta > 0$ such that $t(p) > R$, then the 
complete solution $\bm{u}= (u_1, u_2)$ to equation \eqref{HitEuc} satisfies 
	\begin{align}\label{FirstBound} 
		\begin{cases} 
		\; \; 0 \leq &u_1 - \frac{5}{6}\log(\; c |q| \;) < Ct^{-\beta} \\
		\; \; 0 \leq &u_2 - \frac{1}{6} \log(\; d | q|\; )  < Ct^{-\beta} ,
		\end{cases} 
	\end{align} 
where $ \beta = \frac{12\,n}{n+6} > 1$.
\end{corollary}

\begin{proof}
Let $p_0: = $ the origin in $\C$ and denote $d_n(p) := \mathsf{dist}_{g_n}(p,p_0)$, where $g_n := |z^\frac{n}{6} | |dz|^2$ is the flat metric associated to $q_n = z^n dz^6$.

Using the family of natural $q_n$-half-planes from Proposition \ref{StandardHalfPlanes} of the form $(\zeta_k, U_k)$ with
$\zeta_k = \alpha_k \frac{n+6}{6} z^{\frac{n+6}{6}} $, one finds that $d_n$ satisfies $ d_n \asymp \, |p|^{\frac{n+6}{6}} $. Since 
$q$ and $q_n$ are comparable away from the zeros of $q$, for $|z|$ sufficiently large: 
\begin{align}\label{Comparable_q} 
	C_2 |z|^n < |q(z) | < C_3 |z|^n. 
\end{align} 
Similarly, we have $t, d_n$ comparable for $|z| $ sufficiently large. In particular, we have 
\begin{align}\label{metric_bound} 
	t(p) < C_4 |z|^{\frac{n+6}{6}}. 
\end{align} 
Observe that $ 0 \leq \bm{u} - \bm{u}^- \leq \bm{u}^+ -\bm{u}^-$ by Theorem \ref{ExistenceTheorem}. Expanding the right-hand side, we find that for $D_1 = \frac{5}{12}, D_2 = \frac{1}{12},$ 
\begin{align}\label{Super-Sub}
	u^+_i - u_i^- = D_i \log \left( 1+ \frac{A}{|q|^2} \right) \leq D_i \frac{A}{|q|^2}. 
\end{align} 
Combining \eqref{Comparable_q}, \eqref{metric_bound}, we find that for some constants $C',C''$, 
$$ |q|^2 \geq C' |z^{2n} | \geq C'' \, t^{\frac{12n}{n+6}}. $$ 
Applying this inequality in \eqref{Super-Sub} gives the desired inequalities \eqref{FirstBound}. 
\end{proof}

We now write out the new system that the error $\bm{\tilde{u}} $ solves in a natural coordinate $w$ for $q$. 
Recall that $\tilde{u}_1 = u_1 - \frac{5}{6} \log( c|q|) $ and $\tilde{u}_2 = u_2 - \frac{1}{6}\log( d |q|)$. 
In $w$ coordinates, $\tilde{u}_1(w) = u_1(w) - \log( c^{5/6} )$ and $ \tilde{u}_2(w) = u_2(w) - \log( d^{1/6} )$. 
Since $ \Delta \equiv \Delta_{g_q}$ in $w$-coordinates, for $b = 3 \,d^{1/3} $ from earlier, we find $\bm{\tilde{u}}$ solves
	 \begin{align}\label{PDE_NaturalCoordinate}
 	\begin{cases} 
		 \Delta  \tilde{u}_1 & = b \left( \, e^{\tilde{u}_1-3\tilde{u}_2} -  \, e^{-2\tilde{u}_1}\right)  \\
		\Delta  \tilde{u}_2 &= b \left(   e^{2\tilde{u}_2 } - e^{\tilde{u}_1-3\tilde{u}_2 } \right). \end{cases} 
 	\end{align} 
 The linearization of the system \eqref{PDE_NaturalCoordinate} at zero is  
 \begin{align}
 	\begin{cases} \label{PDE_LinearizedNaturalCoordinate} 
 		\Delta \tu_1 &= b ( 3 \tu_1 - 3\tu_2)  \\
		\Delta \tu_2 &= b ( 5 \tu_2-\tu_1  ) 
	 \end{cases}
\end{align} 
 
Set $x_1 = \tu_1 + \tu_2$ and $x_2 = \tu_1 - 3\tu_2$. Then up to first order, $x_1, x_2$ are eigenvectors of the linearized system \eqref{PDE_LinearizedNaturalCoordinate},
with eigenvalues $2b, 6b$, respectively.  We are led to recall some results from \cite[Lemma 5.8]{DW15} on the growth rate of solutions to 
$\Delta u = k u$ on the upper half-plane $\Ha$ with prescribed boundary values. 
 
 \begin{lemma}[Classical Estimates]\label{FundamentalEstimate} 
 Let $0 \leq g \in C^{0}(\R) \cap L^1(\R)$ and $k> 0$. Then the boundary value problem 
 $$ \begin{cases} \Delta h&= kh \\ h|_{\R} &= g \end{cases} $$ 
 has a solution $h \in C^{\infty}(\Ha) \cap C^0(\R)$ with $ 0 \leq h \leq \sup g$ globally. Moreover, in the standard $z = x +i y$ coordinate,
 $$ h \in  O \left( \frac{e^{ -2\sqrt{k} y} }{y} \, \right) \;\; \text{as} \;\; y \rightarrow \infty. $$ 
 \end{lemma}
 
Let us recall a result from \cite[Lemma 3.13]{TW25} on super-solutions of a quadratic version of the same equation.
 
 \begin{lemma}[Quadratic Super-solution]\label{QuadraticSuperSolution}
 Let $ g \in C^0(\R) \cap L^1(\R)$ with a global bound $0 \leq g \leq \frac{1}{4k'}$ and $k > 0$. Then there is a smooth super-solution
 $v \in C^{\infty}(\Ha) \cap C^0(\R)$ satisfying, for $z = x + i y$ on $\Ha$, 
 $$ \begin{cases} \Delta v \leq k v - k k'v^2 \\  v|_{\R} \geq g \end{cases}$$ 
as well as 
 $$ v \in O \left ( \frac{e^{ -2 \sqrt{k} y} }{\sqrt{y} } \right) \; \; \text{as} \; \; y \rightarrow \infty .$$
 \end{lemma}
 
 We can now show the precise asymptotics of the error terms $x_1, x_2$, using both Lemmas \ref{FundamentalEstimate} and \ref{QuadraticSuperSolution}. 
 It is essential to the proof that the error $\tilde{\bm{u}}$ is a priori integrable; for this fact we need Proposition \ref{G2PDE_Bound}. 
 
 \begin{lemma}[Fundamental Asymptotics]\label{FundamentalAsymptotic} 
Let $q \in H^0(\K_\C^6)$ be a polynomial and $t: \C \rightarrow \R_{\geq 0}$ be the distance from the zeros of $q$. Then for $ \alpha = 2^{1/3}3^{1/4}5^{1/3} $, as $t \rightarrow \infty$, 
$$ x_1 \in O \left( \,\frac{e^{-2\alpha \, t}}{\sqrt{t}} \, \right) \; \; \; \text{and} \; \;  \; x_2 \in O \left ( \, \frac{e^{-2 \alpha \sqrt{3} \, t }}{\sqrt{t}} \, \right) .$$
Moreover, $\Delta x_i \in O(x_i)$ satisfies the same asymptotics as $x_i$. 
 \end{lemma} 
 
 \begin{proof}
Recall by Corollary \ref{G2PDE_Bound}, that $\tilde{u}_i$ are uniformly bounded away from the zeros of $q$. Thus, we select a compact set $K$ containing the zeros of $q$ such that
\begin{align} \label{GlobalBound} 
	0 \leq \tu_i \leq \frac{1}{16} \; \text{on} \; \C - K. 
\end{align} 
By iteratively applying Proposition \ref{t_distance}, we may take a collection $\{ (w_k, U_k)\}_{k=1}^N$
of $q$-half-planes covering $ \{p \in \C \; : \; t(p) > R \}$, for $R$ sufficiently large, such that the $t$-asymptotics of Proposition \ref{t_distance} hold. Thus, it suffices to prove the asymptotics of the lemma in each half-plane $(w_k, U_k)$. Fix a $q$-half-plane $(w_k, U_k)$ and identify $\Ha = U_k$ and $\R = \partial \Ha = \partial U_k$. We write $w = w_k$ for notational simplicity and set $w = x +iy$. In these coordinates, $\tu_1, \tu_2$ satisfy the system \eqref{PDE_NaturalCoordinate}
and $\tilde{u}_i \geq 0$.

We first show the asymptotic for $x_1 = \tu_1 + \tu_2$. By direct calculation, $x_1$ satisfies
\begin{align}
	\Delta x_1= b e^{2\tu_2} - b e^{-2\tu_1}.
\end{align} 
Since $\tu_i \geq 0$, we have $ e^{2 \tu_2} \geq 1 + 2\tu_2 + 2\tu_2^2 $ and $ e^{-2\tu_1} \leq 1 - 2\tu_1 + 2\tu_1^2$. 
Combining these inequalities, 
\begin{align}
	\Delta x_1 \geq 2b\, x_1 - 2b \,(\tu_1^2 - \tu_2^2) \geq 2b \, x_1 - 2b \, x_1^2.
\end{align} 

Since $x_1 \geq 0$, we see $g_1 := x_1|_{\R}$ is non-negative and continuous too. By Corollary \ref{G2PDE_Bound}, $x_1 \in O(t^{-\beta}) $ 
for $\beta > 1$, for $p$ sufficiently large. Also, for $ p \in \R$ sufficiently large,
 $t(p) > a | \Re( w(p)) |$, for some $a > 0$, by Proposition \ref{t_distance}. We conclude
that $g_1 \in L^1(\R)$ is integrable. The bounds \eqref{GlobalBound} on $\tu_i$ imply $ g_1 \leq \frac{1}{4}$. 
We note here that $\sqrt{2b} = \sqrt{6d^{1/3}} = \alpha $. Hence, we may apply
Lemma \ref{QuadraticSuperSolution} with $k = 2b $ and $ k' = 1$ to get a function $v$ satisfying $ v \in O \left ( \frac{ e^{-2\alpha \, y} }{\sqrt{y}} \right) $ as well as 
\begin{align}\label{vBounds}
	 \begin{cases} 
		\Delta v \leq 2b \, v - 2 b \, v^2 \\
		v|_{\R} \geq g_1 .
	\end{cases}
\end{align} 

Next, recall the uniform inequality $t \leq y + A$ from Proposition \ref{t_distance} part (ii), for $t$ sufficiently large, for some $A > 0$. This inequality implies 
$v \in O \left ( \frac{ e^{-2\alpha t} }{\sqrt{t}} \right) $. Thus, it suffices to show $ \eta = x_1 - v \leq 0$.

By \eqref{vBounds}, $\eta \leq 0$ on $\R$ and $\eta \rightarrow 0$ as $ |w| \rightarrow \infty$. 
Suppose $\eta > 0$ somewhere. Then for some $\epsilon > 0$, there is a $q \in \Ha$ with $ \eta(q) > \epsilon $. Since $\eta \rightarrow 0$, we find 
$ Q = \eta^{-1}( \, [\epsilon, \infty) \, )$ is a compact region. Thus, $\eta$ attains a maximum at some $p \in Q$.
Hence, $ \Delta \eta(p) \leq 0 $. On the other hand, 
\begin{align} \label{inequality} 
\Delta \eta = \Delta x_1 - \Delta v = 2b \, x_1 - 2b\, x_1^2 - 2b\, v + 2 b v^2 %= 2b\, \eta - 2b\, (x_1^2 - v^2)
= 2b\, \eta - 2b\, \eta (x_1 +v ).
\end{align} 
Now, $\eta(p) > \epsilon > 0$ implies $ v(p) < x_1(p) $. Hence, $ v(p) + x_1(p) < \frac{1}{2} $ by \eqref{GlobalBound}. Evaluating at $p$,
	$$ 0 \geq \Delta \eta(p) \geq  b \, \eta(p) > 0,$$
a contradiction. We conclude that $\eta \leq 0$ as desired, proving the asymptotic for $x_1$.

We must be more careful with $x_2 = \tilde{u}_1 -3\tilde{u}_2$, as it may not be non-negative, but we can now leverage the asymptotic for $x_1$.
 We work in the same $q$-half-plane $(w,U)$ as before.
Set $x_2' := |x_2| $. Then the global bounds \eqref{GlobalBound} now give $ 0 \leq x_2' \leq \frac{1}{2}$ on $\Ha$. By direct calculation with \eqref{PDE_NaturalCoordinate},
$x_2$ satisfies 
\begin{align}\label{Deltax2}
	\Delta x_2 = 4b\, e^{x_2} - b\, e^{-2\tu_1} - 3 b\, e^{2\tu_2}.
\end{align} 
We now proceed in cases by sign. Suppose 
$x_2(p) \geq 0$ at some point $p \in \Ha$. Then we apply the following inequalities on $\R_{\geq 0}$: (i) $e^{x_2} \geq  1 + x_2 $ \; 
(ii) $e^{-2\tu_1} \leq 1 - 2\tu_1 + 2\tu_1^2 $ as well as (iii) $ e^{2 \tu_2} \leq 1 + 2\tu_2 + 4\tu_2^2$, which holds
for $ t $ sufficiently large since $\tu_2 \rightarrow 0$. Applying (i) - (iii) in \eqref{Deltax2}, if is $t$ sufficiently large and $x_2(p) \geq 0$,
\begin{align}\label{x2Positive}
	\Delta x_2 \geq 6b \, x_2 - 2b\, \tu_1^2 - 12b \, \tu_2^2 \geq 6b  \,x_2 - 12b \, x_1^2.
\end{align}

Suppose instead $ x_2(p) < 0$. Then $e^{x_2(p)} \leq 1 + x_2(p) + \frac{1}{2}x_2^2(p) $. We can
also apply the (global) inequalities $e^{-2\tu_1} \geq 1 - 2\tu_1$ and $e^{2\tu_2} \geq 1 + 2\tu_2$
in \eqref{Deltax2}. Finally, with $x_2^2 \leq 9 x_1^2$, we find that for $x_2(p) < 0$, 
\begin{align}\label{x2Negative}
	\Delta (-x_2) \geq 6b \, (-x_2) - 2b \, x_2^2 \geq 6b \, (-x_2) - 18b \, x_1^2.
\end{align} 

Combining \eqref{x2Positive}, \eqref{x2Negative}, we have $\Delta x_2' \geq 6b \, x_2' - 18b \, x_1^2$ for $t$ large. The first part of the proof implies $ x_1^2 \in O \left ( \frac{ e^{-4 \alpha \, t} }{t} \right )$.
Set $ g_2 := x_2' |_{\R} $ and $ 0 \leq g_2 \leq \frac{1}{2}$ from $\tilde{u}_i \leq \frac{1}{8}$. As before, $g_2$ is integrable from Corollary \ref{G2PDE_Bound} and Proposition \ref{t_distance}. 
We observe that $x_2'$ is a sub-solution to the boundary value problem \eqref{BVP_x2} on the complement of a compact set.
\begin{align}\label{BVP_x2}
	\begin{cases}
		\Delta h = 6b \, h - 18b\, x_1^2 \\
		h|_{\R} = g_2. 
	\end{cases}
\end{align}

To finish the proof, we must show $x_2' \in O \left ( \, \frac{e^{-2 \alpha \sqrt{3} \, t }}{\sqrt{t}} \right)$. To this end, we construct a super-solution $h \in C^{\infty}(\Ha-K)$ in Lemma \ref{SuperSolutionForKeyAsomptotic} to \eqref{BVP_x2} that satisfies the desired asymptotic. The maximum principle shows $ x_2' \leq h$ holds on the domain of $h$, as earlier
in the proof with $\eta$, which proves the asymptotic for $x_2'$. 

For the final claim on $\Delta x_i$, we recall that $x_i$ satisfy 
$\Delta x_i = C_i x_i + O(x_i^2)$. We conclude that $\Delta x_i$ satisfies the same $t$-asymptotics as $x_i$. 
 \end{proof} 

Here is the lemma that completes the proof of the fundamental asymptotic. 
 
 \begin{lemma}\label{SuperSolutionForKeyAsomptotic} 
 Let $g \in C^0(\R) \cap L^1(\R) $ with $0 \leq g \leq 1$. Let $t$ denote distance to the origin in $\Ha$ and take $f \in O \left( \frac{ e^{-4\alpha t } }{t }\right) $ as $t \rightarrow \infty$. Then there is a compact set $K \subset \Ha \cup \R$ and a super-solution $u \in C^{\infty}(\Ha) \cap C^0(\R)$ satisfying
$\Delta u \leq 6b \, u - f$ on $ \Ha - K$ as well as $u |_{\R} \geq g$. 
Moreover, as $ t \rightarrow \infty,$ for $ \alpha = 2^{1/3}3^{1/4}5^{1/3} $, we have
 $$u \in O \left ( \frac{ e^{-2\sqrt{3}\alpha \, t} }{\sqrt{t} } \right) .$$
 \end{lemma} 
 
 \begin{proof}
First, select $h$ from Lemma \ref{FundamentalEstimate} with respect to $k = 6b$ and the same $g$ as its boundary values.
Then $h \in O \left ( \frac{ e^{-2\sqrt{3}\alpha \, t} }{\sqrt{t} } \right) $ since $ 2\sqrt{6b} = 2\sqrt{3}\alpha$, using again the 
argument from Lemma \ref{FundamentalAsymptotic} to replace $y$ with $t$. 
We will construct our desired super-solution by $u = 2h - h^\beta$ for $ 1 < \beta < \frac{2}{\sqrt{3}}$. 
First, note that since $g \leq 1$ and $ \beta > 1$, we have $ u|_\R = 2g - g^\beta \geq g $. We need only show
$\Delta u - 6b \, u + f \leq 0$ to finish the proof.

 By an elementary calculation, $\Delta (h^\beta) = \beta (\beta -1) h^{\beta-2} (h_{z}^2 + h_{\zbar}^2)+ \beta h^{\beta-1} \Delta h$. 
Since $\Delta h = 6b \, h $ by construction of $h$, we get the inequality 
\begin{align}\label{Inequality} 
\Delta( h^\beta) \geq 6b \beta \, h^\beta.
\end{align} 
By definition of $u$, \eqref{Inequality} implies the inequality $ \Delta u \leq 12b \, h - 6 b \beta \, h^{\beta} .$
It follows that
$$ \Delta u - 6b \,u + f \leq 6bh^\beta \, (1 - \beta ) + f .$$ 
Since $ h \in O \left ( \frac{ e^{-2\sqrt{3} t} }{\sqrt{t} } \right ) $, we see that $f \in o( h^\beta)$ for chosen $\beta$. 
Set $C = 6b (1-\beta) < 0$. 
Then for $R$ sufficiently large, $t > R$ implies $ f < -\frac{C}{2} h^\beta$ and hence $ \Delta u - 6b \,u + f < 0$.
 \end{proof} 
 
  We note as a corollary that $\del x_i$ has the same asymptotics as $x_i$. 
 
 \begin{corollary}\label{DerivativeAsymptotics} 
Under the same hypotheses as Lemma \ref{FundamentalAsymptotic}, we have 
$$ \del x_1 \in O \left( \,\frac{e^{-2\alpha \, t}}{\sqrt{t}} \, \right) \; \; \; \text{and} \; \;  \; \del x_2 \in O \left ( \, \frac{e^{-2 \alpha \sqrt{3} \, t }}{\sqrt{t}} \, \right) .$$
 \end{corollary} 
 
 \begin{proof}
We first note that that $\Delta(x_i) \in O(x_i) $ by the fact that $x_i$ are eigenvectors of the linearized system and $x_i$ decays exponentially. 
Applying the interior estimate for solutions to Poisson's equation \cite[Theorem 3.9]{GT01} to $x_i$ on a ball of radius 1 around a point $p_0$ of distance $t$ from the zeros of $q$, along with Lemma \ref{FundamentalAsymptotic}, shows that $\partial x_i \in O(x_i)$.
 \end{proof} 

\section{Asymptotic Parallel Transport \& Unipotent Transitions}\label{ParallelTransport}

In this section, we use the precise rate of exponential decay of the error term $\bm{\tu} = \bm{u} - \bm{u}^- $ to compute the asymptotic parallel 
translation $\mathcal{P}_z: \V_z \rightarrow \V_0$  in the bundle $\V$. We calculate the limit of $\mathcal{P}_z$ in a natural coordinate for $q$ along quasi-rays.
In particular, we find that the limits $\mathcal{P}_{\theta} := \lim_{t \rightarrow \infty} \mathcal{P}_{\gamma_{\theta}(t)} $ of the parallel translation 
along Euclidean rays (or quasi-rays) $\gamma_{\theta}(t) = te^{i\theta}$ of angle $\theta$ are constant for angles $\theta \in I_k$ for certain intervals $I_k$. 
We call these limits $\mathcal{P}_{\theta}$, for $ \theta \in I_k$, \emph{stable limits}. 
Moreover, when crossing certain angles $\theta \in \mathcal{U}$ for some finite set $\mathcal{U}$, the stable limits change.
We call these angles $\theta \in \mathcal{U}$ \emph{unstable}. Finally, we relate the neighboring stable limits $ \mathcal{P}_{\theta - \epsilon},  \mathcal{P}_{\theta + \epsilon}$
reached by crossing a ray in the direction of an unstable angle $\theta \in \mathcal{U}$. In particular, the transition between the stable limits is given by a unipotent matrix, 
computed in Lemma \ref{Unipotents}.
 
\alignL \textbf{The ODE for Parallel Translation.} We now write the ODE for the matrix representative of $\nabla$-parallel translation. Fix the origin $p_0$ as a distinguished basepoint in $\C$. 
Formally, we view a frame $\mathcal{F}$ as a section $\mathcal{F} \in \Gamma(\C, \; \Hom_{\C}( \uline{\C}^7, \V) \,)$. 
We equip $\Hom_{\C}( \uline{\C}^7, \V)$ with the product connection, where $\V$ is equipped with $\nabla$ and $ \uline{\C}^7$ with $d$.  
We write $(\textbf{e}_i)_{i=3}^{-3}$ as the standard basis of $\uline{\C}^7$, indexed to match the powers of $\K$ in $\V$. 
We define the \emph{standard frame} $ \mathcal{F} = (f_i)_{i=3}^{-3} $ by 
\begin{align} \label{StandardFrame} 
	f_i := \textbf{e}_i \mapsto u_i
\end{align} 
for $\mathcal{B} = (u_i)_{i=3}^{-3}$ from \eqref{BaragliaBasis}.

Define $\widetilde{\mathcal{F}}_z$ as the $\nabla$-parallel translation of $ \mathcal{F}_{p_0}$ to the fiber $\V_z$. 
We search for the map $\psi:\C \rightarrow \GL_7\C$ satisfying $\widetilde{\mathcal{F}} \psi = \mathcal{F}$, i.e., expressing the change of basis from $\widetilde{\mathcal{F}} $ to $\mathcal{F}$.
Equivalently, $\psi(z)$ is the matrix representation of parallel translation $ [\mathcal{P}_z: \V_z \rightarrow \V_{p_0} ]$ in the standard frame $\mathcal{F}$. 
The value of $\psi$ along a Euclidean ray is found as the solution to an ODE along a Euclidean ray $\gamma(t) = te^{i \theta}$ for $t \geq 0$. Going forward, we use the shorthand $\mathcal{F}(t) := \mathcal{F}(\gamma(t)) $. \alignLend

On the one hand, the product rule and the fact that $\widetilde{\mathcal{F}}$ is $\nabla$-parallel yields
$$  \nabla_{\der{t}} \mathcal{F}(t) = \nabla_{\der{t}}(\widetilde{\mathcal{F}} \psi )(t) = \left ( \widetilde{\mathcal{F}} \frac{ d\psi}{dt} \right)(t) .$$ 
Conversely, since $\mathcal{F} = \mathsf{id}$ in the basis $\mathcal{B}$, 
$$ \nabla_{\der{t}} \mathcal{F}(t) =  \mathcal{F} \mathcal{A}( e^{i\theta})= \widetilde{\mathcal{F}} \psi(t) ( \, e^{i\theta}  (\varphi + D_H)+ e^{-i\theta} \varphi^* \, ).$$
Combining these equations, we deduce that $\psi(t)$ solves the following initial value problem. 
\begin{align}\label{Psi_ODE} 
	\begin{cases}
	\psi^{-1} \frac{d \psi}{dt} &= \, (\varphi + D_H) e^{i\theta} + \varphi^*e^{-i\theta} \\
	\psi(0) &= \mathsf{id},
	\end{cases} 
\end{align} 
where $D_H = \nabla_{\delbar, h}^{1,0} = H^{-1} \del H$. This means the almost-complex curve $\hat{\nu}: \C \rightarrow \quadric$ depends on both 
the PDE \eqref{HitchinsEquations_rs} for the harmonic metric $h$ as well as ODE \eqref{Psi_ODE} governing the parallel translation. Indeed, as $u_0 = i$, we have $\hat{\nu}(z) = \Psi(z) \mathbf{e}_0$, in the basis \eqref{BaragliaBasis} for $\imoct^\C$. 

\subsection{Parallel Translation when $q = 1$} \label{ModelParallelTranslation} 

In this section, we compute the $\nabla$-parallel translation endomorphism $\psi_0$ explicitly when $q_0 \equiv 1$. In the next section, we show the parallel translation endomorphism
$\psi$ for general monic polynomial $q$ is asymptotic to $\psi_0$ in a natural coordinate for $q$. 

Observe that when $q _0= dz^6$, the solution to Hitchin's equations \eqref{HitchinsEquations_rs} is a pair of constants, determined by 
$\log(r^2s) = \log(c^{5/3}), \log(s) = \log(d^{1/3})$, by the equation \eqref{ExactSolution}, which is now a global solution. Using the relations \eqref{cdSystem}, we find 
\begin{align}
	r_0 &= \frac{6}{5} d^{2/3} = \left (\frac{2}{5} \right)^{1/3}\\
	s_0 &= d^{1/3} = \frac{1}{\sqrt{3}} \left(\frac{5}{2} \right)^{1/3}.
\end{align}

Thus, the matrix representative of the harmonic metric $h_0$ is given by

\begin{align}\label{HitchinConstantMetric} 
	H_0 = \mathsf{diag} \,\left (\, \sqrt{3}, \, \left( \frac{5}{2} \right)^{1/3}, \, \sqrt{3} \left( \frac{2}{5} \right)^{1/3},\, 1,\, \frac{1}{\sqrt{3}} \left( \frac{5}{2} \right)^{1/3},\, \left( \frac{2}{5} \right)^{1/3}, \frac{1}{\sqrt{3} } \right) . 
\end{align} 

We now solve explicitly for the parallel translation matrix $\psi_0$ from \eqref{Psi_ODE}. The ODE simplifies considerably in this case as $\varphi$ and $\varphi^{*h}$ commute. Since the Higgs field $\varphi_0$ and its adjoint $\varphi_0^{*}$ are constant matrices
in the standard coordinates, we find $\psi_0(z) = e^{z \varphi_0 +\zbar \varphi^*_0} $ solves the initial value problem \eqref{Psi_ODE}.

Let us simultaneously diagonalize $\varphi_0, \varphi^*_0$. 
The mutual eigenvalues are recorded in the diagonal matrix 
\begin{align}\label{HiggsEigenvalues} 
	D = \diag( \alpha \xi^1, \alpha \xi^3, \alpha \xi^5, \alpha \xi^7, \alpha \xi^9, \alpha \xi^{11} , 0 ),
\end{align}
where  $\xi = e^{\frac{ i \pi }{6}}$ and 
\begin{align} \label{Alpha} 
	\alpha = 5^{1/6}3^{1/4}2^{1/3}.
\end{align} 
The corresponding eigenvectors recorded as a matrix $E_0$. We also obtain
\begin{align}\label{ModelHiggs} 
	 \varphi_0    &= E_0 D E_0^{-1} \\
	 \varphi_0^* &= E_0 \overline{D} E_0^{-1}.
\end{align} 
Denoting $\Re, \Im$ as real and imaginary parts now, the identity $e^{gGg^{-1}} = ge^{G} g^{-1}$ shows 
$$\psi_0  =E_0 e^{ z D + \overline{ z D} } E_0^{-1}= E_0 e^{2 \, \Re(z D) } E_0^{-1}.$$
We can explicitly describe the eigenvector matrix $E_0$ by factoring it. To this end, set

 \begin{align}\label{ConstantMetric}
 	(H_0)^{-1/2}  &=  \mathsf{diag} \left ( \; 3^{-\frac{1}{4} }, \; \frac{2^{1/6} }{5^{1/6}}, \; \frac{5^{1/6}}{3^{\frac{1}{4}} 2^{1/6} },  \;1, \; \frac{3^{1/4} 2^{1/6} }{5^{1/6}}, \; \frac{5^{1/6}}{2^{1/6}} \; 3^{\frac{1}{4} } \right) \\
	\label{HiggsEigenbasisInOrder}  S &= \frac{1}{\sqrt{6}} \begin{pmatrix}   \frac{1}{\sqrt{2}} & \frac{1}{\sqrt{2}} & \frac{1}{\sqrt{2}}& \frac{1}{\sqrt{2}}& \frac{1}{\sqrt{2}}& \frac{1}{\sqrt{2}}& \sqrt{-3} \\
				\xi^{11} & \xi^9& \xi^7& \xi^{5} &\xi^3 & \xi & 0 \\
  				 \xi^{10} & -1& \xi^2 & \xi^{10} &-1 &\xi^2& 0\\
  				 -1& 1& -1& 1& -1& 1& 0 \\
				  \xi^{2}& -1& \xi^{10} & \xi^2 & -1 & \xi^{10}& 0 \\
  			 	\xi & \xi^{3} & \xi^5 & \xi^7 & \xi^9 & \xi^{11} & 0\\
				\frac{1}{\sqrt{2}}& \frac{1}{\sqrt{2}}& \frac{1}{\sqrt{2}}& \frac{1}{\sqrt{2}}& \frac{1}{\sqrt{2}}& \frac{1}{\sqrt{2}}& -\sqrt{-3} \end{pmatrix}.
  \end{align} 
A (computer-assisted) calculation shows the following matrix $E_0$ is an eigenvector matrix for $\varphi_0$ satisfying \eqref{ModelHiggs}: 
\begin{align}\label{HiggsEigenbasis} 
	E_0 = H_0^{-1/2} S.
\end{align} 
Switching to polar coordinates $z = re^{i\theta}$,  define $\mathcal{D} = e^{2 \, \Re(zD)  }$, so that 
\begin{align}\label{ModelParallelTranslation} 
	\psi_0 = E_0 \mathcal{D} E_0^{-1}.
\end{align}
Here, 
\begin{align} \label{ModelDiagonalTiteca} 
	\mathcal{D} = \mathsf{diag} ( e^{ 2\alpha r\, \cos(\theta + \frac{\pi}{6} ) },  e^{ 2\alpha r\, \cos(\theta + \frac{3\pi}{6})}, e^{ 2\alpha r\, \cos(\theta + \frac{5\pi}{6})}, 
	e^{ 2\alpha r\, \cos(\theta + \frac{7\pi}{6}) }, e^{ 2\alpha r\,\cos(\theta + \frac{9\pi}{6}) }, e^{ 2\alpha r\,\cos(\theta + \frac{11\pi}{6}) }, 1 ). 
\end{align} 
We also write 
\begin{align}\label{DExponential} 
	\mathcal{D} = e^{2\alpha r \, \mathscr{D} }, 
\end{align} 
where  $\mathscr{D} = \mathsf{diag}(d_i)_{i=1}^7$ is given by $d_i = \frac{1}{|z|} \, \Re(\xi^{2i+1} z)$ for $i \in \{1,2,3,4,5,6\}$ and $d_7 =0$: 
\begin{align}\label{DiagonalCombinatorics} 
 \mathscr{D} = \mathsf{diag} \left ( \cos \left(\theta + \frac{\pi}{6} \right), \; \cos \left(\theta + \frac{3\pi}{6} \right),\; \cos \left ( \theta + \frac{5\pi}{6} \right), \; \cos \left ( \theta + \frac{7\pi}{6} \right), \; \cos \left ( \theta + \frac{9\pi}{6} \right) ,
	\cos \left( \theta + \frac{11 \pi}{6} \right) , \; 0 \right).
\end{align} 

The global bounds of $d_i$ on the intervals $I_k := ( \frac{(k-1)\pi}{6}, \frac{k \pi}{6}) \subset \R$  
will be crucial for the asymptotics of $\nabla$-parallel translation along rays in the general case. We will use $d_{j+3} = -d_j$ repeatedly. Note the following bounds:
\begin{align}\label{CombinatorialEstimate} 
	\; \; \; \; \; \; \; \text{for} \; \theta \in I_k, \; \;   | d_i(\theta) - d_j(\theta) | <  \begin{cases}  1 & i,j \leq 6, \, | i-j | \in \{1,5\} \mod 6 \\
														  \sqrt{3} & i, j \leq 6, \, | i-j | \in \{2,4\} \mod 6 \\
														 1 & i = 7\; \text{or} \; j = 7\end{cases} 
\end{align} 

The case of $i =7 $ or $j =7$ is clear since $d_7 = 0$. The remaining cases are 
handled by elementary calculus. 
The asymptotic geometry of the model surface $\nu_0$ associated to $q_0 \equiv 1$ is discussed in Section \ref{ModelSurface},
using the calculations from this current section. 

\subsection{Asymptotic Parallel Translation} \label{ComputeAsymptoticParallelTranslation} 
Now, we return to our parallel translation endomorphism $\psi$ solving \eqref{Psi_ODE} and satisfying $ \widetilde{\mathcal{F}}\psi = \mathcal{F} $. 
We will show the solution $\psi$ with respect to monic polynomial $q$ is asymptotic to $\psi_0$, the solution when $q \equiv 1$, along certain sectors of a natural coordinate
$(w,U)$ for $q$. The asymptotics for $\psi$ are a consequence of the asymptotics of the error term from Section \ref{ErrorEstimates}. \\

We now derive a new ODE for $B := \psi \psi_0^{-1}$ along a ray $\gamma(s) = se^{i\theta} $ in a $q$-half-plane $(w,U)$. 
We aim to show the limit of $B$ exists along \emph{stable rays}, as defined shortly. 
We now recall that that $\psi(s) = \psi(\gamma(s))$
and $\psi_0(s) = \psi_0(\gamma(s))$ satisfy the following initial value problems for some $A, A_0 \in \Gtwo$: 

$$\begin{cases} 
		& \psi^{-1} \frac{d \psi}{ ds} = e^{i\theta}(D_H + \varphi) + e^{-i\theta} \varphi^{*} \\
		&\psi(s_0) = A
	\end{cases}  \; \;\text{and} \; \; \begin{cases} 
	& \psi_0^{-1} \frac{d \psi_0}{ds} = e^{i\theta} \varphi_0+ e^{-i\theta} \varphi^{*}_0 \\
	&\psi_0(s_0) = A_0 \; .
	\end{cases} $$ 
By an elementary calculation, $ B(t) := \psi \psi_0^{-1}(\gamma(t))$ is a solution to the initial value problem 
$$ \begin{cases}
	B^{-1} dB =  \psi_0 \, R \, \psi_0^{-1}  \\
	 B(0) = A A_0^{-1} \end{cases} ,$$
where the error term $R$ is given by 
\begin{align}\label{Rmatrix} 
R  = \psi^{-1} d\psi - \psi_0^{-1} d\psi_0.
\end{align} 
 Since $\varphi \equiv \varphi_0$ in a natural coordinate for $q$, we conclude that 
 \begin{align}
 B^{-1} \frac{dB}{ds} = \psi_0 \, \left (e^{i\theta}D_H + e^{-i\theta} (\varphi^{*} - \varphi^*_0) \, \right ) \, \psi_0^{-1}.
 \end{align}
We denote 
$$\Theta:=  B^{-1} dB =  \psi_0 \, R \, \psi_0^{-1} .$$ 
The asymptotic analysis of $\Theta$ going forward involves a comparison of the exponential growth of $\psi_0$ in these coordinates
to the exponential decay of the matrix $R$. 

Split the error matrix $R = e^{i\theta} D_H + e^{-i\theta} (\varphi^{*} - \varphi^*_0)$ into two terms as follows: 
\begin{align} \label{Rpieces1}
	R' &= e^{i\theta} D_H\\
	R'' &= e^{-i\theta} (\varphi^{*}-\varphi^*_0).
\end{align}
Recalling the expression \eqref{ModelParallelTranslation} for $\psi_0$, re-write $\Theta$ once more as
\begin{align}\label{ThetaRewriten} 
	 \Theta = E_0 (\mathcal{D} (R_1 + R_2) \mathcal{D}^{-1}) E_0^{-1},
\end{align} 
where 
\begin{align}\label{Rpieces2}
	R_1 :&= E_0^{-1} R' E_0 = S^{-1} R'S\\
	 R_2 :&= E_0^{-1} R''E_0.
\end{align} 
Here, we recall $S$ from \eqref{HiggsEigenbasisInOrder}. The first equation follows from the expression $E_0 = H_0^{-1/2}S$ for $E_0$ and the fact that $H_0^{1/2}$ and $R'$ commute. 
Since $E_0$ is a constant-coefficient matrix, the asymptotics of $\Theta$ depend only on $\Ad_{\mathcal{D}}(R_i)$. 
Using the expression \eqref{DExponential} for $\mathcal{D}$, we observe that as we conjugate the error matrix $R_k$ by $\mathcal{D}$, the $(i,j)$ entry of $R_k$ is multiplied by $ e^{2\alpha \, s(d_i - d_j) }$.
By the bounds \eqref{CombinatorialEstimate}, we have upper bounds on the exponent of $e^{2\alpha \, s(d_i - d_j)(\theta) }$ for $\theta \in I_k
= ( \frac{(k-1)\pi}{6}, \frac{\pi}{6})$. We now determine the precise rate of decay of $(R_k)_{ij}$.

Recall that the components of the harmonic metric $h$ are given by $r = e^{ u_1 - u_2 }, s = e^{2u_2}$.\footnote{We note the abuse of notation that the component $s$ of the metric $h$ appears here, instead of $s$ used to denote a parameter of the curve $\gamma$.} Since $u_i, \tilde{u}_i$ differ only by constants in a natural half-plane, we find $R', R''$ are given by 
	\begin{align}
		R'&= 	 \begin{pmatrix} - \del ( \tilde{u}_1 + \tilde{u}_2) &  & & & & &\\
				      			 & - \del( \tilde{u}_1 - \tilde{u}_2) & & & & & \\
			  	     			 & & -2 \, \del \tilde{u}_2 & & & &\\
			 	     			 & && 0& & &\\
				     			 & & & &  2 \, \del \tilde{u}_2 &  &\\
				     			 & & & && \del( \tilde{u}_1 - \tilde{u}_2) & \\
				     			 & & & & & &  \del ( \tilde{u}_1 + \tilde{u}_2) \\
				     \end{pmatrix} 	\\
		    R'' &= 	 \begin{pmatrix} 0& \sqrt{3}\, (s -s_0 )& & & & &\\
				       & 0& \sqrt{5} (\frac{r}{s} - \frac{r_0}{s_0})& & & & \\
			  	      & & 0 &\sqrt{-6}(s - s_0)  & & &\\
			 	      & && 0&\sqrt{-6}(s-s_0) & &\\
				      & & & & 0& \sqrt{5} (\frac{r}{s} -\frac{r_0}{s_0})&\\
				      \frac{1}{ r^2s }- \frac{1}{r_0^2s_0}& & & &&0& \sqrt{3} \,(s-s_0) \\
				      & \frac{ 1}{ r^2s } -\frac{1}{r_0^2s_0}& & & & & 0\\
				     \end{pmatrix} .		\label{R''} 
	\end{align}

Now, by (computer-assisted) matrix multiplication, we find that the entries of $R_1 = E_0^{-1}R'E_0$,\; $ R_2= E_0^{-1}R''E_0$ are scalar multiples of $x_1 = \tu_1 + \tu_2, \, x_2 = \tu_1 - 3\tu_2$
and their derivatives. In particular, we have a combinatorial description of the entries of $R_1$ as follows: for some 
constants $c_{ij} \in \C^*$: 
\begin{align}\label{R1Asymptotics} 
	(R_1)_{ij} = \begin{cases}  c_{ij} ( \del x_1) & i, j \leq 6, \, | i-j | \in \{1,5\} \mod 6 \\  c_{ij} ( \del x_2)&  i, j \leq 6, \, | i-j | \in \{2,4\} \mod 6\\  0 &  i, j \leq 6, \,  | i-j | \in \{0,3 \}\mod 6  \\
		c_{ij} ( \,(1-\delta_{ij})  \del x_1) & i = 7 \; \text{or} \; j =7.\end{cases}
\end{align}

One must work a little harder for the analogous description of $R_2$. Using the 
asymptotics from \eqref{FundamentalAsymptotic}, note the following:
	\begin{align}  \label{R''AlgebraFirst}
		s-s_0 &= d^{1/3}(e^{2\tu_2}-1) = 2 d^{1/3} \tu_2 + O( x_1^2)  \\
		 \frac{1}{ r^2s }- \frac{1}{r_0^2s_0} &= c^{-5/3} (e^{-2\tu_1}-1) = -2c^{-5/3}\tu_1 + O( x_1^2) \\
		\frac{r}{s} - \frac{r_0}{s_0} &= c^{5/6}d^{-1/2}(e^{\tu_1 - 3\tu_2} - 1) = c^{5/6}d^{-1/2} x_2 + O(x_2^2). \label{R''AlgebraLast} 
	\end{align}
Combining the asymptotics \eqref{R''AlgebraFirst} -- \eqref{R''AlgebraLast} with equation \eqref{R''}, a (computer-assisted) matrix multiplication shows
$R_2 = E_0^{-1}R''E_0$ admits a similar asymptotic description to that of $R_1$. Indeed, 
\begin{align}\label{R2Asymptotics} 
	(R_2)_{ij} = \begin{cases}  \asymp( x_1) &  i, j \leq 6,\,  | i-j | \in \{1,5\} \mod 6 \\  \asymp( x_2)& i, j \leq 6, \, | i-j | \in \{2,4\} \mod 6\\  0 &  i, j \leq 6, \, | i-j | \in \{0,3\} \mod 6 \\
		\asymp( (1-\delta_{ij})\, x_1) & i =7 \; \text{or} \; j =7. \end{cases}
\end{align}

We can now prove the existence of the limit $\psi \psi_0^{-1}$ along \emph{stable} rays or quasi-rays in a natural coordinate.
We discuss the appropriate notion of stability here. 

\begin{definition} 
Let $\gamma(t) = te^{i\theta}$ be a ray in a standard $q$-half-plane. We call $\gamma$ \textbf{stable} when $\theta \notin (k \frac{\pi}{6} )_{k=0}^6$. We call a quasi-ray $\gamma'$ stable if and only if an associated ray of $\gamma'$ is stable. If $\gamma: \R_{+} \rightarrow \C$ is a $q$-ray, then we call
$\gamma$ stable if it is stable in any half-plane $(w,U)$, which is independent of this choice by Remark \ref{angle}. 
\end{definition} 

As seen in the following lemma, stability of $\gamma$ tells us that the exponential decay of $\mathcal{R}$ is faster than the exponential growth of $\psi_0$ along $\gamma$. 
We now show stability guarantees us the existence of a limit of $B$ in the sectors the sectors 
$S_{k} = \{ z \; | \; \mathsf{arg}(z) \in ( \frac{(k-1) \pi}{6}, \frac{ k \pi}{6} ) \; \}$ for $k \in \{1,2, \dots, 6\}$.
 
\begin{lemma}[Asymptotic Parallel Translation in Stable Sectors] \label{AsymptoticParallelTranslation} 
Let $\gamma$ be a stable ray or quasi-ray (eventually contained) in a standard $q$-half plane $(w,U)$. Then $\lim_{s \rightarrow \infty} \psi \psi_0^{-1}(s)$ exists.
Moreover, if $\psi$ is eventually contained in the sector $S_k$, then $\lim_{s \rightarrow \infty} \psi \psi_0^{-1}(\gamma(s)) = L_{k} $ for some $L_k \in \Gtwo$.
\end{lemma}

\begin{proof}
First, we suppose $\gamma$ is a ray and show the limit $\lim_{s \rightarrow \infty}\psi \psi_{0}^{-1}(\gamma(s))$ exists.

Set $B(s) = \psi \psi_0^{-1}(\gamma(s))$ for $\gamma(s) = se^{i\theta}$ a stable ray of angle $\theta$. Then $B$ solves 
	$$ \frac{dB}{ds} = B \Theta(s) ,$$
where $\Theta = E_0\mathcal{D} (R_1 + R_2) \mathcal{D}^{-1} E_0^{-1}$, with some appropriate initial value $B(0) = B_0 \in \Gtwo$.
Combining the asymptotics \eqref{R1Asymptotics}, \eqref{R2Asymptotics} with Lemma \ref{FundamentalAsymptotic}, 
and Corollary \ref{DerivativeAsymptotics}, the matrix $\mathcal{R} = R_{1} + R_{2}$ satisfies 
\begin{align}\label{R_Asymptotic} 
	 \mathcal{R}_{ij} = \begin{cases}  O( \frac{e^{-2\alpha s} }{\sqrt{s}} ) & i, j \leq 6, \, | i-j | \in \{1,5\} \mod 6 \\  O(  \frac{e^{-2\alpha \, \sqrt{3} s} }{\sqrt{s}}  )& i, j \leq 6,\, | i-j | \in \{2,4\} \mod 6\\  0 &   i, j \leq 6, \, | i-j | \in \{ 0,3\} \mod 6 \\
	 O( \, (1-\delta_{ij})\,  \frac{e^{-2\alpha s} }{\sqrt{s}} \, ) &  i = 7 \; \text{ or} \; j=7.	\end{cases}
\end{align} 
Recall $\mathcal{D} = e^{2 \alpha s\, \mathscr{D} }$, where $\mathscr{D} = (d_i)_{i=1}^7$ is the matrix \eqref{DiagonalCombinatorics}. 
Now, conjugation by $\mathcal{D}$ changes the $(i,j)$ entry of a matrix by 
the factor $\lambda_{i,j} =  e^{ 2\alpha s\, (d_i-d_j)(\theta) }$. Stability of $\gamma$ allows us
to apply the strict inequalities from \eqref{CombinatorialEstimate}. Combined with the asymptotics \eqref{R_Asymptotic}, we find 
\begin{align}\label{KeyDecay} 
	\mathcal{D} \mathcal{R}_{ij} \mathcal{D}^{-1} \in O( \lambda_{i,j} \mathcal{R}_{ij} ) = O \left( \frac{ e^{-c_{ij} s} }{\sqrt{s}} \right),
\end{align} 
for some $ c_{i,j} > 0$. Since the matrix $E_0$ is constant, we conclude $\int_{a_0}^{\infty} || \Theta(\gamma(t)||_{\infty} dt < \infty$. 

By Standard ODE results, \cite[Lemma B.1]{DW15}, we conclude that the limit $\lim_{s \rightarrow \infty} B(s)$ exists. Since $B(s) \in \Gtwo$ for all $s$, the limit $L_k \in \Gtwo$ as well as $\Gtwo$ is a closed subgroup of $\mathsf{GL}_7\C$. 

We now show the limit is independent of which stable ray we take within a stable sector $S_i$. Let $\gamma_1 = se^{i\theta_1}, \; \gamma_{2} = se^{i\theta_2}$ be two stable rays
with $\theta_1, \theta_2 \in I_k =  (\frac{(k-1)\pi}{6}, \frac{k\pi}{6} )$. Then define $G_i(s) := \psi \psi_0^{-1}(\gamma_i(s) ) $.  We show that $\lim_{s \rightarrow \infty} G_1(s) = \lim_{s \rightarrow \infty} G_2(s)$. 
Define $\eta_s(t) := \gamma_1(s) + (1-t) \gamma_2(s)$ for $t \in [0,1]$ and consider $g_s(t) := G_1^{-1}(s)\, \psi \psi_0^{-1}(\eta_s(t))$ for $ t \in [0,1]$. Observe the following:
$$ \begin{cases}
	g_s(0) &= \mathsf{id} \\
	g_s^{-1} g_s'(t) &= \Theta(\eta_s(t)) \eta_s'(t) \\
	g_s(1) &= G_1^{-1} (s)G_2(s). 
\end{cases} $$
We show that $g_s(1) \rightarrow \mathsf{id}$ as $s \rightarrow \infty$ to prove the claim. However, by the same argument from earlier, $\Theta_{ij}$ satisfies 
the asymptotics of \eqref{KeyDecay} for all $t \in [0,1]$ since $\eta_{s}(t)$ stays within the sector $S_k$. Hence, $\Theta(\eta_s(t)) \eta_s'(t) \in O(e^{-c_{ij} s} \sqrt{s} )$ since $|\eta_s'(t)| \in O(s)$. Again applying \cite[Lemma B.1]{DW15}, we find that
$ g_s(1) \rightarrow \mathsf{id}$ as $s \rightarrow \infty$. \\

The argument is nearly identical in the case of a quasi-ray eventually contained in $S_k$. In this case, set $\gamma_1$ to be a quasi-ray and $\gamma_2$ to be an associated 
stable ray (cf. Definition \ref{QuasiRay}). Then the straight-line homotopy $\eta_{s}(t)$ from $\gamma_1(s)$ to $\gamma_2(s)$ by hypothesis satisfies $|\eta_{s}'(t)| \in o(s) $ and moreover
the same asymptotics \eqref{KeyDecay} apply since we do not cross an unstable ray along $\eta_s(t)$. 
\end{proof} 

\subsection{Unipotent Transitions Between Stable Limits} 

We can now find unipotent matrices relating the stable limits separated by unstable rays.  

\begin{lemma}[Unipotents Relating Limits of Model Surface] \label{Unipotents} 
Let $L_k$ be the limits from Lemma \ref{AsymptoticParallelTranslation}. Then $L_k^{-1}L_{k+1} = E_0U_{k} E_0^{-1}$ for a unipotent matrix $U_k \in \Gtwo$,
where $E_0$ is from \eqref{HiggsEigenbasis}.
\end{lemma} 

\begin{proof}
Along an unstable ray $\gamma$ of angle $\theta$, some of the bounds from \eqref{CombinatorialEstimate} are saturated. Determining for which indices $(i,j)$
the bounds on $d_i -d_j(\theta)$ are saturated will tell us for which indices the exponential growth of $\lambda_{ij}$ matches the exponential decay of $\mathcal{R}_{ij}$. 
We will need this information later.

We first examine $\theta = \frac{\pi}{3}$. 
Recall that $d_{j+3} = -d_j$.  At this $\theta$, we find the only saturated bounds are $ (d_6 - d_2)(\theta) = (d_5- d_3)(\theta)$. 
Translating, $(d_{6-j} - d_{2-j} )(\theta) = d_{5-j} - d_{3-j} (\theta)= \sqrt{3}$ are the saturated bounds for $\theta = \frac{ j\pi}{3}$,
with indices mod 6. 

We now examine multiples of $\theta = \frac{\pi}{6}$. We find $ (d_5 - d_4)(\theta) = (d_1- d_2)(\theta) = 1 $ as well as 
 $(d_6 - d_7)(\theta)   =d_6(\theta)  = -d_3(\theta)  = (d_7 - d_3)(\theta) $ are the only saturated bounds for this $\theta$.
Translating, one finds at $\theta = \frac{\pi}{6} + \frac{j \pi}{3}$ the saturated bounds are $d_{5-j} - d_{4-j}, \; d_{1-j} -d_{2-j}, d_{6-j} - d_7, d_{7} -d_{3-j}$,
indices mod 6 again.

We describe the unipotents transitions across the unstable rays at angles $\frac{\pi}{6}, \; \frac{\pi}{3}$. 
Let $E_{ij}$ be the elementary matrix with entry 1 at index $(i,j)$. Then for constants $c^{ij} \in \C$ to be defined, we show
	\begin{align}
		 U_1 &= \mathsf{exp} \left( \, c^{67} \, E_{67} + c^{73}\, E_{73} + c^{12}\,E_{12} + c^{54}\,E_{54}  \right).\label{UnipotentsEquation1}  \;\\
		U_2 &= \mathsf{exp} \left( \, c^{53} \, E_{53} + c^{62} \,E_{62} \right) .\label{UnipotentsEquation2} 
	\end{align} 
More generally, for indices $k$ mod 6, we have 
	\begin{align}
		U_{1+2k} &=  \mathsf{exp} \left ( c^{6-k, 7} \, E_{6-k,7} +  c^{7, 3-k} \, E_{7, 3-k} + c^{1-k, 2-k}\, E_{1-k,2-k} + c^{5-k, 4-k}\, E_{5-k,4-k} \, \right) \\
		U_{2+2k} &= \mathsf{exp} \left(  c^{5-k,3-k}\, E_{5-k,3-k} + c^{6-k,2-k}\,  E_{6-k,2-k} \; \right) .
	\end{align} 	
We handle the case of $U_1$ as the argument is similar for $U_2$. Fix the stable rays $\gamma_{1}(s) = e^{i \frac{\pi}{12}} s$
and $\gamma_2(s) = e^{i \frac{\pi}{4}}s$. By Lemma \ref{AsymptoticParallelTranslation}, we know that $G_i(s) = \psi\psi_0^{-1}(\gamma_i(s)) \rightarrow L_i$
as $s \rightarrow \infty$. 
Form the arc $\eta_{s}(t) = se^{i (t + \frac{\pi}{12}) \,}$ for $ t \in \left[ 0, \frac{\pi}{6} \right]$ between $\gamma_1(s), \gamma_2(s)$. Define 
$g_s(t) = G_{1}(s)^{-1}\; \psi \psi_0^{-1}(\eta_s(t))$. Just as in the proof of Lemma \ref{AsymptoticParallelTranslation}, we see that $g_s(t)$ satisfies 
$$ \begin{cases} g_s(0) = \mathsf{id} \\ 
			g_s( \frac{\pi}{6} ) = G_1^{-1}(s) G_2(s) \\ 
			(g_s^{-1}g_s')(t) = \Theta(\eta_s(t)) \, \eta'_s(t) \end{cases} ,$$
where $\Theta = B^{-1}dB$ is from \eqref{ThetaRewriten}. We follow the global notation from this section going forwards. 

Recall $\Theta = E_0 \mathcal{D} \mathcal{R} \mathcal{D}^{-1} E_0^{-1}$.
We re-write $\Theta = E_0 A E_0^{-1}$, where $A = \mathcal{D} \mathcal{R} \mathcal{D}^{-1}$, has indices $A_{ij} = \lambda_{ij} \mathcal{R}_{ij}$ and $\lambda_{ij}= e^{ 2\alpha s\, (d_i-d_j)}$. 
Since $\gamma$ is not stable, $A(\eta_s(t))$ is not exponentially small in $s$ for all $t$. Recalling the asymptotic \eqref{R_Asymptotic} for $\mathcal{R}_{ij}$ and \eqref{CombinatorialEstimate} for $d_i-d_j$, along with the
saturation of bounds for $d_i - d_j$ from the beginning of the proof, we find $A( \eta_s(\frac{\pi}{12} ) ) \in O(\frac{1}{\sqrt{s}} ) $, with saturation \emph{only} at the entries
$E_{67}, E_{73}, E_{12}, E_{54}$, where the exponential growth of $\lambda_{ij}$ matches the exponential decay of $\mathcal{R}_{ij}$. 
The remaining entries $A_{ij}$ satisfy $A_{ij} \in O( \frac{e^{-b_{ij} s}}{ \sqrt{s} } \, )$ for some constant $b_{ij} >0$.

We denote $ M_s(t) := g_s^{-1}g_s'(t) $. We separate the terms of $M_s(t)$ that decay exponentially from the remaining terms as follows: 
 \begin{align} \label{ThetaPieces} 
	M_s(t) = M_s^0(t) +  E_0 \,\left [ \mu^{12}_s(t) E_{12} +\mu^{54}_s(t)  E_{54} + \mu^{67}_s(t) E_{67} + \mu^{73}_s(t) E_{73} \right ] \,E_0^{-1},
\end{align} 
where 
$$\mu^{ij}_{s}(t) = \eta_{s}'(t) \, A_{ij}(\eta_{s}(t)) $$
 and $M_s^0(t) \in O\left( \,\sqrt{s} \, e^{-2b\, s} \right) $ for $b = \min_{ij} b_{ij}$, since $|\eta_s'(t) | \in O(s)$.
Next, we note a crucial upper bound on $\lambda_{ij}$. 
 
 By \eqref{R1Asymptotics}, \eqref{R2Asymptotics}, for indices $(i,j) \in \{ \,(1,2), \, (5,4), (6,7), (7,3) \}$,
we have $(R_{1})_{ij} \asymp (\del x_1)$ and $(R_2)_{ij} \asymp x_1$. Hence, $A_{ij} = \lambda_{ij}((R_1)_{ij} + (R_2)_{ij})\in O( \lambda_{ij}x_1)$ by Corollary \ref{DerivativeAsymptotics}. Recall $d_i = \frac{1}{|z|} \, \Re(\xi^{2i-1}z)$, where $\xi = e^{i \pi/6}$. Using that $\xi^{1} -\xi^3 = \xi^{11}$, we find $\lambda_{12} = \lambda_{67}$.
Since $d_7 =0, \, d_j= -d_{j+3}$, we find $ \lambda_{54} = \lambda_{12}$ and $ \lambda_{67} = \lambda_{73} $. Hence, we may define 
$a:= \lambda_{54} = \lambda_{12} =  \lambda_{67} = \lambda_{73}$. 
Note that $\arg( \eta_{s}(t) )= t+\frac{\pi}{12}$ and recall $d_6(\theta) = \cos(\theta + \frac{11 \pi}{6} ) = \cos(\theta -\frac{\pi}{6})$. Thus,
$$ a(\eta_{s}t) = \lambda_{67}(\eta_{s}t) = e^{2\alpha s \, d_6(t+\frac{\pi}{12} ) } = \mathsf{exp}\left( \,2\alpha s \, \cos\left(t - \frac{\pi}{12} \right) \; \right). $$
Using $\cos(x-\frac{\pi}{12}) \leq 1- \frac{1}{4}(x-\frac{\pi}{12})^2$ for $ x \in [0, \frac{\pi}{6}]$, we find the inequality
\begin{align} \label{a}  
 a(\eta_s(t)) \leq \exp \left( \, 2\alpha s \, - \frac{\alpha}{2} s\left( t- \frac{\pi}{12} \right)^2 \, \right).
\end{align} 
for $t \in [0, \frac{\pi}{6} ]$. By Lemma \ref{FundamentalAsymptotic}, we get the following bound for $\mu^{ij}$: 

$$| \mu_{s}^{ij}(t)| \in O( \, |\eta_s'(t) | \, a\,x_1) = O \left(  \sqrt{s} \, \mathsf{exp}\left( - \frac{\alpha}{2} \, s \left( t - \frac{\pi}{12} \right)^2 \, \right) \right).$$
However, this upper bound is an integrable function in $t$, with $L^1$ norm independent of $s$, as it is a (normalized) Gaussian function. 

Set $c^{ij}_s(t) := \int_{0}^{t}  \mu_s^{ij}(\tau)  d\tau$ and define
$$ Q_{s}(t) := E_0  \, \mathsf{exp} \left ( \, c^{12}_s(t) \, E_{12} + c^{54}_s(t) \, E_{54} + c^{67}_s(t) \, E_{67} + c^{73}_s(t)\, E_{73} \right) \, E_0^{-1}. $$ 
Observe that $g^{-1}_s g'_s(t) - Q^{-1}_s(t) Q'_s(t) = M_s^0(t)$. 
Then by the same reasoning as \cite[Lemma B.2]{DW15}, we find that $|| g_s(t) - Q_s(t) ||$ is uniformly bounded 
in terms $||M_s^0||$ for all $t$. However, we know that $||M_s^0|| \rightarrow 0$ as $s \rightarrow \infty$. 
On the other hand, $G_1(s)^{-1}G_2(s) = g_s \left( \frac{\pi}{6} \right)\rightarrow L_1^{-1}L_2 $. 
Thus, $\lim_{s \rightarrow \infty} Q_s(\frac{\pi}{6})$ exists and 
$ \lim_{s \rightarrow \infty} Q_s(\frac{\pi}{6}) = L_1^{-1}L_2 $. Of course, this means $\lim_{s \rightarrow \infty} E_0^{-1} Q_{s}(\frac{\pi}{6}) E_0 $ exists too. 
Thus, $c^{ij} := \lim_{s \rightarrow \infty} c^{ij}_s(\frac{\pi}{6}) $ exists and is finite. We conclude that
 that $U_1$ from \eqref{UnipotentsEquation1} satisfies $L_1^{-1}L_2 = \lim_{s \rightarrow \infty}Q_s(\frac{\pi}{6})= E_0 U_1 E_0^{-1}$.  

The argument for $U_2$ is nearly the same; in this case, one uses $\lambda_{53} = \lambda_{62}$
along with $(R_1)_{ij} \asymp \del x_2, \; (R_{2})_{ij} \asymp x_2$ for $(i,j) \in \{ (5,3), (6,2)\}$.
\end{proof}

\section{Polynomial Sextic Differentials to Annihilator Polygons}\label{AsymptoticBoundary} 

We now discuss the boundary of $\Stwofour$ in $\mathbb{P} \imoct$. Here, we regard $\Stwofour = \{ \, [x] \in \mathbb{P} \imoct \; | \; q(x) > 0 \}$. 
Then the topological closure $(\overline{\Stwofour} \subset \mathbb{P} \imoct) = \Stwofour \sqcup \Eintwothree$, where 
$$\Eintwothree := \mathbb{P} Q_0(\imoct) = \{ \, [x] \in \mathbb{P}\imoct \; | \; x \neq 0, \, q(x) = 0\}$$ is the space of null lines in $\imoct $. 
In Section \ref{BoundaryAnnihilatorPolygon}, we show the \emph{asymptotic boundary} $\Delta := \partial_{\infty}(\nu_q)$ of an almost-complex curve $\nu_q$ associated to a 
solution to equation \eqref{HitEuc}, with $q$ a monic polynomial, is a polygon in $\Eintwothree$ with $\deg q + 6$ vertices. 

Moreover, we describe the $\imoct$ properties of these polygons, showing they have the \emph{annihilator} property, which is in turn related to
a discrete metric $d_3 $ on $\Eintwothree$. Recall that in the $\Gtwo$ preliminaries, we discussed \emph{annihilators}. 
In Section \ref{AnnihilatorPolygons}, define the metric $d_3$ and relate it to annihilators. Then we show $\Gtwosplit = \Isom(d_3) \cap \Diff(\Eintwothree)$, characterizing the group $\Gtwosplit$ as the diffeomorphisms of $\Eintwothree$ that also preserve the metric $d_3$
in Section \ref{d3}. We then prepare for the main results. Section \ref{Sec:RealCrossProductBasis} discusses real cross-product bases for $\imoct$ and introduce a real cross-product frame for $\V$ that we use later. Section \ref{ModelSurface} discusses the asymptotic boundary of $\nu_0$, the model surface when $q \equiv 1$. 

\subsection{Annihilator Polygons in $\Eintwothree$} \label{AnnihilatorPolygons}

In this section, we discuss \emph{annihilator polygons}, which are piecewise linear curves in $\Eintwothree$, with constraints on the edges 
related to the geometry of $\Eintwothree$. First, we recall some definitions. 
Again, we denote the null-cone in $\imoct$ as $Q_0(\imoct) := \{x \in \imoct \; | \; q(x) = 0\}$. 
The notion of annihilators is as follows: 

\begin{definition}\label{Defn:Annihilator} 
Given $u \in \imoct$, we define the \textbf{annihilator} of $u$ as $$\Ann(u) := \ker \mathcal{C}_u = \{ v \in \imoct \; | \; u \times v = 0 \}.$$
\end{definition} 
Note that $\Ann(u) = \Ann(c \, u) $ for $c \in \R^*$. Hence, $\Ann( [u] )$ makes sense. We may abuse notation and 
conflate $\Ann( [u] ) $ with $ \mathbb{P} \Ann([u])$, so that $\Ann([u]) \subset \mathbb{P}(\imoct)$. 

Now, there is another Stiefel triplet model for $\Gtwosplit$ that we shall need going forwards. Consider the following Stiefel manifold of null triplets: 
\begin{align} 
\mathcal{N} := \{ (u,v,w) \in (\imoct)^3 \; | \; q(u) =q(v) = q(w) = 0, \; u\cdot v= v\cdot w = u\cdot w= 0, \; u \cdot (v \times w) = +1 \} 
\end{align} 

The action of a $\Gtwosplit$-transformation is precisely determined by its action on a single null triple, as proven in \cite[Theorem 13]{BH14}. 
 
\begin{lemma}[Null Triplet Model]\label{Lemma:NullTriple} 
$\Gtwosplit$ acts simply transitively on $\mathcal{N}$. 
\end{lemma} 

As a corollary, one finds that $\Gtwosplit$ acts transitively on $Q_0(\imoct^\F)$. One then easily checks that $\dim \Ann( u) = 3$ for any $u \in Q_0(\imoct^\F)$ by
the transitivity. For example, if $(x_i)_{i=3}^{-3}$ is any $\R$-cross-product basis, as in Definition \ref{Defn:RealCrossProductBasis}, then $\Ann(x_3) = \spann_{\R} \langle x_3, x_2, x_1 \rangle$.

\begin{proposition}\label{Prop:Annihilators}
$\Gtwosplit$ acts transitively on $Q_0(\imoct).$ Also, if $u \in Q_0(\imoct)$, then $\Ann(u) \subset Q_0(\imoct)$ and $\dim_\R \Ann(u) =3$.\footnote{If $q(w) \neq 0$, then $\Ann(w) = \R \{w\}$ is uninteresting.}  \end{proposition} 

In \cite{BH14}, Baez and Huerta give a beautiful proof that the action of $\Gtwosplit$ on $\Eintwothree \times \Eintwothree$ decomposes into four orbits. 
In fact, the result is a corollary to the proof Lemma \ref{Lemma:NullTriple}, but the \cite{BH14} proof reveals even more, as we shall soon see. 
Moreover, the obvious $\Gtwosplit$-invariants, dot product and cross-product, classify the orbits. We denote these orbits as:
\begin{align} \label{NullOrbits}
	\begin{cases}
		N_0 &=  \{ \, ( \, [u] ,\, [u] \, ) \in \Eintwothree \times \Eintwothree \; | \; u \in Q_0\imoct \}. \\
		 N_1 &= \{ \, ( \, [u] ,\, [v] \, ) \in \Eintwothree \times \Eintwothree\; | \; u \times v = 0, \, u \neq v\}. \\
		 N_2 &= \{ \, ( \, [u], \, [v] \, ) \in \Eintwothree \times \Eintwothree\; | \; u \cdot v = 0, \; u \times v \neq 0 \}. \\
		 N_3 &= \{ \,( \, [u] ,\, [v] \, )\in \Eintwothree \times \Eintwothree \; | \; u \cdot v \neq 0 \}. \\
	\end{cases} 
\end{align} 
In particular, if $u, v \in Q_0(\imoct)$ are null and $u \cdot v \neq 0$, then $u \times v \neq 0$. 
Remarkably, these orbits these define a $\Gtwosplit$-invariant discrete metric $d_{3}$.

\begin{definition} 
Define $d_3: \Eintwothree \times \Eintwothree \rightarrow \{0,1,2,3\}$ by $d_{3}(x, y) := i$ when $ (x,y) \in N_i$.
\end{definition} 

The fact that $d_3$ defines a metric is never explicitly mentioned in \cite{BH14}, though it is a corollary of one of their main results,
as we now show. 

\begin{lemma}[Discrete Metric on $\Eintwothree$.] 
The map $d_3$ defines $\Gtwosplit$-invariant metric on $\Eintwothree$. 
\end{lemma} 

\begin{proof}
The key to the proof is an equivalent characterization of the orbits $N_i$ given in \cite{BH14}. We now define a map $d': \Eintwothree \times \Eintwothree \rightarrow \Z_{\geq 0}$,
which a priori has no relation to $d_3$. Take $x, y \in \Eintwothree$ and define an \emph{annihilator path} (of length $k$) from $x$ to $y$
as a sequence $P := (x_0, x_1, \cdots, x_k) \subseteq \Eintwothree$ such that $x_0 = x, x_{k} = y$, and  $x_i \odot x_{i+1} = 0$ for $ i\in \{0, 1, \dots, k-1\}$. 
Define $\mathsf{length}(P) = k$. Then let $\mathcal{P}_{x,y} $ 
denote the set of all annihilator paths from $x$ to $y$ and we define $d'$ by 
\begin{align}\label{d'Distance} 
	d'(x,y) := \min_{P \in \mathcal{P}_{x,y}} \mathsf{length}(P).
\end{align} 
However, $d_3 = d'$; this is exactly the content of \cite[Theorem 16]{BH14}. Since $(x,y) \in N_i $ if and only if $(y,x) \in N_i$, it is clear that $d_3$ is symmetric. By definition, $d_3(x,y) =0$ 
if and only if $x = y$. 
The triangle-inequality for $d'$ follows immediately from path concatenation of annihilator paths. Thus, $d_3$ is indeed a metric. 
$\Gtwosplit$-invariance of $d_3$ is clear.
\end{proof} 

We now define annihilator polygons. Observe first that if $x, y \in \Eintwothree $ are orthogonal, then the projective line connecting them is also contained in $\Eintwothree$. That is, 
$L = \mathbb{P}(\spann_{\R} \langle x, y \rangle) \subset \Eintwothree$. 

\begin{definition}
Let $P = (p_i)_{i \in \mathbb{Z}_n}$ be a cyclically ordered multiset of points $p_i \in \Eintwothree$, with $p_i \cdot p_{i+1} = 0$ and $p_{i} \neq p_{i+1}$. 
We call a cyclically ordered set $\Delta := (L_i)_{i \in \mathbb{Z}_n}$ of
projective line segments $L_i$ connecting adjacent vertices $p_i, p_{i+1}$ a \textbf{polygon}. 
We say such a polygon $\Delta$ has the \textbf{annihilator property} when $\Ann( p_{i} ) = \spann_{\R}\langle p_{i-1}, p_i, p_{i+1} \rangle$. 

Similarly, given a sequence $S = (p_1, \dots, p_n)$ with $p_i \in \Eintwothree$, we say $S$ is an \textbf{annihilator sequence} when 
$\Ann( p_{i} ) = \spann_{\R} \langle p_{i-1}, p_i, p_{i+1} \rangle $ for $i \in \{2,\dots, n-1\}$.
\end{definition}

In fact, a straightforward argument shows that if $S$ is an annihilator sequence, then any consecutive sextuple of vertices $(p_i, p_{i+1}, \dots, p_{i+5})$ is linearly independent. This linear independence result is tight, in general, 
as one can construct annihilator polygons $\Delta$ with vertices $(p_i)_{i=1}^{2k}$ and $\dim (\spann_\R (p_i)_{i=1}^{2k})  = 6$. \\

We now remark on the connection between null triples and annihilator hexagons. 

\begin{remark}
Let $p:= (u,v,w)$ be a null triple. Then $p$ naturally defines an annihilator hexagon with vertex set 
$ ( \, [u], \, [u \times v], \, [v], \, [v \times w], \, [w], \, [w \times u]) $ \cite[Theorem 12]{BH14}. 
Here, the entire polygon sits in an affine chart for $\Eintwothree$, so the edges are determined uniquely by the vertices. 
\end{remark} 

We state two basic facts about annihilators that follow from Proposition \ref{Prop:Annihilators} and Lemma \ref{Lemma:NullTriple} 
that will allow us to relate $d_3$ to the annihilator property. 

\begin{proposition}[Standardizing Annihilators]\label{StandardizingAnnihilators} 
 Fix $[u] \in \Eintwothree$ and $v, w \in \Ann(u)$.  
 	\begin{enumerate}
		\item $\{u,v,w\}$ is linearly independent if and only if $v \times w = c \, u$ for some $c \neq 0$. 
		Hence, $\Ann(u) = \spann_{\R} \langle u, v, w \rangle$ if and only if $[v] \times [w] = [u]$. 
		\item If $([v], \, [u] ,\, [w],\, [x] )$ is an annihilator sequence, then $d_3([v],[x]) = 3$. 
	\end{enumerate} 
\end{proposition}

\begin{proof}
(1) By the transitivity of the action of $\Gtwosplit$ on $\Eintwothree$, we need only check (1) for
for a single annihilator. To this end, one can use $ \tilde{u} = i+li, \tilde{v} = j-lj, \tilde{w} = k-lk$. Since $\tilde{u}^2 = \tilde{u}\tilde{v} =\tilde{u}\tilde{w} =0$, we have $\Ann(\tilde{u}) = \spann_{\R} \langle \tilde{u},\tilde{v},\tilde{w} \rangle$ by Proposition \ref{Prop:Annihilators}. 
Moreover, one finds $\tilde{v} \times \tilde{w} = \frac{1}{2} \tilde{u}$. Hence, given $w_1 = a\tilde{v} + b\tilde{w} + c\tilde{u}$ and $w_2 = d\tilde{v} + e\tilde{w} + f \tilde{u}$, we see that
$ w_1 \times w_2 = \frac{1}{2} (ae-bd)\tilde{u}$. On the other hand,  
$$ \det (w_1,\; w_2, \; \tilde{u} ) =  \det \begin{pmatrix} a & d & 0 \\ b& e& 0 \\ c & f & 1 \end{pmatrix}  = ae-bd = 2 (w_1 \times w_2).$$ 
Thus, $\{ w_1, w_2, \tilde{u} \}$ is linearly independent if and only if $ w_1 \times w_2 = c u $ for $ c \neq 0 $. 
Since $\Ann(u)$ is 3-dimensional, the final statement of (1) follows. 

(2) We begin by extending $v, \, w$ to a null triplet $(v, w, y)$. By re-gauging with $\varphi \in \Gtwosplit$ by Lemma \ref{Lemma:NullTriple}
so that $\varphi \cdot (v,w) = (\tilde{v}, \tilde{w})$, for the model elements $\tilde{v}, \tilde{w}$ defined above, we can assume that $ u = c\tilde{u}, \, v = \tilde{v}, \, w= \tilde{w}$ for $c \neq 0$ and $x \in \Ann( \tilde{w})$. But then $x \in \spann_{\R}\langle \tilde{u}, \tilde{w}, \tilde{z} \rangle $, where $\tilde{z} = j+lj$. Linear independence of $\{u,v,x\}$
entails that $[x] =[ z+ \tilde{z}]$ for some $z \in \spann_{\R} \langle \tilde{u}, \tilde{v} \rangle$. But since $\tilde{u}, \tilde{w} \; \bot \; \tilde{v}$ and $\tilde{z} \cdot \tilde{v} \neq 0$, the claim holds. 
\end{proof} 

We can now rephrase the annihilator condition in terms of the discrete metric $d_3$. 

\begin{corollary} 
Let $\Delta$ be a polygon in $\Eintwothree$ with vertex set $(p_i)_{i \in \Z_n}$. Then $\Delta$ is an annihilator polygon if and only if any one of the following equivalent
conditions occur:
	\begin{enumerate} 
		\item $\Ann(p_i) = \spann_{\R} \langle p_{i-1}, p_i, p_{i+1} \rangle $
		\item For all i, $d_3(p_i, p_{i+1} ) = 1, d_{3}(p_i, p_{i+2} ) = 2$
		\item For all i, $d_3(p_i, p_{i+1} ) = 1, d_{3}(p_i, p_{i+2} ) = 2, d_3(p_i, p_{i+3} ) = 3. $ 
	\end{enumerate} 
\end{corollary} 

\begin{proof}
If (1) holds, then $d_3(p_{i-1}, p_i), d_3(p_{i+1}, p_i) \leq 1$. We cannot have $d_3(p_{i+1}, p_i) = 0$ or else $\dim \spann_{\R} \langle p_{i-1}, p_i, p_{i+1} \rangle \; \leq 2$. 
Hence, $d_{3}(p_{i-1}, p_i) = 1$; the same reasoning shows $d_{3}(p_{i+1}, p_i) = 1$. 
By Proposition \ref{StandardizingAnnihilators}, $p_{i-1} \times p_{i+1} = p_i$. Since $p_{i-1} \, \bot \, p_{i+1}$ as annihilators are totally isotropic, $d_3(p_{i-1}, p_{i+1} ) = 2$. Hence, (1) implies (2). Proposition \ref{StandardizingAnnihilators} part (2) says that (2) implies (3). 
Finally, we show (2) implies (1). If (2) holds, then 
$ \spann_{\R} \langle p_{i-1}, p_i, p_{i+1} \rangle \, \subseteq \Ann(p_i)$. By Proposition \ref{StandardizingAnnihilators} part (1), we conclude that $\Ann(p_i) = \spann_{\R} \langle p_{i-1}, p_i, p_{i+1} \rangle $, proving (1).
\end{proof} 

We note an interesting observation of Baez \& Huerta \cite{BH14}, rephrased in terms of the metric $d_3$. 

\begin{remark}
If $d_3(x, y) = 2$ for $x,y \in \Eintwothree$, then $x,y$ have a unique \textbf{midpoint} $ w := x \times y$ in the sense that $w \in \Eintwothree$
is the unique solution to $d_3(x, w) = 1 = d_3(x, y)$. Indeed, Proposition \ref{StandardizingAnnihilators} shows 
$\Ann(w) = \spann_{\R} \langle x, w, y \rangle$ only if $ w= x \times y$. To see that equality does hold, note that the double-cross product \eqref{DCP} says 
$w := x \times y$ has $q(w) = q(xy) = 0$ and $w \times x = 0 = w \times y $. 
\end{remark} 

We now summarize all our observations in terms of the polygons. 

\begin{proposition}[Properties of Annihilator Polygons]\label{PropertiesAnnihilatorPolygons}
If $ P$ be an annihilator polygon in $\Eintwothree$ with vertex set $(p_i)_{i \in \Z_n}$, then 
	\begin{enumerate}
		\item $p_i \odot p_i = 0$. 
		\item $p_i \odot p_{i+1} = 0$. 
		\item $p_i \cdot p_{i+1} = 0, \; p_i \cdot p_{i+2} = 0$. 
		\item $p_i \cdot p_{i+3} \neq 0 $. 
		\item $p_i \times p_{i+2} = p_{i+1}$. 
	\end{enumerate} 
\end{proposition} 

\begin{proof}
For $u \in Q_0(\imoct)$, we have $0 =q( u) = -u^2$. So, (1) holds because $p_i$ is a null line, i.e., $p_i \odot p_i = - p_i \cdot p_i = 0$. Since $p_i \times p_{i+1} = 0$, then \ref{Prop:Annihilators} part (2) implies that (3) holds. Thus, $p_i p_{i+1}= p_{i}\times p_{i+1} =0$. The non-degeneracy condition that $\{ p_i, p_{i+1}, p_{i+2}\}$ is independent 
implies (5) by \ref{StandardizingAnnihilators} part (1). Meanwhile, (4) is just Proposition \ref{StandardizingAnnihilators} part (2). 
\end{proof} 

For reference later, we now discuss the structure of the space of all (marked) generic annihilator polygons. By generic, we mean points are in as general a position
as possible while satisfying the annihilator condition. The following definition makes this notion precise. 

\begin{definition} 
Let  $\Delta$ be a marked annihilator polygon in $\Eintwothree$ with ordered vertices $(p_i)_{i =1}^n$. We say $\Delta$ is \textbf{generic} if for all indices $i \leq j$, 
$d_3(p_i, p_j) = 3$ as long as $d_{\mathsf{cyclic}}(i, j )\geq 3 $, where $d_{\mathsf{cyclic}}(i, j) := \min \{ j-i, n-j + i \}.$
\end{definition} 

Consider a generic marked annihilator polygon $\Delta$ with $k \geq 7$ vertices $(p_i)_{i=1}^k$. Then since any two adjacent vertices $p_i, p_{i+1}$ are naturally in affine chart relative to any $x \in p_{i+4}$,
the edges of $\Delta$ are determined by the vertices. Suppose now that $\Delta $ is marked so it has a distinguished first vertex $p_1$ that gives an honest ordering of the vertex set. 
In this case, $\Delta$ can be conflated with its vertex set $(p_i)_{i=1}^k \subset (\Eintwothree)^k$. In this way, the space $\widehat{\mathcal{TP}}_{k}^\mathsf{gen}$ of all marked generic annihilator polygons
obtains the subspace topology from $(\Eintwothree)^k$. 

\begin{lemma}\label{GenericPoylgonModuliSpace}
Let $k \geq 8$. The moduli space $\mathcal{TP}_{k}^\mathsf{gen} = \widehat{\mathcal{TP}}_{k}^\mathsf{gen}/\Gtwosplit$ of all marked generic annihilator polygons up to the $\Gtwosplit$-action 
is homeomorphic to a finite union of cells of dimension $k-7$. 
\end{lemma}

\begin{proof}
Let $P_0 = [ \Delta ]$ be an equivalence class. Choose a representative $\Delta = (p_i)_{i=1}^k$ and we put $\Delta$ into a unique `standard position', to kill off the $\Gtwosplit$-gauge freedom. Fix an $\R$-cross-product basis
$(x_i)_{i=3}^{-3}$ for $\imoct$ such that $(x_3, x_{-1}, x_{-2}) \in \mathcal{N}$ is a null triple. First, by an appropriate null triplet argument with Lemma \ref{Lemma:NullTriple}, one can find $\varphi \in \Gtwosplit$ so that $\varphi \cdot (p_1, p_2, p_3,p_4) = ( [x_3], [x_2], [x_{-1}], [x_{-3} ] )$. By abuse, call the new polygon by $(p_i)_{i=1}^k$ 
as well. Note that $p_5 \cdot x_3 \neq 0$ and $p_5 \cdot x_2 \neq 0 $. We claim that we can find $\varphi \in \Gtwosplit$ such that $\varphi \cdot p_5 = [x_{-2} + x_{-3}]$ and $\varphi \cdot (p_1, p_2, p_3,p_4) = (p_1, p_2, p_3,p_4) $. 
Now, since $p_5 \in \Ann([x_{-3}])$ and $p_5 \times [x_{-1}] = [x_{-3}]$, we know $p_5 = [x_{-2} + a x_{-1} + bx_{-3}]$ for some $a\in \R$ and $b \in \R^*$. 
By Lemma \ref{Lemma:NullTriple} again, take $\varphi \in \Gtwosplit$ such that
$\varphi \cdot (x_{3}, x_{-1}, x_{-2} ) = (x_{3}, x_{-1}, x_{-2} - a x_{-1})$. Then $\varphi$ fixes the tuple $(p_1, p_2, p_3, p_4)$ and $\varphi \cdot p_5 = [x_{-2} + bx_{-3}]$. Update the polygon $\Delta = (p_i)_{i=1}^k$, keeping the same notation for the vertices. 
Then now take $\psi \in \Gtwosplit$ such that $\psi \cdot (x_{3}, x_{-1}, x_{-2} ) = (b x_{3}, \frac{1}{b} x_{-1}, x_{-2} )$ and then $\psi \cdot (p_1, p_2, p_3, p_4) = (p_1, p_2, p_3, p_4)$ and $\psi \cdot p_5 = [x_{-2} +x_{-3}]$ as desired. 
Thus, we have now standardized $(p_1, \dots, p_5 )$. Again, we take $\Delta \in P_0$ to be our updated polygon. We can remove the remaining $\Gtwosplit$-freedom by renormalizing the next vertex. 

The vertex $p_6 \in \Ann([x_{-2} + x_{-3} ])$ must satisfy
$p_6 \times [x_{-3}]= [x_{-2} + x_{-3} ]$. We claim that $p_6$ obtains the form $p_6 = [x_{0} + a_0 x_{1} + b_0 x_{-1} + cx_{-2} + dx _{-3} ]$, where $a_0, b_0 \in \R^*$ are prescribed and $c, d \in \R^*$. 
To see this, note that for some unique $a_0 \in \R^*$, we have 
$[x_0 + a_0 x_{1} ] \times [x_{-3} ] = [x_{-2} + x_{-3}]$. Thus, any solution to $u \times [x_{-3}] = [x_{-2} + x_{-3}]$ of the form $u = [x_0 + a_0 x_{1} + w] $ for $w \in \Ann([x_{-3}])$, nearly proving the claim. 
However, the coefficient $b_0 \in \R^*$ is determined by $q(x_{0} + a x_{1} + bx_{-1} + cx_{-2} + dx_{-3}) =0$ and $c, d \neq 0$ by $p_6 \cdot x_3 \neq 0$ and $p_6 \cdot x_2 \neq 0$. We now note that the remaining $\Gtwosplit$-freedom is expressed by the stabilizer subgroup $H:= \bigcap_{i=1}^5 \Stab([p_i])$. One finds
that $H $ is necessarily diagonal in the basis $(x_i)$ and hence $\psi \in H $ is of the form $\psi = \diag(rs, r, s, 1, \frac{1}{s}, \frac{1}{r}, \frac{1}{rs}) $ for some $r, s \in \R^*$. Then $\varphi \cdot [x_{-2} + x_{-3} ] = [x_{-2} + x_{-3}]$ forces $s =1$. 
We have then proven $H \cong \R^*$. The remaining freedom can be removed by choosing the unique transformation $\varphi \in H$
such that $\varphi \cdot [x_{0} + a_0x_{1} + b_0x_{-1} + cx_{-2} + dx_{-3}] =  [x_{0} + a_0 x_{1} + b_0x_{-1} + x_{-2} + e x_{-3}] $ for some $e \in \R^*$. In summary, we have shown that any such equivalence class $[\Delta]$ has a unique representative $\Delta' = \varphi \cdot \Delta$ with vertex set $(p_i)_{i=1}^k$ and $(p_1, p_2, \dots, p_6)$ constrained as described above. We then see that $e \in \R^*$ topologically represents the possible configurations  
of the vertex $p_6$. 

Next, we enumerate all the possibilities for the remaining vertices of such normalized polygons, thereby determining the topology of $\mathcal{TP}_{k}^\circ$. We examine $p_j$ now for $7 \leq j \leq k-2$. In this case,
$p_j \in \Ann(p_{j-1})$ and $p_j \in \left( \Eintwothree \backslash \bigcup_{i=1}^{j-3} p_i^\bot \right)$, which is topologically a finite collection of open cells in an affine chart for $\RP^2 \cong \Ann( p_{j-1})$. Indeed,
we can view $\Ann( p_{j-1}) \backslash p_1^\bot$ in an affine chart for $\Ann(p_{j-1})$ and then the remaining non-orthogonality conditions amounts to removing a number of lines from this affine plane.  We then must be slightly more careful as the polygon
closes up. We see that $p_{k-1} \in \Ann( p_{k-2} ) \cap [p_1]^\bot \cong \RP^1$, with $p_{k-1} \in \left( \Eintwothree \backslash \bigcup_{i=2}^{k-4} p_i^\bot \right)$. Hence, topologically, $p_{k-1}$ is chosen from a collection of open intervals in $\R$. 
In fact, since $p_k = p_{k-1} \times p_1$ is forced, we have a finite number of extra demands to also force that $p_{k} \in \left( \Eintwothree \backslash \bigcup_{i=3}^{k-3} p_i^\bot \right)$, but this does not change the fact that $p_{k-1}$ is topologically
chosen from a finite set of disjoint intervals in $\R$. These choices for $p_{k-1}$ force that $(p_i)_{i=1}^k$ completes to a generic annihilator polygon. 

We conclude that each equivalence class $[\Delta] \in \mathcal{TP}_{k}^\mathsf{gen}$ has a unique normalized representative $(p_i)_{i=1}^{k}$ of the form previously described and also that each such normalized representative is in a distinct 
equivalence class. It follows that $\mathcal{TP}_{k}^\mathsf{gen}$ is topologically a finite disjoint union of cells of dimension $k-7 = 1 + 2(k-8) + 1 $.  

\end{proof} 

\subsection{$\Gtwosplit$ and Isometries of $d_3$} \label{d3}

In this section, we show that $\Gtwosplit = \mathsf{Isom}(d_3) \cap \mathsf{Diff}(\Eintwothree)$, where $\mathsf{Diff}(\Eintwothree)$ denotes the diffeomorphism group of $\Eintwothree$. The result is a Corollary of work of Cartan \cite{Car10},
whose work gives an equivalence between normal, regular $(\Gtwosplit, P)$ Cartan geometries and 5-dimensional manifolds $M$ equipped with a (2,3,5)
distribution \cite[Proposition 5.1.2]{Bar10}. We refer the reader to \cite{CS00} for further details on parabolic geometries. Here, $P = \Stab(x) < \Gtwosplit$ is the parabolic subgroup of $\Gtwosplit$ preserving a point $x \in \Eintwothree$. An essential structure for this section is the canonical (2,3,5) distribution $\mathscr{D}$ on $\Eintwothree$, which we now define
abstractly, using the identification $\Eintwothree \cong \Gtwosplit/P$. We then connect $\mathscr{D}$ to annihilators. 

For this discussion, we fix an $\R$-split Cartan subalgebra $\h < \g_2'$ and base $\Pi = \{\alpha, \beta\}$ for the $\h$-roots $\Delta$.
Here, $\beta$ is the short root and $\alpha$ the long root.
We denote the root space of $\sigma \in \Delta$ as $\g_{\sigma}$.  
Write $\sigma \in \Delta$ uniquely as $\sigma = a\, \alpha + b \beta$. Then the $\beta$-height of $\sigma$ is $b$. 
The $\beta$-height function allows us to turn the root space decomposition
$\g_2' = \h \oplus \bigoplus_{\sigma \in \Delta} \g_{\sigma}$ into a grading $\g_2' = \bigoplus_{i=-3}^{3} \g_{\beta, i} $. Note that $\h \subset \g_{\beta, 0}$. 
The parabolic subalgebra $\mathfrak{p} < \g_2'$ associated to $\beta$ is $\mathfrak{p} := \bigoplus_{i \geq 0} \g_{\beta, i} $.
Hence, $\g_2'/\mathfrak{p} \cong \g_{-} := \g_{\beta, -1} \oplus  \g_{\beta, -2} \oplus   \g_{\beta, -3} $.  
The grading $ \bigoplus_{i=-3}^{3} \g_{\beta, i} $ respects the Lie algebra structure, i.e., $[\g_{\beta, i},\g_{\beta, j} ] \subseteq \g_{\beta, i+j}$. 
Now, $\dim_{\R} \g_{\beta, -1} =2, \; \dim_{\R} \g_{\beta, -2} =1, \; \dim_{\R} \g_{\beta, -3}= 2 $. 
Thus, the filtration $\mathcal{F}_{-} := \left( \, \g_{(-1)} \subset \g_{(-2)} \subset \g_{(-3)} \, \right)$ on  $\g_2'/\mathfrak{p}$, where $\g_{(-i)} := \bigoplus_{j=-i}^{-1} \g_{\beta, j}$,
has subspaces of dimensions 2,3,5, respectively. We now show $\g_{\beta,-1} \subset \g_2'/\mathfrak{p}$ manifests as the (2,3,5) distribution $\mathscr{D}$ on $\Ein^{2,3}$. 

First, take $x \in \Eintwothree$ and define the subgroup $P := \Stab(x) < \Gtwosplit$. Then we identify $\Gtwosplit/ P \cong_{\mathsf{Diff}} \Eintwothree$ 
by the map $g P \mapsto g(x) $ since $\Gtwosplit$ acts transitively on $\Eintwothree$ by Proposition \ref{Prop:Annihilators}. 
On the other hand, $\g_2' $ admits a $\beta$-height filtration $\mathcal{F}_{\beta}$ by
$(\g_2')^{(i)} := \bigoplus_{j = i}^{3} \g_{\beta, j}$, satisfying $[(\g_2')^{(i)}, (\g_2')^{(j)}]=  (\g_2')^{(i+j)} $. 
Then the subgroup $P ':= \Aut(\g_2', \mathcal{F}_{\beta}) $ satisfies $P = P'$ by \cite[Proposition 5.1.1]{Bar10}
and the Lie algebra of $P$ is $\mathfrak{p}$ \cite[page 45]{CS00}. 

Using the homogeneous projection $\pi: \Gtwosplit \rightarrow \Gtwosplit/P $, we get the
identification $T\Eintwothree \cong T (\Gtwosplit/P) \cong \Gtwosplit \times_{P} \g_{-} $,
where $P$ acts by the adjoint action on $\g_- = \g_2'/\mathfrak{p}$. With this identification, we define the distribution as $\mathscr{D} := \Gtwosplit \times_{P} \g_{\beta, -1}$. 
Let us explain why $\mathscr{D}$ is a (2,3,5) distribution. 
By identifying the filtration $\g_{(-1)} \subset \g_{(-2)} \subset \g_{(-3)} $ on $\g_-$ 
as descending from the filtration $\mathcal{F}_{\beta}$ on $\g_2'$, we find $P \subset \Aut( \g_-, \mathcal{F}_{-} )$. 
Since  $[ \g_{\beta, -1}, \g_{\beta, -1} ]= \g_{\beta,- 2}, \; [ \g_{\beta, -1}, \g_{\beta, -2} ] = \g_{\beta, -3}$, 
we find $\mathscr{D}^1 := \mathscr{D} + [\mathscr{D}, \mathscr{D}] = \Gtwosplit \times_{P} \, \g_{(-2)} $
and $\mathscr{D}^1 +[ \mathscr{D}, \mathscr{D}^1] = \Gtwosplit \times_{P} \, \g_{(-3)}$ by $\Ad(P)$-invariance of $\mathcal{F}_{-}$. Thus, $\mathscr{D}$ is a (2,3,5) distribution. 

Next, we need a quick aside on parabolic geometries. 

\begin{definition}
Let $(\mathcal{P}, M) $ be a principal $P$-bundle for $P< G$ a parabolic subgroup of $G$. Suppose $\omega \in \Omega^1(\mathcal{P}, \g)$
such that $\pi_{\frakp} \circ \omega$ is a $P$-connection and we have $\omega: T_p\mathcal{P} \rightarrow \g$ a linear isomorphism for all $p \in \mathcal{P}$. 
Then we call the tuple $(\pi, \mathcal{P}, M, P, G, \omega)$, or just $(\pi, \mathcal{P}, M, \omega)$ if $(P, G)$ are implied, a \textbf{parabolic geometry}. Removing the hypothesis on $P$ as a parabolic subgroup, we just have a \textbf{Cartan geometry}. The 1-form $\omega$ is called a \textbf{Cartan connection}. 
\end{definition} 

These structures are relevant to us as 
parabolic geometries are related to filtrations of tangent bundles, as we now describe. Given a parabolic geometry $(\pi, \mathcal{P}, M, P, G ,\omega)$,
we recall that $P$ preserves some filtration on $\g$. Indeed, there is a grading $\g = \bigoplus_{i=-k}^k \g_i$ such that $\frakp = \bigoplus_{i \geq 0} \g_i$
and $P$ is the $\Ad$-stabilizer in $G$ of the filtration $\g_{(-k)} \supset \cdots \supset \g_{k}$, where $\g_{(j)} := \bigoplus_{i \geq j} \g_i$. 
Then the Cartan connection gives a filtration on $TM$. First, split $T\mathcal{P} = HP \oplus VP$ under the connection $\pi_{\frakp} \circ \omega$. 
Then we see that under the Cartan connection $\omega$, we must have that $\omega|_{VP}: VP \rightarrow \frakp$ is an isomorphism and hence
$\omega|_{HP} \rightarrow \g_- \cong \g/ \frakp$ an isomorphism as well, where $\g_- = \bigoplus_{i < 0} \g_i$. 
All of this motivates the fact that we have an isomorphism of $\g/ \frakp$-bundles $\mathcal{P} \times_{\Ad} \g/\frakp \cong TM$ coming from the Cartan connection. 
In particular, $TM$ receives a filtration and a grading coming from the induced filtration and grading on $\g/\frakp $ from $\g$. 
Call the Cartan connection \emph{regular} if 
\begin{enumerate}
	\item Lie bracket on $TM$ respects the height, $[T^{-k}M, T^{-l}M] \subset T^{-k-l}M$, so that the Lie bracket descends to define a \emph{Levi bracket} $\mathcal{L}$
	 of the graded Lie algebra $\gr(TM)$. 
	\item The Lie algebra bundle $(\gr(TM), \mathcal{L})$ is pointwise isomorphic to $\g_-$. 
\end{enumerate} 
The question one would like to answer is if we can reverse this process and recover a Cartan connection just from 
a filtration of $TM$ which induces a grading $\gr(TM)$ and a Levi bracket $\mathcal{L}$ such that $(\gr(T_xM), \mathcal{L}) \cong \g_-$. 
The answer is yes, but with some caveats. Firstly, a cohomological assumption is needed on the pair $(P, \g)$. Fortunately, this demand is satisfied for
$(P, \g_2')$. Second, from such a filtration of $TM$, we may only demand a unique \emph{normal} Cartan connection $\omega$ which induces the filtration, 
where normality is a condition related to how the curvature of $\omega$ respects the grading. In summary, 
a normal, regular Cartan connection $\omega: T\mathcal{P} \rightarrow \g$ on a principal $P$-bundle $\mathcal{P} \rightarrow M$ 
induces a filtration and grading of $TM$ such that $(\gr(T_xM), \mathcal{L})\cong \g_-$ and conversely, given such a filtration $\mathcal{F}$ on $TM$ inducing these structures, 
we uniquely recover a normal, regular Cartan connection $\omega$ that induces $\mathcal{F}$ \cite[Corollary 3.23]{CS00}. 
 
We may apply the correspondence to the case of 5-manifolds $M$ with $(2,3,5)$-distributions and normal, regular parabolic 
geometries $(\pi, \mathcal{P}, M, P, \Gtwosplit, \omega)$.
Now, the pair $(\Gtwosplit, \Gtwosplit/P)$ is itself a Cartan geometry of type $(P, \Gtwosplit)$ that is normal, regular.
Moreover, in this correspondence 
the pair $(\Eintwothree, \mathscr{D})$ corresponds to the homogeneous parabolic geometry $\Gtwosplit \rightarrow \Gtwosplit/P$ \cite{Sag06}. 
Thus, if we take $\varphi \in \Diff(\Eintwothree)$ such that $\varphi^*\mathscr{D} = \mathscr{D}$, then 
$\varphi$ is also an automorphism of $(\Gtwosplit \rightarrow \Gtwosplit/P, \, \omega)$ and hence
$\varphi = g|_{\Eintwothree}$ for some $g \in \Gtwosplit$. We record this result now. 
We will translate this theorem to a result about annihilators by re-interpreting
the distribution $\mathscr{D}$. 

\begin{lemma}[$\Gtwosplit$ and the Annihilator Distribution]\cite{Car10, Sag06}\label{InfinitesimalAnnihilator} 
Let $F: \Gtwosplit \rightarrow \mathsf{Diff}(\Eintwothree)$ be the restriction map $\psi \mapsto \psi|_{\Eintwothree}$. 
Then $\mathsf{im} \, F = \{ \, \varphi \in \mathsf{Diff}(\Eintwothree) \; | \; \varphi^* \mathscr{D} = \mathscr{D} \}$. 
\end{lemma} 

We note here that the map $F$ is injective. Indeed, suppose $F(\varphi ) = F(\psi)$. Then fix a null triplet $(u,v,w) $ as in Lemma \ref{Lemma:NullTriple}. By hypothesis $\varphi \cdot ([u], [v], [w]) = \psi \cdot ([u], [v], [w])$. But then $\psi = \varphi \circ \alpha$ for $\alpha \in  \Gtwosplit$
satisfying $\alpha(u)  =au, \; \alpha(v) = bv, \; \alpha(w) = \frac{1}{ab} w$, with $a,b \in \R\backslash \{0\}$. Then 
$$[\varphi(u) + \varphi(v) ] = \varphi \cdot [u +v ] = \psi \cdot [u+v] = [ \, a\varphi(u) + b \varphi(v) ] .$$
Hence $a= b$. Similar reasoning shows $a= \frac{1}{ab} = b$. Thus, $a=b=1$, which means $\varphi = \psi$. \\

Theorem \ref{InfinitesimalAnnihilator} says that any diffeomorphism of $\Eintwothree$ that infinitesimally preserves $\mathscr{D}$ must be the induced action of a $\Gtwosplit$-transformation. However, by work of Sagerschnig \cite{Sag06}, the distribution $\mathscr{D}$ is intimately related to annihilators. In fact,
if we identify $T_{[u]} \Eintwothree \cong \Hom_{\R} ( \, [u], u^\bot / [u] )$ and define the projection map $\pi_{[u]}: u^\bot \rightarrow u^{\bot}/ [u]$,
then the distribution $\mathscr{D}$ at $[u]$ may be realized as
\begin{align}\label{AnnihilatorDistribution} 
	\mathscr{D}_{[u]} = \{ \,X \in T_{[u]}\Eintwothree \;\; | \; \; X( [u] ) \subset \pi_{[u]}(\, \Ann([u]) \,) \; \}.
\end{align} 

In other words, the distribution $\mathscr{D}$ on $\Eintwothree $ at $[u]$ be realized as the set of variations of the line $[u]$ in the direction
of some $w \in \Ann([u])$. In fact, \cite{Sag06} instead defines an equivalent distribution 
$$\mathcal{H}_{[u]} := \{ \,X \in T_{[u]}\Eintwothree \; \;  | \; \; \Omega( [u], X, Y) = 0 \; \text{for all} \; Y \in \imoct \; \},$$
where $\Omega$ is the 3-form \eqref{Calibration} on $\imoct$. However, since $u \times x = u \odot x = 0$ for $u \in \Ann(x)$, it is evident that $\mathscr{D}_{[u]}\subseteq \mathcal{H}_{[u]}$. Since $\dim_\R \mathscr{D}_{[u]}= \dim_{\R}  \mathcal{H}_{[u]} =2$, we have $\mathcal{H}_{[u]} =\mathscr{D}_{[u]}$.

We now state this result a bit more precisely. The proof of the main theorem from \cite{Sag06} shows the following,
where we use here that $x_3$ is a highest weight vector in the representation of $\g_2'$ on $\imoct$ in the basis $C= (x_i)_{i=3}^{-3}$. 

\begin{lemma}[Annihilator Distribution]\label{Lem:AnnihilatorDistribution}
Let $C = (x_i)_{i=3}^{-3}$ be a real cross product basis for $\imoct$. Fix the Cartan subalgebra $\h$ of diagonal matrices in the basis $C$. Then $X \in \h $ is of the form $X = \mathsf{diag} (r+s, r, s, 0, -s, -r, -r-s)$. 
Define primitive roots $\Pi:=\{\alpha, \beta\}$ for $\alpha := r^*-s^*, \beta: = s^*$ so that $\beta$ is the short root. Define $\mathfrak{p} $ as the parabolic subalgebra associated
to $(\h, \Pi, \{\beta\})$ and $P$ the subgroup preserving the $\beta$-height filtration on $\g_2'$. Then the diffeomorphism $\Phi: \Gtwosplit/P \rightarrow \Eintwothree$
by $ gP \mapsto g \cdot x_3 $, satisfies that $d\Phi: \mathscr{D}_{[u]} \rightarrow \Gtwosplit \times_{P} \g_{\beta, -1} $ is a linear isomorphism,
where $ \mathscr{D}_{[u]}$ is given by \eqref{AnnihilatorDistribution}. 
\end{lemma} 

There is a natural global notion of annihilators in $\Eintwothree$ opposed to the infinitesimal notion of annihilator just described. Given $[u] \in \Eintwothree$, then $\mathbb{P}\mathsf{Ann}([u]) \cong_{\mathbf{Diff}} \RP^2$ 
is a 2-dimensional submanifold of $\Eintwothree$ by Proposition \ref{Prop:Annihilators} (2). Thus, we define 

\begin{definition}\label{AnnihilatorLanguage} 
We say $\varphi \in \mathsf{Diff}(\Eintwothree)$ preserves annihilators globally in $\Eintwothree$ when $\varphi( \mathbb{P}\Ann( \, [x] ) ) = \mathbb{P}\Ann( \, \varphi([x]) \, ) $
for all $x \in \Eintwothree$. We say $\varphi$ preserves annihilators infinitesimally when $ \varphi^*\mathscr{D} = \mathscr{D}$, i.e. $d\varphi|_{\mathscr{D}_{x} }: \mathscr{D}_{x} \rightarrow \mathscr{D}_{\varphi(x)} $ is a linear isomorphism for $x \in \Eintwothree$.  
\end{definition} 

In terms of Definition \ref{AnnihilatorLanguage}, Cartan's Theorem \ref{InfinitesimalAnnihilator} says that $\Gtwosplit$ is the subgroup of $\Diff(\Eintwothree)$ preserving annihilators infinitesimally. It is obvious that $\Gtwosplit$ preserves annihilators globally, since $\Gtwosplit$ preserves $\odot$. Conversely, it is not clear that any $\varphi \in \Diff(\Eintwothree)$ that preserves annihilators globally is the restriction of a $\Gtwosplit$-transformation. However, we now show this is the case by Theorem \ref{InfinitesimalAnnihilator}.

\begin{theorem}[$\Gtwosplit$ and Annihilators] \label{GlobalLocalAnnihilators}
Suppose $\varphi \in \Diff(\Eintwothree)$. Then $\varphi$ preserves annihilators globally if and only if $\varphi$ preserves annihilators infinitesimally. 
\end{theorem}

\begin{proof}
One direction is a consequence of Theorem \ref{InfinitesimalAnnihilator}. If $\varphi \in \Diff(\Eintwothree)$ preserves $\mathscr{D}$, then $\varphi = \psi|_{\Eintwothree}$
for a unique $\psi \in \Gtwosplit$. But for $u \in \Eintwothree$, we have $\mathbb{P}\mathsf{Ann}(u) = \{ \, u \} \sqcup \{ \, v \in \Eintwothree \; | \; d_{3}(u,v ) = 1 \}$.
Thus, we we immediately see that $\varphi $ preserves annihilators globally since $\Gtwosplit$ preserves $d_3$.

The converse comes from exponentiating tangent vectors. Take $\varphi \in \Diff(\Eintwothree)$ preserving annihilators globally and take $[u] \in \Eintwothree, X \in \mathscr{D}_{[u]} \subset T_{[u]} \Eintwothree$ and we show that $d\varphi ( X ) \in \mathscr{D}_{\varphi([u])} $. If we can show this, then $d\varphi ( \mathscr{D}_{[u]} ) \subseteq \mathscr{D}_{\varphi([u])} $. Since $\varphi$ is a diffeomorphism, we must have equality. 
Observe that the parametrized projective line $\gamma := \gamma_X(t) := [ \hat{u} + t \, X(\hat{u} ) ] $ is well-defined independent of choice of $\hat{u} \in [u]$. 
Moreover, $\gamma(0) = [u], \, \gamma'(0) = X$. We note two key facts. First, $\gamma(t) \subset \mathbb{P} \Ann([u])$ for all $t \in \R$. 
Next, defining this construction of $\gamma_{X}$ for any $X \in \mathscr{D}_{[u]}$, one finds $T_{[u]} \mathbb{P} \Ann( [u] ) = \mathscr{D}_{[u]}$. We emphasize that
this equality is true only at $[u]$, so we do not contradict the non-integrability of $\mathscr{D}$. 

Denote $[v] := \varphi( [u] )$. Then since $\varphi$ preserves annihilators globally, $(\varphi \circ \gamma)(t) \in \mathbb{P} \Ann([v])$ for all $t \in \R$. On the other hand, 
$ d\varphi(X) = d(\varphi \circ \gamma)(0)$. Hence, $d \varphi(X) \in T_{[v] }\mathbb{P} \Ann([v]) = \mathscr{D}_{[v]}$, completing the proof. 
\end{proof} 

We note here one final notion of preserving annihilators before proving our desired corollary. Given $x, y \in \Eintwothree$ with $d_3(x,y) =1$,
we define the projective line $\mathscr{L}_{x,y}$ between $x,y$. Then we say $\varphi \in \Diff(\Eintwothree)$ preserves annihilator lines when
$\varphi( \mathscr{L}_{x,y}) = \mathscr{L}_{\varphi(x), \varphi(y)} $ for all such $x,y \in \Eintwothree$.

\begin{corollary}\label{GtwoEquivalence} 
The following subgroups of $\Diff(\Eintwothree)$ coincide with $\mathsf{image}(F) \cong \Gtwosplit$ from Theorem \ref{InfinitesimalAnnihilator}:
	\begin{enumerate}
		\item $H_1: = \{ \; \varphi \in \Diff(\Eintwothree) \; | \; \varphi^* \mathscr{D} = \mathscr{D} \}$. 
		\item $H_2:= \{ \; \varphi \in \Diff(\Eintwothree) \; | \; \varphi( \, \mathbb{P}\Ann( [x] ) \, ) = \mathbb{P} \Ann( \, \varphi([x]) \, ) \; \}$. 
		\item $H_3 := \{ \; \varphi \in \Diff(\Eintwothree) \; | \; \varphi^*d_3 = d_3. \}$. 
		\item $H_4 := \{ \; \varphi \in \Diff( \Eintwothree) \; | \; \varphi( \mathscr{L}_{x,y}) = \mathscr{L}_{\varphi(x), \varphi(y)} \; \text{for} \; x,y \; \text{such that} \; \;d_3(x,y) =1 \}$. 
	\end{enumerate} 
\end{corollary} 

\begin{proof}
By Theorem \ref{InfinitesimalAnnihilator}, we have $H_1 = \mathsf{image}(F) \cong \Gtwosplit$. Thus, it suffices to see the defining conditions of each $H_i$ are equivalent.

By Theorem \ref{GlobalLocalAnnihilators}, we have that $H_1 = H_2$. It is clear that $H_3 \subseteq H_2$. But since $H_2 = \mathsf{image}(F)$ and $\mathsf{image}(F)$ preserves $d_3$,
we have $H_2 \subseteq H_3$. Moreover, $H_1 = \mathsf{image}(F)$ implies $H_1 \subseteq H_4$, since $\mathsf{image}(F)$ preserves annihilator lines. On the other hand, we note that $\mathbb{P}(\Ann(x)) = \bigcup_{y \in \{d_3(x,y) = 1\} } \mathscr{L}_{x,y}$. Hence, $\varphi \in H_4 $ satisfies $\varphi( \mathbb{P}\Ann(x)) \subseteq \mathbb{P} \Ann(\varphi(x))$.
Since $\varphi$ is a diffeomorphism, $\varphi( \mathbb{P}\Ann(x)) = \mathbb{P} \Ann(\varphi(x))$. Thus, $H_4 \subseteq H_2$.
\end{proof} 

In particular, Corollary \ref{GtwoEquivalence} shows that a diffeomorphism $\varphi $ of $\Eintwothree $ that preserves the projective lines $\mathscr{L}_{x,y}$
when $d_3(x,y) =1$ (note: these are \emph{not} all the projective lines) lifts to a linear map $\hat{\varphi} \in \GL(\imoct)$ such that $\varphi = \hat{\varphi}|_{\Eintwothree}$. 

\begin{remark} 
We note here that if $y \in \Ann(x)$, then $ax + by \in \Ann(cx +dy)$ for any $a, b, c, d \in \R$. Hence, one finds that the projective lines $\mathscr{L}_{x,y}$
are integral curves of $\mathscr{D}$, realized as \eqref{AnnihilatorDistribution}. That is, $\gamma(t) = [tx + (1-t) y] $ has $\gamma'(t) \in \mathscr{D}_{\gamma(t)}$ for all $t$.
Conversely, an integral curve $\gamma$ of $\mathscr{D}$ that is also a projective line is necessarily of the form $\gamma = \mathscr{L}_{x,y}$ for some $x,y \in \Eintwothree$ with $x \times y = 0$.  
\end{remark}

 \subsection{Real Cross-Product Bases for $\imoct$ }\label{Sec:RealCrossProductBasis}
 
 We now discuss the notion of a real cross-product basis, analogous to the complex cross-product basis from Proposition \ref{PropBaragliaBasis}. 
 We then define a frame $C$ for $\V$ that is a real cross-product basis. 
It will be the correct basis to use to show that the boundary at infinity of $\nu_q$ is an annihilator polygon.

\begin{definition}\label{Defn:RealCrossProductBasis} 
Let $X = (x_i)_{i=3}^{-3}$ be a basis for $\imoct$ or $\imoct^\C$. 
\begin{itemize} 
	\item We call $X$ \textbf{standardized} when 
 	the eigenspaces of $\mathcal{C}_{x_0} = (w \mapsto x_0 \times w) $ are given by 
	$$\mathbb{E}_{+1}(\mathcal{C}_{x_0}) = \spann_{\F} \langle x_3, x_{-1}, x_{-2} \rangle,  \; \; \mathbb{E}_{-1}(\mathcal{C}_{x_0}) =\spann_{\F} \langle x_{-3}, x_{1}, x_{2} \rangle$$
	and the dot product satisfies $q_{\imoct}(x_i, x_j) = a_{i}\delta_{i, -j}$ for some $a_i \neq 0$.
	\item We call $X$ a \textbf{real cross-product basis} when $ x_i \times x_j = c_{ij} x_{i+j}$ for $c_{ij} \in \R$.
\end{itemize}
\end{definition} 

Next, we recall the algebraic structure of a certain $\Z_3$-grading on the cross-product. 

\begin{lemma}[$\Z_3$-Cross-Product Grading]\label{Z3CrossProductGrading}
Let $z \in Q_- \imoct$. Then define $V^1 := \mathbb{E}_{+1}(\mathcal{C}_z)$ and $V^2 := \mathbb{E}_{-1}(\mathcal{C}_z)$
and $V^0 = \R\{z\}$. Then $V^1, V^2$ are isotropic subspaces and $q: V^1 \times V^2 \rightarrow \R$ is a non-degenerate pairing. 
We also have $V^i \times V^j \subseteq V^{i+j}$ for $i, j \in \mathbb{Z}_3$ and, in fact, the inclusion is an equality if $i \neq j$. 
\end{lemma} 

\begin{proof}
By the Stiefel model from Proposition \ref{StiefelTripletModel}, $\Gtwosplit$ acts transitively on $Q_-(\imoct)$. Thus, we need only prove
the statement for a model point. But this is addressed by the existence of a standardized $\R$-cross-product basis. 
For example, the following basis will suffice: 
$$\mathcal{B} = (k+lk, \, j-lj, \, i-li, \, l, \, i+li, \, j+lj, \, k-lk). $$
\end{proof}

Some remarks are in order on the definitions. Recall that we say the dot product $q$ is \emph{anti-diagonal} in the basis $X = (x_i)_{i=3}^{-3}$ when 
$q(x_i, x_j) = a_{i}\delta_{i, -j}$ for some constants $a_i \neq 0$. 
We now show that if $X$ is a real cross-product basis for $\imoct$, then $q$ is anti-diagonal in the basis $X$. 

\begin{proposition}\label{Anti-Diagonal}
Let $X = (x_i)_{i=3}^{-3}$ be a $\R$-cross-product basis for $\imoct$. Then $x_i \cdot x_j = c_{i,j}\delta_{i, -j}$, with $c_{i,-i}$ nonzero. 
\end{proposition} 

\begin{proof}
The key point is that we have $\mathcal{C}_{x_0}: \imoct \rightarrow \imoct$ diagonalizable by our hypothesis. This is only possible if $q(x_0) < 0$ by
\eqref{DCP}. Indeed, if $q(x_0)>0$, then zero is its only real eigenvalue. If $q(x_0) = 0$, then \eqref{DCP} says zero is the only eigenvalue of $\mathcal{C}_u$ and
$\dim \ker \mathcal{C}_u = 3$ says $\mathcal{C}_u$ is not diagonalizable. Thus, setting $\hat{x}_0 := \frac{x_0}{|q(x_0)|^{1/2}}$, we have $q(\hat{x}_0) = -1$. 
Then for $i \neq 0$, each $x_i$ is a $\pm 1$ eigenvector of $\mathcal{C}_{\hat{x}_0}$. 
Now, note that $\langle x_3, x_2, x_1 \rangle \subset \Ann(x_3)$ and $\langle x_{-3}, x_{-2}, x_{-1} \rangle \subset \Ann(x_{-3})$ from the cross-product
grading. The inclusions are equality by dimension count. Thus, $x_{3} \times x_{-3} = c_{3, -3} x_0$ with $c_{3,-3} \neq 0$. 
Now, applying Corollary \ref{Z3CrossProductGrading} (whose proof is independent of the current result), we conclude that $x_3, x_{-3}$
must be in opposite eigenspaces of $\mathcal{C}_{\hat{x}_0}$. Without loss of generality, suppose that $x_3 \in \mathbb{E}_{+1}(\mathcal{C}_{\hat{x}_0})$.
Define $V^0 = \R \{x_0\}, V^1 :=  \mathbb{E}_{+1}(\mathcal{C}_{\hat{x}_0}), V^2:=  \mathbb{E}_{-1}(\mathcal{C}_{\hat{x}_0})$. 
By Corollary \ref{Z3CrossProductGrading}, we have $V^2 \times_{\imoct} V^2 = V^1$. However, from the cross-product grading, the only possible
$u, v \in \spann (x_i)_{i=-3}^{2} $ with $u \times v \in \R \{x_3\}$ is for $ u, v \in \langle x_2, x_1 \rangle $. We conclude that $x_1 \times x_2 = c_{1,2} x_3$
for some $c_{1,2} \neq 0$. This also means $x_1, x_2 \in V^2$. Similarly, we must have $x_{-1} \times x_{-2} = c_{-1,-2} x_{-3}$ for $c_{-1,-2} \neq 0$
and hence $x_{-1},x_{-2} \in V^1$ since $x_{-3} \in V^2$. Next, we claim the following, denoted by $(\star)$ -- elements $x \in V^1$ and $y \in V^2$ are orthogonal if and only if 
$x \times y =0$ by Corollary \ref{Z3CrossProductGrading} again. This holds by a 3-form argument with $\Omega$. We know $x \times y \in \R\{x_0\}$. Thus, $x \times y = 0 \iff (x \times y) \cdot x_0 = 0
\iff \Omega(x,y,x_0) = 0 \iff \Omega(x, x_0, y) = 0 \iff x \cdot y = 0$. Fix an index $i$ and suppose for simplicity that $x_i \in V^1$. Then
$x_i \cdot V^1 = 0$ since $V^1$ is isotropic. The only elements $y \in V^2$ with $x_i \cdot y \neq 0$ are of the form $y \in \R \{ x_{-i} \}$ by $(\star)$. 
Hence, $x_i \cdot x_j =0$ unless $j = -i$, in which case their dot-product is nonzero. \end{proof}

\begin{remark} 
In fact, the proof of that proposition shows that the standardized condition 
very nearly holds automatically as well. Indeed, if $X$ is a real-cross-product basis, then $q(x_0) < 0$ is forced. 
Define $\hat{x}_0 := \frac{x_0}{|q(x_0)|^{1/2}}$ and then 
either $\mathbb{E}_{+1}(\hat{x}_0) = \spann_\R \langle x_3, x_{-1}, x_{-2} \rangle$ or $\mathbb{E}_{+1}(\hat{x}_0)  = \spann_{\R} \langle x_{-3}, x_{1}, x_{2} \rangle$. 
\end{remark}

Lemma \ref{Z3CrossProductGrading}, which produces a $\Z_3$-cross-product grading for $\imoct$ by taking 
$x \in Q_-\imoct = \{x \in \imoct \; | \; q(x) = -1\}$ and setting $V_0 := \R \{x\} , V_1 := \mathbb{E}_{+1}(\mathcal{C}_{x}), V_{2} := \mathbb{E}_{-1}(\mathcal{C}_{x})$,
such that $V^i \times V^j \subset V^{i+j}$ (with equality if $i \neq j$), will allow us to prove that a standardized basis is automatically
a real cross-product basis. We now show this is the case and discuss further structure of standardized real cross-product bases 
(compare to Proposition \ref{PropBaragliaBasis}). 

\begin{lemma}[Standardized Bases]\label{Lem:RealCrossProductBasis} 
Let $C = (v_i)_{i=3}^{-3}$ be a standardized basis. Then $C$ is a real cross-product basis and 
\begin{enumerate} 
	\item the 3-form $\Omega$ is given by $\Omega = \sum_{i+ j + k = 0} \; c_{ijk}\, v^{i, j, k} $ for some $c_{ijk} \in \R$. In particular, $ D = \mathsf{diag}(a_i)_{i=3}^{-3} \in \Gtwosplit$ if $a_ia_ja_k =1$ for $ i  + j+ k = 0$. 
	\item If $v_i \in \mathbb{E}_{\varepsilon}(\mathcal{C}_{v_0})$ for $\varepsilon \in \{+1, -1\}$, then $\Ann(v_i )= \C\{v_i\} \oplus (v_i^\bot \cap \mathbb{E}_{-\varepsilon}(\mathcal{C}_{v_0}) \,)$.
\end{enumerate} 
\end{lemma} 

\begin{proof} We prove (2) before showing $C$ is a real cross-product basis. Denote $V_1 := \mathbb{E}_{+1}(\mathcal{C}_{v_0}), \;V_2 := \mathbb{E}_{-1}(\mathcal{C}_{v_0})$ and $V_0 = \R\{x_0\}$. We show that 
$$\Ann(v_3) = \spann_{\R} \langle v_3, v_2, v_1 \rangle \; = \C \{v_3\} \oplus (v_3^\bot \cap \mathbb{E}_{-1}(\mathcal{C}_{v_0}) \,). \; (*)$$
The same argument given here works for any other $\Ann(v_i)$. 
Now, note that for $i \in \{1,2\}$, we have $v_{3} \times v_i \in V_0$ by Lemma \ref{Z3CrossProductGrading}. Hence, $ v_3 \times v_i = c_{i} \, v_0$ for some $c_i \in \C$. On the other hand, $v_3 \cdot v_i = 0$ since $q$ is anti-diagonal in the basis $C$. Thus, by the multiplicativity of $q$, we have 
$$c_{i}^2 q(v_0) = q(c_{i} \, v_0) = q(v_3 \times v_i ) = q(v_3v_i) = q(v_3)q(v_i) = 0 .$$
Hence, $v_3 \times v_i = 0$. Of course, $v_3 \in \Ann(v_3)$.  
We conclude $\spann_{\R} \langle v_3, v_2, v_1 \rangle = \Ann(v_3)$ by Lemma \ref{Prop:Annihilators}. Since $v_3^\bot \cap \mathbb{E}_{-1}(\mathcal{C}_{v_0})  = \spann_\C\langle v_1, v_2 \rangle$,
this proves the claim. 

To show $C$ is a real cross-product basis, we need to verify relations of the form $v_i \times v_j = c_{i,j} v_{i+j}$. Write $v_i \in V_{k}, v_j \in V_l$.
We distinguish cases based on the pair $(k,l)$. Note that if $k =0$ or $l=0$, then we are done by hypothesis of $C$ being standardized. For the case $k =1, l =2$,
we note that $ v_i \times v_j =0$ unless $j= -i$ by the annihilator statement (2). But in this case $v_i \times v_{-i} = c_{i,-i} v_0 \in V_0$ by the $\Z_3$-grading of $(V_0, V_1, V_2)$. 
Finally, we check the cases $k=l=1$ and $k= l =2$. The argument is nearly the same, so we handle the former. The key is to note that $v_{-i} = \beta_ i v_i^*$ for some $\beta_i \neq 0$ with respect to $q$. 
We show $v_{3} \times v_{-1} = c_{3,-1} \,v_2$, as the same argument works for the pair $(v_3, v_{-2})$. Now, $ v_3 \times v_{-1} = av_{-3} + bv_{1} + cv_{2} \in V_2$ by the $\Z_3$-grading. Thus, if $(v_{3} \times v_{-1} ) \cdot v_{3} = 0 = (v_{3} \times v_{-1} ) \cdot v_{-1}$, then it follows that $v_{3} \times v_{-1} \in \R\{v_{-2} \}$. 
Now, using the 3-form $\Omega$ from \eqref{Calibration}, we find
$$(v_{3} \times v_{-1}) \cdot v_3 = \Omega( v_{3}, v_{-1}, v_3) = 0 = \Omega(v_3, v_{-1}, v_{-1} ) = (v_{3} \times v_{-1} )\cdot v_{-1}.$$
We conclude that $C$ is a real cross-product basis. 

(1) Note that $x_i \times x_j = c_{ij}x_k$ and $x_k \cdot x_l = 0$ unless $k =-l$. This proves (1).
\end{proof} 

We now define an $h$-unitary, $\hat{\tau}_{\V}$-real cross-product basis in $\V$.
Set $\xi = e^{i \pi/6}$ and define $C$ in terms of a unitary matrix $W$ as follows: 
\begin{align}
	C &:= H^{-1/2} W  \label{RealCrossProductBasis} \\
	W &:= \frac{1}{\sqrt{6}} \begin{pmatrix}   \frac{1}{\sqrt{2}} & \frac{1}{\sqrt{2}} & \frac{1}{\sqrt{2}}&\sqrt{-3} & \frac{1}{\sqrt{2}}& \frac{1}{\sqrt{2}}& \frac{1}{\sqrt{2}}\\
				\xi^{9} & \xi^7& \xi^{11} & 0& \xi^{5} &\xi^1 & \xi^3  \\
  				 -1 & \xi^{2} & \xi^{10} &0& \xi^{10} & \xi^{2} & -1\\
  				1& -1& -1& 0& 1& 1& -1 \\
				  -1 & \xi^{10} & \xi^2 &0 & \xi^2 & \xi^{10} & -1 \\
  			 	\xi^3 & \xi^{5} & \xi^1 & 0& \xi^7 & \xi^{11} & \xi^{9} \\
				\frac{1}{\sqrt{2}}& \frac{1}{\sqrt{2}}& \frac{1}{\sqrt{2}}& -\sqrt{-3}  &\frac{1}{\sqrt{2}}& \frac{1}{\sqrt{2}}& \frac{1}{\sqrt{2}}\end{pmatrix}. \label{ComplicatedRealCrossProduct} 
\end{align} 
The matrix $W$ is a re-ordering of $S$ from \eqref{HiggsEigenbasisInOrder}, a matrix used earlier to factorize the Higgs field eigenbasis. 
Using the expression from Proposition \ref{PropBaragliaBasis} for $q_{\Oct'}^\C$ in the basis $\mathcal{B}$, a (computer-assisted) calculation shows that in the basis $W$, the
dot product $q_{\Oct'}^\C$ is represented by the matrix $-Q $, where $Q$ is from \eqref{Qmatrix}. For reasons soon to be apparent, write 
$$W =: (w_3, \; w_5, \; w_1, \; w_0, \; w_7, \; w_{11}, \; w_{9} ),$$
in column-vector form. 

A calculation shows $W$ is unitary and hence $C$ is $h$-unitary. We now show that $C$ is also $\hat{\tau}$-real. 
Observe the symmetry $ Q\overline{W} = W$. By direct calculation, $QHQ^{-1} = H^{-1}$. Hence, 
$$\hat{\tau}( C ) = H^{-1}Q (H^{-1/2} \overline{W} ) = Q H H^{-1/2} \overline{W} = Q H^{1/2} \overline{W} = H^{-1/2}Q\overline{W} = H^{-1/2} W = C .$$
Thus, $C$ is $h$-unitary and $\hat{\tau}$-real. We now build up to showing $C$ is a real cross-product basis.

Let us compute the cross-product endomorphism $\mathcal{C}_{w_0}: \imoct^\C \rightarrow \imoct^\C$.
Using all the cross product relations 
from Proposition \ref{PropBaragliaBasis}, a tedious calculation shows that for $y = \sum_{i=3}^{-3} y_i u_i $ in the in the basis $\mathcal{B}$, we have 
\begin{align}\label{x7CrossProduct}
	 w_0 \times y= \left ( \frac{1}{\sqrt{2} } \, y_0, \, \xi^3 \, y_{-1}, \, \xi^3 \, y_{-2}, \frac{1}{\sqrt{2}} \, y_{-3} + \frac{1}{\sqrt{2}} y_{+3}, -\xi^3 \, y_{2}, \, -\xi^3 \, y_{1} \,
\frac{1}{\sqrt{2}} y_0  \right ).
\end{align}

Consider the following transformation, which will be crucial going forwards:
\begin{align}
	R_{\xi}:= \mathsf{diag}( 1, \xi^{-2}, \xi^{-4}, -1, \xi^4, \xi^{2}, 1),
\end{align} 
in the basis $\mathcal{B}$ from Proposition \ref{PropBaragliaBasis}. Write $R_{\xi} = (r_i)_{i=3}^{-3}$. Examining the coefficients, we find $-r_{i}r_j r_{k} = 1$ for $i + j + k = 0$. By Proposition \ref{PropBaragliaBasis} part (5), we see that
$-R_{\xi} \in \Gtwo.$ Here is the key fact we will repeatedly apply: $w_{k+2} = R_{\xi} w_{k}$. This equation can be directly verified. 

We now claim that 
	\begin{align} \label{CrossProductEigenspaces} 
		\begin{cases} 
			W_1 := \spann_\C \langle w_3, w_7, w_{11} \rangle \; = \; \mathbb{E}_{+1}(\mathcal{C}_{w_0}) \\
			W_2 := \spann_\C \langle w_1, w_5, w_{9} \rangle \; = \; \mathbb{E}_{-1}(\mathcal{C}_{w_0})
		\end{cases}
	\end{align} 
By a calculation with equation \eqref{x7CrossProduct}, one finds $w_0 \times w_1 = -w_1$.
Applying $-R_{\xi}$ to the relation $w_0 \times w_k = a_k \, w_k$, we inductively find $w_0 \times w_{k+2} = - a_k w_{k+2}$, with the signs of the eigenvectors
alternating. Hence, $W_1 \subseteq \; \mathbb{E}_{-1}(\mathcal{C}_{w_0}), \; W_2 \subseteq \mathbb{E}_{+1}(\mathcal{C}_{w_0})$. 
Since $W_1 \oplus W_2 = w_{0}^\bot$ and $w_0 \times w_0  = 0$, the equality \eqref{CrossProductEigenspaces} must hold.
 
We can now apply Lemma \ref{RealCrossProductBasis}. 
If we re-label the column vectors of $W$ as $W = (a_3, \, a_2, \, a_1, \, a_0, \, a_{-1}, \, a_{-2}, \, a_{-3} )$, then $(a_i)_{i=3}^{-3}$ is standardized by our 
previous observations. Hence, $W$ is a real cross-product basis by Lemma \ref{RealCrossProductBasis}. 
Furthermore, the lemma shows the annihilators of $w_i$ are found by (indices mod 12)
$$	\Ann(w_{2k+1} ) \; = \spann_{\R} \langle w_{2k-1}, w_{2k+1}, w_{2k+3} \rangle.$$
Of course, since $C = H^{-1/2}W$, with $H^{-1/2} \in \Gtwo$, the basis $C: = (x_3, \, x_5, \, x_1, \, x_0,\,  x_{7}, \, x_{11}, \, x_9) $ is also a real cross-product basis (if indexed correctly) 
with
\begin{align}\label{EigenbasisAnnihilator}
	 \Ann(x_{2k+1} ) \; = \spann_{\R} \langle x_{2k+1}, x_{2k-1}, x_{2k+3} \rangle.
\end{align} 
Moreover, since $C$ is a $\tau_\V$-real and $H$-unitary frame for $\V$, the subspace $\spann_\R (x_i)_{i=3}^{-3}$ is a fiber-wise copy of $\imoct $.

\subsection{Boundary Values of the Model Surface $\nu_0$} \label{ModelSurface} 

In this section, we directly compute the coordinate expression for the almost-complex curve $\nu_0$ associated to $q \equiv 1$ by solving \eqref{HitEuc}. 
We call $\nu_0$ the \emph{model surface}.

For later purposes, we compute the change of basis from the frame $\Mh = H^{-1/2} \mathcal{M}$ to the $\tau$-real frame $E$ defined below.
Here, we recall $S$ is from \eqref{HiggsEigenbasisInOrder}, given by $S = (w_1, \, w_3, \, w_5, \, w_7, \, w_9, \, w_{11}, \,w_0)$ in our new notation. 
\begin{align}\label{ReorderedRealBasis} 
	E := H^{-1/2} S = (x_1, \, x_3, \, x_5, \, x_7, \, x_9, \,x_{11}, \,x_{0}).
\end{align}

Note that $E$ is a re-ordering of the columns of $C$. While $E$ is not a real cross-product basis
in this ordering, the basis $E$, as opposed to $C$, will nevertheless be more convenient to work with as we see in Lemma \ref{BoundaryVertices}.

To compute the change of basis from $\Mh$ to $E$, it suffices to compute the change of basis from $\M$ to $S$. 
A (computer-assisted) calculation finds the change of basis matrix $S^{-1} \mathcal{M} = E^{-1} \Mh$.

\begin{proposition}[Change of Basis $\mathcal{M}_h$ to $E$] \label{ChangeOfBasis} 
The following linear relations hold: 
 \begin{itemize}
	\item $i = \frac{1}{\sqrt{6}} \left(\; ( -w_1+w_7) + \, (w_3 -w_9) \, + (-w_5+ w_{11}) \, \right)$
	\item $l = \frac{1}{2\sqrt{3}} \left( \;(w_1+w_7) - 2 (w_3 +w_9) + (w_5 +w_{11}) \; \right) $
	\item $k = \frac{1}{2\sqrt{3}} (  \,(w_1 - w_7) + \, 2(w_3-w_9) + \,(w_5-w_{11}) ) $.
	\item $li = \frac{1}{2} \left( \, -( w_1 + w_7) + (w_5 + w_{11}) \, \right) $.
	\item $j = \frac{1}{2} \left( \, (w_1 - w_7) + (-w_5 +w_{11} ) \right) $.
	\item $lj = -\frac{1}{\sqrt{6}} \left( (w_1+w_7) + (w_3+ w_9) + (w_5+w_{11}) \right) $
	\item $lk = w_{0} $.
\end{itemize}
\end{proposition} 

Now, let $q_0 \equiv 1$ be the constant sextic differential. Then we can express the tautological section $\sigma$ corresponding to $\nu_0$ 
by $\sigma(z) = i = \frac{1}{\sqrt{6}} (-1, +1, -1, +1, -1, +1, 0)^T$ in the basis $E = H_0^{-1/2}W$, where $H_0$ solves equation \eqref{HitEuc} for $q_0$.
Define $\bm{v} := \frac{1}{\sqrt{6}} (-1, +1, -1, +1, -1, +1, 0)^T \in \C^7$.
Recall that the $\nabla$-parallel
translation is given by \eqref{ModelParallelTranslation} in the standard frame $\mathcal{F} = (\mathbf{e}_i \mapsto u_i)$ for $u_i$ from \eqref{BaragliaBasis}. Hence, $\nabla$-parallel translation $\mathcal{P}: (\V_\R)_z \rightarrow (\V_\R)_{p_0}$, where $p_0$ is the origin, is represented by the matrix $\mathcal{D}$ from \eqref{ModelDiagonalTiteca} in the basis $\tau$-real basis $E_0$ from \eqref{HiggsEigenbasis}. 
We conclude that $\nu_0(z)$ is given in $E_0$-coordinates by 
	\begin{align}\label{ModelSurfaceCoordinateExpression} 
		\nu_0 = \frac{1}{\sqrt{6}} \sum_{k=1}^{6} (-1)^k e^{ 2\alpha \, \mathsf{Re}(z \xi^{2k-1} )} x_{2k-1}  = \mathcal{D} \,\bm{v}.
	\end{align} 

It will be convenient for us to have the following expressions for $c_k := 2\alpha \, \mathsf{Re}(z \xi^{2k-1} )$. 
	\begin{align} \label{Coefficients} 
		\begin{cases}
			c_1 &= 2 \alpha \mathsf{Re}(z\xi^1) =  \alpha (\sqrt{3}x-y)  \\
			c_2 &= 2 \alpha \mathsf{Re}(z\xi^3) = -2 \alpha \, y. \\
			c_3 &= 2 \alpha \mathsf{Re}(z\xi^5) =  \alpha( - \sqrt{3}x-y) \\
			c_4 &= 2 \alpha \mathsf{Re}(z\xi^7)  =  \alpha(-\sqrt{3}x+y)  \\
			c_5 &= 2 \alpha \mathsf{Re}(z\xi^9) = 2\alpha \, y  \\
			c_6 &= 2 \alpha \mathsf{Re}(z\xi^{11} ) =  \alpha(\sqrt{3}x+y) .
		\end{cases}
	\end{align} 
We can now compute the projective limits of the model surface $\nu_0$ along Euclidean rays in the standard $z$ coordinate. 
Here, $[u]$ denotes the projective class of $u \in \imoct \cong \R^{3,4}$. 

\begin{proposition}[Limits of the Model Surface] \label{ModelSurfaceLimits} 
Let $\nu_0$ be the model surface \eqref{ModelSurfaceCoordinateExpression} with respect to $q \equiv 1$. 
Then the projective limits of $\nu_0$ along Euclidean rays are given below. 
In particular, $\partial_{\infty} \nu_0 =(\, [x_1] , \, [x_3],\, [x_5], \, [x_7], \, [x_9], \, [x_{11}] \, ) $ is an annihilator polygon with 6 vertices. 
			\begin{center}
			\begin{tabular}{||c c c||} 
				 \hline
			\textbf{Type of path} $\gamma$ & \textbf{Direction} $\theta$ & \textbf{Projective limit} $p_{\gamma}$ of $\nu_0$ \\ [0.5ex] 
			 \hline\hline
			 Quasi-ray & $ \theta \in ( 0, \frac{\pi}{3} )$ & $ p_{\gamma} = [x_{11} ]$ \\ 
	 			Ray $\gamma(t) = te^{i\theta}+iy$\;& $\theta =\frac{\pi}{3} $  & $p_{\gamma} = [x_{11} -e^{\alpha \, y} x_{9} ]$ \\ 
			 \hline
		 	Quasi-ray & $ \theta \in ( \frac{\pi}{3}, \frac{2\pi}{3} )$ & $ p_{\gamma} = [-x_{9} ]$ \\
			Ray $\gamma(t) = te^{i\theta}+iy$\; & $\theta =\frac{2\pi}{3} $  & $p_{\gamma} = [-x_{9} + e^{-\alpha \, y} \, x_{7}]$\\
			 \hline
			 Quasi-ray & $ \theta \in ( \frac{2\pi}{3}, \pi )$ & $ p_{\gamma} = [x_{7} ]$ \\
			 Ray $\gamma(t) = te^{i\theta}+iy$\;& $\theta = \pi $  & $p_{\gamma} = [x_7 -e^{-2\alpha \, y}x_{5} ]$\\
 			\hline
 			Quasi-ray & $ \theta \in ( \pi, \frac{4\pi}{3} )$ & $ p_{\gamma} = [-x_{5} ]$ \\
 				Ray $ \gamma(t) = te^{i\theta}+iy$\; & $\theta =\frac{4\pi}{3} $  & $p_{\gamma} = [-x_{5} +e^{-\alpha \, y} \, x_{3}]$\\
 				\hline
 				Quasi-ray & $ \theta \in (\frac{4\pi}{3} , \frac{5\pi}{3} )$ & $ p_{\gamma} = [x_{3} ]$ \\
 			Ray $ \gamma(t) = te^{i\theta}+iy$\; & $\theta =\frac{5\pi}{3} $  & $p_{\gamma} = [x_3 -e^{\alpha \, y} x_{1}]$\\
 			\hline
 	 Quasi-ray & $ \theta \in ( \frac{5\pi}{3} , 2\pi)$ & $ p_{\gamma} = [-x_{1} ]$ \\
 			Ray $ \gamma(t) = te^{i\theta}+iy$\; & $\theta = 0$  & $p_{\gamma} = [-x_{1} +e^{2\alpha\, y} \, x_{11}]$ \\
			 \hline
		\end{tabular}
	\end{center}
\end{proposition} 

\begin{proof}
These calculations are all straightforward, using the expressions \eqref{Coefficients} for $c_k$ along with the expression \eqref{ModelSurfaceCoordinateExpression} for $\nu_0$. 
\end{proof} 

\begin{remark}[Global Polar]\label{GlobalPolar} 
We note here $N := x_0$ identifies as a global orthogonal vector to $\mathsf{image}(\nu_0)$. Hence, $\nu_0$ is not linearly full (compare to Proposition \ref{LinearlyFull}). We note further the algebraic relationship between $N$ and the vertices $(p_i)_{i=1}^6$ of the asymptotic boundary $\Delta$: the vertices $p_{i}$ alternate between $\mathbb{E}_{+1}(\mathcal{C}_N)$ and $\mathbb{E}_{-1}(\mathcal{C}_N)$. 
\end{remark} 

\subsection{$\partial_{\infty} \nu_q$ is an Annihilator polygon} \label{BoundaryAnnihilatorPolygon} 

We now show the asymptotic boundary $\Delta:= \partial_{\infty}( \, \nu_q)$ of the almost-complex curve $\nu_q$ associated to monic polynomial sextic differential 
$q \in H^0(\K_\C^6)$
is an annihilator polygon in $\Eintwothree$. The idea is to compare $\nu_q$ to appropriately re-normalized versions of the model surface $\nu_0$ from \eqref{ModelSurfaceCoordinateExpression}, in a natural coordinate $w$ for $q$. 

\begin{definition}
Let $ \nu_q: \C \rightarrow \quadric$ be an almost-complex curve associated to polynomial $q \in H^0(\K^6_\C)$. Then we define the \emph{asymptotic boundary}
$\partial_{\infty}\nu_q := \{ \,\lim_{s \rightarrow \infty} [\nu_q(\gamma(s))] \; | \; \gamma \; \text{is a } \; q-\text{ray} \}$ (recall Definition \ref{QuasiRay}).
\end{definition}

We follow the approach of \cite{DW15} for polynomial affine spheres and split the proof into three pieces:
	\begin{itemize}
		\item \emph{Step 1}(Local vertex structure) In each $q$ half-plane $(w_k, U_k)$, we detect 3 stable limits $p_{1,k}, p_{2,k}, p_{3,k}$ of $\nu_q$.
		Moreover, we verify the annihilator property locally $ \Ann(p_{2,k}) = \spann_{\R} \langle p_{1,k}, p_{2,k}, p_{3,k} \rangle $ in this half-plane.
		\item \emph{Step 2}(Local edge structure) In the $w_k$-half-plane, we find linear edges $e_{i,k}$ between $p_{i,k}, p_{i+1, k}$ in $\Eintwothree$
		with $e_{i, k} \subset \partial_{\infty} \nu_q$.
		\item \emph{Step 3} Using the family of $q$-half-planes from Lemma \ref{StandardHalfPlanes}, we show the local pictures 
		give a global description of $\partial_{\infty}\nu_q$ as an annihilator polygon in $\Eintwothree$. 
	\end{itemize} 
	
We now complete Step 1 of the argument. We follow the notation from Section \ref{ParallelTransport}. In particular, $S_i := \, \{ \frac{(i-1) \pi}{6} < \arg(z) < \frac{i \,\pi}{6} \}$
denotes a sector in an upper-half plane $(w,U)$, where $U$ is identified with the upper half-plane $\Ha$. 

\begin{lemma}[Boundary Vertices]\label{BoundaryVertices} 
Let $q \in H^0(\K^6_\C)$ be a monic polynomial sextic differential. Let $(w, U)$ be a natural coordinate for $q$. 
Then the projective limits of $\nu_q$ along stable quasi-rays in $U$ are shown below, where $M_k = (E_{p_0})^{-1}L_k E_0$ and $L_k$ is from Lemma \ref{AsymptoticParallelTranslation}.
	 \begin{center}
			\begin{tabular}{||c c c||} 
				 \hline
			\textbf{Type of path} $\gamma$ & \textbf{Direction} $\theta$ & \textbf{Projective limit} of $\nu_q$ \\ [0.5ex] 
			 \hline\hline
			 Quasi-ray & $ \theta \in ( 0, \frac{\pi}{3} ) - \{\frac{\pi}{6} \}$ & $ p_{1}: = M_3[x_{11} ]$ \\ 
			 \hline
		 	Quasi-ray & $ \theta \in ( \frac{\pi}{3}, \frac{2\pi}{3} ) - \{\frac{4\pi}{6} \}$ & $ p_{2} := M_3 [x_{9} ]$ \\
			 \hline
			 Quasi-ray & $ \theta \in ( \frac{2\pi}{3}, \pi ) - \{\frac{5\pi}{6} \}$ & $ p_{3} := M_3[x_{7} ]$ \\
 			\hline 
		\end{tabular}
	\end{center}
Moreover, we have $\Ann(p_2) = \spann_{\R} \langle p_1, p_2, p_3 \rangle$.
\end{lemma} 

\begin{proof}
We first express $\sigma(z) = u_0 \in \V_z$ in terms $\tau$-real frame $E$ from \eqref{ReorderedRealBasis}. 
By Proposition \ref{ChangeOfBasis}, we have $u_0 = \frac{1}{\sqrt{6}} (-1, +1, -1, +1, -1, +1, 0)^T =: \bm{v}$ in the frame $E$. Recall the endomorphism $\psi$ expressing $\nabla$-parallel translation in the standard basis. Let $z_0 := 0 \in \C$, which we treat as a distinguished basepoint. Since $\widetilde{E} := \widetilde{\mathcal{F}} E_{z_0}$ is the $\nabla$-parallel translation of $E_{z_0}$, one find the $\nabla$-parallel translation $\mathcal{P}_{z}: (\V^\R)_z \rightarrow (\V^\R)_{z_0}$ is expressed in the frame $E$ by the matrix 
\begin{align}
	[\mathcal{P}_z]_E = (E_{z_0})^{-1} \psi E_z,
\end{align} 
where $E_{z_0}$ denotes the frame $E$ at the origin and $E_z$ the frame at $z$.\footnote{Here, we once again apologize to the reader for the poor choices in notation. 
We emphasize that $E_0 = H_0^{-1/2} S$ is \emph{not} the same thing as $E_{z_0} = (H_q)_{z_0}^{-1/2}S$.} 
We suggestively re-write  
\begin{align}\label{ParallelTransportRewrite1}
 	 [\mathcal{P}_z]_E = (E_{z_0})^{-1} \, (\psi H_q^{-1/2} H_0^{1/2}\psi_0^{-1} ) ( \psi_0 H_0^{-1/2} S) ,
 \end{align} 
 where $H_q $ is the hermitian metric with respect to $q$ and $H_0$ is \eqref{ConstantMetric}, the hermitian metric with respect to $q_0 \equiv 1$. 
 Thus, in the basis $E_{z_0}$ for $\V_{z_0}$, we find $\nu(z)$ is given by 
$\nu(z) =  [\mathcal{P}_z]_E \bm{v}$. 

 Now, suppose that $\gamma: \R \rightarrow \C$ is a stable quasi-ray eventually contained in a stable
 sector $S_{k} \subset \, U$ of a standard half-plane $(U,w)$. We claim that 
 \begin{align}\label{KeyEq}
   \lim_{s \rightarrow \infty}  \left [ \psi H_q^{-1/2} H_0^{1/2}\psi_0^{-1} \right ](\gamma(s)) = \lim_{s \rightarrow \infty} \left [ (\psi \psi_0^{-1}) \psi_0\, (H_q^{-1/2} H_0^{1/2}) \, \psi_0^{-1} \right ](\gamma(s))  = L_k.
  \end{align} 
By Lemma \ref{AsymptoticParallelTranslation}, it suffices to show 
$$ \lim_{s \rightarrow \infty} \left[\psi_0 ( H_q^{-1/2} H_0^{1/2}) \psi_0^{-1} \right ] (\gamma(s)) = \mathsf{id} .$$ 

For the remainder of the proof we conflate $\nu$ with $\nu \circ w^{-1}: \Ha \rightarrow \quadric$ and work in the $w$-coordinate. 
Recall the error term $\tilde{\bm{u}}:= \bm{u} - \bm{u}^-$ as well as the expressions $u_1 = \tilde{u}_1 + \log(c^{5/6}) $ and $u_2 = \tilde{u}_2 + \log(d^{1/6} )$ 
in a natural coordinate as well as $r = e^{u_1 - u_2}$,
$s = e^{ 2u_2}$. We now set $ \tilde{r} = e^{\tilde{u}_1 - \tilde{u}_2}$ and $\tilde{s} = e^{ 2\tilde{u}_2}$. Then one finds 
$$ H_q \, H_0^{-1} = \diag \left ( \frac{1}{\tilde{r} \tilde{s}},  \; \frac{1}{\tilde{r} }, \; \frac{1}{ \tilde{s}}, \; 1,\;  \tilde{s}, \; \tilde{r}, \; \tilde{r} \tilde{s}  \right).$$
Now, in terms of $x_1 = \tilde{u}_1+ \tilde{u}_2$ and $ x_2 = \tilde{u}_1 - 3\tilde{u}_2$, taking an asymptotic approximation, which amounts to taking 
the linear terms, gives us 
\begin{align}\label{Rtilde} 
 (H_qH_0^{-1})^{-1/2} - \mathsf{id} \asymp \widetilde{R} := \diag \left( - \frac{1}{2}x_1 , \; -\frac{1}{4}(x_1- x_2), \; -\frac{1}{4}(x_1+x_2), \; 1, \; \frac{1}{4}(x_1+x_2), \; 
\frac{1}{4}(x_1- x_2),\;  \frac{1}{2} x_1 \right)
\end{align} 

We show that $\lim_{s \rightarrow \infty}  \psi_0 \, \widetilde{R} \, \psi_0^{-1}(\gamma(s)) = 0$ to finish proving \eqref{KeyEq}. 
Recall the expression $\psi_0 = E_0\mathcal{D}E_0^{-1}$ and then write
$$  \psi_0 \, \widetilde{R} \, \psi_0^{-1} =  E_0\mathcal{D}E_0^{-1} \widetilde{R}  E_0 \mathcal{D}^{-1} E_0^{-1}. $$ 
Then a tedious matrix multiplication shows $A:= E_0^{-1} \widetilde{R}  E_0 $ has the following description, for some $\alpha_{ij} \in \C^*$: 
\begin{align}\label{R2Asymptotics_new} 
	A_{ij} = \begin{cases}  \alpha_{ij} \,  x_1 &  i, j \leq 6, \, | i-j | \in \{1,5\} \mod 6 \\  \alpha_{ij} \, x_2&  i, j \leq 6, \, | i-j | \in \{2,4 \} \mod 6\\  0 & i, j \leq 6, \,  | i-j |  \in \{ 0,3 \} \mod 6 \\
		\alpha_{ij} (1-\delta_{ij})\, x_1) & i =7 \; \text{or} \; j =7 \end{cases}.
\end{align}

Recall $\lambda_{ij} = e^{2\alpha s \,(d_i-d_j)}$ as in Lemma \ref{AsymptoticParallelTranslation}.
Then each entry $ \mathcal{D} A_{ij} \mathcal{D}^{-1} = \lambda_{ij}  A_{ij} $ exponentially decays by the stability of the quasi-ray $\gamma$. 
Since $E_0$ is constant, it follows that $\lim_{s \rightarrow \infty} \psi_0 \, \tilde{R} \, \psi_0^{-1}(\gamma(s)) = 0$.

We now relate the projective limit $\lim_{s \rightarrow \infty} [ \nu_q(\gamma(s))]$ to $\lim_{s \rightarrow \infty} [ \nu_0(\gamma(s))]$. 
First, we note that $\psi_0 H_0^{1/2}S = \psi_0 E_0 = E_0 \mathcal{D}$. By equation \eqref{ModelSurfaceCoordinateExpression}, $[\nu_0] =  \mathcal{D} [\bm{v}]$
in the coordinates $E_0$. Hence, combining all the previous steps, we find 
\begin{align}
	\lim_{s \rightarrow \infty} [ \nu(\gamma(s))] &= \lim_{s \rightarrow \infty} [\mathcal{P}_z]_E \, \cdot [\bm{v}](\gamma(s)) \\
		&= \lim_{s \rightarrow \infty} \left [ \, (E_{z_0})^{-1} \, (\psi H_q^{-1/2} H_0^{1/2}\psi_0^{-1} ) ( \psi_0 H_0^{-1/2} S)  \cdot [\bm{v}] \, \right ](\gamma(s)) \\ 
		&= \left[ \, \lim_{s \rightarrow \infty} (E_{z_0})^{-1} \psi H_{q}^{-1/2} H_0^{1/2} \psi_0^{-1}(\gamma(s)) \, \right] \cdot  \left( E_0 \lim_{s \rightarrow \infty} \mathcal{D}[\bm{v}](\gamma(s)) \right) \\
		&= (E_{z_0})^{-1} \, L_k \, E_0 \,\lim_{s \rightarrow \infty} \mathcal{D}[\bm{v}](\gamma(s))  =(E_{z_0})^{-1} \, L_k \, E_0 \,\lim_{s \rightarrow \infty} [\nu_0](\gamma(s)) .
\end{align}
We set 
$$M_k = (E_{z_0})^{-1}L_k E_0 .$$
We now change notation and write $E_0 = (e_i)_{i=1}^7$, as this indexing is simpler going forwards. We write (sloppily) $\bm{e}_i \in \C^7$ for the $i^{th}$ standard basis vector. By Proposition \ref{ModelSurfaceLimits}, the limits $X_k$ of $[\nu]$ along stable quasi-rays in $S_k$, in $E_0$ coordinates, are given by: 
\begin{align}
	\begin{cases}
		X_1 &= M_1[\bm{e}_6], \; \;  X_2 = M_2 [\bm{e}_6 ] \\ 
		X_3 &= M_3[\bm{e}_5 ] , \; \, \; X_4 = M_4 [\bm{e}_5] \\ 
		X_5 &= M_5[\bm{e}_4 ],  \; \, \; X_6 = M_6 [\bm{e}_4] 
	\end{cases} 
\end{align} 
We claim that, in fact, $X_1 = X_2, \, X_3= X_4, \, X_5 = X_6$. These claims are verified by using the unipotents from Lemma \ref{Unipotents}.
For example, $L_2 = L_1 E_0U_1 E_0^{-1}$. Thus, $M_2 = (E_{z_0})^{-1}L_1 E_0 U_1 = M_1 U_1$. Since $U_1(\bm{e}_6) = \bm{e}_6$, we have $X_2 = X_1$. 
More generally, for $k > j$, one inductively finds $M_k = M_{j} \prod_{i=j+1}^{k} U_i$. 
The remaining equalities are proved similarly. To finish the proof of the Lemma, we must show $M_3[\bm{e}_6] = M_2[\bm{e}_6] = X_1$
and $M_3[\bm{e}_4] = M_5[\bm{e}_4] = X_5$. 

Applying $U_2 \mathbf{e}_6 = \mathbf{e}_6$ shows that $X_1 = [M_3 \mathbf{e}_6]$.
Similarly, using that $U_3, U_4 \cdot \bm{e}_4 = \bm{e}_4$, we find $X_5 = [M_5 \textbf{e}_4] = [M_3 U_3U_4 \mathbf{e}_4] = [M_3 \mathbf{e}_4]$.
 
 Thus, the 3 unique limits along stable quasi-rays eventually contained in a sector $S_k$ are given by $p_1 = M_3 [x_{11}], \; p_2= M_3[x_9], \; p_3= M_3 [x_7]$,
 since $\bm{e}_6,  \bm{e}_5,  \bm{e}_4$ correspond, respectively, to $ x_{11},  x_{9},  x_{7}$ under our identifications (recall the column vectors in \eqref{ReorderedRealBasis}). Now, as $M_3 \in \Gtwo$ and $\Ann(x_9) = \spann_{\R} \langle x_{7}, x_{9}, x_{11} \rangle$ by \eqref{EigenbasisAnnihilator} we conclude that $\Ann(p_2) = \spann_{\R} \langle p_1, p_2, p_3 \rangle$.
  \end{proof}

We now detect the edges between the stable limits $p_i$ found in Lemma \ref{BoundaryVertices}. 
The argument here is a combination of the arguments from Lemma \ref{Unipotents} and from Lemma \ref{BoundaryVertices}. 

\begin{lemma}[Boundary Edges]\label{BoundaryEdges} 
Let $q \in H^0(\K^6_\C)$ be a monic polynomial sextic differential. Let $(w, U)$ be a natural coordinate for $q$. 
Let $p_1, p_2, p_3$ be the stable limits of $\nu_q$ in $(w,U)$ as in Lemma \ref{BoundaryVertices}. 
Denote $e_{k}$ for the projective line $\spann \langle p_k, p_{k+1} \rangle$ and $e_k^\circ$ for the interior of $e_k$. Then the following limits hold. In particular,
there is a projective line segment $L_i$ with endpoints $p_i, p_{i+1}$ such that $L_i \subset \partial_{\infty} \nu_q$. 
 \begin{center}
			\begin{tabular}{||c c c||} 
				 \hline
			\textbf{Type of path} $\gamma$ & \textbf{Direction} $\theta$ & \textbf{Projective limit} of $\nu_q$ \\ [0.5ex] 
			 \hline\hline
			 Ray $\gamma(s) = se^{i \theta} + iy_0$ & $ \theta = \frac{\pi}{3}$ & $ p_{\gamma}: = M_3 [x_{11}-e^{\alpha \, y_0} x_9] \in e_1^{\circ} $ \\ 
			 \hline
		 	Ray $\gamma(s) = se^{i \theta} + iy_0$ & $ \theta  = \frac{2\pi}{3}$ & $ p_{\gamma} := M_3 [-x_{9} + e^{-\alpha \, y_0} x_7]  \in e_2^{\circ} $ \\
			 \hline
		\end{tabular}
	\end{center}

\end{lemma} 

\begin{proof} 
Again, we work with the local parametrization $\nu := \nu \circ w^{-1}: \Ha \rightarrow \quadric$. 
We handle the argument for $e_{1}$ between $p_1, p_2$. Thus, we take an unstable ray $\gamma:= \gamma_{y_0}(s) = se^{i \frac{\pi}{3} } + iy_0$ of angle $\frac{\pi}{3}$ in $U$.
We show that for $M_1 = (E_{p_0})^{-1} L_1 E_0$ from Lemma \ref{BoundaryVertices}, 
$$  \lim_{s \rightarrow \infty} [ \nu(\gamma(s)) ] = M_1 [ \bm{e}_6 - e^{\alpha \, y_0} \bm{e}_5] .$$ 
Perturb $\gamma $ to form the stable rays $\gamma_1 $ and $\gamma_2$
by $\gamma_1(s) =s e^{\frac{3\pi i}{12} } + i y_0$ and $\gamma_2(s) = se^{ \frac{5\pi i}{12} } + i y_0$. As in Lemma \ref{Unipotents}, we study $\nu_q$ along the 
circular arc $\eta_s(t) = se^{i ( \frac{3 \pi }{12} + t)} + i y_0$ for $ t \in [0,\frac{\pi}{6}]$, between $\gamma_1(s)$ and $\gamma_2(s)$. We use the same notation
 from the proof of Lemma \ref{Unipotents}. 
Define $G_i(s) := \psi \psi_0^{-1} (\gamma_i(s))$. 
Again, we find that the same construction of the (re-normalized) map $g_s(t) : = G_{1}(s)^{-1} \, \psi\psi_0^{-1}(\eta_s(t))$ for $0 \leq t \leq \frac{\pi}{6}$ satisfies
\begin{align}\label{Asymptotic1}
  g_s(t) = E_0 \, \mathcal{U}_1(s,t) E_0^{-1} + A_1(s), 
  \end{align}  
where $A_1(s) \in O(\, || M_s^0 || \, ) = O(e^{-b \, s})$ for some $b >0$ as long as $s$ is sufficiently large, and 
$$ \mathcal{U}_{1}(s,t) = \exp \left[ \, c^{12}_s(t) \, E_{12} + c^{54}_s(t) \, E_{54} + c^{67}_s(t) \, E_{67} + c^{73}_s(t)\, E_{73}  \right] .$$
We recall that $c^{ij}_s(t) = \int_{0}^{t} \mu^{ij}_s(\tau) d\tau $, where $\mu^{ij}_s(t) \in O( \sqrt{s} e^{ -\frac{\alpha}{2} s(t - \frac{\pi}{12}) } )$ are each integrable functions of $t$,
the integral independent of $s$. 
In particular, $c^{ij}_s( \frac{\pi}{12} ) \in O(1)$. Now, by construction $\eta_{s}(\frac{\pi}{12} ) = \gamma(s)$. We will apply two more asymptotic approximations,
after first recalling the parallel transport matrix of interest.

By re-writing the matrix $[\mathcal{P}_z]_E$ from \eqref{ParallelTransportRewrite1}, we obtain the following expression in preparation for our asymptotic approximation:
\begin{align}\label{ParallelTranslationExpression} 
	[ \mathcal{P}_z]_E =(E_{p_0})^{-1} \, (\psi \psi_0^{-1} ) \, (\psi_0 H_q^{-1/2} H_0^{1/2}\psi_0^{-1} ) (E_0 \mathcal{D} ). 
\end{align} 

Now, $\psi \psi_0^{-1}(\gamma(s)) = G_{1}(s) g_{s}(\frac{\pi}{12} ) $ and by the stability of $\gamma_1$, we have 
$G_1(s) \rightarrow L_1$ as $s \rightarrow \infty$ from Lemma \ref{AsymptoticParallelTranslation}. 
Thus, to compute $\lim_{s \rightarrow \infty} [\nu(\gamma(s))] =\lim_{s \rightarrow \infty} [\mathcal{P}_z]_E \,[\bm{v}]$,
it suffices to compute
\begin{align}\label{EssentialLimit}
	\widetilde{L}_1:= \lim_{s \rightarrow \infty} \left [ g_{s}\left( \frac{\pi}{12} \right) (\psi_0 H_q^{-1/2} H_0^{1/2}\psi_0^{-1} ) (E_0 \mathcal{D} ) [\bm{v} ] \, \right](\gamma(s)).
\end{align} 

The above limit is a product of 3 terms. We previously handled the first term and now we approximate the middle term. 

We claim that 
\begin{align} \label{Asymptotic2} 
	(\psi_0 H_q^{-1/2} H_0^{1/2}\psi_0^{-1} )(\gamma(s)) \asymp \mathsf{id} + A_2, 
\end{align}
where $A_2 \in O(\frac{1}{\sqrt{s}})$.
Start with the same error matrix $\widetilde{R}$ from \eqref{Rtilde} such that $H_q^{-1/2} H_0^{1/2} = \mathsf{id} + \widetilde{R}$.
Then $A_2 = \psi_0\widetilde{R} \psi_0^{-1}$. Recall $\psi = E_0 \mathcal{D} E_0^{-1}$. We find that matrix $A_3 = E_0^{-1}  \widetilde{R} E_0$ has 
$( \mathcal{D} E_0^{-1} \widetilde{R} E_0 \mathcal{D}^{-1} )_{ij} = \lambda_{ij} (A_{3})_{ij} \in O(\frac{1}{\sqrt{s}})$ from the instability of $\gamma$. Since $E_0$ is constant,
it follows that 
$$A_2 = E_0 \mathcal{D} A_3 \mathcal{D}^{-1} E_0^{-1} \in O \left ( \frac{1}{\sqrt{s}} \right).$$

Next, we revisit the model surface $\nu_0(z) =\mathcal{D} \, \bm{v}$, where $\bm{v} =\frac{1}{\sqrt{6}}(-1, +1, -1, +1, -1, +1, 0)^T$, expressed in the basis
$(E_{p_0})$. In particular, by the coordinate expression \eqref{ModelSurfaceCoordinateExpression} and \eqref{Coefficients}, we see that 
$$ \nu_0(\gamma(s)) \asymp e^{2\alpha \sqrt{3} \, s }  \left ( \bm{e}_6 - e^{\alpha \, y_0} \bm{e}_5 \right) + o(e^{2\alpha \sqrt{3}\, s }) \bm{w},$$
where again $\gamma(s) = se^{ i\frac{\pi}{3}} + iy_0$. 
The essential observation here is that $U_1(s,t)$ fixes $\bm{e}_5, \bm{e}_6$ for all $s,t$. Since $\bm{e}_5, \bm{e}_6$ contain the asymptotically
dominant terms of $\mathcal{D} \bm{v} $, the remaining action of the bounded matrix $U_1(s,t) \in O(1)$ on $[\nu] = [\mathcal{D} \bm{v} ]$ is irrelevant once we projectivize $\nu$.
Hence, we can finally compute the limit \eqref{EssentialLimit},
using \eqref{Asymptotic1}, \eqref{Asymptotic2} and $A_1, \, A_2 \in o(1)$. We remark that we do \emph{not} need to compute the limit of $A(s) := [ g_{s}\left( \frac{\pi}{12} \right) (\psi_0 H_q^{-1/2} H_0^{1/2}\psi_0^{-1}  ) (E_0 \mathcal{D} ) ](\gamma(s))$ as $s \rightarrow \infty$; we only need the projective limit of $A(s) \cdot [\bm{v}]$ as $s \rightarrow \infty$. 
Now, observe that 
\begin{align}
	\widetilde{L}_1 &= \lim_{s \rightarrow \infty}  \left [ g_{s}\left( \frac{\pi}{12} \right) (\psi_0 H_q^{-1/2} H_0^{1/2}\psi_0^{-1} ) (E_0 \mathcal{D} ) [\bm{v} ] \, \right](\gamma(s)) \\
	 &= \lim_{s \rightarrow \infty} \left ( \, \left(\, E_0 \, \mathcal{U}_1\left(s, \frac{\pi}{12} \right) E_0^{-1} + A_1 \, \right)( \mathsf{id} + A_2 ) (E_0 )\left[ e^{2\alpha \sqrt{3} s } ( \bm{e}_6 - e^{\alpha \, y_0} \bm{e}_5  ) + o(e^{2\alpha \sqrt{3} \, s}) \bm{w} \right]  \, \right) \\
		&= \lim_{s \rightarrow \infty} \left ( \, \left (E_0 \, \mathcal{U}_1\left(s, \frac{\pi}{12} \right) + o(1) \right) \left[ ( \bm{e}_6 - e^{\alpha \, y_0} \bm{e}_5  ) + o(1) \bm{w} \right]
		   \, \right ) = E_0 [ \bm{e}_6 - e^{\alpha \, y_0} \bm{e}_5].
	\end{align}  
Since $\lim_{s \rightarrow \infty} [\nu(\gamma(s))] = (E_{p_0})^{-1} L_1 \widetilde{L}_1$, we then immediately find 
$$  \lim_{s \rightarrow \infty} [ \nu(\gamma(s)) ] = M_1 [ \bm{e}_6 - e^{\alpha \, y_0} \bm{e}_5] .$$ 
Hence, as $y_0$ varies, the limits $p(y_0) := \lim_{s \rightarrow \infty} [\nu(\gamma_{y_0}(s)) ]$ saturate the interior $e_{1}^{\circ}$ of a projective line segment between $p_1$ and $p_2$.
The argument for the edge $e_2$ is very similar, so we omit it. However, we note that the same proof shows that for the unstable curve
$\gamma_{y_0, 2}(s) = se^{\frac{2\pi i}{3}} + i y_0$, we have $\lim_{s \rightarrow \infty} [ \,\nu(\gamma_{y_0, 2}(s)) \, ] = M_2 [- \bm{e}_5 +e^{-\alpha \, y_0} \bm{e}_4]$ 
by Proposition \ref{ModelSurfaceLimits}. 

To complete the proof, we show that we can replace $M_2, M_1$ in the unstable limits of angles $\frac{2\pi}{3}, \frac{\pi}{3}$,
respectively, with $M_3$. Recall $M_j = M_i \, \prod_{l=i+1}^j U_l$ for $j > i$. We again recall that the unipotents $U_1, U_2$ both fix $\bm{e}_5, \bm{e}_6$, so that $ M_1[ \bm{e}_6 - e^{\alpha \, y_0} \bm{e}_5] = M_3[ \bm{e}_6 - e^{\alpha \, y_0} \bm{e}_5] $. Similarly, $U_2$ fixes $\bm{e}_4$. Hence, the table of limits in the statement of the lemma hold
after recalling that $ (\bm{e}_i)_{i=1}^7 $ corresponds to $E_{p_0} = (x_1, \, x_3, \, x_5, \, x_7, \,x_9, \, x_{11}, \, x_0)$. 
  \end{proof} 

We retain the hypotheses and notation of the proof.
We now explain that as a corollary to the proof, the projective limit of $\nu$ does \emph{not} change when we cross
the unstable angles $\theta \in \{ \frac{\pi}{6}, \frac{3\pi}{6}, \frac{5\pi}{6} \}$, despite the fact that the limit of $\psi \psi_0^{-1}$ does change. 
Let $(w, U)$ be a natural coordinate for $q$. Define $\gamma_{\theta}(s) : = se^{i\theta}$ in $\Ha$, identified with $U$.
We explain the case of $\frac{\pi}{6}$ as a model. By Lemma \ref{Unipotents}, the unipotent $U_1$ at this angle satisfies $U_1 \cdot \bm{e}_6 = \bm{e}_6$. 
Then we can apply the same asymptotics from the Lemma \ref{BoundaryEdges}, except now the model surface
satisfies $[\nu_0(\gamma_{\theta}(s))] \asymp [ \bm{e}_6+ o(1) \bm{w} ]$ along the ray $\gamma$ at height $\frac{\pi}{6}$. The same argument then shows that
$\lim_{s\rightarrow \infty} [ \, \nu( \gamma_{\theta}(s))  \, ]  = p_1 $. 
In fact, the same argument can be applied to the unstable curves $\gamma_{\theta}(s) : = se^{i\frac{\pi}{6}} + iy_0$ for any $y_0 \in \R$.
Applying this argument to the other unstable angles $\theta \in \{ \frac{\pi}{6}, \frac{3\pi}{6}, \frac{5\pi}{6} \}$, we find that 
$ \lim_{s\rightarrow \infty} [ \, \nu( \gamma_{\frac{3\pi}{6}}(s))  \, ] = p_2$ and $ \lim_{s\rightarrow \infty} [ \, \nu( \gamma_{\frac{5\pi}{6}}(s))  \, ]= p_3$
for any $q$-rays of angle $\frac{3\pi}{6}$ and $\frac{5\pi}{6}$, respectively. \\

Thus, we have computed the projective limit of $\nu$ along all $q$-rays eventually
contained in half-plane, stable or unstable, with the exception of angle $0$ or $\pi$, which we handle below. We now prove the main theorem. 

\begin{theorem}[Annihilator Polygon]\label{THM:AnnihilatorPolygon}
Let $q \in H^0(\K_\C^6)$ be a polynomial and $\nu_q: \C \rightarrow \quadric$ be its associated almost-complex curve. Then 
$\partial_{\infty} \nu_q$ is an annihilator polygon with $\deg q + 6$ vertices. 
\end{theorem} 

\begin{proof}
Set $n:= \deg q$ and we use the family $\{ (U_k, w_k) \, \}_{k=1}^{n+ 6}$ of $q$-half-planes from Lemma \ref{StandardHalfPlanes}, where $\bigcup_{i=1}^{n+6} U_i$
covers $\C$ up to a compact set. Let $\gamma: \R_{+} \rightarrow \C$ be a $q$-ray. 
For $s_0$ sufficiently large, we have $\gamma(s) \in U_j$ for some $j$ also by Lemma \ref{StandardHalfPlanes}.

Next, we discuss the compatibility of the local pictures. To find $\partial_{\infty} \nu_q$, it suffices to compute the 
limit of $\nu := \nu_q$ along $q$-rays eventually contained a half-plane $(w_k, U_k)$. On the other hand, we have already done this, except for $q$-rays
of angle 0 or $\pi$. However, in these cases, such $q$-rays are also (eventually) contained in the half-planes $U_{k-1}$ or $U_{k+1}$, respectively, 
where they now have angle $\theta \in (0, \pi)$. We now first discuss the limits along stable $q$-rays. 
By Lemma \ref{BoundaryVertices}, we have exactly three stable limits $(p_{1, k}, p_{2,k}, p_{3,k} )$ in the $q$-half-plane $(w_k, U_k)$. These 
 limits can be found by taking the limit of any stable ray or quasi-ray of angle $\frac{\pi}{12}, \frac{5\pi}{12}, \frac{11\pi}{12}$, respectively, in $U_k$. 
We are led to define the Euclidean rays $\alpha_k$ in the standard $z$ coordinate for $\C$ by $\alpha_{k}(s) :=se^{ \frac{2\pi i \,}{n+6} \left(k + \frac{1}{4} \right) } $. 
By Lemma \ref{StandardHalfPlanes}, we have that $\alpha_{k-1}, \alpha_{k}, \alpha_{k+1}$ are stable quasi-rays eventually contained in $U_k$, each of respective angles $\frac{\pi}{12}, \frac{5\pi}{12}, \frac{11\pi}{12}$ in the $w_k$-coordinate. 
We define $ \gamma_{1,k} := w_k \circ \alpha_{k-1}, \;  \gamma_{2,k} := w_k \circ \alpha_{k}, \; \gamma_{3,k} := w_k \circ \alpha_{k+1}$
and write $\nu_k := \nu \circ w_k^{-1} : \Ha \rightarrow \quadric$ for the local parametrization of $\nu$ in the $w_k$-coordinate. 
In the $w_k$-coordinate, we see the picture from Figure \ref{Fig:NaturalHalfPlaneRays}, except each $\gamma_i$ is now a quasi-ray. 
Hence, for $i \in \{1,2,3\}$, we have 
\begin{align}
	 \lim_{s \rightarrow \infty} [\nu_k( \gamma_{i, k}(s)) ] = p_{i, k }. 
\end{align} 
On the other hand, we then see for $s$ large enough that $\alpha_{k-1}(s) \in U_k \cap U_{k-1}$, we have 
$$ \nu_{k-1}(\gamma_{2,k-1}(s)) = \nu(\alpha_{k-1}(s)) = \nu_{k}(\gamma_{1, k}(s)).$$
 Hence, $ p_{2,k-1} = p_{1, k}$. Similarly, one finds 
that $p_{3, k} = p_{2, k+1}$. We now denote $p_i := \lim_{s \rightarrow \infty} [ \nu(\alpha_i(s)) ] = p_{2, i}$. Then the 
multiset of $q$-stable limits is given by $L = \{ p_i \}_{i=1}^{n+6}$. Since $ p_{i+1} = p_{3, i}, \; p_{i-1} = p_{2,i}$, with cyclic indices mod $n$,
we have that $\Ann(p_i) = \spann_{\R} \langle p_{i-1}, p_{i}, p_{i+1} \rangle$ by Lemma \ref{BoundaryVertices}. 

If $\gamma$ is an unstable $q$-ray eventually contained in $U_k$, then 
$\lim_{s \rightarrow \infty} [\nu(\gamma(s))] \in L_{i,k}$, where 
 $L_{i,k}$ is the projective line segment with endpoints $p_{i,k}, p_{i+1,k}$ described in Lemma \ref{BoundaryEdges}.
Define $\Delta $ to be the union of the projective line segments $L_{i,k}$. 
Then we have shown $\Delta =\partial_{\infty}\nu$. The polygon $\Delta$ is an annihilator polygon with $n +6$ vertices, with vertex set
$(p_{2,k })_{k=1}^{n+6}$. 
 \end{proof} 

\section{Future Directions} 

We now discuss how Theorem \ref{THM:AnnihilatorPolygon} allows us to define a map of moduli spaces. 
Define $\mathsf{TP}_{k} $ as the moduli space of \emph{marked} annihilator polygons with $k$ vertices up to the $\Gtwosplit$-action, where 
$(\Delta, p_0) \sim (\Delta', q_0)$ when there exists $\varphi \in \Gtwosplit$ such that $ \varphi(\Delta) = \Delta'$ and $\varphi(p_0) = q_0$.
We may also define $\mathsf{TS}_{k} := $ the vector space of monic, centered polynomial sextic differentials of degree $k$. 
Here, a polynomial $q(z) = \sum_{i=0}^N c_i z^i$ is centered when $c_1 = 0$, i.e., the sum of the roots of $q$ is zero. 
Then Theorem \ref{THM:AnnihilatorPolygon} allows us to define a map $\alpha: \mathsf{TS}_{k} \rightarrow \mathsf{TP}_{k+6}$ by
$\alpha( q) = ( \partial_{\infty} \nu_q, p_0 )$, where $ p_0 := \lim_{s \rightarrow \infty} \nu_q( se^{i\theta_0})$ and $\theta_0 = \frac{2 \pi}{k+6}\left(1 + \frac{1}{4} \right) $.
Here, we use that $\gamma(s) = se^{i \theta_0}$ is a stable $q$-quasi-ray for \emph{any} monic polynomial sextic differential $q$ by Lemma \ref{StandardHalfPlanes}.

A straightforward argument shows that $\mathsf{TS}_{k} \cong_{\mathbf{Top}} \R^{2(k-1)}$. 
Next, one counts that the \emph{generic} locus in $\mathsf{TP}_{k+6}$ is also of dimension $2(k-1)$. 
Here, by \emph{generic} we mean that $d_3(p_i, p_{l}) = 3$ if the induced distance between $p_i, p_l$ in the graph structure
of the polygon is $\geq 3$. Intuitively, this means the points are in as general condition as possible while still satisfying the annihilator condition.
 We denote the moduli space of all such marked polygons, up to the $\Gtwosplit$-action, by 
$\mathsf{TP}^{\mathsf{gen}}_{k+6}$. As proved in Lemma \ref{GenericPoylgonModuliSpace}, the space $\mathsf{TP}^{\mathsf{gen}}_{k+6}$ consists of a (finite) collection of disjoint cells of real dimension $2(k-1)$. 
Without this demand of non-degeneracy, there are lower dimensional strata in the moduli space $\mathsf{TP}_{k+6}$. 
In particular, this dimension count motivates the question of
what the image of $\alpha$ is and furthermore whether $\alpha$ is a homeomorphism onto
one of the $2(k-1)$-cells of $\mathsf{TP}^{\mathsf{gen}}_{k+6}$. A problem here is that the techniques we have used seem to only describe the local structure of the polygon,
yet the generic condition is a global one. 

The question of describing the image of $\alpha$ is further motivated by the related work in \cite{DW15, TW25, Tam19}, where 
similar (geometric) homeomorphisms were realized between moduli spaces of holomorphic differentials and
moduli spaces of polygons by taking the asymptotic boundary $\Delta := \partial_{\infty}(\phi_q)$ of geometric objects $\phi_q$ associated to $q$.
The role played by $\Gtwosplit$ was instead filled by other rank two semi-simple real Lie groups: $\mathsf{SL}_3\R, \, \mathsf{SO}_0(2,3), \, \mathsf{PSL}_2\R \times \mathsf{PSL}_2\R$ in each respective paper. These results lead us to conjecture: 

\begin{conjecture} 
The map $\alpha:  \mathsf{TS}_{k} \rightarrow \mathsf{TP}_{k+6}$ has image in $\mathsf{TP}^{\mathsf{gen}}_{k+6}$ and is a homeomorphism onto its image.
\end{conjecture} 

\appendices
\section{$\g_2$ Lie Theory} 

\subsection{Root Space Decomposition} \label{RootSpace} 

We fix the Cartan subalgebra (CSA) $\mathfrak{h}$ of $\g_2$ given by the set of diagonal matrices in the basis \eqref{BaragliaBasis}. Each element $X \in \mathfrak{h}$ is of the form 
\begin{align}\label{G2_CSA} 
	A = \mathsf{diag}(r+s, r, s, 0, -s, -r, -r - s) \; \; \text{for} \; \; r, s \in \C.
\end{align} 
We can see \eqref{G2_CSA} holds by Proposition \ref{PropBaragliaBasis} part (4).
Infinitesimally, then $ A=\mathsf{diag}(a_i)_{i=3}^{-3}  \in \End(\C^7)$ stabilizes $\Omega$ iff $ a_i + a_j + a_k = 0 $ for pairwise distinct triplets $(i,j,k)$ with $i + j + k = 0$. 
Let $\Delta$ denote the roots of $\mathfrak{h}$. We have a root space decomposition 
$$ \g_2 = \h \oplus \left( \bigoplus_{\delta \in \Delta} \g_{\delta} \right) .$$ 
Here, we fix positive, primitive roots $ \Pi := \{\alpha, \beta\}$ to be 
\begin{align}
		\alpha &:= r^*-s^* \\
		\beta &:= s^*,
\end{align} 
where $*$ indicates the dual map.
The root space of $\g_2$ is shown in Figure \ref{G2Roots}. The highest root $\gamma = 2\alpha + 3 \beta
$ is given by $\gamma = 2r^* + s^*$. We will need that our principal co-roots are 
\begin{align} \label{principal_coroots} 
	t_{\beta} = \mathsf{diag}(1,-1, 2, 0, -2, 1, -1) \; \mathrm{and} \; t_{\alpha} = \mathsf{diag}(0, 1, -1,0, 1, -1, 0 ) .
\end{align} 
Note that $\alpha$ is the long root and $\beta$ the short root. \\

\begin{figure}[ht]
	\centering
	\includegraphics[width=.45\textwidth]{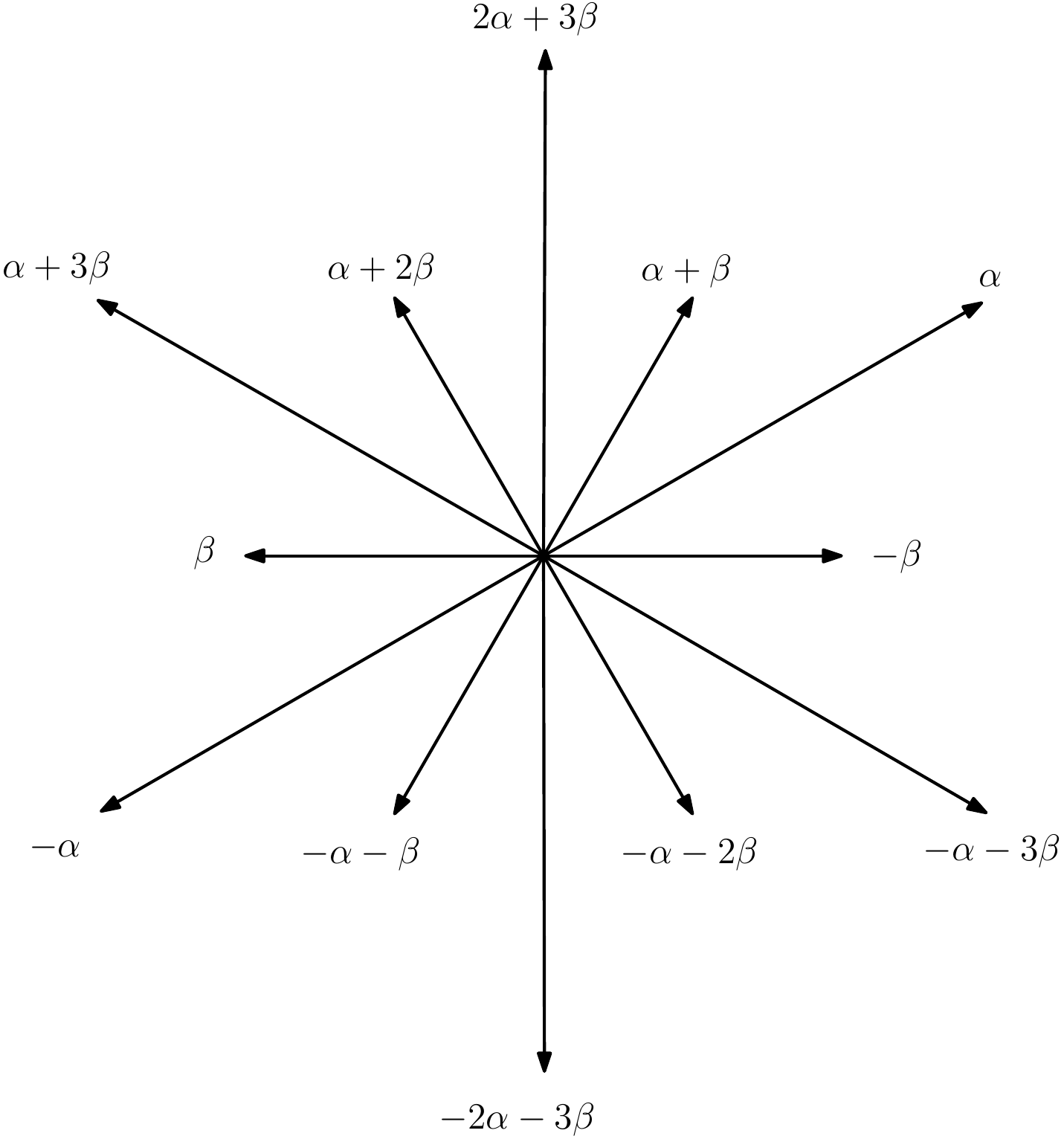}
	\caption{\small \emph{Root space decomposition of $\g_2$.}}
	\label{G2Roots}
\end{figure}

With respect to $\mathfrak{h}, \Delta, \Pi$, we now further choose 
\emph{Chevalley} generators $t_{\delta}, e_{\delta}, e_{-\delta}$ for the Lie algebra, $ \delta \in \Delta^+$.\footnote{As opposed to Kostant 
normalizations whence $\delta(t_{\delta} ) = 1$ instead of $ \delta(t_{\delta} ) = 2$.} These generators $t_{\delta} \in \h, e_{\delta} \in \g_{\delta}$ satisfy following normalizations:
	\begin{align}\label{ChevalleyNormalizations} 
		\begin{cases}
			\; [t_{\delta}, e_{\delta} ] &= 2 e_{\delta} \\
			\; [t_{\delta}, e_{-\delta}] &= -2e_{\delta} \\
			\; [e_{\delta}, e_{-\delta} ] &= t_{\delta}. \\
		\end{cases} 
	\end{align}

\subsection{Chevalley Generators} \label{ChevalleyGenerators} 

Here, we give a complete description of the \emph{Chevalley} generators $e_{\sigma}, e_{-\sigma}, t_{\sigma}$, $\sigma \in \Delta^+$, for our Lie algebra,
satisfying the normalizations \eqref{ChevalleyNormalizations}. \\

Let $E_{ij}$ be the elementary matrix with 1 at entry $(i,j)$, 0's elsewhere and consider the map $\U^k: \C^{7-k} \rightarrow \gl_{7} \C$ by 
$$(a_1, \dots, a_{7-k}) \mapsto \sum_{i=1}^{7-k} a_i E_{k, k+i} .$$ 
That is, $\U^k(a_1, \dots, a_{7-k})$ is the matrix with $a_i$ appearing in order on the $k^{th}$ upper diagonal.
A convenient feature of our matrix model for $\g_2$ is that each root vector $e_{\alpha}$ with $\mathsf{height}(\alpha)=k$ identifies as a matrix $e_{\alpha} \in \mathsf{im}(\U^k)$
Here, we write $\gamma= 2\alpha + 3\beta$ for the highest root. First of all, our principal co-roots are 
$$t_{\beta} = \mathsf{diag}(1,-1, 2, 0, -2, 1, -1) \; \mathrm{and} \; t_{\alpha} = \mathsf{diag}(0, 1, -1,0, 1, -1, 0 ) \; \mathrm{and} \; \; t_{\gamma} = \diag(1,1,0,0,0,-1,-1)$$

The root vectors are 	
\begin{enumerate} 

	\item $e_{\beta} = \begin{pmatrix} 0& 1 & & & & &\\
				      &0 & 0 & & & &\\
			  	      & &0 & \sqrt{-2}& & &\\
			 	      & & &0 & \sqrt{-2} & &\\
				      & & & & 0& 0&\\
				      & & & &&0& 1\\
				      & & & & &&0\\ \end{pmatrix} , \;\;e_{\alpha} = \begin{pmatrix} 0& 0 & & & & &\\
				      &0 & 1 & & & &\\
			  	      & &0 & 0& & &\\
			 	      & & &0 & 0 & &\\
				      & & & & 0& 1&\\
				      & & & &&0& 0\\
				      & & & & &&0\\ \end{pmatrix} $
	\item $e_{\alpha +\beta} = \begin{pmatrix} 0& 0 &1 & & & &\\
				      &0 & 0 & -\sqrt{-2}& & &\\
			  	      & &0 & 0& 0& &\\
			 	      & & &0 & 0 & \sqrt{-2} &\\
				      & & & & 0& 0& -1\\
				      & & & &&0& 0\\
				      & & & & &&0\\ \end{pmatrix} $
	\item $e_{\alpha +2 \beta} = \begin{pmatrix} 0& 0 &0 & \sqrt{-2}& & &\\
				      &0 & 0 &0& 1& &\\
			  	      & &0 & 0& 0& 1&\\
			 	      & & &0 & 0 & 0 & \sqrt{-2}\\
				      & & & & 0& 0& 0\\
				      & & & &&0& 0\\
				      & & & & &&0\\ \end{pmatrix} $
	\item $e_{\alpha +3 \beta} = \begin{pmatrix} 0& 0 &0 & 0& 1& &\\
				      &0 & 0 &0& 0& 0&\\
			  	      & &0 & 0& 0& 0&-1\\
			 	      & & &0 & 0 & 0 & 0\\
				      & & & & 0& 0& 0\\
				      & & & &&0& 0\\
				      & & & & &&0\\ \end{pmatrix} $	
	\item $e_{2\alpha +3 \beta} = \begin{pmatrix} 0& 0 &0 & 0& 0& 1&\\
				      &0 & 0 &0& 0& 0&1\\
			  	      & &0 & 0& 0& 0&0\\
			 	      & & &0 & 0 & 0 & 0\\
				      & & & & 0& 0& 0\\
				      & & & &&0& 0\\
				      & & & & &&0\\ \end{pmatrix} $
\end{enumerate} 

Finally, $e_{-\delta} := \overline{(e_{\delta})}^T$ for each root $\delta \in \Delta^+$.

\subsection{A Principal 3-Dimensional Subalgebra $\fraks$} \label{KostantTheory} 

We now recall the notion of a \emph{principal 3-dimensional subalgebra} \cite{Kos59}. A subalgebra $\mathfrak{s}$ of $\g$ is
called a 3-dimensional subalgebra (3DS) when $\mathfrak{s} \cong \mathfrak{sl}_2\C$. We are interested in a 3DS equipped with distinguished
generators $x, e,\tilde{e}$ that satisfy the normalizations $[x, e] = e, [x, \tilde{e}], = - \tilde{e}, [e, \tilde{e} ] = x$.
In \cite{Kos59}, a distinguished conjugacy class of 3DS is found for any complex simple Lie algebra.
Any member of this conjugacy class is called a \emph{principal 3DS}.  \\

 We now review a recipe of Kostant to construct a convenient P3DS for our purposes. Fix $x \in \h$ to be a \emph{weighting element} distinguished by
  $ [x, e_{\alpha} ] = \mathsf{height}(\alpha) \, e_{\alpha}$ for $e_{\alpha} \in \g_{\alpha}$. 
Re-write $x$ as $ x = q_{\alpha} t_{\alpha} + q_{\beta} t_{\beta}$ for $t_{\alpha}, t_{\beta}$ the co-roots 
 of our principal roots. We then define $e := \sqrt{q_{\alpha}} e_{\alpha} + \sqrt{q_{\beta}} e_{\beta}$
 and $ \tilde{e} := \sqrt{q_{\alpha}} e_{-\alpha} + \sqrt{q_{\beta}} e_{-\beta}$. The Chevalley relations of our root vectors
 guarantee $\fraks$ is a 3DS. Moreover, $q_{\alpha}, q_{\beta} \neq 0$ imply that $\fraks$ is principal (c.f. Corollary 3.7 \& Theorem 5.3 of \cite{Kos59}).
In our matrix model, $x, e, \tilde{e}$ are: 
 \begin{align}\label{Principal3DS} 
	x &= \mathsf{diag}(3, 2, 1, 0, -1, -2, -3) = 3t_{\beta} + 5 t_{\alpha}  \\ 
	e &= \begin{pmatrix} 0& \sqrt{3} & & & & &\\
				      & 0& \sqrt{5} & & & &\\
			  	      & & 0& \sqrt{-6}& & &\\
			 	      & & & 0& \sqrt{-6}& &\\
				      & & & & 0& \sqrt{5}&\\
				      & & & & & 0&\sqrt{3}\\
				      & & & & & &0\\
				     \end{pmatrix} \\     
	\tilde{e} &= \begin{pmatrix} 0&  & & & & &\\
				     \sqrt{3} & 0&  & & & &\\
			  	      & \sqrt{5}&0 & & & &\\
			 	      & & -\sqrt{-6}&0 && &\\
				      & & & -\sqrt{-6}& 0& &\\
				      & & & & \sqrt{5} &0&\\
				      & & & & & \sqrt{3}&0\\
				     \end{pmatrix} 
\end{align}

\subsection{Hitchin's (Cartan) Involution $\sigma$} \label{HitchinTheory} 

We now recall Hitchin's construction of a particular Cartan involution $\sigma$ on a complex simple Lie algebra $\g$. 
The involution $\sigma$ is unique with respect to a fixed P3DS $\fraks= \spann_{\C} <x, e, \tilde{e} >$.
The complex-linear involution $\sigma$ is uniquely constrained by 
\begin{enumerate}
	\item $\sigma(e_i) = - e_i$ for $(e_i)_{i=1}^k$ the centralizer of $e$, where $k = \mathsf{rank} \,\g$.
	\item $\sigma( \tilde{e} ) = -\tilde{e}$. 
\end{enumerate}
Every complex simple Lie algebra $\g$ is given by $\g = \bigoplus_{i=1}^k \mathsf{Sym}^{2m_i}\, V $ as a $\fraks$-module, where $V$ is the standard representation of 
$\SL_2\C$ on $\C^2$ and the positive integers $m_i$ are the \emph{(Kostant) exponents} of $\g$ (Theorem 6.7 \cite{Kos59}).  
Since $\sigma$ preserves Lie bracket, a simple induction shows $\sigma$ is then completely determined by 
$$ \sigma( \, \ad_{\tilde{e}}^{\circ k}(e_i) \, ) = (-1)^{k+1} \ad_{\tilde{e}}^{\circ k}(e_i).$$
Hitchin proved that this involution $\sigma$ is indeed a Lie algebra involution (Proposition 6.1 \cite{Hit92}). 
In $\g_2$, conveniently both Kostant exponents, 1 and 5, are odd. Hence, 
$\g_2 = \fraks \; \oplus \; \mathsf{Sym}^{10} V$. If we decompose $\g_2$ instead into root spaces, then $\sigma$ is equivalently described by
its action on $\h$ and root vectors $e_{\alpha}$:
\begin{align} 
	\sigma( e_{\alpha} )& = (-1)^{\mathsf{height}(\alpha)} e_{\alpha}\\
	 \sigma|_{\h} &= \mathsf{id}_{\h}.
\end{align} 
We emphasize that such a description of $\sigma$, of its action depending only on the height of the root vectors, exists only because all the exponents are odd. \\

We explicitly describe $\sigma$ in this model for $\g_2$. 
Set $$Q = \begin{pmatrix} & & & & & &1\\
					& & & & &1&\\
				   	& & & & 1& & \\
				  	& & & 1& & & \\
				  	& & 1&& & &\\
				   	& 1& & & & & \\
				   	1 & & & & & & \end{pmatrix}$$ and define
				   
\begin{align}
	\sigma'(A) = -Q{A}^TQ.
\end{align} 
Then, in fact, $\sigma = \sigma'$. One can see this by the symmetries of the root vectors. Write $e_{\alpha} = \mathcal{U}^k(a_i)_{i=1}^{7-k}$.
Then by direct inspection, for $k = \mathsf{height}(\alpha)$, reading the root vector backwards, one finds
$$-\mathcal{U}^k(a_i)_{i=7-k}^{1} = (-1)^k e_{\alpha}.$$
On the other hand, $e_{\alpha}^T =\mathcal{U}^{-k}(a_i)_{i=1}^{7-k}$. Hence, by a matrix calculation, we find 
$$ \sigma'( e_{\alpha} ) = -Q(e_{\alpha})^TQ = -Q\mathcal{U}^{-k}(a_i)_{i=1}^{7-k}Q = -\mathcal{U}^k(a_i)_{i= 7-k}^1 =  (-1)^k e_{\alpha}.$$

\subsection{Model Compact \& Split Real Forms} \label{ModelRealForms}

We fix the compact real structure $\rho: \g_2 \rightarrow \g_2$ given by 
\begin{align}\label{ModelCompactRealForm} 
	\rho(A) = - \overline{A}^T.
\end{align} 
An easy calculation shows that $\rho\sigma = \sigma \rho$ so that $\tau = \rho \sigma$ defines another real form. 
By Hitchin, $\tau$ is a split real form (c.f. Proposition 6.1 \cite{Hit92}). Consequently, our model split real structure $\tau$ is given by 
\begin{align}
	 \tau(A) = Q\overline{A}Q.
\end{align} 

The model split real form $\tau$ preserves the distinguished copy of $\R^7$ in $\imoct^\C$ given by $\imoct$. 
We explain this now, by defining another involution $\tau_{V}: \imoct^\C \rightarrow \imoct^\C$ by 
$\tau_V( \mathbf{x} ) = P\overline{\mathbf{x}}$, where we write the vector $\mathbf{x}$ in our given $\imoct^\C$ basis $\mathcal{B}$ from Proposition \ref{PropBaragliaBasis}.
Then note that 
$$ \tau_V(A\mathbf{x}) = \tau(A) \tau_V(\mathbf{x} ),$$
 so that $\tau$ is `multiplicative'. Hence, the copy of $\g_2'$ given by
$\Fix(\tau)$ preserves the subspace $\R^7 = \Fix(\tau_V)$. Indeed, if $A \in \Fix(\tau)$ and $\mathbf{x} \in \Fix(\tau_V)$,
then $Ax \in \Fix(\tau_V)$. \\

Furthermore, the $\R^7$ fixed is $\Fix(\tau) = \imoct $, the fixed points of the complex conjugation $x \mapsto \overline{x}$ on $(\imoct)^\C$. One 
can see this by checking that $\tau$ fixes the real span of the following elements in $(\imoct)^\C:$
$e_0$ and $e_{s} + e_{-s}, \sqrt{-1} (e_s - e_{-s} )$ for $s \in \{1,2,3\}$. Each of these split-octonions
is (up to a real constant) one of our standard multiplication basis $\mathcal{M}$ for $\imoct$ by \eqref{StandardMultiplicationFrame}.
Thus, $\Fix(\tau_V) = \imoct.$

\printbibliography 

\end{document}